\documentclass[a4paper, english, 10pt]{smfart} 

\usepackage{placeins} 

\usepackage{setspace}
\usepackage[british]{babel}

\usepackage[OT2,T1]{fontenc}

\usepackage{amssymb} 
\usepackage{mathrsfs} 

\usepackage{stmaryrd}
\usepackage{amsthm}

\usepackage{amsmath}
\usepackage[latin1]{inputenc}

\usepackage{graphicx}
\usepackage{manfnt} 

 \usepackage{booktabs} 

\usepackage[margin=3cm]{geometry}

\usepackage[all]{xy}

\usepackage{math tools} 

\usepackage{enumitem}

\usepackage[x11names]{xcolor}

\usepackage[pagebackref=true, colorlinks=true, linkcolor=black, citecolor=black, urlcolor=black]{hyperref}
\usepackage{smfthm}

\usepackage[msc-links, alphabetic]{amsrefs}

\DeclareSymbolFont{cyrletters}{OT2}{wncyr}{m}{n}
\DeclareMathSymbol{\Sha}{\mathalpha}{cyrletters}{"58}

\newtheorem{theoA}{Theorem}

\newtheorem{innercustomconj}{Conjecture}
\newenvironment{customconj}[1]
  {\renewcommand\theinnercustomconj{#1}\innercustomconj}
  {\endinnercustomconj}

\newtheorem{innercustomthm}{Theorem}
\newenvironment{customthm}[1]
  {\renewcommand\theinnercustomthm{#1}\innercustomthm}
  {\endinnercustomthm}

\newtheorem*{coro*}{Corollary}
\newtheorem*{conj*}{Conjecture}
\newtheorem*{lemm*}{Lemma}

\providecommand{\twomat}[4]{\left(\begin{array}{cc}#1&#2\\#3&#4\end{array}\right)}
\providecommand{\smalltwomat}[4]{\left(\begin{smallmatrix}#1&#2\\#3&#4\end{smallmatrix}\right)}

\theoremstyle{definition}

\theoremstyle{remark}
\newtheorem{remark*}{Remark}

\NumberTheoremsIn{subsection}

\numberwithin{equation}{subsection}
\numberwithin{table}{subsection}

\newcommand{\ot}{\otimes}

\newcommand{\ts}{\times}
\newcommand{\cd}{\cdot}

\newcommand{\beq}{\begin{equation}\begin{aligned}}
\newcommand{\eeq}{\end{aligned}\end{equation}}

\newcommand{\beqq}{\begin{equation*}\begin{aligned}}
\newcommand{\eeqq}{\end{aligned}\end{equation*}}

\newcommand{\lb}[1]{\label{#1}}

\newcommand{\nek}{Nekov{\'a}{\v{r}}}

\newcommand{\rhobar}{\overline{\rho}}

\newcommand{\one}{\mathbf{1}}
\newcommand{\Q}{\mathbf{Q}}

\newcommand{\GL}{\mathrm{GL}}
\newcommand{\G}{\mathrm{G}}
\renewcommand{\H}{\mathrm{H}}

\newcommand{\X}{\mathscr{X}}

\newcommand{\cE}{\mathscr{E}}

\newcommand{\cH}{\mathscr{H}}

\newcommand{\cV}{\mathscr{V}}

\newcommand{\calP}{\mathscr{P}}
\newcommand{\cP}{\mathscr{P}}

\newcommand{\calL}{\mathscr{L}}

\newcommand{\cW}{\mathscr{W}}
\newcommand{\cF}{\mathscr{F}}

\newcommand{\R}{\mathbf{R}}
\newcommand{\Z}{\mathbf{Z}}

\newcommand{\uw}{\underline{w}}
\newcommand{\ur}{\underline{r}}
\newcommand{\ul}{\underline{l}}

\newcommand{\frakm}{\mathfrak{m}}
\newcommand{\frakp}{\mathfrak{p}}

\newcommand{\ad}{\mathrm{ad}}

\newcommand{\exc}{\mathrm{exc}}

\newcommand{\into}{\hookrightarrow}

\newcommand{\Y}{\mathscr{Y}}

\newcommand{\la}{\langle}
\newcommand{\ra}{\rangle}

\newcommand{\Nm}{\mathrm{N}}

\newcommand{\al}{\alpha}
\newcommand{\tht}{\theta}
\newcommand{\lm}{\lambda}
\newcommand{\Lm}{\Lambda}
\newcommand{\sg}{\sigma}
\newcommand{\Sg}{\Sigma}

\newcommand{\tord}{\gamma_{H'}^{\ord} }
\newcommand{\tordp}{\gamma_{H'}^{\ord} }

\newcommand{\calA}{\mathscr{A}}

\newcommand{\cQ}{\mathscr{Q}}
\newcommand{\cK}{\mathscr{K}}
\newcommand{\calN}{\mathscr{N}}

\newcommand{\cI}{\mathscr{I}}
\newcommand{{\calG}}{\mathscr{G}}

\newcommand{\cR}{\mathscr{R}}

\newcommand{\cM}{\mathscr{M}}

\newcommand{\cD}{\mathscr{D}}
\newcommand{\cU}{\mathscr{U}}

\newcommand{\bC}{\mathbf{C}}

\newcommand{\N}{\mathbf{N}}
\newcommand{\OO}{\mathscr{O}}
\newcommand{\A}{\mathbf{A}}

\newcommand{\bks}{\backslash}
\newcommand{\baar}{\overline}

\newcommand{\eps}{\varepsilon}

\newcommand{\vpi}{\varpi}

\newcommand{\Up}{\mathrm{U}}

\newcommand{\wtil}{\widetilde}
\newcommand{\Res}{\mathrm{Res}}
\newcommand{\B}{\mathbf{B}}

\newcommand{\ord}{\mathrm{ord}} 
\newcommand{\cl}{\mathrm{cl}} 

\newcommand{\vol}{\mathrm{vol}}
\newcommand{\Tr}{\mathrm{Tr}}

\newcommand{\Gal}{\mathrm{Gal}}

\newcommand{\Ker}{\mathrm{Ker}\,}

\newcommand{\Hom}{\mathrm{Hom}\,}
\newcommand{\End}{\mathrm{End}\,}

\newcommand{\llb}{\llbracket}
\newcommand{\rrb}{\rrbracket}

\newcommand{\Spec}{\mathrm{Spec}\,}

\newcommand{\id}{\mathrm{id}}

\newsavebox\tempbox
\let\svwidetilde\widetilde
\renewcommand\widetilde[1]{\sbox\tempbox{$#1$}\svwidetilde{\usebox{\tempbox}}}

   \def\XXint#1#2#3{{\setbox0=\hbox{$#1{#2#3}{\int}$}
        \vcenter{\hbox{$#2#3$}}\kern-.5\wd0}}

\title{The universal $p$-adic Gross--Zagier formula}
\author{Daniel Disegni} 
\address{Department of Mathematics, Ben-Gurion University of the Negev, Be'er Sheva 84105, Israel}
\email{disegni@bgu.ac.il}

\begin{document}

\begin{abstract} Let $\G $ be the group $ (\GL_{2}\ts {\rm GU}(1))/\GL_{1}$ over a totally real field $F$, and let $\X$ be a Hida family for~$ \G$.
Revisiting a construction of Howard and Fouquet, we construct an explicit section $\cP$ of a sheaf of Selmer groups  over $\X$. We show, answering a question of Howard, that $\cP$ is a universal Heegner class, in the sense that it interpolates geometrically defined Heegner classes at all the relevant classical points of $\X$. We also propose a `Bertolini--Darmon' conjecture for the leading term of $\cP$ at classical points.

We then prove that the  $p$-adic height of $\cP$ is given by the cyclotomic derivative of a $p$-adic $L$-function. This formula over $\X$ (which is an identity of functionals on  some universal ordinary  automorphic representations)     specialises at classical points to all the  Gross--Zagier formulas for $\G$ that may be expected from representation-theoretic considerations.

 Combined with a result of Fouquet, the formula implies the $p$-adic analogue of the Beilinson--Bloch--Kato conjecture in analytic rank one, for the selfdual motives attached to Hilbert modular forms and their twists by CM Hecke characters. It also implies one half of the  first example of a non-abelian Iwasawa main conjecture for derivatives, in $2[F:\Q]$ variables. Other applications include two different generic non-vanishing results for Heegner classes and  $p$-adic heights.
\end{abstract} 
\thanks{This work was supported by ISF grant 1963/20 and BSF grant 2018250. During the preparation of a first draft of this paper, the author was supported by a public grant  of the Fondation Math\'ematique Jacques Hadamard.}

\maketitle
\tableofcontents

\section{Introduction and statements of the main results}

A beautiful construction of Heegner and Birch, based on the modularity of elliptic curves and the theory of complex multiplication,  attaches to an elliptic curve $A/\Q$ and an imaginary quadratic field $E$ a point $P\in A(E)$. The work of Gross--Zagier \cite{GZ} related the height of $P$
   to   the derivative of the  $L$-function $L'(A_{E}, 1)$, with striking applications to the Birch and Swinnerton-Dyer conjecture. An analogous result in $p$-adic coefficients was proved  by Perrin-Riou \cite{pr} soon thereafter, if $A$ has good ordinary reduction  at the prime~$p$.   
   
   \medskip

The decade following those works saw a pair of 
   similar results, by \nek\ \cite{nek-heeg} and Zhang \cite{zhang-hcycles},  relating   Heegner cycles on Kuga--Sato varieties to ($p$-adic) $L$-functions of higher-weight modular forms.  
We may single out two  major innovations in the approach to Heegner points and Gross--Zagier formulas since then,\footnote{Two other recent  ideas  that our work does not touch upon are nevertheless too important to be ignored: the conjecture of Darmon and Guitart--Madseu--{\c Seng\"un}  that there should exist Heegner points attached to any quadratic extension of number fields (see \cite{darmon-hhp}, \cite{gms}), and the 
formulas for the $p$-adic \emph{logarithms} of Heegner points of  \cite{bdp, LZZ}.} both answering the question of what `other' Heegner points there are and how they fit together.

 The first one starts from the observation by Mazur \cite{mazur-icm} and Perrin-Riou \cite{PR2} that Heegner points should vary $p$-adically in anticyclotomic families, in the same way that the $L$-function of the elliptic curve $A_{E}$ does; this observation inspired  Howard \cite{howard} 
to prove a generalisation to such families of Perrin-Riou's formula. Howard later significantly expanded the scope of Mazur and Perrin-Riou's  idea by proving that the Kummer classes of Heegner points  also vary in Hida families of modular forms \cite{howbig}; the question of the relation of the resulting `big' classes to Heegner cycles was left open.

The second innovation was the observation  by Gross \cite{gross-msri} that Heegner points can be viewed as elements of spaces of $\H'$-invariant linear functionals on an  automorphic representation of $(\G\ts\H)'$ (these reductive groups will be defined below),\footnote{N.B.: the notation $\G$ used in the informal abstract differs from the notation of the paper.}
 so that the tools of representation theory may be brought in to conceive and prove more  general formulas: a programme whose main achievement, in complex coefficients,  is the work of Yuan--Zhang--Zhang \cite{yzz} on Heegner points on Shimura curves. 

\medskip

In this  work, we combine those two approaches. We construct  Heegner classes for the Galois representation over a Hida family for $(\G\ts\H)'$, show that they specialise to  (cohomological) Heegner cycles at all classical points, and prove a formula for their $p$-adic heights that is universal in the sense that it specialises to all the $p$-adic formulas suggested by the framework of Gross. (The analogous complex Gross--Zagier formulas  are not currently known\footnote{See however the very recent \cite{Qiu}. (Note added during revision.)}  for motives of higher weight.)
 We obtain various applications to the arithmetic of motives attached to Hilbert modular forms.

\medskip

In the rest of this first section we state our main theorems, and complete the discussion of their history.

We begin in \S~\ref{sec: BK} by  presenting  the results concerning the $p$-adic  {Be\u\i linson}--Bloch--Kato conjecture   
(Theorem \ref{BK thm}); they are applications of  the  general $p$-adic Gross--Zagier formula for a fixed representation, stated as 
 Theorem \ref{GZ thm} in \S~\ref{sec: GZ}.

In \S~\ref{sec: uH} we outline the construction and properties of the universal  family of Heegner classes (Theorem \ref{theo heeg univ}), referring to the ``Bertolini--Darmon'' conjecture of \S~\ref{sec bd conj} for a further study of its classical specialisations.  In \S~\ref{sec: ugz}  we state the universal  formula  of the title (Theorem \ref{ugz thm}); a complementary `Waldspurger' analogue will be proved in \S~\ref{sec: wald} (Theorem \ref{wald univ}).
 
  Finally, in \S~\ref{sec: appl} we discuss some further applications:  the first non-abelian example of an Iwasawa main conjecture for  derivatives of $p$-adic $L$-functions (Theorem \ref{iw}); and two results on the generic non-vanishing of $p$-adic heights and Heegner cycles: one for CM motives (Theorem \ref{nv CM}), the other for Hida families containing a rank-$0$ elliptic curve with split multiplicative reduction (Theorem \ref{nv exc}).
   A further application, to a criterion for certain Bloch--Kato Selmer groups to be of rank \emph{zero}, will appear separately.

\subsection{The $p$-adic {Be\u\i linson}--Bloch--Kato conjecture in analytic rank~$1$} \label{sec: BK}
 The primary  motivation for our work comes from  the generalisations of  the Birch and Swinnerton-Dyer (BSD) conjecture and its $p$-adic analogue, as proposed by {Be\u\i linson}, Bloch--Kato, and Perrin-Riou \cite{bei, bk, PRbook}. Recall that if $A/\Q$ is an elliptic curve, (BSD) is equivalent to the following statements. Denote by   $r_{\rm an}$ and $L^{*}(A, 1)$ the order of vanishing and leading term  of $L(A, s)$ at $s=1$. Then 
  $L^{*}(A,1) >0$ and for every prime  $p$,
  \begin{enumerate}[label=(\alph*)]
  \item
 the Selmer group ${\rm Sel}(V_{p}A):=(\varprojlim_{n} {\rm Sel}_{p^{n}}(A)) \otimes_{\Z_{p}} \Q_{p}$ has dimension equal to $r_{\rm an}$;
\item  the divisible part  of $\Sha(A)[p^{\infty}]$ vanishes;
\item the $p$-adic valuations of $L^{*}(A, 1)/\Omega_{A} R_{A}$ and $|\Sha(A)[p^{\infty}]|\prod_{v\nmid \infty}c_{v}(A)$ are equal.
  \end{enumerate}

\subsubsection{Selmer groups according to Bloch--Kato and \nek} If $E$ is a number field and $V$ is a geometric $p$-adic  representation of its Galois group $G_{E}$, Bloch  and Kato \cite{bk} have proposed an analogue 
$$H^{1}_{f}(E,V)$$ of the Selmer group of $A$; it is an $L$-vector-subspace (where $L$ is the field of scalars for $V$) of the first Galois cohomology group of $V$, consisting of those classes satisfying certain local conditions. According to the resulting variant of the  conjecture of  {Be\u\i linson} \cite{bei}, the dimension  
$\dim_{L} H^{1}_{f}(E, V)$ should equal   the    order of vanishing of the $L$-function $L(V^{*}(1), s)$ at $s=0$.\footnote{Provided $V$ contains no copies of the trivial representation. Of course in general the   meromorphic continuation of $L(V,s)$ is itself conjectural. Note that when $V$ is self-dual, or $E$ is a CM field and $V$ is   conjugate-self-dual, we have $L(V,s)=L(V^{*}(1), s)$.}  

Another definition of Selmer groups   was proposed by Greenberg when $V$ satisfies an ordinariness condition at the places above  a prime $p$; specialised to the cases of interest to us,  it recovers the Bloch--Kato Selmer groups. \nek\ observed that a variation of Greenberg's definition works well in $p$-adic families, and developed this observation into the theory of Selmer complexes \cite{nek-selmer},  that provides the foundation for the  present work (\S~\ref{big sec: sel}). For nice $p$-adic families of $G_{E}$-representations, the theory allows to define groups 
$$\wtil{H}^{i}_{f}(E, V).$$
 for all $i$.

\subsubsection{The $p$-adic  {Be\u\i linson}--Bloch--Kato conjecture for Hilbert modular forms}  
Our main arithmetic results concern the $p$-adic analogue
of the {Be\u\i linson}--Bloch--Kato conjecture   for the Galois representations attached to Hilbert modular forms and their twists by Hecke characters of CM fields.

 Fix throughout the rest of this paper a rational prime  $p$. Let $F$ be a totally real field,  let $E$ be  a CM quadratic extension of $F$, and  let
$$\G_{0}:= \Res_{F/\Q}\GL_{2}, \quad \H:= \Res_{E/\Q}\G_{m}.$$
  Let $L$ be a finite extension of $\Q_{p}$ splitting $F$.
A pair of cohomological weights for $\G_{0}$ and $\H$ is a pair of tuples $\uw:=(w; (w_{\sigma})_{\sigma\colon F\into L})$, $\ul=(l;(l_{\sigma})_{\sigma\colon F\into L} )$, each consisting of   $[F:\Q]+1$ integers  of the same parity, such that $w_{\sigma}\geq 2$ for all $\sigma\colon F\into L$.  In this paper we will only consider cohomological weights and therefore omit the adjective `cohomological'. By a ``Hilbert modular form over $L$ of weight $\uw$'' (respectively a  ``Hecke character of $E$ over $L$ of weight $\ul$'') we mean a cuspidal automorphic  representation of $\G_{0}(\A)$ (respectively $\H(\A)$) over $L$ of weight $\uw$ (respectively weight $\ul$) as defined in Definition \ref{aut-def} below.

If $\pi_{0}$ is a Hilbert modular form and $\chi$  a Hecke character over $L$, we denote by $\Pi_{0}=\pi_{0}\otimes \chi$ the associated representation of $\G_{0}\times \H$. We denote by  $V_{\pi_{0}}$ and $V_{\chi}$ 
the corresponding $2$- (respectively $1$-) dimensional representations of $G_{F}$ (respectively $G_{E}$), normalised so that $L(V_{\pi_{0}}, s)= L(s+1/2, {\pi_{0}})$,  and we let 
$$V:= V_{\Pi_{0}}:=V_{\pi_{0}|G_{E}}\otimes V_{\chi}.$$
Let  $\omega_{\pi_{0}}$ be the central character of $\pi_{0}$ and let $\omega_{\chi}:=\chi|_{F_{\A^{\infty}}^{\times}}$. If $\omega_{\pi_{0}}\omega_{\chi}=1$, then $V$  is  conjugate-self-dual  and pure of weight $-1$,  and  the  epsilon factor $\eps(V)\in  \{\pm 1\}$.

Let $\Gamma_{F}:= F_{\A^{\infty}}^{\times}/F^{\times}\hat{\OO}^{p, \ts}_{F}$ (identified with the Galois group of the maximal abelian  extension of $F$ unramified outside $p$ by class field theory), and  let 
$$\cE_{{\rm Z}/L}:=\Spec(\Z_{p}\llb \Gamma_{F}\rrb_{L}).$$ 
(We will also simply write $\cE_{\rm Z}$ for $\cE_{\rm Z/\Q_{p}}$.)
 Suppose that $ \pi_{0}$ is \emph{ordinary}  in the sense of Definition \ref{ord-def}; equivalently, for all $v\vert p$ the associated $G_{F_{v}}$-representation $V_{\pi_{0}, v}$ reduces nontrivially as
$$0\to V_{\pi_{0}, v}^{+}\to V_{\pi_{0, v}}\to V_{\pi_{0}, v}^{-}\to 0,$$
and $G_{F_{v}}$ acts on $V_{\pi_{0},v}^{+}$ by the product of the cyclotomic character $\chi_{\rm cyc}$ and a character $\alpha_{v}^{\circ}$ valued in $p$-adic units.
We may attach to $V$ a meromorphic $p$-adic $L$-function
$$  \calL_{p}(V_{(\pi_{0}, \chi)}, s) \qquad \in \cK(\cE_{{\rm Z}/L})$$
where the variable  $s\in \cE_{{\rm Z}/L}$ may be thought of as a $p$-adic character of $\Gamma_{F}$; we use the synonym  $\chi_{F, s}$ when we want to emphasise such nature of $s$, and we denote by ``$s=0$'' the trivial character $\chi_{F, 0}=\one$.\footnote{Other authors consider $p$-adic $L$-functions of a variable $s'\in\Z_{p}$. In our language this corresponds to restricting $\calL_{p}(V, s)$ along the embedding $\Z_{p}=\Spec \Z_{p}\llb \Z_{p}\rrb_{\Q_{p}}(\Q_{p})\to \cE_{{\rm Z}/L}(\Q_{p})$,  $s'\mapsto \chi_{{\rm cyc}, F}^{s'}$ where $\chi_{{\rm cyc}, F}=\eqref{chi cyc}$ is the cyclotomic character of $F$.}
More precisely, working in terms of the multivariable  function $\calL_{p}(\cV^{\sharp})$ of Theorem \ref{conj Lp} below,   we may define  $\calL_{p}(V_{(\pi, \chi)})$ as the restriction
 \beq \lb{def Lp}
 \calL_{p}(V_{(\pi, \chi)}, s):= \calL_{p}(\cV^{\sharp})(z_{s})
 \eeq
  where  $z_{s}$ corresponds to the family of representations $V_{\pi|G_{E}}\ot\chi \chi_{F, s|G_{E}}$.
   
  If $\eps(V)=-1$, then $\calL_{p}(V_{(\pi_{0}, \chi)}, 0)=0 $ and we denote  by $\calL_{p}'(V_{(\pi_{0}, \chi)}, 0)={\rm d}\calL_{p}((V_{(\pi_{0}, \chi)})(0) \in T_{0}\cE_{{\rm Z}/L}= \Gamma_{F}\hat{\ot}L$  its first derivative.

\begin{theoA}\label{BK thm} Let $\pi_{0}$ be a Hilbert modular form over $L$ of weight $\uw$, and let $\chi$ be a Hecke character of  $E$ over $L$ of weight $\ul$.
Let $V:=V_{\pi_{0}|G_{E}}\otimes V_{\chi}$ .   Suppose that:
\begin{itemize}
\item[$(\textup{wt})$]  $  |  l_{\sg}|<w_{\sg}$ {for all $\sg\colon F\into L$};
\item[$(\textup{sd})$] 
$\omega_{\pi_{0}}\omega_{\chi}=1$ (which implies $w+l=0$);
 \item[$(\eps)$]  $\eps(V)=-1$;
\item[$(\textup{ord})$] $\pi_{0}$ is ordinary;
\item[$(\textup{n-exc})$] $V$ is \emph{not exceptional}: for no place $w\vert v\vert p$  of $E$ is $V_{w}^{-}:=V^{-}_{\pi_{0}, v|G_{E_{w}}}\ot \chi_{w}$ the trivial representation. \end{itemize}
\begin{enumerate}
\item\lb{dimgeq1}
We have
$$\calL_{p}'(V_{(\pi, \chi)}, 0)\neq 0 \   \Longrightarrow  \ \dim_{L} \wtil{H}^{1}_{f}(E, V)\geq 1,$$
and we can exhibit an explicit  nonzero element of $\wtil{H}^{1}_{f}(E, V)  = {H}^{1}_{f}(E, V)$, whose $p$-adic height (cf. Proposition \ref{prop ht})  is also non-zero.
\item \lb{2fou}
Let $T\subset V$ be a stable lattice. If $\calL_{p}'(V_{(\pi, \chi)}, 0)\neq 0$ and moreover the conditions of \cite[Theorem B.(i)]{fouquet} 
 are satisfied, then:
  \begin{enumerate}
\item \lb{2foua}
 we have
\beqq
 \dim_{L} \wtil{H}^{1}_{f}(E, V)=1;
\eeqq
\item 
let $R_{T} \in \OO_{L}\hat{\ot}_{\Z_{p}}\Gamma_{F} $ be the regulator of the height pairing \eqref{ht V}  on $\wtil{H}^{1}_{f}(E, T)\ts\wtil{H}^{1}_{f}(E, T^{*}(1))$. Then 
\beqq
\calL_{p}'(V_{(\pi, \chi)}, 0) \succeq_{\Z_{p}} R_{T}\cdot | \wtil{H}^{2}_{f}(E, T)_{\rm tors}|
\eeqq
in $L\hat{\ot}\Gamma_{F} $.
\end{enumerate}
\end{enumerate}
\end{theoA}

In the last formula we have used the following suggestive notation.
\paragraph{Notation} For a domain $A$ with fraction field $K$  and two $A$-submodules $m_{1}, m_{2}$ of a $K$-vector space $M$   we write $m_{1}\succeq_{A} m_{2}$ if  $m_{1}\subseteq m_{2}$; the notation is extended to the case where some $m_{i}$ is an element of $M$, in which case we interpret it as $Am_{i}$.

 Part 1    will be an immediate consequence of Theorem \ref{GZ thm},  the Jacquet--Langlands correspondence, and the observation following \eqref{eps v} below.
For a list of previous results in the direction of part 1 we refer to the discussion following Theorem~\ref{GZ thm}. Let us  note, for now, that an analogue of this result in complex coefficients is not known.

Part 2 follows from invoking the results of Fouquet in \cite{fouquet}, that generalise the bounds on Tate--Shafarevich groups of elliptic curves obtained by Kolyvagin using the methods of Euler systems.

\begin{rema} Condition $(\textup{n-exc})$ guarantees that $\calL_{p}(V_{(\pi, \chi)}, s)$ has no exceptional zeros at $s=0$, and it is equivalent to the identity $\wtil{H}^{1}_{f}(E, V)= {H}^{1}_{f}(E, V)$.  We will also  equivalently say that $\Pi$ is not exceptional.  For a characterisation of this condition, see Lemma \ref{6444}.
\end{rema} \begin{rema} 
 In the simplest case  where $F=\Q$, $\pi_{0}$ is a modular form with rational Fourier coefficients of weight $w_{\sg}=2$, and   $\chi=\one$,
 the representation $V_{\pi_{0}}=V_{\frakp}A$ is  the rational $p$-adic  Tate module of an elliptic curve $A/\Q$. In this case ${H}^{1}_{f}(E, V)={\rm Sel}(V_{p}A_{E})$,
and letting $T=T_{p}A_{E}$, the group  $\wtil{H}^{2}_{f}(E, T)_{\rm tors}$ equals (\cite[(1.36)]{bu-fl}) the quotient of $\Sha(A_{E})[p^{\infty}]$ by its divisible submodule $\Sha(A_{E})_{p\textup{-div}}$. 

The group $\Sha(A_{E})_{p\textup{-div}}$, conjecturally $0$,  measures the failure of ${\rm Sel}(V_{p}A_{E})$ to be generated by the classes of points in $A(E)$. We do not address in this paper the analogous conjecture from \cite{bk} that ${H}^{1}_{f}(E, V)$ should be generated by the classes of algebraic cycles.
Nevertheless our construction of a generator is sufficiently geometric to provide a good  starting point to establish this conjecture, cf. Remark \ref{rem-P in H1f}.
\end{rema}

\subsubsection{A variant for selfdual Hilbert modular forms} Suppose that $\pi_{0}$ is an ordinary  Hilbert modular form,  $\omega_{\pi_{0}}=1$ (so that $w=0$), and $\eps(V_{\pi_{0}})=-1$. Assume that either $[F:\Q]$ is odd or there is a place $v\nmid p\infty$ of $F$ such that $\pi_{0,v}$ is not a principal series. Suppose that for no $v\vert p$ is $\pi_{0,v}$ the Steinberg representation. Let $L_{p}(V_{\pi_{0}}, s)$ be the $p$-adic $L$-function of $V_{\pi_{0}}$ constructed in \cite{dimitrov}.
 If   $L'_{p}(V_{\pi_{0}},0)\neq 0$, then  the conclusions \eqref{dimgeq1} and  \eqref{2foua} of the previous theorem hold with $(E, V)$ replaced by $(F, V_{\pi_{0}})$. (This is proved by a standard argument based on the choice of a suitable auxiliary $E$ to reduce to the previous theorem.) A similar remark (at least for part  \eqref{dimgeq1}) applies when $\pi_{0}$ has CM by $E$, cf. the proof of Theorem \ref{nv CM} in \S~\ref{sec: nvCM}

\subsubsection{Addendum to the historical overview: higher-rank cases}
The general  overview sketched in our opening page ignored a third important theme:
 Gross's framework has been generalised in \cite{ggp} to study  special cycles attached to other pairs of groups $(\G, \H)$. Several works have explored the  consequences towards the Be\u\i linson--Bloch--Kato conjecture of the possible non-vanishing of those cycles, most notably \cite{five}. On the other hand, non-vanishing \emph{criteria} in terms of  $L$-functions have been obtained in a considerably more limited set of cases, mostly related to  triple-product $L$-functions  \cite{YZZ3, xue, dr, bsv, bcf}.\footnote{In a related context,  see also  the very recent breakthrough of Li--Liu \cite{LL}. (Note added during revision.)}  The relation with cyclotomic $p$-adic $L$-functions has not been studied beyond Heegner cycles.

\subsection{The $p$-adic Gross--Zagier formula for arbitrary weight} \label{sec: GZ}
Theorem \ref{BK thm}, like analogous previous results \cite{pr, nek-heeg, dd-ant, shnidman, dd-pyzz, nonsplit}, is an application of an explicit formula for the $p$-adic heights of a certain Selmer class (here rather a collection of classes). When the weights are \emph{trivial}, that is $\uw=(0; (2,\ldots, 2))$ and $\ul=(0; (0, \ldots, 0))$, this is the class of a \emph{Heegner} $0$-cycle coming from   CM points on quaternionic   Shimura curves; this is the case studied in \cite{dd-pyzz, nonsplit}, and earlier  in complex coefficients by Yuan--Zhang--Zhang \cite{yzz}.
 In general, it is the class of a   $0$-cycle supported at CM points,  with coefficients in a local system corresponding to the weight of the representation. The specific choice of the (tower of) Shimura curves is dictated by the local root numbers of $V$, see the discussion preceding Definition \ref{loc dist}.

\subsubsection{Algebraic groups and Shimura varieties} 
Let $\B$ be a quaternion algebra over $F_{\A}$ (where $\A$ denotes the ad\`eles of $\Q$) with ramification set $\Sigma\sqcup \{v\vert \infty\} $ satisfying $|\Sigma| \equiv [F:\Q] -1\pmod{2}$.  Then $\G(\A):= \B^{\ts}$ is   \emph{not} the points of an algebraic group `$\G$' over $\Q$, but we will still find convenient to   use this  suggestive notation and   refer to  $\G$ as an \emph{incoherent} algebraic group over $\Q$ (see \S~\ref{ssec incoh} for a more formal treatment).
 Let $\H=\Res_{E/\Q}{\bf G}_{m}$ as above, and let  ${\rm Z}:= \Res_{F/\Q}{\bf G}_{m}$, that admits natural central embeddings in $\G$ and $\H$.

The list of (coherent or incoherent) groups of interest in this paper, often denoted collectively by $\G_{*}$, is
\beq \label{list}
\G, \qquad  \H, \qquad \G\times \H, \qquad (\G\times \H)':= (\G\times \H)/{\rm Z}, \qquad\H':=\H/{\rm Z},
\eeq
where ${\rm Z}$ is embedded diagonally in the product group.\footnote{In fact, the (incoherent) group that truly underlies our constructions is $(\G\ts_{\rm Z}\H)' = \{(g, h)\, | \, \nu_{\G}(g)=\nu_{\H}(h)\}$ (where $\nu_{?}\colon ?\to {\rm Z}$  arises from  the reduced norm map of $\B$ (for $?=\G$) or from the norm of $E/F$ (for $?=\H$). That is, the universal Heegner class and the other associated objects described below descend to the ordinary eigenvariety for $(\G\ts_{\rm Z}\H)'$ (a quotient of the one for $(\G\ts\H)'$). Nevertheless, for the sake of simplicity we will content ourselves with working with $(\G\ts\H)'$.} 
 We suppose that for every $v\in \Sg$, $E_{v}/F_{v}$ is nonsplit. Then  there is unique $\B^{\times}$-conjugacy class of $F_{\A}$-embeddings $E_{\A}\into \B$, of which we fix one. It induces an embedding ${\rm e}\colon \H\into \G$.

To the above groups and suitable Shimura data (\S~\ref{ssec shi}), we associate
corresponding towers of   compactified Shimura varieties $X_{*}$, respectively denoted 
\beq \lb{shlist}X_{/F}, \qquad Y_{/E}, \qquad X\times_{F} Y_{/E}, \qquad Z_{/E}, \qquad Y'_{/E}.
\eeq
  They are curves except for $Y$, $Y'$ that have dimension~$0$. The embedding ${\rm e}$ induces a diagonal embedding $\H'\into (\G\times \H)'$, hence a morphism of Shimura varieties 
$${\rm e}' \colon  Y' \to Z.$$

\subsubsection{$p$-adic automorphic representations}
It is  more natural to parametrise ``cohomological automorphic representations over a $p$-adic field $L$'' of a group $\G_{*}$ by irreducible  algebraic representations $W$  of $\G_{*}$.\footnote{See Definition \ref{aut-def}: the  $W/L$ of interest to us are in bijection with (finite) $G_{L}$-orbits of cohomological `numerical'  weights as defined above. From now on all numerical or representation-theoretic weights will be tacitly understood to be cohomological.} 

Let $G_{*, \infty}$ be $\G_{*}(\Q_{p})$ with the Zariski topology (and for later purposes let $G_{*, p}:=\G_{*}(\Q_{p})$ with the $p$-adic topology, $G_{*}:=G_{* , p} \ts G_{*,\infty}$).  We \emph{redefine} throughout this work
$$\G_{*}(\A):= \G_{*}(\A^{\infty})\ts G_{*, \infty}.$$
 Let $W$ be an (algebraic) representation of $G_{*, \infty}$ over $L$, and let $\cW $ be the corresponding  \'etale local system on the tower $X_{*}$. Then we define  a (cuspidal, cohomological)  \emph{automorphic representation of $\G_{*}(\A)$ over $L$ of weight $W$} to be a representation  $$\Pi=\Pi^{\infty}\ot W$$
of $\G_{*}(\A)$  occurring in $H^{\bullet}(X_{*, \baar{E}}, \cW^{\vee})\ot W$.\footnote{This approach is inspired by the work of Emerton \cite{emerton}.}
  (Here and in the rest of the paper, groups and Hecke algebras act on Shimura varieties and their homology on the \emph{right}, on cohomology and  on automorphic forms on the \emph{left}. Left and right algebraic representations $W$ are identified via $w.g:=g^{-1}.w$.)

\subsubsection{Automorphic and Galois representations}
Let $\Pi=\pi\otimes \chi$ be a cuspidal automorphic representation of $(\G\times \H)'(\A^{})$ over $L$ of weight $W=W_{\G}\ot W_{\H}$.
    Let $V=V_{\Pi}=V_{\pi|G_{E}}\otimes V_{\chi}$ be the associated $G_{E}$-representation.

For a smooth proper variety $Z'$ of dimension $d$ over a characteristic-zero field $F'$ and a $p$-adic local system $\cW'$, define $$H_{i}(Z', \cW') := H^{2d-i}_{\text{\'et}}(Z', \cW'(d))$$
 for all $0\leq i\leq 2d$.       For each level $K\subset (\G\ts\H)'(\A^{\infty}) $, let  $\baar{Z}_{K}:=Z_{K}\ts_{\Spec E}\Spec{\baar E}$.  
 We use the notation
$$\H_{i}(\baar{Z}_{K}, \cW):=H_{i}(\baar{Z}_{K}, \cW)\ot W^{\vee}$$
and similarly for the other Shimura varieties over $F$, $E$, $\baar{E}$ under consideration.
Thanks to work of Carayol we can construct an injection (an isomorphism unless $V$ is decomposable) of $(\G\ts\H)'(\A^{})$-representations
\beq
\label{cara Z}
\Pi \into \varinjlim_{K} \Hom_{L[G_{E}]}(\H_{1}(\baar{Z}_{K}, \cW), V_{\Pi}) .
\eeq

\subsubsection{Heegner cycles}   
 Suppose that $W$ satisfies (wt), then $W^{H'_{\infty}}\cong W_{H'_{\infty}}$ is $1$-dimensional, and ${\rm e'}$ induces   a canonical system of maps  
 $$\H_{0}(Y'_{V'}, L)\to \H_{0}(Z_{K}, \cW)$$
  for all $V'\subset \H'(\A^{\infty})\cap K$. The  image $\Delta^{\circ}_{W, f_{\infty}}\in  \H_{0}(Z_{K}, \cW)$  of the normalised fundamental class 
$$[Y'_{V'}]= |Y'(\baar{E})|^{-1}\cdot \sum_{y\in Y'(\baar{E})} [y] \quad \in \varprojlim_{V'} \H_{0}(Y'_{V}, L)$$
is well-defined and (after a modification if $W$ is trivial) belongs to the kernel  $\H_{0}(Z_{K}, \cW)_{0}$ of 
$  \H_{0}(Z_{K}, \cW)\to \H_{0}(\baar{Z}_{K}, \cW)$.
 The images of $\Delta^{\circ}_{W, f_{\infty}}$ under the Abel--Jacobi maps ${\rm AJ}\colon \H_{0}(Z_{K}, \cW)_{0} \to H^{1}(E, \H_{1}( \baar{Z}_{K}, \cW)) $ are compatible under pushforward along the tower $Z_{K}$ and invariant under the $\H'(\A^{})$-action,  hence they define an element 
 $$P_{W}:= \lim {\rm AJ}(\Delta_{W, -}^{\circ})\in \varprojlim_{K} H^{1}(E, \H_{1}(\baar{Z}_{K}, \cW)^{H'(\A^{})})$$
Via  \eqref{cara Z},  $P_{W}$ defines  an $\H'(\A)$-invariant functional
\beq \lb{cP Pi 0} 
P_{\Pi}\colon \Pi_{H'_{\infty}}
 \to H^{1}(G_{E}, V_{\Pi}),
  \eeq
whose  image should   lie in $H^{1}_{f}(E, V_{\Pi})\subset H^{1}(E, V_{\Pi})$ (see Remark \ref{rem-P in H1f} for a stronger conjecture). We show in Proposition \ref{in H1f} that this is the case if  $\B_{p}$ is split and $\Pi$ is \emph{ordinary} and not exceptional, which we define to mean that $\B_{p}$  is split and the Jacquet--Langlands transfer $\Pi_{0}$ of $\Pi$  to $\G_{0}\ts \H$ (which  is thus the `identity' at $p$) satisfies those properties. 

Our formula will give a criterion for the nonvanishing of  $P_{\Pi}$.

\subsubsection{Multiplicity one}\lb{sec m1} Representation theory provides a necessary condition.  The space 
$$(\Pi)^{*,\H'(\A)}=\Hom_{\H'(\A)} (\Pi, L)$$
 is known, by a theorem of Waldspurger,  Tunnell, and H. Saito \cite{tunnell, saito-h}, to be nonzero if and  only if 
the following condition is satisfied for all $v$:
\begin{itemize}
\item[($\eps_{v}$)]  Define $ \eps(\B_{v}):=+1 $ (respectively $-1$) if $\B_{v}$ is split (respectively nonsplit).  Let
$\eps(V_{v}):= \prod_{w\vert v}\eps(V_{w})$, $\chi_{v}(-1):=\prod_{w\vert v}\chi_{w}(-1)$; then 
\beq \label{eps v}
\eps_{v}^{\G}(V):= \eps(V_{v})\chi_{v}(-1)\eta_{v}(-1)\eps(\B_{v}) =+1.
\eeq
\end{itemize}
If this is the  case, $(\Pi)^{*,\H'(\A)}$ is $1$-dimensional and moreover the global root number $\eps(V)=-1$.  Conversely, if $V$ is as in Theorem \ref{BK thm} and in particular  satisfies $\eps(V)=-1$, there exists a unique incoherent totally definite quaternion algebra $\B$ verifying $(\eps_{v})$.
 
 The conditions ($\eps_{v}$) for a finite $v$ generalise the classical ``Heegner condition''. For $v\vert p$, if $\pi$ is ordinary the condition $(\eps_{v})$  is satisfied  unless   $v$ is nonsplit in $E$ and $\pi$ is  exceptional at $v$ (Lemma \ref{6444}). 
The condition ($\eps_{\infty}$)   is equivalent to (wt). 
 \begin{defi}\lb{loc dist} 
 We say that $\Pi$ is \emph{locally distinguished by $\H'$}, or simply \emph{locally distinguished}, if it satisfies conditions $(\eps_{v}) $ for all $v$.
 \end{defi}

\subsubsection{Local toric periods}\lb{ssec loc tp} Assume that $\Pi$ is locally distinguished, and let $\Pi^{\vee}$ denote the contragredient representation of $\Pi$.  Then we know an explicit a generator of 
\beq\label{pi*h}  
\Pi^{*,\H'(\A)}\ot (\Pi^{\vee})^{*,\H'(\A)}
\eeq
as a product of local pairings, which we now define. 
The  pair $P_{\Pi}\ot P_{\Pi^{\vee}}$ will be measured against this generator.

For $v$  a finite  place of $F$, let $\Pi_{v}$ be the local component of $\Pi$, a representation of $(\B_{v}^{\ts}\ts E^{\ts}_{v})/F_{v}^{\ts}\supset H'_{v}:= E_{v}^{\ts}/F_{v}^{\ts}$; let   $dt_{v}$ be a Haar measure on $H_{v}'$.
For $v=\infty$, let $\Pi_{\infty}=W$ and let $dt_{\infty}$ be a formal symbol  synonymous with a constant $\vol(H'_{\infty}, dt_{\infty})\in L$. In all cases, let $\Pi_{v}^{*, H'_{v}}:=\Hom_{H_{v}'}(\Pi_{v}, L)$ and let $(\ , \ )_{v}$ be an invariant pairing on $\Pi_{v}\ot \Pi_{v}^{\vee}$. 

Let $V_{v}$ (respectively $V_{\pi, v}$) be the  restriction to $G_{E_{v}}:=\prod_{w\vert v}G_{E,w}$ (respectively $G_{F_{v}}$) of the Galois representation 
 associated with $\Pi$ (respectively  $\pi$)
   if $v$ is finite, and the Hodge structure associated with $W$ (reps. $W_{\G}$) if $v=\infty$. Let us also introduce the convenient notation 
$$\text{``$V_{(\pi, \chi), v}:= (V_{\pi, v}\ot{\rm Ind}_{F_{v}}^{E_{v}}\chi_{v})\ominus \ad(V_{\pi})(1)$''}$$
(to be thought of as referring to a `virtual motive').

Let  $\eta\colon F^{\ts}_{\A}/F^{\ts}\to \{\pm 1\}$ be the character associated with $E/F$, and let
\beq\lb{calLv}
\calL(V_{(\pi, \chi), v},0):= {\zeta_{F,v}(2) L(V_{v}, 0) \over {L(1, \eta_{v})L(\ad(V_{\pi, v}) , 1)}}
\cdot  \begin{cases} 1 \quad &\text{if $v$ is finite}\\ \pi^{-[F:\Q]}  &\text{if
 $v=\infty$}\end{cases}
\quad \in L.
\eeq
Then 
\beqq
 \label{Qpair intro} 
 Q_{v,(, )_{v}, dt_{v}}(f_{1, v}, f_{2,v})&:=
\calL(V_{(\pi, \chi),v},0)^{-1}\
 \int_{H_{v}'} (\Pi_{v}(t)f_{1,v}, f_{2, v})_{v}\, dt_{v} 
\eeqq
is an explicit generator of $\Pi_{v}^{*, H'_{v}}\ot_{L} (\Pi_{v}^{\vee})^{*, H'_{v}}$. Here for $v\nmid \infty$ the integral is absolutely convergent (after making any choice of $L\into\bC$), and  for $v =\infty$ we understand $$\int_{H'_{\infty}}\Pi_{\infty}(t)dt_{\infty}:={\vol(H'_{\infty}, dt_{\infty})} \cdot {\rm p}_{H'_{\infty}} \colon W\to W_{H'_{\infty}} =W^{H'_{\infty}},$$ where ${\rm p}_{H'_{\infty}}$ is  the natural projection.

Given $f_{3,v}, f_{4,v}\in \Pi_{v}\ot\Pi_{v}^{\vee}$ such that $(f_{3,v}, f_{4,v})_{v}\neq 0$, the quantity
\beq\label{Qv intro} 
Q^{}_{v, dt_{v}}\left( {  f_{1, v} \ot  f_{2, v}\over f_{3,v}\ot f_{4, v}}\right) := {Q_{v,( , )_{v}, dt_{v}}(f_{1,v}, f_{2,v})\over ( f_{3,v}, f_{4,v})_{v}}
\eeq
 is independent of the choice of $(\ ,  \ )_{v}$; it equals $\vol(\OO_{E,v}^{\ts}/\OO_{F, v}^{\ts}, dt_{v})$ if all the data are unramified.

Fix a choice of   measures $dt_{v}$ such that for $dt=\prod_{v }dt_{v}$,
\beq \lb{dt norm}\textstyle\vol(\H'(\Q)\bks \H'(\A), dt):= \vol(\H'(\Q)\bks \H'(\A^{\infty}), \prod_{v\nmid \infty} dt_{v}) \cdot \vol(H'_{\infty}, dt_{\infty})=1.\eeq
Then we define  for $f_{1}\in \Pi_{H'_{}}$, $f_{2}\in \Pi^{\vee}_{H'_{}}$, $f_{3}\in \Pi$, $ f_{4}\in \Pi^{\vee}$
such that  $\prod_{v}(f_{3,v}, f_{4,v})_{v}\neq 0$:
\beqq 
Q \left( {  f_{1} \ot  f_{2}\over f_{3}\ot f_{4}}\right)
&:= 
\prod_{v}Q_{v, dt_{v}}\left( {  f_{1, v} \ot  f_{2, v}\over f_{3,v}\ot f_{4, v}}\right).
\eeqq

\subsubsection{Global pairings and $p$-adic heights} Let $V^{\iota}:=V_{\Pi^{\vee}}$. Fix a Galois-equivariant  pairing 
\beq\label{pair V}
V\ot V^{\iota }\to L(1).\eeq
 Poincar\'e duality provides a canonical Galois-  and Hecke- equivariant pairing 
 $ \H_{1}(\baar{Z}_{K},\cW) \ot  \H_{1}(\baar{Z}_{K}, {\cW}^{\vee})\to L(1)$.
  Via \eqref{cara Z} and \eqref{pair V}, it induces  dual  pairings $(\ , \ )_{\Pi}^{K}\colon \Pi^{K}  \ot \Pi^{\vee, K}
  \to L $ for all $K$. Letting $L_{K}$ be the Hodge bundle on $Z_{K}$, the following pairing (\eqref{pair pi} in the text) is well defined:
  $$(\ , \ )_{\Pi}:= \lim_{K} \ (\dim W\cdot \deg(L_{K}))^{-1}\cdot (\ ,\ )_{\Pi^{K}} \colon \Pi
  \ot \Pi^{\vee}
  \to L .$$

On the other hand, if $\pi$ is ordinary  the restriction $V_{w}$ of $V$ to $G_{E_{w}}$, $w\vert p$,  is reducible
\beq\label{ord w}
0\to V^{+}_{w}\to V_{w}\to V_{w}^{-}\to 0,
\eeq
and there is an analogous reduction  for $V^{\iota}$ such that $V^{+}_{w}$ and $V^{\iota, +}_{w}$ are exact orthogonal of each other under \eqref{pair V}.  These data
allow 
  to define a height  pairing
\beq \label{ht V}
 h_{V}\colon \wtil{H}^{1}_{f}(E, V)\ot \wtil{H}^{1}_{f}(E, V^{\iota})\to L\hat{\ot}\Gamma_{F} 
 \eeq
on \nek's Selmer groups as in Proposition \ref{prop ht}.
When  $W$ is trivial,  the representation  $V=V_{\frakp}A_{E}\ot\chi$ is a factor of the Tate module of an abelian variety, and (under   (n-exc))  the pairing $h_{V}$ coincides with all other $p$-adic height pairings on abelian varieties defined in the literature: see \cite{dd-pyzz} for a review.
\subsubsection{The formula}   We can now state the $p$-adic Gross--Zagier formula for $V$. 
\begin{theoA}\label{GZ thm} Let $\Pi=\pi\ot \chi$ be an ordinary, locally distinguished, non-exceptional automorphic  representation of $(\G\ts\H)'(\A^{})$ over $L$.
 Let $V=V_{\Pi}$.

  The image of $P_{\Pi}$ lies in $H^{1}_{f}(E, V)$, and 
for all $f_{1}\in \Pi_{H'_{\infty}}$, $f_{2}\in\Pi^{\vee}_{H'_{{\infty}}}$, $f_{3}\in \Pi$, $f_{4}\in\Pi^{\vee}$
such that $(f_{3}, f_{4})_{\Pi}\neq 0$, we have
$$   {  h_{V}(P_{\Pi}(f_{1}), P_{\Pi^{\vee}}(f_{2}))    \over   (f_{3}, f_{4})_{\Pi} }  
= 
e_{p\infty}(V_{(\pi, \chi)})^{-1}
 \cdot \calL_{p}'(V_{(\pi, \chi)}, 0)
\cdot
Q\left( {  f_{1} \ot  f_{2}\over f_{3}\ot f_{4}}\right),
$$
where $e_{p\infty}(V_{(\pi, \chi)}) \in L^{\ts}$ is  the $p$-interpolation factor  for  $\calL_{p}(V_{(\pi ,\chi)}, s)$ defined in \eqref{epinf L} below. 
\end{theoA} 
 When $\G=\GL_{2/\Q}$, $V$ is crystalline at $p$, $p$ splits in $E$,   $\chi$ is unramified  and the $f_{i}$  are newforms, a version of this result was proved by Perrin-Riou \cite{pr} when $W$ is trivial, and by   \nek\ \cite{nek-heeg} and Shnidman \cite{shnidman} when $W$ has even weights. The general case with trivial $W$ was proved in \cite{dd-pyzz, nonsplit}. 
\begin{rema} Establishing Gross--Zagier formulas in this generality has proven useful for arithmetic applications, such as those in  \cite{tian, bur-d, BuTian} and Theorem \ref{nv CM} below.

Explicit versions of the formula can be obtained by evaluating the functional $Q$  at well-chosen $f_{i}$. This is a local problem, solved in  \cite{cst}.
\end{rema}

\begin{rema} For a variant of Theorem \ref{GZ thm} that is valid  in the exceptional case as well, see Theorem \ref{GZ thm'}. That variant is often trivially  $0=0$ in the exceptional case, but not always, and indeed Remark \ref{rmk GS} sketches a new proof of the Greenberg--Stevens theorem \cite{GS} based on it. For a further discussion going beyond any trivial or non-trivial vanishing, see Remark \ref{go to bd} and \S~\ref{sec bd conj}.
\end{rema}

\subsection{The universal Heegner classes}\label{sec: uH} We explain the interpolation of the Heegner cycles $P_{\Pi}$ as $\Pi$ varies over a Hida family for $(\G\ts \H)'$.

Suppose from now on that $\B_{p}$ is split and  fix an isomorphism $\G_{\Q_{p}}\cong \Res_{F_{p}/\Q_{p}}\GL_{2}$, giving a model of $\G$ (hence $(\G\ts\H)'$) over $\Z_{(p)}$. We let $N_{\G, 0}:=\smalltwomat 1{\OO_{F, p}}{}1\subset \G(\Q_{p})$ and $N_{0}$ be the image of $N_{\G, 0}$ in $(G\ts H)'_{p}$. Finally we denote by $\Up_{p}$ the usual operator in the Iwahori--Hecke algebra of $(G\ts H)'_{p}$, and by $\Up_{p\infty}$ its product with $ (\smalltwomat p{}{}1, 1)\in (G\ts H)'_{\infty}$.

For a localisation ${\rm M}$ of a finite $\Z_{p}$-module ${\rm M}^{\circ}$ on which the operator $\Up_{p\infty}$ acts (on the left or the right), we denote by ${\rm M}^{\ord}$ the image of ${\rm M}$ under  Hida's ordinary projector 
$$e^{\ord}=\lim \Up_{p\infty}^{n!}.$$

\subsubsection{Hida families for $(\G\ts\H)'$}  
Pick an arbitrary  $(\G\ts\H)'(\Z_{p})$-stable lattice $W^{\circ}\subset W$, yieding a sub-local system $\cW^{\circ}\subset \cW$.
  Then we define, for any $K=K^{p}K_{p}$ with $K_{p}\supset N_{0}$,
\beq\label{MWK}
M_{W,K}^{\circ}:= (H^{1}(\baar{Z}_{K}, \cW^{\circ})\ot (W^{\circ, \vee})_{N_{0}})^{\ord},  \qquad M_{W,K}:= M_{W, K}^{\circ}\ot_{\OO_{L}}L.
\eeq
    
Let $K^{p}\subset (\G\ts\H)'(\A^{p\infty})$ be an open compact subgroup. Consider the ordinary completed homology of $\baar{Z}_{K^{p}}$
$$M_{K^{p}}:=( \varprojlim_{K_{p}\supset N_{0}} M_{K^{p}K_{p}}^{\circ})\ot_{\Z_{p}}\Q_{p},$$
where $ M_{K}^{\circ}= $\eqref{MWK} with $W$ the trivial representation, and the limit is over $K$ such that $K_{p}\supset N_{0}$ (``level $\Gamma_{1}^{1}(p^{\infty})$''). By the work of Hida, $M_{K^{p}}$ is a finite flat module over a certain weight algebra $\Lm=\Lm_{K^{p}} \simeq \Q_{p}[\Delta]\ot_{\Z_{p}} \Z_{p}\llb T_{1}, \ldots , T_{2[F:\Q]+1+\delta_{F ,p}}\rrb$ where $\Delta $ is a   finite group and $\delta_{F, p}$ is the Leopoldt defect of $F$. 

Let ${\bf T}^{\rm sph, \ord}_{K^{p}, \Q_{p}}\subset \End_{\Lm}(M_{K^{p}})$ be the image of the algebra generated by the  spherical Hecke operators and the operators $\Up_{v}$, $v\vert p$. The `ordinary eigenvariety' 
$$\cE^{\ord}=\cE^{\ord}_{K^{p}}:=  \Spec {\bf T}^{\rm sph, \ord}_{K^{p}, \Q_{p}}$$ contains a dense subset  $\cE^{\ord, \cl}$ (more precisely a reduced $0$-dimensional ind-subscheme) of regular points,  in bijection with the   set of $G_{\Q_{p}}$-orbits of those  ordinary automorphic representations $\Pi$ of $(\G\ts\H)'$ over $\Q_{p}$ such that  $\Pi^{K^{p}}\neq 0$.

Let us  fix an irreducible component 
$$\X\subset\cE^{\ord}_{K^{p}}$$
that is a  \emph{Hida family} for $(\G\ts\H)'$. We let  $\X^{\cl}:=\X\cap\cE^{\ord, \cl}  $.

 \begin{defi}\lb{hida loc dist} 
A {Hida family} $\X$ for $(\G\ts \H)'$ 
is said to be \emph{locally distinguished} (by $\H'$) if it satisfies  the conditions  
\begin{itemize}
\item[$(\eps_{v})'$]  for every (equivalently,\footnote{By \cite[Corollary 5.3.3]{LLC}.} one) classical point $z\in \X$ (of weight satisfying (wt)), the Galois representation $\cV_{z}$ attached to the representation $\Pi_{z}$ satisfies ($\eps_{v}$)
\end{itemize}
for all $v\nmid p\infty$.
\end{defi}

\subsubsection{Sheaves on $\X$} 
 The Hida family $\X$ comes with a coherent sheaf   $\cM_{K^{p}}$ corresponding to $M_{K^{p}}$; moreover in fact for each $K^{p}{}'\subset K^{p}$ the module $M_{K^{p}{}'}$ gives rise to a coherent $\OO_{\X}$-module
$$\cM_{K^{p}{}'}$$
 with $\OO_{\X}$-linear Hecke- and Galois actions. Fix an arbitrary  $K^{p}{}'\subset K^{p}$, `sufficiently large' at the places in $\Sg$.\footnote{In the sense that for each $z\in \X^{\cl}$,  $v\in\Sg$, the finite-dimensional constituent  $\Pi_{z,v}$ of $\Pi_{z}$ is fixed by $K_{v}$.}  Let $S$ be a finite set of primes, not containing those above $p$, such that all data $\G, \H,K^{p}{}'$ are unramified outside $Sp$.
 Let $G_{E, Sp}$ be the Galois group of the maximal extension of $E$ unramified outside $Sp$.
 We prove in the text that the following statements are true up to  replacing $\X$ by an open subset  containing $\X^{\cl}$: 
\begin{itemize}
\item there exists a locally free sheaf $\cV$ of rank $2$ with a $G_{E, Sp}$-action, such that for all $z\in \X^{\cl}$, the representation $\cV_{|z}$ is associated with $\Pi_{z}$ via  the  Langlands correspondence;
\item for each $w\vert p$ there is an exact sequence  of  $\OO_{\X}[G_{E,w}]$-modules 
\beq\label{ord fam}
 0\to \cV_{w}^{+}\to \cV_{w}\to \cV^{-}_{w}\to 0,\eeq 
where the $\cV_{w}^{\pm} $ are line bundles  over $\X$, specialising to  \eqref{ord w} at all $z\in \X^{\cl}$;
\item assume from now on that $\X$ is locally distinguished. There is a locally free $\OO_{\X}$-module 
$$\Pi^{K^{p}{}', \ord}_{H'_{\Sg}}$$ 
interpolating the spaces of $(E_{\Sg}^{\ts}/F_{\Sg}^{\ts})$-coinvariants,
 $K^{p}{}'$-invariants of $\Pi_{z}^{\ord}$ for $z\in \X^{\cl}$; 
\item we have a map of Hecke modules over $\OO_{\X}$
\beq \label{cara hida}
 \Pi^{K^{p}{}', \ord}_{H'_{\Sg}}\to \Hom_{\OO_{\X}[G_{E, Sp}]}(\cM_{K^{p}{}'}^{H_{\Sg}'}, \cV)
\eeq
whose specialisations over $\X^{\cl}$ are deduced by \eqref{cara Z}.
\end{itemize}

\subsubsection{The universal Heegner class} 
We construct in the appendix (Proposition \ref{gHo})  an  operator $\gamma_{H'}^{\ord}$, that is the key to the interpolation of Heegner cycles.  It is a limit of    of Hecke operators at $p\infty$, 
 intertwining toric and ordinary  parts:
\beqq\lb{op gamma} {\H}_{1}(\baar{Z}_{K^{p}}, \cW)^{H_{}'} &\stackrel{\cdot\tordp}{\longrightarrow} \H_{1}(\baar{Z}_{K^{p}}, \cW)^{\ord} =M_{W, K^{p}}^{\ord}\\
\Pi^{{K^{p}}}_{H_{}'} &\stackrel{\tordp\cdot }{\longleftarrow} \Pi^{K^{p},\ord}.
\eeqq

Consider 
 the class
$$P^{\ord}_{W, K^{p}{}'} := P_{W,K^{p}{}'} \ \gamma_{H'}^{\ord} 
\in  H^{1}(G_{E, Sp},   M_{W, K^{p}{}'}).$$
It is   invariant under $\H'(\A^{p\infty})$, hence:
\begin{itemize}
\item as $K^{p}{}'$ varies, it defines an $\H'(\A^{p\infty})$-invariant functional
\beq \lb{cP Pi}
P_{\Pi}^{\ord}= P_{W}\circ \tord\colon \Pi^{\ord}\to H^{1}(E, V_{\Pi})
\eeq
and in fact, as we shall prove, valued in $H^{1}_{f}(E, V_{\Pi})$. 
\item
restricting  (without loss of generality as we will see in a moment) to the case where $W$ is trivial,  its localisation over 
$\X$ defines a global section $\cP_{K^{p}{}'}$ of
$  H^{1}(G_{E, Sp},\cM_{K^{p}{}'}^{H_{\Sg}'})$. 
\end{itemize}

Using \nek's theory of Selmer complexes  we show that
 the universal class $\cP_{K^{p}{}'}$ is a section of a sheaf of Selmer groups   $\wtil{H}^{1}_{f}({E},\cM_{K^{p}{}'}^{H_{\Sg}'})$, where the subscript $f$ signifies  a local condition at $p$ coming from \eqref{ord fam}, and for Selmer groups we use $E$ in place of $G_{E, Sp}$ for short. Then by \eqref{cara hida} the class $\cP_{K^{p}{}'}$   defines a map of $\OO_{\X}$-modules
$$\cP_{K^{p}{}'}\colon \Pi^{K^{p}{}', \ord}_{H'_{\Sg}} \to \wtil{H}^{1}_{f}({E}, \cV).$$

When $\G=\GL_{2/\Q}$, the value of $\cP_{K^{p}{}'}$ on a family of newforms is the class  originally defined by Howard in \cite{howbig}. (The statement  that  the fibre of $\cP_{K^{p}{}'}$  at \emph{all} classical points lands in the Selmer group is in new even in the context of \cite{howbig}.) There, Howard    asked whether his class  interpolates  Heegner cycles at all classical points of $\X$. 
The first part of the  following theorem summarises the results described above. The second part, whose proof is  simple and direct,  provides an affirmative answer to the  generalisation of Howard's question.\footnote{The question in \cite{howbig} was phrased in terms of the Abel--Jacobi classes of  Heegner  cycles  in a suitable Chow group, defined in that case in \cite{nek-heeg}; these  classes are identical to the  $P_{\Pi}(f)$ from  \eqref{cP Pi 0}: see \cite[\S~ I.2]{nek-heeg}.}

\begin{theoA}\lb{theo heeg univ} Let $\X$ be a locally distinguished Hida family for $(\G\ts\H)'$.
  There exist an open subset $\X'\subset  \X$ containing $\X^{\cl}$ 
  and a map
 $$\cP_{K^{p}{}'}\colon \Pi^{K^{p}{}', \ord}_{H'_{\Sg}} \to {H}^{1}_{f}(G_{E, Sp}, \cV)$$
of sheaves over $\X'$,  satisfying the following properties:
\begin{enumerate}
\item
$\cP_{K^{p}{}'}$ is invariant under the action of the away-from-$p\Sg$-Hecke algebra of $\H'$;
\item 
for all  $z\in \X^{\cl}$ corresponding to a representation $\Pi_{z}$  satisfying $(\textup{wt})$,  denote by $P^{\ord}_{\Pi_{z}, K^{p}{}'}$ the restriction of \eqref{cP Pi} to  $(\Pi_{z})_{H'_{\Sg}}^{K_{p}', \ord }$; then
$$\cP_{K^{p}{}'|z}  = P^{\ord}_{\Pi_{z},K^{p}{}'}$$
under the  natural  map ${H}^{1}(G_{E, Sp}, \cV)_{|z}\to {H}^{1}_{}(G_{E, Sp}, V_{\Pi_{z}})$.
\end{enumerate}
\end{theoA}

 An answer to Howard's question in its original context was earlier given by Castella (\cite{cast}, \cite{cast2})  by an indirect method, under the assumption that   $p$ splits in $E$. 

\begin{rema} It follows from the results of \cite{CV} that, under mild conditions, the class $\cP$ is non-torsion over $\X$, cf. the discussion after \cite[Theorem B]{fouquet}.
\end{rema}
\begin{rema} \lb{go to bd}
Theorem \ref{theo heeg univ} is far from being the last word on $\cP$: first,  the class $\cP$ may vanish at some classical points; second, we can consider its specialisation in \nek's Selmer group $\wtil{H}^{1}_{f}(E, V_{\Pi_{z}})$, which equals $H^{1}_{f}(E, V_{\Pi_{z}}) $ when $z$ is not exceptional but is larger otherwise. In \S~\ref{sec bd conj}, we address both problems by proposing a conjecture for the order of vanishing and leading term of $\cP$ at any classical point, generalising conjectures by Bertolini--Darmon. The same Conjecture  \ref{pf conj} will also give a prediction for the leading terms of universal toric periods on distinguished Hida families for \emph{coherent} quaternionic groups, discussed in \S~\ref{sec: wald}, and in that case we will describe some new evidence in higher rank coming from the `plectic' world  via \cite{FG}.
\end{rema}

\subsection{The universal formula}\lb{sec: ugz}   We first recall the $p$-adic $L$-function constructed in \cite{dd-pLf}, then  state our formula for the $p$-adic height of $\cP_{K^{p}{}'
}$.  

At times we refer to the main body of the paper for the precise definition of some of the objects.

\subsubsection{Dualities over Hida families}
The space $\cE^{\ord}$  is  endowed with an involution $\iota$ corresponding to $\Pi_{z}\mapsto \Pi_{z}^{\vee}$.
Fix a  locally distinguished Hida family $\X$; then the   constructions of \S~\ref{sec: uH}  can be performed over $\X$.
 Denoting by $(-)^{\iota}$ the  pullback   under $\iota$ of an object over $\X$,
 we have dualities 
\beq\label{pair V fam}
\cV\ot \cV^{\iota}\to\OO_{\X}(1)
\eeq
interpolating \eqref{pair V}. These data, together with their deformation to a Hida family $\X^{\sharp}$ for $\G\ts\H$,  allow to define a height pairing as in Proposition \ref{prop ht2},
\beq\label{ht fam}
h_{\cV/\cV^{\sharp}}\colon \wtil{H}^{1}_{f}({E}, \cV)\ot_{\OO_{\X}} \wtil{H}^{1}_{f}(E,\cV^{\iota})\to\mathscr{N}_{\X/\X^{\sharp}}^{*}\cong \OO_{\X}\hat{\ot}\Gamma_{F}.
\eeq
As usual after possibly restricting to an open subset containing $\X^{\cl}$, we  construct:
\begin{itemize}
\item
pairings 
$$((\ , \ )) \colon \Pi^{K^{p}{}', \ord}_{H'_{\Sg}} \ot_{\OO_{\X}} (\Pi^{K^{p}{}', \ord}_{H'_{\Sg}})^{\iota}\to \OO_{\X}$$
 interpolating the $p\infty$-modification $(\ ,\  )^{\ord}_{\Pi}:=\eqref{()'}$ of $(,)_{\Pi}$;
\item   $\OO_{\X}^{\ts}$-module maps 
$$\cQ \colon( \Pi^{K^{p}{}', \ord}_{H'_{\Sg}} \ot_{\OO_{\X}^{\ts}} \Pi^{K^{p}{}', \ord, \iota}_{H'_{\Sg}})
\ot_{\OO_{\X}^{\ts}} 
( \Pi^{K^{p}{}', \ord}_{H'_{\Sg}} \ot_{\OO_{\X}^{\ts}} \Pi^{K^{p}{}', \ord, \iota}_{H'_{\Sg}})^{\times, -1}\to \cK_{\X}$$
interpolating the $p\infty$-modification $Q^{\ord}= \eqref{Q orddd}$ of $Q$. Here, $\cK_{\X}$ is the sheaf of fractions of $\OO_{\X}$ and  the superscript `$\times, -1$' denotes the subgroup of those $f_{3}\ot f_{4} $ satisfying $((f_{3}, f_{4}))\neq 0$ and suggests  the `denominator' invariance of the pairing in the last two variables.
\end{itemize}
\subsubsection{The $p$-adic $L$-function} \lb{ssec Lp}
 Let $\cE_{0}^{\sharp, \ord}:= \cE_{\G_{0}\ts\H}^{\ord} $ be the ordinary eigenvariety for $\G_{0}\ts\H$ (see \cite{hida-adv, hida-LF}); for appropriate  choices of tame levels,  there is a map $\iota_{\rm JL}\colon \cE^{\ord}_{\G\ts\H} \to \cE_{\G_{0}\ts \H}^{\ord}$, which is a closed immersion onto a union of irreducible components.
Let $\X_{0}^{\sharp}:=\iota_{\rm JL}(\X^{\sharp})\subset \cE^{\sharp, \ord}$.  We recall the $p$-adic $L$-function on $\X^{\sharp}$ constructed in \cite{dd-pLf}.

Let $\X_{0}^{\sharp, \cl} = \iota_{\rm JL}(\X_{}^{\sharp, \cl}) \subset \X_{0}^{\sharp}$ be the ind-scheme of classical points.
If $(x,y)\in \X_{0}^{\sharp, \cl}(\bC)$ is a geometric point corresponding to a closed point  $(x_{0}, y_{0})\in \X_{0}^{\sharp,\cl}$ together with an embedding   $\iota\colon \Q_{p}(x_{0}, y_{0}) \into \bC$, we denote $\pi_{x} =\pi_{x_{0}}^{\iota}, \chi_{y}=\chi_{y_{0}}^{\iota}$, which are  complex automorphic representations of $\G_{0}(\A)$ and $\H(\A)$ respectively. We then  denote $V_{(x_{0},y_{0})}^{\sharp}:= V_{(\pi_{x_{0}}, \chi_{y_{0}})}$ and let  
$$\calL(V_{(x, y)}^{\sharp}, 0)= \prod_{v}\iota \calL(V_{(x, y),v}^{\sharp},0)
$$ be the product (defined by analytic continuation) of \emph{all} of the factors \eqref{calLv}.

Recall that if $W$ is a complex Weil--Deligne representation of the Weil group of a local field $F_{v}$ and $\psi_{v}\colon F_{v}\to \bC^{\ts}$ is a nontrivial character,  the inverse Deligne--Langlands $\gamma $-factor is\footnote{The normalisations of $L$- and $\eps$-factors are as in \cite{tate-nt}.}  
\beq \lb{gamma intro} \gamma(W, \psi_{v})^{-1} = {L(W)/ \eps(W, \psi_{v}) L(W^{*}(1))}, \eeq
and  $\psi_{E, w}=\psi_{v}\circ \Tr_{E_{w}/F_{v}}$.

 If $\pi=\pi_{x_{0}}$, $\chi=\chi_{y_{0}}$ are as just above (with weights $\uw=\uw_{x_{0}}$, $\ul=\ul_{y_{0}}$),  let  $\ad (V_{\pi,v})(1)^{++}:= \Hom(V_{\pi,v}^{-}, V_{\pi, v}^{+})$. Let $\psi=\prod_{v}\psi_{v}\colon F\bks \A_{F}\to \bC^{\ts}$ be the standard additive character such that $\psi_{\infty}(\cd) = e^{2\pi i \Tr_{F_{\infty}/\R}(\cd)}$; let $\psi_{E}=\prod_{w}\psi_{E, w}=\psi\circ \Tr_{\A_{E}/\A_{F}}$. For  a place $v\vert p $ of $F$, let $d_{v}$ be a generator of the different ideal of $F_{v}$, and  define 
\beq \lb{eVv}
e_{v}(V_{(\pi_{x}, \chi_{y})})= |d_{v}|^{-1/2}
 {\prod_{w\vert v}\gamma(\iota{\rm WD}(V^{+}_{\pi, v|G_{E, w}} \ot V_{\chi, w}), \psi_{E,w})^{-1}
 \over \gamma(\iota{\rm WD}(\ad(V_{\pi,v})(1))^{++}, \psi_{v})^{-1}}
 \cd \calL( V_{(\pi^{\iota}, \chi^{\iota}),v})^{-1},\eeq
 where  $\iota {\rm WD}$ is the functor from potentially semistable Galois representations to complex Weil--Deligne representations of \cite{fontaine}. 
Finally,   we define 
\beq\lb{epinf}
e_{\infty}(V_{(\pi^{\iota}, \chi^{\iota})}) &:= i^{(w+l)[F:\Q]},\\
 e_{p\infty}(V_{(\pi^{\iota}, \chi^{\iota})}) &:=  e_{\infty}(V_{(\pi^{\iota}, \chi^{\iota})}) \cd \prod_{v\vert p}e_{v}(V_{(\pi^{\iota}, \chi^{\iota})}).
\eeq
At least if $w+l=0$, these belong to $\iota \Q_{p}(x_{0}, y_{0})$, and we may define  
\beq\lb{epinf L}
e_{p\infty}(V_{(\pi, \chi)}) :=   \iota^{-1}e_{p\infty}(V_{(\pi^{\iota}, \chi^{\iota})}).\eeq

  The following is the main theorem of \cite{dd-pLf}.

\begin{theo} \lb{conj Lp} There exists a  meromorphic function 
 $$\calL_{p}(\cV^{\sharp})\in  \mathscr{K}(\X_{0}^{\sharp})$$
 whose polar locus $\mathscr{D}$  does not intersect $\X_{0}^{\cl}$, uniquely characterised by the following property.

For each $z=(x, y)\in \X_{0}^{\sharp, \cl}(\bC)- \mathscr{D}(\bC)$ corresponding to an automorphic  representation $\pi_{x}\ot \chi_{y}$ of $\G_{0}(\A)\ts \H(\A)$
 of weight $(\uw_{x},  \ul_{y})$ 
 satisfying the conditions 
$$
  \text{ $| l_{y, \sg} | < w_{x, \sg}$, \qquad  $|w_{x}+l_{y}| \leq w_{x,\sg}- | l_{y,\sg}| -2$ \qquad for all $\sg\colon F\into \bC$},$$  
 we have
 \beq
\label{interpol Lp intro} 
\calL_{p}(\cV)(x, y) =
e_{p\infty}(V_{(\pi_{x}, \chi_{y})}^{})
 \cdot  \calL(V_{(\pi_{x} ,\chi_{y})}, 0).
 \eeq
  \end{theo}

\subsubsection{Main theorem}
Under the condition of local distinction of $\X$,   the function  $\calL_{p}(\cV^{\sharp})$ vanishes identically on $\X_{0}$. Let   $\mathscr{N}^{*}_{\X_{0}/\X_{0}^{\sharp}} =\cI_{\X_{0}}/\cI_{\X_{0}}^{2}\ot_{\OO_{\X_{0}^{\sharp}}}\OO_{\X_{0}}  \cong \OO_{\X_{0}}\hat{\ot}\Gamma_{F}$  be the conormal sheaf and let 
$$\qquad \qquad \qquad {\rm d}^{\sharp}\calL_{p}(\cV):={\rm d}_{\X_{0}/\X_{0}^{\sharp}} \calL_{p}(\cV^{\sharp})  \qquad \in \cK(\X_{0})\hat{\ot}\Gamma_{F}= \cK(\X)\hat{\ot}\Gamma_{F}$$ be the  image of $\calL_{p}(\cV^{\sharp})$.

\begin{theoA}\label{ugz thm} Let $\X$ be a locally distinguished Hida family for $(\G\ts\H)'$. Abbreviate $\Pi^{(\iota)}:=  \Pi^{K^{p}{}', \ord, (\iota)}_{H'_{\Sg}}$, $\OO:=\OO_{\X}$, $\cK:=\cK_{\X}$. 	

Then there is an
open subset $\X'\subset \X$ containing $\X^{\cl}$ such that all of  the above constructions   can be made over $\X'$,  and
$$   {  h_{\cV/\cV^{\sharp}}(\cP(f_{1}), \cP^{\iota}(f_{2}))    \over   ((f_{3}, f_{4}))      }  
=
 {\rm d}^{\sharp} \calL_{p}(\cV^{\sharp})
\cdot
\cQ\left( {  f_{1} \ot  f_{2}\over f_{3}\ot f_{4}}\right),
$$
 an equality of $\cK\hat{\otimes}_{\Z_{p}}\Gamma_{F}$-valued $\OO$-linear  functionals on $(\Pi\ot_{\OO} \Pi^{\iota})\ot_{\OO^{\times}} (\Pi\ot_{\OO} \Pi^{\iota})^{\ts, -1}$.
\end{theoA}
The formula of the theorem  in fact also holds at exceptional points $z\in \X^{\cl}$, see Theorem \ref{GZ thm'}.

\subsection{Applications}
\lb{sec: appl}
We turn to some  arithmetic  applications of the main theorems (in addition to Theorem \ref{BK thm}).

\subsubsection{On the Iwasawa Main Conjecture for derivatives}
We use the notation introduced after Theorem~\ref{BK thm}. 
\begin{theoA} \label{iw}
Let $\X$ be a locally distinguished Hida family for $(\G\ts\H)'$, satisfying the further conditions 
of \cite[Theorem B.(iii)]{fouquet}.  Let $\X'\subset \X$ be the open subset of Theorem \ref{ugz thm}; up to shrinking $\X'$ we may assume it is a regular scheme.  Let $\cR\subset \OO_{\X'}\hat{\ot} \Gamma_{F}$ be the regulator of the height pairing \eqref{ht fam}
over $\X'$.
 Then 
\beqq
{\rm d}^{\sharp}\calL_{p}(\cV) \succeq_{\OO_{\X'}} \cR\cdot {\rm char}_{\OO_{\X'}} (\wtil{H}^{2}_{f}(E, \mathscr{\cV})_{\OO_{\X'}{\textup{-tors}}}).
\eeqq
\end{theoA}
The proof, based on Theorem \ref{ugz thm} and \cite[Theorem B.(iii)]{fouquet}, is virtually identical to that  of \cite[Theorem D]{dd-pyzz}, based on Theorem C.4 \emph{ibid.}  and \cite[Theorem B.(ii)]{fouquet}.

\subsubsection{Generic non-vanishing of $p$-adic heights for self-dual CM motives} 
It is conjectured that cyclotomic $p$-adic height pairings are non-vanishing (and even non-degenerate).  
Results in this direction have been quite rare.  The next theorem generalises a variant of the main theorem of \cite{bur-d}, to which we refer for a discussion of the background.

Consider the set of locally algebraic Hecke characters
$$\chi\colon E^{\ts}\bks E_{\A}^{\ts}\to \Q_{p}(\chi)^{\ts}.$$
satisfying the special  self-duality  condition
\beq\lb{sd hecke}
\chi_{|F_{\A}^{\ts}}=\eta \cdot \chi_{{\rm cyc},F}.
\eeq

This is precisely the set of classical points of the closed subspace 
$$\cE^{\ord, {\rm sd}}_{\H}\subset \cE^{\ord}_{\H}:= \bigcup_{V^{p}\subset \H(\A^{p})}\cE^{\ord}_{\H, V^{p}}$$
cut out by the condition \eqref{sd hecke} on continuous characters. 
The space $\cE^{\ord, {\rm sd}}_{\H}$ is a  torsor for $\cE^{\ord}_{\H'}$; in particular it is smooth of  dimension $[F:\Q]$. Let $\Y\subset \cE^{\ord, {\rm sd}}_{\H}$ be an irreducible component; then there is a sign $\epsilon \in \{\pm 1\}$ such that  for all $x\in \Y_{}^{\rm cl}$,  $\eps(1,\chi) =\epsilon$; we then say that $\Y^{\cl} $ has type $\epsilon$.

Denote by $h_E^-=h_E/h_F$  the relative class number of $E/F$ and  by $D_{F}$ the absolute discriminant of $F$.

\begin{theoA}  \lb{nv CM}
Let $\Y\subset \cE^{\ord, {\rm sd}}_{\H}$ be an irreducible component  of type $-1$. 
Suppose that all primes $v\vert p$ of $F$ split in $E$,  the extension $E/F$ is ramified,  and $p \nmid 2 D_{F} h_E^-$. 

Then, there  exists a non-empty open  subset $\Y'\subset \Y$ such that  for all $y\in \Y^{\cl}\cap \Y'$, the Selmer group $H^{1}_{f}({E}, \chi_{y})$ is nonzero and the $p$-adic height pairing 
$$h\colon H^{1}_{f}({E}, \chi_{y})\ot H^{1}_{f}({E}, \chi_{y}^{-1}(1))\to \Q_{p}(y)\hat{\ot}\Gamma_{F}$$
is non-vanishing.
\end{theoA}

\subsubsection{Non-vanishing of universal  Heegner classes along some classical Hida families} Part 3 of the following theorem is also a contribution to the non-vanishing conjecture for $p$-adic heights. Parts 1 and 2 provide, to the best of the author's knowledge, the first piece of theoretical evidence towards conjectures of  Greenberg \cite{greenberg} and Howard \cite{howbig}.

\begin{theoA} \lb{nv exc} Let $\X_{0}$ be a Hida family for ${\rm PGL}_{2/\Q}$, and let $\X_{0}^{\sharp}$ be the Hida family for $\GL_{2/\Q}$ containing $\X_{0}$. Denote by $\cV_{0}$, $\cV_{0}^{\sharp}$ the associated rank-$2$ representations of $G_{\Q}$. 

  Suppose that $\X_{0}$ contains a point corresponding to an elliptic curve $A$ with split multiplicative reduction at~$p$,  satisfying $L(A, 1)\neq 0$.
   Then:
\begin{enumerate}
\item a universal Heegner class  $\calP_{0}$ is nonvanishing along $\X_{0}$;
\item the Selmer group $\wtil{H}^{1}_{f}(\Q, \cV_{0})$ has generic rank~$1$, generated by $\calP_{0}$;
\item the $p$-adic height pairing $h_{\cV_{0}/\cV_{0}^{\sharp}} $ is non-vanishing.
\end{enumerate}
\end{theoA}


\subsection{Outline  of  the proofs} The basic  strategy to prove the main results is very simple. When $W$ is trivial, Theorem \ref{GZ thm} was proved  in \cite{dd-pyzz, nonsplit} under some technical assumptions. As the set of points of trivial weight in $\X^{\cl}$ satisfying those assumptions is still dense in $\X$, this suffices to deduce Theorem \ref{ugz thm} once its terms are defined; by a multiplicity-one argument and an explicit local  computation, this in turn implies Theorem \ref{GZ thm} for all $W$.
Much of this work is therefore  an exercise in $p$-adic interpolation to construct the objects  of   \S\S~\ref{sec: uH}-\ref{sec: ugz}; 
 the table of contents, and the internal references given so far, should suffice to guide the reader through the paper. 
 
 {\begin{figure}[h] $$\includegraphics[width=.7\textwidth]{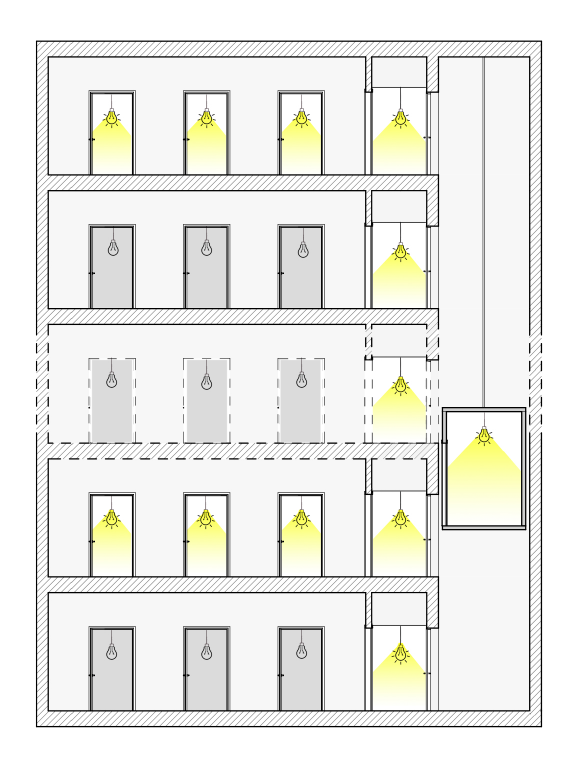} 
$$
\caption{An illustration of the proof of Theorems \ref{GZ thm} and \ref{ugz thm}. Each of the (infinitely many) floors corresponds to a representation $\Pi$ as in Theorem \ref{GZ thm},  each apartment to a quadruple ${\bf f}=(f_{1}, f_{2}, f_{3}, f_{4})$, and the building  to a Hida family. A light being on indicates that the corresponding Gross--Zagier formula is proven. On `most' floors corresponding to a $\Pi$ of trivial weight, all lights are on by \cite{dd-pyzz, nonsplit}. In this paper, we construct the lift corresponding to the formula of Theorem~\ref{ugz thm}, with doors (interpolation statements) onto special  apartments in each floor (the formulas of Theorem \ref{GZ thm'} in \S~\ref{B ord}, equivalent to Theorem \ref{GZ thm} for certain special quadruples ${\bf f}$). As soon as the lights in a dense set of floors in the building are on, the light in the lift is on; this allows to turn on the light in all the special apartments. Finally, the multiplicity-one principle allows to propagate the electricity among different apartments on the same floor.}
\lb{fig1}
\end{figure} }
 
 The proof of Theorem \ref{theo heeg univ} is completed in \S~\ref{645} and the proof of Theorems \ref{GZ thm}, \ref{ugz thm} is completed in  \S~\ref{sec: final}, where we also prove Theorem \ref{nv CM}. Theorem  \ref{nv exc} is proved  at the end of  \S~\ref{sec bd conj} using a known case of the conjecture made there.  Constructions and calculations of a local nature are gathered in  Appendix~\ref{app A}.  
 
 \medskip
 
 We highlight some of the key tools we use (many have already been mentioned):
 \begin{itemize}
 \item
 \nek's theory of Selmer complexes and $p$-adic heights (\cite{nek-selmer}, see also \cite[Appendix C]{venerucci-thesis}), applied to Hida theory;
\item the local Langlands correspondence in families as described in \cite{LLC}, that is necessary for the interpolation of the terms $Q_{v}$;
\item Emerton's point of view \cite{emerton} on $p$-adic cohomological automorphic representations as having a component at `infinity' that is an algebraic representation of the relevant group; in our context, this further allows to properly consider  `incoherent' reductive groups;
\item the multiplicity-one result for $\H'(\A)$-invariant functionals;
\item the definition and study of semi-local operators at $p\infty$, as the key to transitioning between ordinary and anti-ordinary or toric parts of a module; 
\item the explicit evaluation of certain local toric periods in terms of gamma factors.
 \end{itemize}
 
We view the framework introduced in the appendix as the main technical novelty contributed by the present work, and we hope that the underlying approach will prove useful in many other contexts.\footnote{Cf. the work \cite{loeffler} discussed in \S~\ref{contemp} below.}

 \subsubsection*{Further directions} We have not paid attention to the integral aspects; doing so may also remove the need to  restrict to open subsets of $\X$ at various points, e.g. by restricting to newforms or using the  local Langlands correspondence in integral families of Emerton, Helm, and Moss (see references in \cite{LLC}).  (However, this would require imposing some residual irreducibility assumptions for the representation $\cV_{v}$.)  This may lead to non-vanishing results for higher-weight Heegner cycles, automorphic toric periods,  and $L$-values: an example we have in mind is the anticyclotomic non-vanishing result of \cite{CH}, based on a construction not unlike that of Theorem \ref{wald univ}. 
 
  In a different direction, all of the constructions of this paper could be generalised, with work, to the context of eigenvarieties;   the  Gross--Zagier formulas should also extend to that context.
 
\subsection{Related contemporary work}\lb{contemp} After a first version of this paper was made publicly available, the following partly related works have appeared.
\begin{itemize}
\item In  \cite{JLZ}, the authors construct universal Heegner classes for Coleman families of elliptic modular forms (with classical restrictions); then they prove that these classes interpolate the images of Heegner cycles, by a method not dissimilar to that of the present work. Similar results are also  independently proved in \cite{ota} in the ordinary case, and (by a  different method) in   \cite{BL} in the case where $p$ splits in the field of complex multiplications.
\item In \cite{BPS},  the authors
use \cite{JLZ, BL} and a strategy similar to the one of the present paper to prove the $p$-adic Gross--Zagier formula for critical-slope refinements of elliptic modular forms, conditionally on  work in preparation  of Kobayashi on such formula for small-slope refinements. Their idea is to deduce, from the latter,  a $p$-adic Gross--Zagier formula in a Coleman family, within which the objects considered by Kobayashi form a dense subset; then specialise the formula to other  classical points.
\item In \cite{loeffler}, Loeffler gives a method to construct  $p$-adic  families of cohomology classes attached to   inclusions of reductive groups $\H_{1}\subset \H_{2}$ such that 
$\H_{2}/\H_{1}$ is  a spherical variety. His local-at-$p$ construction vastly generalises the one of Proposition \ref{gHo}. A difference is that in \cite{loeffler}, the weight variation is not addressed (accordingly, that construction does not use the `infinite' place).
\end{itemize}

 \FloatBarrier

\subsection{Acknowledgements} I would like to thank  Jo\"el Bellaiche, Ashay Burungale, Ga\"etan Chenevier,  Olivier Fouquet, Ming-Lun Hsieh, David Loeffler, Michele Fornea, Jonathan Pottharst,  Ye Tian,  Rodolfo Venerucci, and Sarah Zerbes for useful conversations or correspondence or mathoverflow answers. I am also grateful to the referee for coaxing me to write \S~\ref{sec bd conj}, and to Simone Dell'Ariccia for Figure~\ref{fig1}.
 Finally,  I would  like to thank Shouwu Zhang for  a  vague question  he asked me in 2010; this paper may be a partial answer.

\subsection{Notation}\label{sec: not1} Throughout the paper we use the  following notation unless otherwise noted.
\begin{itemize}
\item $\A$ is the ring of ad\`eles of $\Q$;
\item   the fields $F$ and $E$ are as in the introduction, $\eta=\eta_{E/F}\colon F_{\A}^{\ts}/F^{\ts}\to \{\pm 1\}$ is the associated quadratic character,  and we denote by  $\baar{E}$ a fixed  algebraic closure {of $E$};
\item we denote by $G_{E}$ the absolute Galois group of a field $E$; if $E$ is a number field and $S$ is a finite set of places, we denote by $G_{E, S}$ the Galois group of the maximal extension of $E$ unramified outside $S\infty$;
\item  for a place $w $ of  a number field $E$, we denote by $\vpi_{w}$ a fixed uniformiser at $w$, and by $q_{w} $ the cardinality of the residue field;
\item the class field theory isomorphism is normalised by sending uniformisers to geometric Frobenii; for $E$  a number field (respectively a local field), we will then identify   characters of $G_{E}$ with characters of $E_{\A}^{\times}/E^{\times}$ (respectively $E^{\times})$ without further comment;
\item let $\mu\subset \Q_{p}^{\times}$ be the subgroup of roots of unity, and let $\langle \cdot\rangle_{p}\colon \Q_{p}^{\times}\to 1+2p\Z_{p}\subset \Q_{p}^{\times} $ be the unique continuous character such that $x_{p}\langle x\rangle_{p}^{-1}$ has values in $\mu$. 
The $p$-adic  cyclotomic character of $\Q$ is 
 $$\chi_{{\rm cyc}, \Q}(x):=|x|_{\A^{\infty }}\la x_{p}\ra_{p} , $$ a character on $\A^{\infty,\times}/\Q^{\times}$.
   If $E$ is a number field, the $p$-adic cyclotomic character of $E$ is the character
    \beq\lb{chi cyc}\chi_{{\rm cyc}, E}=\chi_{{\rm cyc}, \Q } \circ N_{E/\Q}\colon E_{\A^{\infty}}^{\times}/E^{\times}\to \Q_{p}^{\times}.\eeq
\end{itemize}

\section{Automorphic and Galois representations }
In this section we define the basic set up regarding  ordinary automorphic representations for our groups, and the associated Galois representations.

\subsection{Groups}\label{sec: not} We introduce our notation on groups and related objects.

\subsubsection{Incoherent reductive groups} \lb{ssec incoh} 
Let $F$ be a global field. For the purposes of this discussion, a `coherent' reductive  group over $F$ is just a reductive  algebraic group in the usual sense. The following notion is probably appropriate only in the context of orthogonal or unitary groups, cf. \cite{gross-incoh}; we do not explicitly restrict to that case just for the sake of brevity.

 An \emph{$F$-incoherent reductive group $\G$ over $F$} is a collection of reductive  groups $\G_{v}/F_{v}$, for $v$ a place of $F$, such that for each place $w$ of $F$ there is a coherent reductive  group $\G(w)/F$ that is $w$-nearby to $\G$ in the following sense: for each place $v\neq w$, $\G(w)\ts_{F}F_{v}\cong \G_{v}$, and the groups $\G(w_{})\ts_{F}F_{w}$ and  $\G_{w} $ are non-isomorphic inner forms of each other.  

 Let $F/F_{0}$ be a finite extension of global fields. An \emph{$F$-incoherent reductive  group $\G$ over $F_{0}$} is a collection of reductive groups $\G_{v_{0}}/F_{0, v_{0}}$, indexed by the places $v_{0}$ of $F_{0}$, satisfying the following. For each $v_{0}$, we may write $\G_{v_{0}}=\Res_{F_{v_{0}}/F_{0,v_{0}}}\G_{F, v_{0}}:=\prod_{v\vert v_{0}}\Res_{F_{v}/F_{0, v_{0}}} G_{F, v}$ for a collection of reductive groups $G_{F, v}/F_{v}$ that forms an $F$-incoherent algebraic group $\G_{F}$ over $F$. 
In this  situation, we write $\G=\Res_{F/{F_{0}}} \G_{F}$. We write just `incoherent' when $F$ is unimportant or understood from context. We also write $\G(F_{v_{0}}):= \G_{v_{0}}(F_{v_{0}})$ for short.

  By definition,  for all but finitely many $v_{0}$, the group $\G_{v_{0}}$ is unramified. In particular, if $S$ is a finite set of places of $F_{0}$, it makes sense to consider the restricted tensor product $\G(\A^{S}):= \prod'_{v_{0}\notin S}\G(F_{v_{0}})$. 
  
   It will be convenient to consider a $p$-adic variant in the case where $F=\Q$ and $\G_{F, \infty}$ is anisotropic modulo its centre (so that all its admissible representations are finite-dimensional). In this case we \emph{redefine}, for any finite set of finite places $S$,
  $$\G(\A^{S}):= \G(\A^{S\infty})\ts G_{\infty},$$
  where  $ \G_{\infty}:=\G_{p}(\Q_{p}) $ with the Zariski topology.

The  main example of interest to us is the following: $F$ is our totally real number field, $F_{0}=\Q$, and $\G_{F_{v}}=\B_{v}^{\ts}$. The conditions are satisfied since, for each place $w$, there is a quaternion algebra $B(w)$ over $F$ such that $\B_{v}\cong B(w)_{v}$ if and only if $w\neq v$.   Other examples are obtained as follows: if $\G$ is an incoherent group and $\H$ is a coherent group, the product $\G\ts\H$ (whose precise definition is left to the reader) is an incoherent group.

\subsubsection{Hecke algebras}   Let  ${\G}$ be a coherent or incoherent reductive group over $\Q$, $A$ a ring.

 If  $S$ is a finite set of primes of $\Q$ different from $p$,
let
$${\mathscr H}_{{\G, A}}:=C^{\infty}_{\rm c}({\G}(\A^{p\infty}),A), \qquad {\mathscr H}^{S}_{{\G}}:=C^{\infty}_{\rm c}({\G}(\A^{Sp\infty}),A)$$
be the Hecke algebras. If $U\subset \G(\A^{\infty})$ is a compact open subgroup we let $\cH_{\G,U,A}$ and $\cH_{G, U,
A}^{S}$ be the respective subalgebras of functions that are bi-$U$-invariant. If $S$ is \emph{$U$-spherical} in the sense  that $U_{v}$ is maximal for all $v\notin S$, we say that $ \cH_{\G,U}^{S}$ is a \emph{spherical} Hecke algebra. 

If $M$ is an $A$-module with a smooth $A$-module action by ${\mathscr H}= \cH_{{\G}}$, $\cH_{{\G}}^{S}$, $\cH_{{\G},U}$, or  $\cH_{{\G},U}^{S}$, we let $\cH({M})\subset \End_{A}(M)$ by the image of $\cH$. We define the \emph{spherical Hecke algebra} acting on $M$ to be 
$$\cH_{\G}^{\rm sph}:=\varprojlim_{S,U} \cH_{{\G},U}^{S}(M)\subset  \End_{A}(M)$$
if the limit, taken over pairs $(S,U)$ such that $S$ is $U$-spherical, stabilises. It is equipped with an involution $\iota $ coming from the involution on $\G_{*}(\A)$.

 \subsubsection{Subgroups of $\G_{*}$} We restrict,  for the rest of this subsection,  to the groups in \eqref{list}, denoted collectively by $\G_{*}$. 
Assuming that $\B_{p}$ is split, we fix an  identification $G:=\G(\Q_{p})\cong \GL_{2}(F_{p})$ for the rest of the paper, by which we obtain  $\Z_{p}$-models $\G_{*/\Z_{p}}$ for all of the groups $\G_{*/\Q_{p}}$.

 Let $N_{\G}\subset \G(\Q_{p})\cong \GL_{2}(F_{p})$ be the subgroup of unipotent matrices. Let $N_{\H}=\{ 1\}\subset \H(\Q_{p})$, and for  $?=\emptyset$ (respectively $?= {}'$), let $N_{(\G\times \H)^{?}}:=N_{\G}\times \{1\}$ (respectively its image in $(\G\times\H)'(\Q_{p})$). Finally, let $N_{\G_{*},0}:= N_{\G_{*}}\cap \G_{*/\Z_{p}}(\Z_{p})$.

Let $T_{\G_{*}}\subset \G_{*}(\Q_{p})$ be the maximal torus consisting of diagonal matrices when $\G_{*}=\G$ and compatible with this choice when $\G_{*}$ is any other group. Let $T_{\G,0}:= T_{\G_{*}}\cap \G_{*}(\Z_{p})$ the integral subgroup.
Let    $T_{\G_{*}}^{+}\subset T_{\G_{*}}$
 be the normaliser of $N_{\G_{*}, 0}$ in $T_{\G_{*}}$, so that   $T_{\G}^{+}:=\prod_{v\vert p}T_{\G, v}^{+}$ with 
 $$T_{\G, v}^{+}:=\{ \smalltwomat {t_{1}}{}{}{t_{2}} \, :\,  v(t_{1})\geq v(t_{2})\} .$$

\subsubsection{Involutions} \lb{ssec invo} We denote by $\iota$ the involutions on 
 $\cH^{\rm sph}_{\G}$ induced by $g\mapsto g^{\rm T,-1}$, and on $\H$ induced by $h\mapsto h^{-1}$.   

We also denote by $\iota$ the involution of $T_{\G_{*}}$  deduced by the involutions 
\beq 
\label{involution}
t\mapsto t^{\iota}:= t\nu(t)^{-1},
\eeq
where  $\nu$ denotes the reduced norm if $\G_{*}=\G$, the norm $N_{E/F}$ if $\G=\H$. 
It preserves the sub-semigroups  $T_{\G_{*}}^{+}$.

\subsubsection{Congruence subgroups}  \lb{ssec cong}
Let $G=\GL_{2}(F_{p})$, $H=E_{p}^{\times}$, $H'=E_{p}^{\times}/F_{p}^{\times}$,  $(G\times H)':=(G\times H)/F_{p}^{\times}$ where $F_{p}^{\times }$ is identified with the centre of $G\times H$.

For $r\in \N$, define the compact subgroups $U(\vpi_{v}^{r})\subset U^{1}_{1}(\vpi_{v}^{r})\subset \GL_{2}(F_{v})$ by 
\begin{align*}
U^{1}_{1}(\vpi_{v}^{r}) &:= \{ \smalltwomat abcd \in \GL_{2}(\OO_{F,v})\, :\,  a-1\equiv d-1 \equiv c\equiv 0 \pmod{\vpi_{v}^{r}}\},\\
U(\vpi_{v}^{r}) &:= \{ \smalltwomat abcd \in \GL_{2}(\OO_{F,v})\, :\,  a-1\equiv d-1 \equiv b\equiv c\equiv 0 \pmod{\vpi_{v}^{r}}\}.
\end{align*}

For each place  $v\vert p$  of $F$, we  fix $\epsilon_{v}\in F_{v}^{\times}$ such that  $E_{v}=F_{v}(\sqrt{\epsilon_{v}})$;
for technical reasons it will be convenient to \emph{assume that $v(\epsilon_{v})\geq 1$.}

 For $\ur=(r_{v})\in \N^{\{v\vert p\}}$, we define  the   compact open subgroups   $V_{F,v, r_{v}}:=1+\vpi^{r_{v}}_{v}\OO_{F,v}\subset F_{v}^{\times}$ and
$$V_{p, \ur}=\prod_{v\vert p}V_{v, r_{v}} \subset H=\prod_{v\vert p}E_{v}^{\times},
 \qquad U_{p,\ur}=\prod_{v\vert p}U_{v, r_{v}} \subset G=\prod_{v\vert p} \GL_{2}(F_{v})$$
  as follows:
\beqq\lb{Vr}
V_{ v, r_{v}}&:=\begin{cases}
 V_{F,v, r_{v}}(1+\vpi^{r_{v}}\OO_{E,w}) &\text{if $v $  splits in $E$}\\
 V_{F,v, r_{v}}(1+\sqrt{\epsilon_{v}}\vpi^{r_{v}}\OO_{E,w}) &\text{if $v$ is nonsplit  in $E$}\\
 \end{cases},\\
  \qquad U_{ v, r_{v}}&:=
  U_{1}^{1}(\vpi_{v}^{r_{v}}).
\eeqq
We also define  $V_{p, \ur}':= V_{p, \ur}F_{p}^{\times}/F_{p}^{\times}\subset H'$, and  
$$K_{p, \ur}, K_{p}(p^{\ur})\subset (G\times H)'$$
 to be the images of $U_{p, \ur}\times V_{p, \ur}$, $U_{p}(p^{\ur})\times V_{p, \ur}$ respectively.

 We also denote  $$    T_{\G_{*},r}:=T_{\G_{*}}\cap U_{*,p}(p^{r }).$$

If $p\OO_{F,p}=\prod_{v}\vpi_{v}^{e_{v}}\OO_{F,v}$, we associate to an integer $r$ the tuple $\ur:= (e_{v}r)_{v\vert p}$. Denoting by $U_{*}$ any of the symbols $U, V, K$, we then let $U_{*,p,r}:=U_{*,p,\ur}$, $U_{*, p}(p^{r}):=U_{*,p}(p^{\ur})$.

\subsection{Algebraic representations}
We set up some notation for algebraic representations of a (coherent or incoherent) reductive group $\G$ over $\Q$,  then  discuss in some more detail the situation for the groups of interest to us.
Let $L$ be an extension of $\Q_{p}$, $W$ a finite-dimensional irreducible algebraic (left) representation of $\G$ over $L$.  
Throughout the paper, we tacitly identify left and right algebraic representations of $\G$ via $g.w=w. g^{-1}$.

\subsubsection{Highest-weight character} We suppose that $\G=\G_{*}$ is one of the groups of \S~\ref{sec: not}.
Let  $T_{\G_{*}}\subset G_{*}$ be the fixed  torus and let $N_{\G_{*}}\subset G_{*}$ be the fixed unipotent subgroup. If $W$ is an irreducible left  (respectively right) representation of $\G_{*}$, we denote by $\sigma_{W}$ the character by which $T_{\G_{*}}$ acts on the line  of highest-weight vectors $W^{N_{\G_{*}}}$ (respectively highest-weight covectors $W_{N_{\G_{*}}}$). 

The highest-weight character of $W$ is related to that of its  dual by 
\beq \label{highest weight}
\sigma_{W^{\vee}}(t)
=\sigma_{W}(t^{\iota}),
\eeq
where ${\iota}$ is the involution \eqref{involution}.

\subsubsection{Quaternionic  special case} Suppose that $\G(\A^{\infty})$ is 
 the group of units of a quaternion algebra  $\B^{\infty}$ over $\A^{\infty}$. 
Let $L$  be an extension of $\Q_{p}$ splitting $F$ and $B_{p}$.  A \emph{(cohomological)  weight} for $\G$ over $L$ is a list  $\underline{w}=(w; (w_{\sigma})_{\tau\colon F\into L})$ of   $[F:\Q]+1$ integers  of the same parity such that $w_{\sigma}\geq 2$ for all $\sigma\colon F\into L$. Denote by ${\rm Std}_{\sigma}\cong (L^{\oplus 2})^{*}$ (respectively, ${\rm Nm}_{\sigma}\cong L$) the standard (respectively, reduced norm) representation of ${\G}(\Q_{p})=B_{p}^{\times}$ factoring through $(B_{p}\otimes_{F_{p}, \sigma} L)^{\times}\cong \GL_{2}(F_{p}\otimes_{\sigma} L)$. We associate to the weight $\uw$ the algebraic representation 
\begin{align}\label{Ww}
W_{{\G},\underline{w}}:= \bigotimes_{\sigma\in \Hom(F, L)}{\rm Sym}^{w_{\sigma}-2}{\rm Std}_{\sigma}\otimes {\rm Nm}_{\sigma}^{(w-w_{\sigma}+2)/2}
\end{align}
of $\G_{/\Q_{p}}$, whose dual is $W_{\G, \underline{w}^{\vee}}$ with $\underline{w}^{\vee}=(-w; (w_{\sigma}))$. 

 Suppose for a moment  that $L/\Q_{p}$ is Galois, then  $\Gal(L/\Q_{p})$ acts on the set of all weights $\uw$ and,   letting $L(\uw)\subset L$ be the fixed field of the stabiliser of $\uw$, 
 the representation $W_{\G, \uw}$ descends to a representation 
 over $L(\uw)$.  It is then  convenient to use the following terminology: if $W$ is an algebraic representation of $\G$ over $L$ and $\uw$ is a cohomological weight over a finite extension $L'/L$, we say that $W$ is of weight $\uw$ (with respect to $L\into L'$) if $W\otimes_{L}L' \cong W_{\G,\uw}$.

Explicitly,  $W_{\G,\uw}$ may be described as  the space of tuples  $p=(p_{\sigma})_{\sigma\colon F\into L}$ such that $p_{\sigma}\in L[x_{\sigma}, y_{\sigma}]$ is a homogeneous polynomial of degree $w_{\sigma}-2$,
 with action  on each $\sg$-component by
\beq\lb{ghp}
g.p_{\sg}(x, y)= \det(\sg g)^{w-w_{\sg}+2\over 2}\cdot  p_{\sg}((x, y)\sg g).
\eeq
The representation $W_{\G, \uw}$ admits a natural   $\OO_{L}$-lattice, stable under the action of a maximal order in $\G(\Q_{p})$,
\beq \label{homog poly lattice}
W_{\G, \uw}^{*,\circ}\subset W_{\G, \uw}^{*}
\eeq
consisting of tuples of polynomials with coefficients in $\OO_{L}$.

If $W=W_{\G, \uw}$, we have   $\sigma_{W}:=\otimes_{v}\sigma_{W,v}\colon T_{v}\to L^{\times}$ with 
\beq\label{alg char}
\sigma_{W,v}\colon \smalltwomat {t_{1}} {}{}{t_{2}} \mapsto \prod_{\sigma\colon F_{v}\into L} \sigma(t_{1})^{{w+w_{\sigma}-2\over 2}}\sigma(t_{2})^{w-w_{\sigma}+2\over 2}.
\eeq
By abuse of notation we still denote by $\sigma_{W}=\otimes \sigma_{W,v}$ the algebraic character of $F_{p}^{\times}$ defined by 
\beqq
\sigma_{W, v}(x):=\sigma_{\uw, v}\left(\smalltwomat x{}{}{1}\right).
\eeqq

\subsubsection{Toric special case}\lb{hecke 1} Let $L$ be a finite extension of $\Q_{p}$ splitting $E$. A \emph{cohomological  weight } for $\H$ is a list  $\ul:= (l, (l_{\sigma})_{\sigma\colon F\into L})$ of $[F:\Q]+1$ integers of the same parity.  For each $\sigma\colon F\into L$, fix  an arbitrary extension $\sigma'$ of $\sigma$ to $E$ (this choice will only intervene in the numerical labelling of representations of $\H$).
We let
\begin{align}\label{Wl}
W_{{\H},\underline{l}}:= \bigotimes_{\sigma\in \Hom(F, L)}{\sigma'}^{l_{\sigma}}\otimes \sigma\circ{\rm Nm}_{E_{p}/F_{p}}^{{l-l_{\sigma}\over 2}},
\end{align}
as a $1$-dimensional vector space over $L$ with action by $\H(\Q_{p})=E_{p}^{\times}$.  After choosing an identification of this space  with $L$, it admits a lattice $W_{\H,\underline{l}}^{\circ}$, stable under the action of $\OO_{E, p}^{\times}$.  If$W$ is an algebraic representation of $\H$ over $L$ and $\ul$ is a cohomological weight over a finite extension $L'/L$, we say that $W$ is of weight $\ul$ (with respect to $L\into L'$) if $W\otimes_{L}L'\cong W_{\H,\ul}$.

\subsection{Shimura varieties and   local systems}
We  again write $\G_{*}$ to denote any of the groups \eqref{list}.

\subsubsection{Shimura varieties}\lb{ssec shi}
 For $\tau$ an infinite place of $F$,  let $\G_{\tau}=\Res_{F/\Q}\G_{F}(\tau)$ be the $\tau$-nearby group as  in \S~\ref{ssec incoh}.  Consider the  Shimura datum $(\G_{\tau}, \{h_{\G, \tau}\})$, where $h_{\G, \tau}\colon {\rm S}:=\Res_{\bC/\R}{\bf G}_{m}\to \G_{\R}$ the Hodge cocharacter of  \cite[\S0.1]{carayol}.   Let $h_{\H}\colon  {\rm S} \to \H_{\R}$ be the unique cocharacter such that ${\rm e}_{\R}\circ h_{\H}=h_{\G}$. By products and projections we deduce Hodge cocharacters $h_{\G_{*}, \tau}$, hence Shimura data $(\G_{*, \tau}, h_{\G_{*, \tau}})$, for any of the groups \eqref{list}; from $h_{\H, \tau}$ we also obtain an extension of $\tau $ to an embedding $\tau\colon E\into \bC$.  Then we obtain towers of Shimura varieties  $X_{*, \tau}/\tau E_{*}$, where  the reflex field $E_{*}:=E$ unless $\G_{*}=\G$, in which case  $E_{*}=F$.  
  These data descend to $E_{*}$: there are towers  
  $$X_{*}/E_{*}$$ such that $X_{*}\ts_{\Spec E_{*}} \Spec \tau E_{*}=X_{*,\tau}$, see \cite[\S~3.1]{yzz}. Throughout this paper, we will also use the notation $\baar{X}_{*}:=X\ts_{\Spec E_{*}}\Spec{\baar{E}_{*}}$.

 We will use also the specific names \eqref{shlist} for those varieties; an explicit description of some of them is as follows:
\beq\lb{z quot}
X_{U, \tau}(\bC)&\cong B(\tau)^{\ts}\bks \mathfrak{h}^{\pm}\ts \B^{\infty\ts}/ U \cup \{\text{cusps}\}, \qquad Y_{V}(E^{\rm ab}) \cong E^{\ts}\bks E_{\A^{\infty}}^{\ts}/V, 
\\ Z_{K} &\cong X_{U}\ts Y_{V}/\Delta_{U, V}, \qquad \Delta_{U, V}:= F_{\A^{\infty}}^{\ts}/F^{\ts}\cdot ((U\cap  F_{\A^{\infty}}^{\ts})\cap (V\cap  F_{\A^{\infty}}^{\ts})) 
\eeq
if $K\subset (\G\ts\H)'(\A^{\infty})$ is the image of $U\ts V$.

\subsubsection{Automorphic local systems}
 Let $W$ be an irreducible cohomological {right} algebraic representation of $\G_{*}$ over $L$, let $U_{*}\subset \G_{*}(\A^{\infty})$ be a sufficiently small (in the sense of Lemma \ref{free action} below)  compact open subgroup, let $W^{\circ}\subset W $ be a $U_{*,p}$-stable  $\OO_{L}$-lattice, and let $U_{*,p, n}\subset U_{*,p}$ be a subgroup acting trivially on $W^{\circ}/p^{n}W^{\circ}$.
 
\begin{lemm}\label{free action} If $U^p_*$ is sufficiently small (a condition independent of $n$), then:
\begin{enumerate}
\item The quotient $\baar{\G}_n:=U_{*,p}/U_{*,p, n}({\rm Z}_{\G_*}(\Q) \cap U_*)_p$ acts freely on $X_{U_*^pU_{*,p,n}}$, hence $X_{*,U_*^pU_{*,p,n}}\to X_{*, U_*}$ is an \'etale  cover with Galois group $\baar{\G}_{*,n}$. 
\item The group ${\rm Z}_{\G_*}(\Q) \cap U_*$ acts trivially on $W^{\circ}$.
\end{enumerate}
\end{lemm}
\begin{proof} The first assertion is \cite[Lemme 1.4.1.1]{carayol}  when $\G_*=\G$ (other cases are similar or easier).
For the second assertion, we may reduce to the case $\G_*=\G $ or $\G_*=\H$, with centre ${\rm Z}_{\G_*}={\rm Res}_{E_*/\Q}\G_m$. For any $U_*$, the group ${\rm Z}_{\G_*}(\Q) \cap U_*$ has has finite index in $\OO_{E_*}^\times$, therefore  for sufficiently small $U_*^p$ 
it  is contained in the finite-index subgroup $\OO_{F}^{\times, 1}:=\{z\in \OO_F^\times \, :\, N_{F/\Q}(z)=1\}\subset \OO_{E_*}^\times$. But since $W$ is of cohomological weight, the group $\OO_{F}^{\times, 1}$ acts trivially. 
\end{proof}

 Assume first that $X_{*}$ is compact.   
  Then, by the lemma,
 \begin{align}\label{Upn}
 (X_{*,U_{*}^{p}U_{*,p, n}}\times W^{\circ}/p^{n}W^{\circ})/\baar{\G}_{*,n}
 \end{align}
 defines a locally constant \'etale $\OO_{L}/p^{n}\OO_{L}$-module  $\cW^{n}$over $X_{*,U_{*}}$. We let 
 $$ \cW^{\circ}:=  (\varprojlim_{n} \cW^{ n}),$$
   an $\OO_{L}$-local system on $X_{*, U_{*}}$, and consider 
   $$\cW:=\cW^{\circ}\otimes_{\OO_{L}}L.$$ The $L$-local system $\cW$  is compatible with pullback in the tower $\{X_{*, U_{*}}\}$ and, up to isomorphism, independent of the choice of lattice $W^{\circ}$.  When $X_{*}$ is the compactification of a noncompact Shimura variety $X_{*}'$ (essentially only when $\G=\GL_{2/\Q}$), we perform the above construction on  $X'_{*}$ and then push the resulting sheaf forward to $X_{*}$.

\subsection{Ordinary automorphic representations}  Keep the assumption that $\G_{*}$ is one of the groups in \eqref{list}.
\subsubsection{$p$-adic automorphic representations}
Let $L$ be an extension of $\Q_{p}$, $W$ a finite-dimensional irreducible algebraic left representation of 
$G_{*,\infty}=\G_{*}(\Q_{p})$  over $L$. 
\begin{defi}\label{aut-def} A (regular algebraic cuspidal) \emph{automorphic representation of $\G_{*}(\A^{})$ over $L$ } of weight $W$
is an irreducible admissible locally algebraic representation $\pi$ of 
$$\G_{*}(\A^{}):=\G(\A^{\infty})\ts G_{*,\infty}$$
 that can be factored as $$\pi=\pi^{\infty}\ot W$$
such that $\G_{*}(\A^{\infty})$ acts smoothly on $\pi^{\infty}$, $\G_{*, \infty}$ acts algebraically, and $\pi$ occurs as a subrepresentation of $$\H^{\bullet}(\baar{X}_{*}, \cW^{\vee})=\varinjlim_{U\subset \G_{*}(\A^{\infty})}H^{\bullet}(\baar{X}_{*, U}, \cW^{\vee})\ot W,$$
where $X_{*}$ is the compactified Shimura variety attached to $\G_{*}$, and $\cW^{\vee}$ is the local system on $X$ attached to $W^{\vee}$.

In the quaternionic or toric case, we say that $\pi$ is of weight $\uw$ (a cohomological weight for $\G$ over some finite extension $L'/L$) if $W$ is of weight $\uw$ as defined after \eqref{Ww} (respectively \eqref{Wl}).\footnote{These notions depend of course on $L\into L'$; nevertheless they will only be used to impose conditions on the weights that are invariant under the Galois group of $L$.}
\end{defi}
We will use subscripts $p$, respectively $\infty$, respectively $p\infty$, to denote an element of $\G(\A)$ in the copy of $\G(\Q_{p})$ contained in $\G(\A^{\infty})$, respectively in the  `algebraic copy' $G_{\infty}$, respectively the diagonal copy in the product of the previous two. 

\begin{rema} The previous definition follows the work of Emerton  (\cite{emerton}). It slightly departs from it in  that in \cite{emerton}, one restricts to considering the action of the product of $\G(\A^{p\infty})$ and the diagonal copy of $\G(\Q_{p})$. While this is indeed the part that acts integrally, we do have use for the non-integral action of each individual copy (cf.  \S~\ref{sec: A1}). The corresponding local notions are introduced in Definition \ref{def admiss}.
\end{rema}

\subsubsection{Quaternionic  special case and ordinarity} \lb{242}
Suppose that $\G_{*} =\G$ and $\B_{p}$ is split,  or that  $\G_{*}=\G_{0}=\Res_{F/\Q}\GL_{2}$ for a totally real field $F$.  
 An automorphic representation  $\pi$ over  $L$ of weight $W_{\G_{*}, \uw}$ is also said to be  of weight $\uw$.

\begin{defi}\label{ord-def} 
We say that an automorphic representation $\pi$ of $\G_{*}(\A^{})$ over $L$ of classical weight $W=W_{\G_{*}, \uw}$ is \emph{ordinary} at $ v$ with unit character $\alpha_{v}^{\circ}$
if there exists  a  smooth character $\alpha_{v}$ of $T_{v}$ 
such that  $\pi_{v}$ is 
the unique irreducible subrepresentation of
${\rm Ind}(\alpha_{v}\cdot( |\ |_{v} , |\ |_{v}^{-1}))$ and
the locally algebraic character
\begin{align}\label{alphacirc}
\alpha_{v}^{\circ}:=\alpha_{v}
\sigma_{W, v}\colon T_{v}\to L^{\ts}
\end{align}
 takes values in $ \OO_{L}^{\times}$.\footnote{This notion agrees with the notion of $\pi$ being \emph{nearly ordinary} as defined in the work of Hida (e.g. \cite{hida-adv}).}

(It follows from  the parity conditions on the weights 
   that the indicated subrepresentation is always infinite-dimensional; moreover if $\pi_{v}$ is ordinary then the character 
$\alpha_{v}$ of $T_{v}$    is uniquely determined by $\pi_{v}$.) 
We  say that $\pi$ is \emph{ordinary} if it  is ordinary at all $v\vert p$. 
\end{defi}

 Let  $v\vert p $ be a prime of $F$ and $\vpi_{v}$ a uniformiser.
For $t\in T_{v}^{+}$ or $x\in F_{v}^{\times}$ with $v(x)\geq0$, define the double coset operators
\beq\label{aut-Up}
 \Up_{t}&:=[U_{1}^{1}(\vpi_{v}^{r}) \,  t \, U_{1}^{1}(\vpi_{v}^{r})_{v}], \\
 \Up_{x}&:= \Up_{\smalltwomat x{}{}1},\\
 \Up_{v} &:=\Up_{\vpi_{v}},
\eeq
which act on the $N_{0}$-fixed vectors of any locally algebraic representation of $\GL_{2}(F_{v})$ (see also \S~\ref{A1} for further details).
  Then $\pi$ is ordinary at $v$ with  unit character $\alpha_{v}^{\circ}$ if and only if,  for sufficiently large $r$,
   the space of  $U_{1}^{1}(\vpi_{v}^{r})$-fixed vectors in the  locally algebraic representation
$$\pi_{v}\otimes_{L} W$$
of $\GL_{2}(F_{v})$ contains a (necessarily unique) line of   eigenvectors for the diagonal action of the operators   $\Up_{x}$,  $x\in F_{v}^{\times}$,  with eigenvalue $\alpha^{\circ}(x)$. Specifically, if $w_{v}\in \pi_{v}$ is a $\Up_{x}$-eigenvector of eigenvalue $\alpha_{v}(\vpi_{v})$, then such line is 
  $$\pi_{v}^{\ord}:= L w_{v}\otimes W^{N_{v}} $$
where  $W^{N_{v}}$ is the line of  highest-weight vectors of $W$.

If $\pi$ is an automorphic representation that is ordinary at  all $v\vert  p$, extend $\alpha_{v}^{\circ}$ to a character of $T_{v}^{+}$ by the formula \eqref{alphacirc}, and let $\alpha^{\circ}=\prod_{v\vert p}\alpha_{v}^{\circ}\colon T^{+}=\prod T_{v}^{+}\to L^{\times}$; then we define
\beq\label{piord}
\pi^{\ord}:=\pi^{p\infty} \otimes\otimes_{v\vert p} \pi_{v}^{\ord},
\eeq
as a smooth representation of $\G(\A^{p\infty})$ and  a locally algebraic representation of $T^{+}$ on which $ T^{+}$ acts by $\Up_{t}\mapsto \alpha_{v}^{\circ}(t)$.

\subsubsection{Toric special case}  \lb{hecke 2}
Suppose now that $E$ is a CM field and that $\G_{*}=\H:={\rm Res}_{E/\Q}\G_{m}$. 
Then a $p$-adic automorphic representation of $\H$ of weight $W_{\H, \ul}$ is 
simply the space of scalar multiples of a locally algebraic character $\chi\colon E^{\times}\bks E_{\A^{\infty}}^{\times}\to L^{\times}$ whose restriction to a sufficiently small open subgroup of $E_{p}^{\times}$ coincides with  the character of $W_{\H, \ul}$.

\subsubsection{Convention} We use the convention  that all automorphic representations of $\H(\A^{\infty})$ are ordinary, and that a  representation $\pi\otimes \chi$ of $\G\times \H$ of cohomological weight is cuspidal and  ordinary  if $\pi$ and $\chi$ are.

\subsection{Galois representations} Let $\G$ be as in \S~\ref{242}. 
\subsubsection{Galois representations attached to automorphic representations of $\G(\A^{})$}
The following notation is used throughout the paper: if $V$ is a representation of $G_{F}$ and $v$ is a prime of $F$, we denote by $V_{v}$ the restriction of $V$ to a decomposition group at $v$. 

\begin{theo}
[Ohta, Carayol, Saito]\label{galois for hilbert}
 Let $L$ be a finite extension of $\Q_{p}$, let $W$ be an irreducible algebraic representation of $\G$ over $L$, and   let $\pi$ be an automorphic representation of $\G(\A^{\infty})$ of weight $W$ over $L$.   Let $S$ be a finite set of non-archimedean places of $F$ containing all the places at which $\pi$ is ramified and the places above $p$. 
 There exists a  $2$-dimensional 
$L$-vector  space $V_{\pi}$ and an absolutely irreducible Galois representation
$$\rho=\rho_{\pi}\colon G_{F,S}\to{\rm Aut}(V_{\pi})$$
uniquely determined by the property that for every finite place $v\notin  S$ of $F$, 
\begin{align}\label{trrho}
{\rm Tr}(\rho({\rm Fr}_{v}))= q_{v}^{-1}\lambda_{\pi_{v}}(T_{v})\end{align}
where ${\rm Fr}_{v}$ is a geometric Frobenius,   $T_{v}\in \mathscr{H}^{\rm sph}_{\GL_{2}(F_{v})}$ is the element corresponding to the double class $K_{0}(1)\smalltwomat {\vpi_{v}}{}{}1K_{0}(1)$, and $\lambda_{\pi_{v}}\colon \cH_{\GL_{2}(F_{v})}^{\rm sph}\to L$ is the character giving the action  on $\pi_{v}^{K_{0}(1)}$. 

For a prime $v$ of $F$, let $\rho_{v}$ be the restriction of $\rho$ to a decomposition group at $v$. 
\begin{enumerate}
\item The representation $\rho_{v}$ is unramified for almost all $v$ and potentially semistable for $v\vert p$. For every finite place $v$,  the Weil--Deligne representation $r_{v}$ attached to $\rho_{v}$ is associated with $\pi_{v}$ via the local Langlands correspondence normalised ``\`a la Hecke'' \cite[\S~ 3.2]{deligne}:
$$L(s, r_{{v}})=L(s+1/2, \pi_{v}).$$
\item For  every finite place $v$, $r_{v}$
 satisfies the weight-monodromy conjecture: its monodromy filtration is pure 
 of weight $w-1$. The monodromy filtration is trivial if and only if $\pi_{v}$ is not a special representation.
 \item For any archimedean place $v$, the representation $\rho_{v}$ is odd, that is if  $c_{v}\in G_{F_{v}}$ is the complex conjugation, $\det \rho_{v}(c_{v})=-1$.
 \item 
If $W=W_{\G, \uw}$ with  $\uw=(w; (w_{\sigma})_{\sigma\colon F\into \baar{L}})$, then for each $v\vert p$ and $\sigma\colon F_{v}\into  \baar{L}$, 
\begin{itemize}
\item 
 the   $\sigma$-Hodge--Tate weights \footnote{If $V$ is a Hodge--Tate representation of $G_{F_{v}}$ over $\baar{L}$ and $\sigma \colon F_{v}\into \baar{L}$, the $\sigma$-Hodge--Tate weights of $V$ are the degrees in which the graded module  
$$( \oplus_{n} \bC_{v}(n)\otimes_{\baar{F}_{v},\sigma} V)^{G_{F_{v}}}$$
is nonzero; here $\bC_{v}$ is a completion of $\baar{F}_{v}$ and, in the tensor product, $\sigma $ is extended to an isomorphism $\baar{F}_{v}\to \baar{L}$.
In particular our convention is that the Hodge--Tate weight of the cyclotomic character of $\Q_{p}$ is $-1$.}
of $\rho_{v}\otimes_{L }\baar{L}$ are
$$-1-{w+w_{\sigma}-2\over 2}, \qquad -{w-w_{\sigma}+2\over 2}.$$
\item\label{phig} if $\pi$ is ordinary   at $v$ in the sense of Definition \ref{ord-def},
 then there is a unique  exact sequence  in the category of  $G_{F,v}$-representations
\beq\lb{fil}0\to V^{+}_{\pi, v} \to V_{\pi, v} \to V_{\pi,v}^{-}  \to 0, \eeq
 such that  $V^{\pm}_{\pi, v}$ is $1$-dimensional.
 
 The Galois group $G_{F, v}$ acts on $V^{+}_{\pi,v}$ by the character
 $$\alpha_{\pi,v}^{\circ}\chi_{{\rm cyc},v}\colon F_{v}^{\ts}\to L^{\times},$$
  where $\alpha_{\pi,v}^{\circ}$ is \eqref{alphacirc}.
\end{itemize}
\end{enumerate}
\end{theo}
\begin{proof}
The construction and statements 1 and  2 for $v\vert p$ are the main results of  Carayol in \cite{carayol}. Statements  1  and 2 for $v\vert p $  were proved  by T.  Saito \cite[Theorems 2.2, 2.4]{saito}. 
For the last two statements, we refer to \cite[Proposition 6.7]{TX} and references therein;  note  that  
in comparison with  the notation of \cite{TX}, our $\rho$ equals their $\rho_{f}(1)$,
and  our $(w; \underline{w})$ is their  $(w-2, {\underline k})$.
\end{proof}

\subsubsection{Realisation in the homology of Shimura varieties} \lb{ssec real}
Let $\G_{*}$ be again one of the groups of \eqref{list}. 
We introduce a new piece of notation. Let 
$$G_{F, E}:=G_{F}\ts G_{E},$$
and similarly  for a finite set of places $S$, $G_{F, E, S}:=G_{F, S}\ts G_{E,S}$. If $\G_{*}=\G\ts \H$ or $(\G\ts \H)'$, we \emph{redefine}
$$G_{E_{*}}:= G_{F, E }.$$
(This is an abuse of notation, as we have not redefined $E_{*}$.)
This product of Galois groups  acts on the homology $\overline{X}_{*}$: this is clear by the K\"unneth formula in the case of $\G\ts \H$, and follows from that case and the Galois-invariance of the quotient map  for  $(\G\ts \H)'$. 

The following is  the main result of \cite{carayol-h} in the special case $\G_{*}=\G$; the general case may be deduced from the special case together with the case $\G_{*}=\H$ (that is  class field theory).

\begin{prop}[Carayol]   Let $U_{*}\subset \G_{*}(\A^{\infty})$ be a compact open subgroup, $W$ be an irreducible   right algebraic representation of $
\G_{*}$ over $L$,  $\cW$ the local system on $X_{*, U_{*}}$ associated with $W$. Let $L'$ be a sufficiently large finite Galois extension of $L$.

 Then there is an isomorphism of $\cH_{\G_{*}, U_{*}, L}[G_{E_{*},S}]$-modules
\beq\label{carayol iso 2}
\H_{d}(\baar{X}_{*, U_{*}}, \cW) \ot_{L}L' \cong \bigoplus_{\pi } \pi^{\vee, U_{*}}\otimes V_{\pi},
\eeq
equivariant for the action of $\Gal(L'/L)$, 
where $\pi$ runs through all   equivalence classes of automorphic representations of $\G_{*}(\A)$ of weight $W$ over $L'$. 
\end{prop}

\section{Sheaves on Hida families}\label{section 3}
We construct the universal Hecke- and Galois- modules over Hida families for $(\G\ts\H)'$ and prove a local-global compatibility result. We claim no originality for the  results of \S\S~\ref{3.1}-\ref{3.2new}.
\subsection{Hida theory}\lb{3.1}
We let   $\G_{*}$ denote any of the groups $\G$, $\H$, $(\G\times \H)$, $(\G\times \H)'$, and let $r\in \N^{\{v\vert p\}}$. We will use the notation from \S~\ref{sec: not}.
 For $U_{*}^{p}\subset \G(\A^{p\infty})$ we let $X_{*, U_{*}^{p},r}:=X_{*, U_{*}^{p}U_{*,r}}$ be the corresponding Shimura variety.

  When $M$ is a $\Z_{p}$-module with   action by $T_{\G_{*}}^{+}$, arising as limit of ordinary parts of   $p$-adic coadmissible  $\G_{*}(\Q_{p})\ts G_{*, \infty}$-modules (see Definition \ref{def coadmiss} and \S~\ref{ssec ord}),
  we denote this action by 
  $$t\mapsto \Up_{t}$$
  and adopt the notation of \eqref{aut-Up}.

\subsubsection{Weight spaces} \lb{sec wtsp} Let $U_{*}^{p}\subset \G_{*}(\A^{p\infty})$ be a compact open subgroup, and define $Z_{\G_{*}, U^{p}_{*}}
\subset  {\rm Z}_{\G}(\Q)$ by 
\beqq
Z_{\G, U^{p}}&:= {\rm Z}_{\G}(\Q)\cap U^{p}T_{\G, 0}, 
&Z_{\H, V^{p}}&:=\H(\Q)\cap V^{p}T_{\H, 0}=\OO_{E}^{\times}\cap V^{ p }, & {\ } \\
Z_{(\G\times \H)', K^{p}}&:= \text{the image in $T_{(\G\times \H)'}$ of }  &Z_{\G\times H, U^{p}\times V^{p}} : =  Z_{\G, U^{p}}\times Z_{\H, V^{p}} & \text{if $K^{p}$ is the image of $U^{p}\times V^{p}$}.
\eeqq 
In all cases, let 
$\baar{T}_{\G_{*},U_{*}^{p},0}:=T_{\G,0}/\baar{Z_{\G_{*}, U_{*}}}$,
where $\baar{\Box}$  denotes the closure for the $p$-adic topology, and let  $ \baar{T}_{\G_{*}, U^{p}_{*}, r}\subset \baar{T}_{\G, U^{p}_{*},0}$ be the image of $T_{\G_{*},U^{p}_{*},r}$.
Let
$$\Lm^{\circ}_{\G_{*}, U^{p}_{*}}:=\Z_{p}\llb \baar{T}_{\G_{*},0}\rrb, 
$$
and for an irreducible algebraic representation $W$ of $\G_{*}$ consider the ideals 
\beq 
\label{Ir}
I_{\G_{*},U^{p}_{*}, W, r, L}:= ([t]-\sigma_{W}^{-1}(t))_{t\in \baar{T}_{\G_{*},  U^{p}_{*},r}})\subset \Lm^{\circ}_{\G_{*}  ,U^{p}_{*}}\otimes \OO_{L}
\eeq
For each fixed $W$ and varying $r$, the ideals $I_{\G_{*},U^{p}_{*}, W, r}$ form a fundamental system of neighbourhoods of zero in $\Lm^{\circ}_{\G_{*}, U^{p}_{*}}\otimes \OO_{L}$, so that  
\beq\label{alg density}
\Lm_{\G_{*}, U^{p}_{*}}:= \Lm^{\circ}_{\G_{*}, U^{p}_{*}}\otimes_{\OO_{L}} L=(\varprojlim_{r}\Lambda_{\G_{*}, U^{p}_{*}, W,r})\ot_{\OO_{L}} L\eeq
with
\beq\label{lambda r}
\Lambda^{\circ}_{\G_{*}, U^{p}_{*}, W, r}:= \Lambda^{\circ}_{\G_{*}, U^{p}_{*}}
/ I_{\G_{*}, U^{p}_{*},W, r}\cong
 \OO_{L}[\baar{T}_{\G_{*}, U^{p}_{*}, 0}/\baar{T}_{\G_{*}, U^{p}_{*}, r}],
\eeq
where the isomorphism is given by $[t]\mapsto \sigma_{W}^{-1}(t)[\baar{t}]$.
When $W=\Q_{p}$, we omit $W$ from the  notation. We also omit the subscript $U^{p}_{*}$  when it is unimportant or understood from context.

Writing $\baar{T}_{\G_{*},U^{p}_{*},0}\cong \Delta\times \Z_{p}^{{\rm d}(\G_{*})}$ for a finite torsion group $\Delta$, we have an isomorphism $\Lambda_{\G_{*}, U^{p}_{*}}\cong \Z_{p}[\Delta]\otimes  \Z_{p}\llb X_{1}, \ldots X_{{\rm d}(\G_{*})}\rrb$ for  an integer ${{\rm d}(\G_{*})}$  given by\footnote{We would have ${\rm d}((\G\ts_{\rm Z}\H)'=2[F:\Q]$ for the group of the footnote after \eqref{list}, which explains the way the number of variables of Theorem \ref{iw} is counted in the abstract.}
$${\rm d}(\G)= {\rm d}(\H)= [F:\Q]+1+\delta, \quad {\rm d}((\G\times \H)')= 2[F:\Q] + 1+\delta,$$
where $\delta=\delta_{F, p}$ is the Leopoldt defect of $F$ at $p$; see \cite[\S~2.2.3.3]{fouquet} for ${\rm d}(\G)$.

\begin{defi} \label{clwt} 
The \emph{weight space} is 
$$  \mathfrak{W}_{\G_{*}, U^{p}_{*}}:= \Spec \Lambda_{\G_{*}, U^{p}_{*}}\otimes \Q_{p}.$$
Let $W$ be an irreducible cohomological  algebraic representation of $\G_{*}$. The zero-dimensional subscheme of \emph{classical points of weight $W$}  and level $r$  is 
$$\mathfrak{W}_{\G_{*}, U^{p}_{*},r}^{\cl, W}:=  \Spec \Lambda_{\G_{*}, U^{p}_{*}, r, W}.$$
The ind-subschemes   of all classical points of weight $W$ and of of all classical points  are respectively
$$\mathfrak{W}_{\G_{*}, U^{p}_{*}}^{\cl, W}:=\bigcup_{r\geq 0}  \mathfrak{W}_{\G_{*}, U^{p}_{*},r}^{\cl, W}, \qquad\mathfrak{W}_{\G_{*}, U^{p}_{*}}^{\cl}:=\bigcup_{W} \mathfrak{W}_{\G_{*}, U^{p}_{*}}^{\cl, W},$$
where as usual the union runs through the algebraic representations of cohomological weight.
\end{defi}

\subsubsection{Ordinary completed homology} 
Let $W$ be an irreducible right   algebraic representation of $\G_{*}(\Q_{p})$ over $L$, and fix a $\G(\Z_{p})$-stable $\OO_{L}$-lattice $W^{\circ}\subset W$. 
Let $\cW$ be the  local system attached to $W$,  and for $U_{*}^{p}\subset \G_{*}(\A^{p\infty})$, $r\geq 0$ consider the \emph{ordinary parts}
\beqq\H^{\textrm{\'et}}_{d}(\baar{X}_{*,U_{*}^{p}U_{*,r}}, \cW^{\circ})^{\ord} &
 := (H^{\textrm{\'et}}_{d}(\baar{X}_{*,U_{*}^{p}U_{*,r}}, \cW^{\circ})\ot_{\OO_{L}[T^{+}_{\G_{*}}]} (W^{\circ,\vee})_{N_{0}})^{\ord}
 \\
\H^{\textrm{\'et}}_{d}(\baar{X}_{*,U_{*}^{p}U_{*,r}}, \cW)^{\ord}&=\H^{\textrm{\'et}}_{d}(\baar{X}_{*,U_{*}^{p}U_{*,r}}, \cW^{\circ})^{\ord}\ot_{\OO_{L}}L
\eeqq
with respect to the action of $T_{\G_{*}}^{+}$ by $\Up_{t}\ot t$, as defined in \S~\ref{ssec ord}.
The \emph{ordinary completed homology} of $X_{\G_{*}, U^{p}_{*}}$ is 
$$M^{\circ}_{\G_{*}, U_{*}^{p}, W}:=  
\varprojlim_{r } \H^{\textrm{\'et}}_{d}(\baar{X}_{*,U^{p}_{*}U_{*,p,r}}, \cW^{\circ})^{\ord},$$
an $\OO_{L}$-module. It depends on the choice of lattice $W^{\circ}\subset W$, whereas the $L$-vector space 
$$M_{\G_{*}, U_{*}^{p}, W}:=M_{\G_{*}, U_{*}^{p}, W}^{\circ}\otimes_{\OO_{L}}L$$
does not.
When $\cW={\Q_{p}}$ is the trivial local system, we omit it from the notation, thus
$$M_{\G_{*}, U_{*}^{p}}=
M_{\G_{*}, U_{*}^{p}, {\Q_{p}}}.$$

\subsubsection{Independence of weight and Control Theorem} 
For  a $\Z_{p}$-algebra $A$, let  $\cH_{\G_{*}, U^{p}_{*}, p, A}^{\ord}:= A[T_{\G_{*}}]\otimes_{\Z_{p}[T^{+}_{\G_{*},0}]}\Lambda^{\circ}_{\G_{*}, U^{p}_{*}}$. 
For $?= S,\emptyset, {\rm sph}$, consider the $\Lambda_{\G_{*}, U^{p}_{*},A}$-algebra
\begin{align} \label{hecke}
\cH_{\G_{*}, U^{p}_{*},A}^{?, \ord}:= \cH^{?}_{\G_{*}, U^{p}_{*}, A} \otimes_{A} \cH_{\G_{*}, U^{p}_{*}, p, A}^{\ord}.
\end{align}
For every irreducible algebraic representation $W$ over $L$ and $\OO_{L}$-algebra $A$, the space $M^{\circ}_{\G_{*}, U_{*}^{p}, W}\otimes A$ is  a module over $\cH^{\ord}_{\G_{*}, U^{p}_{*},A}$, where $[t]\in A[T_{\G_{*}}^{+}]$ acts by the double coset operator $\Up_{t}$. 

The base ring $A$ will be omitted from the notation when it can be understood from the context.

Let $U_{*, r}=U^{p}_{*}U_{*, r, p}$ be as in \S~\ref{sec: not} and let $X_{*, r}:=X_{*, U_{*,r}}$.

\begin{prop} \label{free indep} Let $W$ be an irreducible right algebraic representation of $\G_{*/\Q_{p}}$ over $L$, $\cW$ the corresponding local system. Then:
\begin{enumerate} 
\item If
 $\G_{*}=\G$, $\H$, then   $M_{\G_{*}, U_{*}^{p},W}^{\circ}$ is a projective  $\Lm^{\circ}_{\G_{*}, U^{p}_{*}}\otimes \OO_{L}$-module of finite type. For all  of the groups $\G_{*}$, the  $\Lm^{\circ}_{(\G\times \H)'}\otimes \OO_{L}$-module $M_{(\G\times \H)',K^{p}, W}^{\circ}$  is of finite type, and $M_{(\G\times \H)',K^{p},W}$  is a projective $\Lambda_{(\G\times \H)'}\otimes L$-module of finite type.
\item \label{indep}  We have natural $\cH_{\G_{*}, U^{p}_{*}}^{\ord}$-equivariant isomorphisms  
\beq\label{iotaW}
j_{W}\colon M_{\G_{*}, U_{*}^{p}}\otimes \OO_{L}\cong M_{\G_{*}, U_{*}^{p},W}.  \eeq
\item Consider
\beq\label{MGr} 
M_{\G_{*}, U^{p}_{*}, W,r}:= M_{\G_{*}, U^{p}_{*}}\otimes_{\Lambda_{\G_{*}, U^{p}_{*}}}\Lambda_{\G_{*}, U^{p}_{*}, W, r}.
\eeq
There is a natural $\cH_{\G_{*}, U^{p}_{*}}^{\ord}$-equivariant isomorphism 
$$M_{\G_{*}, U_{*}, W, r} \cong \H_{d}(\baar{X}_{*, r}, \cW)^{\ord}.$$
 \end{enumerate}
\end{prop}
\begin{proof} We first treat part 1 when $W=\Q_{p}$. Then we will deal with part 2, which implies that part 1 holds for any $W$.

If $\G_{*}=\G$, the result is proved  in \cite[Thoerem 1.2, cf. also Remark 1.1]{hida luminy}.
 If $\G_{*}=\H$, $U^{p}_{*}=V^{p}$, then the module under consideration is  isomorphic to   $\Z_{p}
\llb E^{\times}\bks E_{\A^{\infty}}^{\times}/V^{p}\rrb$, which is finite free over $\Lm^{\circ}_{\H,V^{p}}=\Z_{p}\llb \baar{\OO_{E}^{\times }\cap V^{p}}\bks \OO_{E, p}^{\times}  \rrb$ as $ \baar{\OO_{E}^{\times }\cap V^{p}}\bks \OO_{E, p}^{\times}\subset E^{\times}\bks E_{\A^{\infty}}^{\times}/V^{p}$ is a subgroup of finite index.

If $\G_{*}=\G\times \H$, by the K\"{u}nneth formula we have $M^{\circ}_{\G\times \H, U^{p}\times V^{p}}= M^{\circ}_{\G, U^{p}}\hat{\otimes }    M^{\circ}_{\H, V^{p}}$, which by the previous results is a finite type projective module over $\Lm^{\circ}_{\G\times \H} = \Lm^{\circ}_{\G}\hat{\otimes} \Lm^{\circ}_{\H}$. Finally, if $\G_{*}=(\G\times \H)'$ and $K^{p}$ is the image of $U^{p}\times V^{p}$,  by the description of $Z_{K}$ in \eqref{z quot} we have 
\begin{align}\label{mgh}
M_{(\G\times\H)', K^{p}}^{\circ}= (M^{\circ}_{\G\times \H, U^{p}\times V^{p}})  /  (F_{\A^{\infty}}^{\ts}/ F^{\ts}\cd ((U^{p}\cap F_{\A^{p\infty}}^{\ts}) \cap (V^{p}\cap F_{\A^{p\infty}}^{\ts})))
\end{align}
As $M^{\circ}_{\G\times\H, U^{p}\times V^{p}}$ is a projective $\Lm^{\circ}_{\G\times \H, U^{p}\times V^{p}}$-module of finite type, the quotient
$M^{\circ}_{\G, \times\H, U^{p}\times V^{p}}/{F, p}^{\times} = M_{\G, \times\H, U^{p}N_{\G,{0}}\times V^{p}} \otimes_{\Lm^{\circ}_{\G\times \H}}\Lm^{\circ}_{(\G\times \H)'}$ is a projective  $\Lm^{\circ}_{(\G\times \H)', K^{p}}$-module of finite type, and $M^{\circ}_{(\G\times\H)',K^{p}}$ is its quotient by the free action of the finite group 
$F_{\A^{\infty}}^{\ts}/ F^{\ts}\cd  F_{p}^{\ts}((U^{p}\cap F_{\A^{p\infty}}^{\ts}) \cap (V^{p}\cap F_{\A^{p\infty}}^{\ts}))).$
After inverting $p$, the quotient map admits a section, hence  $M_{(\G\times \H)',K^{p}}$ is projective over $\Lambda_{(\G\times \H)'}$. 

We now turn to part 2. As above it suffices to prove the result  when $\G_{*}=\G, \H$.  Let $\G_{*}=\G$,  and suppose that $W=W_{\G, \uw}^{*}$.
Let $W^{\circ}\subset W$  be the lattice of \eqref{homog poly lattice}, $r\geq 1$. We have a $\Lm_{r}$-linear map
\beq\label{jWr}
j_{W, r}\colon H_{1}(\baar{X}_{r  }, \Z/p^{r}\Z)^{\ord}
\ot_{\Lm_{r}}   W^{\circ,N_{0}}/p^{r}
\to H_{1}(\baar{X}_{r}, \cW^{\circ}/p^{r}\cW^{\circ})
\eeq
induced by cap product\footnote{I am grateful to David Loeffler and Sarah Zerbes for explaining to me this point of view on the Control Theorem.} via the isomorphism of $\Lm_{r}$-modules $ H^{0}(\baar{X}_{r}, 
\cW^{\circ}/p^{r} \cW^{\circ})\cong W^{\circ,N_{0}}/p^{r}$

The maps \eqref{jWr}  are compatible with variation in $r$, and taking limits  we obtain the map \eqref{iotaW}, which Hida  (\cite[\S~ 8]{Hida88},   \cite[Theorem 2.4]{hida luminy}) proved to be an isomorphism; the asserted equivariance  properties are clear from the construction.

When $\G_{*}=\H$ the construction is similar but  easier, as each $W$ is $1$-dimensional and each of the analogous maps $j_{W,r}$ is an isomorphism.

Finally, we address part 3. As above we may reduce to the case $W=\Q_{p}$ and $\G_{*}=\H$ or $\G_{*}=\G$. The former is clear, and the latter is, in view of part 2, equivalent to the statement
$$M_{\G, U^{p}, \cW}\otimes_{\Lm_{\G, U^{p}}} \Lm_{\G, U^{p}, r}\cong \H_{d}(\baar{X}_{r},  \cW)^{\circ},$$
which is the control theorem of 
 \cite[Theorem 1.2 (3)]{hida luminy}.
\end{proof}

\subsubsection{Ordinary eigenvarieties}

The space  $M^{\circ}_{\G_{*}, U^{p}_{*}}$ has the structure of an $\cH_{\G_{*}, U^{p}_{*}}^{ \ord}$-module (in particular of $\Lambda^{\circ}_{\G_{*}, U^{p}_{*}}$-module), and for $?=\emptyset, {\rm sph}$ and $A$ a $\Z_{p}$-algebra,   we let
$${\bf T}_{\G_{*}, U^{p}_{*}, A}^{{\rm sph}, \ord}$$
   be the image of  $\cH_{\G_{*}, U^{p}_{*}, A}^{{\rm ?}, \ord}$ in $\End_{A}(M^{\circ}_{\G_{*}, U^{p}_{*}}\otimes A)$,  that  is independent of the particular spherical Hecke algebra chosen when $?={\rm sph}$.  When $A=\Z_{p}$ we omit it from the notation.

We may now define 
  $$\cE^{\ord}_{\G_{*}, U_{*}^{p}} :=\Spec {\bf T}_{\G_{*}, U^{p}_{*}, \Q_{p}}^{{\rm sph}, \ord}.$$
  When $\G_{*}=\H$, we will omit the superscript `${\ord}$'. 
 
Let
   $$\kappa_{\G_{*}}\colon \cE^{\ord}_{\G_{*}, U_{*}^{p}}  \to  \mathfrak{W}_{\G_{*}, U^{p}_{*}}  .$$ 
Referring to Definition \ref{clwt}, the zero-dimensional  (ind)-subscheme of classical points (respectively classical points of weight $W$, for an algebraic representation $W$ of $\G_{*}$, respectively classical points of weight $W$ and level $r$) is 
  $$\cE^{\ord, \cl, (W)}_{\G_{*}, U^{p}_{*}, (r)} := \mathfrak{W}_{\G_{*}, U^{p}_{*},(r)}^{\cl, (W)}     \ts_{ \mathfrak{W}_{\G_{*}, U^{p}_{*}}, \kappa} \cE^{\ord}_{\G_{*}, U_{*}^{p}}. $$
  
  We denote by 
  $$\cM_{\G_{*}, U^{p}_{*}}$$
  the sheaf on $ \cE^{\ord}_{\G_{*}, U_{*}^{p}}$ corresponding to $M_{\G_{*}, U^{p}_{*}}$. 
\begin{enonce*}{Notation} \textup{When $\G_{*}= (\G\times \H)'$, we omit the subscripts, thus e.g. for $K^{p}\subset  (\G\times \H)'(\A^{p\infty})$ 
we write
  $$\cE_{K^{p}}^{\ord}:= \cE_{(\G\times \H)', K^{p}}^{\ord}.$$}
\end{enonce*}

By \eqref{mgh},
  ${\bf T}_{K^{p}}^{{\rm sph}, \ord}$ is a quotient of ${\bf T}_{\G\times \H, U^{p}\times V^{p}}^{{\rm sph}, \ord}$ and correspondingly we have a closed immersion
\begin{align}\label{immers}
\cE_{K^{p}}^{\ord}\into \cE_{\G\times\H, U^{p}\times V^{p}}^{\ord}.
\end{align}

\begin{prop}\label{semi-local} 
The ring $ {\bf T}_{\G_{*}, U^{p}_{*} }^{{\rm sph}, \ord}$ is finite  flat
 over $\Lambda^{\circ}_{\G_{*}, U^{p}_{*}}$, hence semi-local. 
  The  maximal ideals of $ {\bf T}_{\G_{*}, U^{p}}^{{\rm sph}, \ord} $ are in bijection with $G_{{\bf F}_{p}}$-orbits of  characters $\baar{\lambda}\colon   
 {\bf T}_{\G_{*}, U^{p}_{*}}^{{\rm sph}, \ord}\to \baar{\bf F}_{p}$. 
  \end{prop}
 \begin{proof}
The first statement is easy for the group $\H$ and it is  proved in \cite{hida luminy} for the group $\G$. Together they imply the statement for $\G\times \H$ and hence $(\G\times \H)'$.  
As  $ {\bf T}_{\G_{*}, U^{p}}^{{\rm sph}, \ord} $  is topologically finitely generated over $\Z_{p}$, the residue fields of its maximal ideals are finite extensions of ${\bf F}_{p}$; this implies the second statement.
\end{proof}

\begin{lemm}\label{dense}
Let $W$ be an irreducible algebraic representation of $\G_{*}$. The set $\cE^{\ord, \cl, W}_{\G*, U_{*}}$
 of classical points of weight $W$
 is Zariski-dense in $\cE^{\ord}_{G_{*}, U_{*}^{p}}$.
\end{lemm}
\begin{proof}
By the previous proposition, the map $\kappa_{\G_{*}}$ is finite hence closed. Then the Zariski-density of
 $\cE^{\ord, \cl, W}_{\G*, U_{*}}=\kappa_{\G_{*}}^{-1} ( \mathfrak{W}_{\G_{*}, U^{p}_{*}}^{W})$ reduces to the Zariski-density of $\mathfrak{W}_{\G_{*}, U^{p}_{*}}^{W} \subset \mathfrak{W}_{\G_{*}, U^{p}_{*}}$, which follows from \eqref{alg density}; cf. aso \cite[Lemma 3.8]{SW99}.
 \end{proof}

\subsubsection{Abelian case}
The structure  of the eigenvariety for the abelian groups $\H:={\rm Res}_{E/\Q}\G_{m}$ and ${\rm Z}={\rm Res}_{F/\Q}\G_{m}$
 is very simple, and we make it explicit  for the group $\H$: we have
$$M_{\H, V^{p}}:= \hat H_{0}(\baar{Y}_{V^{p}})=  \Z_{p}\llb Y_{V^{p}}( \baar{E})\rrb\otimes\Q_{p},$$
the set $ Y_{V^{p}}(\baar{E})$ is a principal homogeneous space for $\Gamma_{E,V^{p}}:=\H(\Q)\bks \H(\A^{\infty})/V^{p}= E^{\times}\bks E_{\A^{\infty}}^{\times}/V^{p}$, and 
$$\cE_{\H, V^{p}}= \Spec \Z_{p}\llb{\Gamma_{E, V^{p}}}\rrb_{\Q_{p}}.$$
(We omit the superscript `$\ord$' which is meaningless here.)
The  classical points 
 $\cE_{\H, V^{p}}^{\cl}\subset \cE_{\H, V^{p}}(\baar\Q_{p})$ parametrise locally algebraic characters of $\Gamma_{E, V^{p}}$.
Finally, the sheaf $\cM_{\H, V^{p}}$ is a trivial line bundle, with actions  by $G_{E}$ given by the universal character 
\begin{align}\label{chiuniv}
\chi_{\rm univ}\colon G_{E}\to \Gamma_{E, V^{p}}\to\Z_{p}\llb\Gamma_{E}\rrb^{\times},
\end{align} and by  $\H(\A^{\infty})$  given by the inverse  $\chi_{\H, \rm univ}^{-1}$ of the  corresponding automorphic character. We may formally write 
\begin{align}\label{chiHuniv}
\cM_{\H, V^{p}}= \chi_{\H, \rm univ}^{-1}\otimes \chi_{\rm univ}
\end{align}
as a tensor product of two trivial sheaves, the first one endowed with the $\H(\A^{})$-action only, and the second one with the Galois action by $ \chi_{\rm univ}$ only.%

\subsubsection{Fibres of the sheaves $\cM$} 
Let 
$$(\lambda^{p},\lambda_{p})\colon {\bf T}_{\G, U^{p}}^{{\rm sph}, \ord}\to \OO(\cE_{\G, U^{p}})$$ be the tautological character, 
 and define
  \beq
\begin{aligned}\label{alphauniv}
\alpha^{\circ}\colon  F_{p}^{\times}&\to \OO(\cE_{\G, U^{p}})^{\times}\\
x&\mapsto \lambda_{p}(\Up_{x}).
\end{aligned}
\eeq
\begin{prop}\label{fibreM2} 
Let $x\in \cE^{\ord, \cl,W}_{\G, U^{p}}$  be a classical point  of weight $W$ and level $r$. 
 Then:
\begin{enumerate}
\item  Let $U:=U^{p}U_{p,r}$, and let $\cW$ be the local system on $X$ associated with $W$. 
 We have an isomorphism of $\cH^{\ord}_{\G, U^{p}, \Q_{p}(x)}$-modules
$$(\cM_{\G, U^{p}})_{x}\cong  \H_{1}(X_{U, \baar F}, \cW_{\G, \uw})^{\ord}\otimes_{\cH^{\rm sph}_{\G, U^{p}}, \lambda_{x}} \Q_{p}(x).$$
\item 
There exists  a  unique  automorphic representation ${\pi}_{x}$ of $\G(\A^{})$  over ${\Q}_{p}(x)$ of spherical character $\lambda_{x}^{p}$, weight $W$, and unit character $\alpha_{x}^{\circ}$. It satisfies the property
\beq\label{property}
\pi_{x}^{ {\ord,U^{p}}}
\cong
\Hom_{\Q_{p}(x)[G_{F, S}]}((\cM_{\G, U^{p}})_{x} ,
 \rho_{x})
\eeq
as left $\cH^{\ord}_{\G, U}\otimes \Q_{p}(x)$-modules.
\end{enumerate}\end{prop}
\begin{proof}
Part 1 follows from Proposition \ref{free indep}.3. 

For part 2, 
fix an embedding $\Q_{p}(x)\into \baar{\Q}_{p}$. By strong multiplicity-one, a representation $\overline{\pi}$ over $\baar{\Q}_{p}$ with character $\lambda_{x}^{p}$ is unique if it exists. By comparing part 1 with \eqref{carayol iso 2}, we find that $\baar{\pi}$ exists and that for such $\baar{\pi}$  property \eqref{property}  holds after base-change to $\baar{\Q}_{p}$. Let $V_{\baar{\pi}}$ be the Galois representation associated with $\pi$ by Theorem \ref{galois for hilbert}, then by looking at Frobenius traces, we see that  $V_{\baar{\pi}}$  has a model $V_{\pi}$ over $\Q_{p}(x)$. It follows again from \eqref{carayol iso 2} that $\pi :=\lim_{U'}\Hom(H_{1}(\baar{X}_{U}, \cW), V_{\pi}) $ is a model of $\baar{\pi}$ that satisfies \eqref{property}.
\end{proof}

In the rest of the paper, we will use without further comment the notation $\pi_{x}^{}$ for the
 representation of $\G(\A^{})$ defined above, for  $x\in \cE_{\G, U^{p}}^{\ord,\cl }$.

\begin{coro}\label{fibreMad}
Let $z\in \cE_{K^{p}}^{\ord, \cl}$ be a classical point, and write $z=(x,y)$ via \eqref{immers} and $L:=\Q_{p}(z)$. Let $\omega_{x}$ be the central character of $\pi_{x}$, let $\chi_{\H,y}$ be the  character of $\H(\A^{\infty})$ obtained by specialising $\chi_{\H, \rm univ}$, and let $\chi_{y}$ be the corresponding locally algebraic character of $G_{E,S}$. Write $L:=\Q_{p}(z)$. Then $\omega_{z}:=\chi_{y|F_{\A^{\infty, \times}}}\omega_{x} = \one $, and 
$$\cM_{K^{p},{z}}\cong
(  (\pi_{x}^{ p,\vee, })^{U^{p}}\otimes_{L}\chi_{\H,y}^{-1,p})  \otimes_{L} (V_{x|_{G_{E,S}}}\otimes_{L[G_{E, S]}} \chi_{y})
$$
as $\cH^{K_{S}}_{S, L}\otimes_{L} L[G_{E,S}]$-modules. Here, $G_{E, S}$ acts trivially on the first two tensor factors, and the natural action of 
 $\cH_{\G\ts\H}^{U_{S}\ts V_{S}}$ 
on the first two factors is extended trivially to the whole tensor product, and  it factors to an action of $\cH^{K_{S}}_{S}$.
\end{coro}

\begin{proof}
Let  $\lambda_{x, F}^{p}$ and $\lambda_{y, F}^{p}$ be the restrictions of the characters $\lambda_{x}$, $\lambda_{y}$ to $\Z[F^{\ts}_{\A^{Sp\infty, \times}}/K_{F}^{Sp}]$, and let $\lambda_{F}$ be the restriction of $\lambda_{F,x}\lambda_{F,y}$ to $\Delta'=F^{\ts}_{\A^{Sp\infty, \times}}/K_{F}^{Sp}$. As this groups acts trivially on $M_{K^{p}}$ by \eqref{mgh}, we have $\lambda_{F}=\one$. On the other hand $\lambda_{F}$ equals the restriction of $\omega_{z}$ to $\Delta'$. We deduce that $\omega_{z}$ factors through $C=F^{\ts}_{\A^{\infty}}/F^{\times}F_{\A^{Sp\infty}}K_{F, S}K_{F,p}$ for some open compact $K_{F,p}\subset F_{p}^{\times}$. By weak approximation, $C=\{1\}$, therefore $\omega_{z}=\one$.

By Proposition \ref{fibreM2}, \eqref{chiHuniv}, and \eqref{mgh}, the asserted result holds provided we quotient the right-hand side by the action of $F_{\A^{\infty}}^{\ts}$,  however  this group acts  by $\omega_{z}$, hence trivially.
\end{proof}

\begin{prop}\label{etale} The natural map $\kappa\colon \cE^{\ord}_{K^{p}} \to \mathfrak{W}_{K^{p}}$ is \'etale over a neighbourhood of  the classical  points in $\mathfrak{W}^{\cl}_{K^{p}}$. In particular, the space $\cE^{\ord}_{K^{p}}$ is regular at all $z\in \cE^{\ord, \cl}_{K^{p}}$.
\end{prop}
\begin{proof} As $\kappa $ is finite  flat by Proposition \ref{semi-local}, it suffices to check that the fibre of $\kappa$ over any $x\in \mathfrak{W}^{\cl}_{K^{p}}(\baar{\Q}_{p})$ is isomorphic to  $\baar{\Q}_{p}^{m}$ for some $m$. By \ref{free indep}, \ref{MGr} and \eqref{carayol iso 2}, this fibre  is the spectrum of the image $A_{x}$ of $\mathscr{H}_{K^{p}}^{{\rm sph}, \ord}$
in   $\bigoplus_{z\in \kappa^{-1}(x)} (\Pi_{z}^{K^{p}, \ord})^{\oplus 2}$, where the $\Pi_{z}$ form   a list  of distinct  irreducible representations of $(\G\times \H)'(\A^{p\infty})$ over $\baar{\Q}_{p}$. By strong multiplicity-one, we have $A_{x}\cong \oplus_{z}\baar{\Q}_{p}$. This proves \'etaleness. As $\mathfrak{W}_{K^{p}}$ is regular, we deuce that so is $\cE_{K^{p}}^{\ord}$ in a neighbourhood of classical points.
\end{proof}

\subsection{Galois representations in families} 

We recall the existence of a  universal family of Galois representations over $\X$.
\subsubsection{Representations associated with irreducible pseudocharacters}\lb{3.2new} Recall that an $n$-dimensional   pseudocharacterof $G$ over  a scheme $\X$ is a function $T\colon G\to \OO(\X)$  that `looks like' the trace of an $n$-dimensional representation of $G$ over $\OO(\X)$, see  \cite{rou} for the precise definition. A pseudocharacter $T$ is said to be (absolutely) irreducible  at a point $x\in \X$ if, for any (equivalently, all) geometric point $\baar{x}$ of $\X$ with image $x$, the pullback $\baar{x}^{*}T$  is not the sum of two  pseudocharacters of dimensions $k$, $n-k$ with $0< k<n$. The \emph{irreducibility locus} of $T$ is the set of points of $\X$ at which $T$ is irreducible; it is open  (\cite[\S~7.2.3]{chenevier}).
 
We start by proving that, if $T$ is irreducible, a  representation with trace $T$ is essentially unique when it exists.

\begin{lemm} Let $\X$ be an integral scheme and let  $\cV_{1}$, $\cV_{2}$ be vector bundles of rank $n>0$ over $\X$. Suppose that there is an isomorphism $F\colon \End_{\OO_{\X}}(\cV_{1})\to \End_{\OO_{\X}}(\cV_{2})$. Then there is an invertible $\OO_{\X}$-module $\calL$ and an isomorphism  
$$g \colon \cV_{1}\cong \cV_{2}\otimes \calL$$ inducing $F$ in the sense that 
$F(T)\otimes{\rm id}_{\calL} =g Tg^{-1}$
for all sections $T$ of $\End_{\OO_{\X}}(\cV_{1})$.
\end{lemm}
\begin{proof}
By \cite[Ch. IV]{knus}, any automorphism of an Azumaya  algebra (such as $ \End_{\OO_{\X}}(\cV_{i})$) is Zariski-locally inner.  Therefore there exists an open cover $\{U_{i}\}$ of $\X$ and isomorphisms $g_{i}\colon \Gamma(U_{i}, \cV_{1})\to \Gamma(U_{i}, \cV_{2})$ such that 
\begin{align}
\label{pippo}
F(T)=g_{i}Tg_{i}^{-1}\end{align} for all $T \in \End_{\OO_{\X}(U_{i})}(V_{1})$. Let $U_{ij}:=U_{i}\cap U_{j}$ and
\begin{align}\label{cij}
c_{ij}:=g_{i}^{-1}g_{j},
\end{align}
an automorphism of $\cV_{1}$ over $U_{ij}$.  By \eqref{pippo},  $c_{ij} $ commutes with every $T \in \End_{\OO_{\X}(U_{ij})}(\cV_{1})$, hence it is a scalar in $\OO_{\X}(U_{ij})^{\times}$. One verifies easily that the $c_{ij}$ form a cocycle in $H^{1}(\X,\OO_{\X}^{\times})$. Let $\calL$ denote the associated invertible sheaf, which is trivialised by  the cover $\{U_{i}\}$. Then we may view  $g_{i}\colon \Gamma(U_{i}, \cV_{1})\to \Gamma(U_{i}, \cV_{2}\otimes \calL)$. By \eqref{cij}, the $g_{i}$ glue to the desired isomorphism $g\colon \cV_{1}\cong \cV_{2}\otimes \calL$.
\end{proof}

\begin{lemm}\label{uniq}
Let $\X$ be an integral scheme and $\mathscr{T}\colon G_{F, S}\to \OO(\X)$ an irreducible pseudocharacter of dimension $n$. Let $\cV_{1}$, $\cV_{2}$ be representations of $G_{F,S} $ with trace $\mathscr{T}$. Then
  there exist    a  line bundle $\mathscr{L}$ with trivial Galois action and a $G_{F, S}$-equivariant isomorphism
 $$\cV_{1}\cong \cV_{2} \otimes \mathscr{L}.$$
\end{lemm}
\begin{proof}
Write $G=G_{F,S}$ and let $\calA:=\OO_{\X}[G]/\Ker(\mathscr{T})$. By \cite[Theorem 5.1]{rou}, $\mathscr{A}$ is an Azumaya algebra of rank $4$. By \cite[Corollary 2.9]{saltman}, 
the two natural injective maps $\alpha_{i}\colon \calA\to {\End}_{\OO_{\X}}(\cV_{i})$ are isomorphisms. Then we conclude by the previous lemma.
\end{proof}

\subsubsection{Galois representations in ordinary families} We prove the analogue in Hida families of Theorem \ref{galois for hilbert}.
\begin{lemm}\label{exist rhobar}
Let $\overline{\lambda}\colon {\bf T}_{\G, U^{p}}^{\rm sph,ord} \to \baar{\bf  F}_{p}$ be a character. Then there is a unique semisimple representation $\rhobar\colon G_{F,S}\to \GL_{2}(\baar{\bf F}_{p})$
such that $\Tr(\rhobar({\rm Fr}_{v}))=q_{v}^{-1}\baar{\lambda}(T_{v})$  for all $v\notin S$.
\end{lemm}
\begin{proof}
The existence follows by lifting $\baar{\lambda}$ to the character $\lambda_{x}$ associated with a classical point $x$ (that is possible thanks to Lemma \ref{dense}), then taking the semisimplification of the reduction modulo $p$ of a lattice in the representation  $\rho_{x}:=\rho_{\pi_{x}}$ of Theorem  \ref{galois for hilbert};   the uniqueness is a consequence of   the Brauer--Nesbitt theorem.
\end{proof}

By Proposition \ref{semi-local} we may decompose ${\bf T}_{\G, U^{p}}^{{\rm sph}, \ord}\cong
 \prod_{\mathfrak{m}}
 {\bf T}_{\G, U^{p}, \frak m}^{{\rm sph} \ord}$
 and consequently  the generic fibre of the associated schemes also decomposes as 
\begin{align}\label{decomp}
\cE_{\G, U^{p}}^{\ord} \cong \coprod
\cE_{\G, U^{p}, \frak m}^{\ord}.
\end{align}
We will  say that a connected subset  $\X \subset \cE_{\G, U^{p}}^{\ord}$  has residual representation $\rhobar\colon G_{F}\to \GL_{2}(\baar{\bf F}_{p})$ if $\X$ is contained in some $\cE_{\G, U^{p}, \frak m}^{\ord}$ such that  the character $\lambda_{\frak m}\otimes_{{\bf F}_{p}(\frak m)}\baar{\bf F}_{p} $ associated with $\frak m$ is the character of $\rhobar$.

\begin{prop}\label{galois2}
  Let $\X_{\G}$ be  an irreducible component of $\cE_{\G}$  (that is, a {Hida family}).
 Then there exist:
 \begin{itemize}
 \item
an open subset $\X_{\G}'\subset \X_{\G}$ containing $\X_{\G}^{\rm cl}:=\X_{\G}\cap\cE_{\G, U^{p}}^{\cl}$;
\item a locally free $\OO_{\X_{\G}'}$-module $\cV_{\G}$ of rank two over $\X'_{\G}$, such that $$\cV_{\G, x}\cong V_{\pi_{x}}$$ for all $x\in \X_{\G}^{\cl}$;
\item  a  filtration
\begin{gather}\label{decatv} 0\to \cV_{\G, v}^{+}\to \cV_{\G, v}\to\cV_{\G,v}^{-}\to0,\end{gather}
where the $\cV_{\G,v}^{\pm}$ are locally free $\OO_{\X'_{\G}}$-modules of rank $1$, and $G_{F_{v}}$ acts on $\cV_{\G, v}^{+}$ by the character associated, via local class field theory, with the character 
\beq\lb{alpha gal}
\alpha_{|F_{v}^{\ts}}^{\circ} \langle \ \rangle_{F_{v}}\eeq
deduced from \eqref{alphauniv}.
 \end{itemize}

The representation $\cV_{\G}$ is uniquely determined up to automorphisms and twisting by line bundles with trivial Galois action.
\end{prop}
The result is due to Hida and Wiles (\cite[\S~ 3.2.3]{fouquet} and references therein), except for the existence of \eqref{decatv} when  the residual Galois representation of $\X_{\G}$ is reducible.
\begin{proof}
Let $\mathscr{T}\colon G_{F, S}\to \OO(\X_{\G})$ be the pseudocharacter defined by $\mathscr{T}({\rm Fr}_{v})=q_{v}^{-1}\lambda(T_{v})$, where $\lambda\colon {\bf T}_{\G, U^{p}}^{\rm sph}\to \OO(\X_{\G})$ is the tautological character. Let $\X_{\G}  ^{\rm irr}\subset \X_{\G}$ be the (open) irreducibility locus. By Theorem \ref{galois for hilbert}, $\X_{\G}^{\cl}\subset \X_{\G}^{\rm irr}$. By Lemma \ref{uniq}, a representation $\cV_{\G}$ is unique up to Galois-trivial twists if it exists. We show existence.

By \cite[Theorem 5.1]{rou}, $\mathscr{A}:=\OO_{\X_{\G}^{\rm irr}}[G_{F, S}]/ \Ker(\mathscr{T})$ is an Azumaya algebra of rank $4$ over $\cE_{\G'}$ and the natural map 
$$\rho\colon G_{F,S}\to \mathscr{A}^{\times}$$ satisfies ${\rm Tr}\circ \rho= \mathscr{T}$ (where $\Tr$ is the reduced trace of $\mathscr A$). Let $c\in G_{F, S}$ be a complex conjugation; we have an isomorphism $\mathscr A={\mathscr A} (\rho(c)-1)\oplus {\mathscr A}(\rho(c)+1)=:\cV_{+1}\oplus \cV_{-1}$. Each of the $c$-eigen-summands $\cV_{\pm1}$ is a locally free $\OO_{\X_{\G}^{\rm irr}}$-module (since so is $\mathscr A$), whose rank is  $2$: indeed at any classical geometric point  $x\in \X_{\G}^{\cl}(\bC_{p})$, the specialisation $\rho_{x}$ is odd, hence we can pick an isomorphism ${\mathscr A}_{x}\cong M_{2}(\bC_{p})$ sending $\rho_{x}(c)$ to $\smalltwomat 1{}{}{-1}$ from which it is immediate that $\cV_{\pm 1, x}$ has rank $2$; since classical points are dense, we conclude that $\cV_{\pm 1}$ also has rank $2$. 

 Let $\cV_{\G}$ be either of $\cV_{\rm 1, \pm}$. By \cite[Corollary 2.9 (a)]{saltman}, the natural map 
 $${\mathscr A}\to \End_{\OO_{\X^{\rm irr}_{\G}}}(\cV_{\G})$$ 
 is an isomorphism; we view it as an identification to obtain a representation $\rho'$ with trace $\mathscr{T}$. As an irreducible  $2$-dimensional Galois representation over a field is uniquely determined by its trace, the representation $\cV_{\G, x}$ is isomorphic to $V_{\pi_{x}}$.

We now show the existence of the filtration up  to further restricting the base. Fix a place $v\vert p$ of $F$, and let  $\det_{v}\colon G_{F_{v}}\to \OO(\X_{\G}^{\rm irr})^{\ts}$ be the character giving the action on $\det \cV_{\G,v}$. Let $\cV_{0,v}^{+}$ be the trivial  sheaf $\OO_{\X^{\rm irr}_{\G}}$ with $G_{F_{v}}$-action by the character \eqref{alpha gal}, $\cV_{0,v}:=\cV_{\G,v}$, $\cV_{0,v}^{-}:=(\cV_{0,v}^{+})^{-1}(\det_{v})$.  Finally, for $?=+, -. \emptyset$, let 
$$\cW_{v}^{?}:= \Hom_{\OO_{\X'_{\G}}}(\cV_{0,v}^{+}, \cV_{0,v}^{?}).$$ Then for all $x\in \X_{\G}^{\rm cl}$, by Theorem \ref{galois for hilbert} we have exact sequences 
\beq \lb{cW ex seq}
0\to \cW_{v,x}^{+}= \Q_{p}(x) \to \cW_{x}\to \cW_{v,x}^{-},
\eeq
which we  wish to extend to a neighbourhood of $\X_{\G}^{\cl}$. 
From a  consideration of weights based on 
Theorem \ref{galois for hilbert}, we see that for all $x\in \X_{\G}^{\cl}$,  $H^{0}({F_{v}}, \cW_{v,x}^{-})= H^{2}(F_{v}, \cW_{v,x}^{-})=0$.
 Then from \eqref{cW ex seq} we deduce 
\beq\lb{h2va}
H^{2}(F_{v}, \cW_{v,x})=0\eeq
for all
$x\in \X_{\G}^{\cl }$,
 and from  the Euler--Poincar\'e formula and \eqref{cW ex seq} we deduce that  \beq\lb{consth1}\dim_{\Q_{p}(x)} H^{1}(F_{v}, \cW_{v,x}^{-})= 1+[F_{v}:\Q_{p}], \qquad H^{1}(F_{v}, \cW_{v,x})= 1+2[F_{v}:\Q_{p}]    \eeq 
for all $x\in \X_{\G}^{\cl }.$

By Proposition \ref{torss}.3 below, \eqref{h2va} and \eqref{consth1} imply that  the natural map 
$$H^{0}(F_{v}, \cW)\ot_{\OO_{\X^{\rm irr}_{\G}}}\Q_{p}(x)\to H^{0}(F_{v}, \cW_{v,x})\cong \Q_{p}(x)$$
is an isomorphism for all $x\in \X_{\G}^{\rm cl}$. Hence the sheaf $\calL:= H^{0}(F_{v}, \cW_{v})$ is locally free of rank one  in a neighborhood $\X_{\G}'\subset \X_{\G}^{\rm irr}$ of $\X_{\G}^{\cl}$. Defining 
$$\cV_{\G, v}^{+}:=\calL\ot_{\OO_{\X'_{\G}}} \cV^{+}_{0,v}, $$
the natural map $\cV_{\G, v}^{+}\to \cV_{\G,v|\X_{\G}'}$ is injective, and its cokernel $\cV_{\G,v}^{-}$ has rank one at each $x\in \X_{\G}^{\cl}$. Up to further restricting $\X_{\G}'$, $\cV_{\G,v}^{-}$ is also locally free of rank one. It follows immediately from the construction that the exact sequence 
$$0\to \cV_{\G, v}^{+}\to \cV_{\G, v}\to\cV_{\G,v}^{-}\to 0$$ 
has the asserted properties.
\end{proof}

\begin{prop}\label{schur} In the situation of Proposition \ref{galois2},  
the natural injective map
$$i\colon \OO_{\X'_{\G}}\to \End_{\OO_{\X'_{\G}}[G_{F, S}]}(\cV_{\G})$$
is an isomorphism over an open subset $\X_{\G}''\subset \X_{\G}$ containing $\X^{\cl}$.
\end{prop}
\begin{proof}
By Theorem \ref{galois for hilbert}, $\rho_{x}$ is absolutely irreducible for all $x\in \X_{\G}^{\cl}$. 
We deduce that for each $x\in \X_{\G}^{\cl}$, the map 
$i_{x}$ is an isomorphism. Then we may take for $\X''_{\G}$ the open complement of the support of  ${\rm Coker}(i)$. 
\end{proof}

\subsection{Universal ordinary representation and local-global compatibility}\label{3.2}
The idealised description of what is achieved in this subsection would be to  define a universal ordinary automorphic representation of $\G(\A^{\infty})$  over an irreducible component $\X$ of $\cE_{\G}^{\ord}$; then show that it decomposes as the product of the representations of the local groups $\B_{v}^{\times}$,
for $v\nmid p$,\footnote{The action of $\B_{p}^{\times}$ has already been traded for an action of the torus, subsumed into the $\cE_{\G}$-module structure}     associated to $\cV|_{G_{F,v}}$ by a local Langlands correspondence in  families. 
The definition should be an elaboration of
\begin{align}\label{fakejpi}
``\Pi_{\G}:=\lim_{U^{p}{}'}\underline{\Hom}_{\OO_{\X}[G_{F, E}]}(\cM_{U^{p}{}'}, \cV_{\G})\text{''}.
\end{align}
For technical reasons, a few modifications are necessary: 
\begin{itemize}
\item the local  Langlands correspondence in families  is not defined for the unit groups of division algebras;\footnote{There is an essential reason for this, namely the possible presence of Schur indices in representations of those groups.}
therefore we  ``remove'' the components at the ramification primes $\Sigma$ of $\B$, in the following way: we consider a component of $\cE$ rather than $\cE_{\G}$, and we take ${H_{\Sg}'}$-coinvariants in an analogue $\Pi$ of \eqref{fakejpi}. For sufficiently large levels, this isolates a local factor of $\Pi$ that is generically free of rank one along  locally distinguished Hida families;
\item in the limit in \eqref{fakejpi}, we fix an arbitrarily large finite set of primes $\Sg'$, disjoint from $\Sigma$ and from $S_{p}$, and we let only the $\Sg'$-component of  $U^{p}{}'$  shrink, so as to get a representation of $\B_{\Sg'}^{\times}$;  
\item 
 we replace the abstractly constructed $\cV=\cV_{\G}\otimes \cV_{\H}$ (where $\cV_{\H}=\chi_{\rm univ}$) by a more geometric incarnation using the sheaf $\cM$ in `new' level (with respect to the chosen irreducible component).
\end{itemize}

We use the correspondence studied in \cite{LLC}, with the caveat that  strictly speaking  the normalisation chosen there differs by the one fixed here in Theorem \ref{galois for hilbert}.1 by a Tate twist. This is only a matter of book-keeping, and in order to avoid excessive notational burden, we do not signal such Tate twists when referring to the results of \cite{LLC} in the rest of this paper.

\subsubsection{Irreducible components}
Let $\X_{\G}\subset \cE^{\ord}_{\G, U^{p}}$ be an irreducible component.  Fix a place $v$ of $F$ not in $\Sigma\cup S_{p}$.
 Recall that the $v$-level of a representation $\pi_{v}$ of $\GL_{2}(F_{v})$ is the smallest $m$ such that $\pi_{v}^{U_{1}(\vpi_{v}^{m})}\neq 0$, where $U_{1}(\vpi_{v}^{m})= \{\smalltwomat abcd \in \GL_{2}(\OO_{F,v})\, :\, c\equiv d-1\equiv 0 \pmod{\vpi_{v}^{m}\OO_{F,v}}\}$.  Let $m_{x,v}$ be the $v$-level of $x\in\X^{\cl}$.
\begin{lemm} The function $x\mapsto m_{x,v}$ is constant on $\X_{\G}^{\cl}$.
\end{lemm}
\begin{proof} 
By \cite{carayol-h}, $m_{x,v}$ equals the conductor of the $G_{F_{v}}$-representation $\cV_{x}$; as all those Galois representations are pure, we may conclude by  \cite[Theorem 3.4]{saha-cond}. 
\end{proof}
 
 We may then define the $v$-level $m_{v}$ of $\X_{\G}$ to be the common value of the $m_{x,v}$ for $x\in \X^{\cl}$.  By the following lemma, it is not restrictive to make the following assumption: \emph{for all $v\notin \Sigma \cup S_{p}$, we have $U_{v}=U_{1}{(\vpi}_{v}^{m_{v}(\X_{\G})})$}. (We say that $\X_{\G}$ is a $v$-new component of $\cE_{\G, U^{p}}^{\ord}$.)
 
\begin{lemm}\label{change level} Let ${}'\X_{\G}\subset \cE^{\ord}_{\G, U^{p}{}'}$ be an irreducible component, and suppose that $U^{p}{}'=\prod_{v\nmid p} U_{v}'$. Let $m_{v}$ be the level of $\X_{\G}$ and let $U^{p}=\prod_{v\nmid p} U_{v}$, with $U_{v}=U_{1}(\vpi_{v}^{m_{v}})\supset U_{v}'$ for all $v\notin \Sigma\cup S_{p}$, and $U_{v}=U_{v}'$.  There exists a unique irreducible component $\X_{\G}\subset \cE^{\ord}_{\G, U^{p}}$ whose image under the natural embedding $\cE^{\ord}_{\G, U^{p}}\subset \cE^{\ord}_{\G, U^{p}{}'}$ is ${}'\X_{\G}$.
\end{lemm}
\begin{proof} Let $x'\in {}'\X_{\G}^{\cl}$ be any classical point. By \cite{carayol},  its level (that is, the level of $\pi_{x,v}$) is $m_{v}$ if and only if  $\pi_{x,v}$ already occurs in the cohomology of $X_{\baar{F}}$ at $v$-level $m_{v}$, equivalently if and only if (the system of Hecke- and $\Up_{v}$-eigenvalues associated with) $\pi_{x,v}$ occurs in a quotient of $\cM_{U^{p}}$; that is, if $x'$ comes from a point $x$ of $\cE_{\G, U^{p}}$. Let $\X_{\G}\subset \cE^{\ord}_{\G, U^{p}}$ be the irreducible component containing $x$, which is unique by Proposition \ref{etale}. As $\cE^{\ord}_{\G, U^{p}}\subset \cE^{\ord}_{\G, U^{p}{}'}$ are equidimensional of the same dimension, the image of $\X_{\G}$  in $\cE_{\G, U^{p}{}'}$ is an irreducible component, necessarily ${}'\X_{\G}$. 
\end{proof}

We now deal with the level at $\Sigma$.
\begin{lemm}\label{levSig}
 Let $v\in \Sigma$. There exists a compact open $U_{v}'\subset U_{v}$ such that for every classical point $x'\in \X_{\G}$, we have $\pi_{x',v}^{U'_{v}}=\pi_{x',v}$,
where $\pi_{x,v}$ is the local component at $v$ of $\pi_{x}\otimes \baar{\Q}_{p}$.
\end{lemm}
\begin{proof} 
Fix a classical point $x\in \X_{\G} \subset \cE^{\ord}_{\G, U^{p}}$, and let $U_{v}'\subset U_{v}$ be such that
 $\pi_{x,v}^{U'_{v}}=\pi_{x,v}$. (This will hold for sufficiently small $U_{v}'$ as $\pi_{x,v}$ is finite-dimensional.) We show that $U_{v}'$ satisfies the desired property at all classical $x'\in \X_{\G}$.
Let $\mathfrak{X}_{v/\Q_{p}}$ be the   Bernstein variety of $\GL_{2}(F_{v})$, a scheme over $\Q_{p}$  (see \cite{LLC}, to which we refer for more background). 
By \cite[Theorem 3.2.1]{LLC}, the representation $\cV_{\G}$ of $G_{F_{v}}$ gives a map $f\colon \X_{\G}\to \mathfrak{X}_{v/\Q_{p}}$,
 compatibly with the local Langlands correspondence in the sense that for all $x\in  \X_{\G}$,   $f(x)$ is the point corresponding to the supercuspidal support of the representation $\pi_{x,v}'$ of $\GL_{2}(F_{v})$ over $\Q_{p}(x)$ attached to the representation $\cV_{\G, x}$. Note that for classical points $x$, $\pi_{x,v}={\rm JL}_{v}(\pi'_{x,v} \otimes_{\Q_{p}(x)}\baar{\Q_{p}})$, where  ${\rm JL}_{v}$ is the Jacquet--Langlands correspondence.

 After base-change to $\baar{\Q}_{p}$, we may consider the  finitely many maps $f_{i}\colon \X_{i}\to \mathfrak{X}_{v/\baar{\Q}_{p}}$, where the $\X_{i}$ are the connected components of $\X_{\G,\baar{\Q}_{p}}$. The image of  $f_{i}$ is contained in a connected component $\mathfrak{X}_{i}$ of  $\mathfrak{X}_{v/\baar{\Q}_{p}}$. These components are in bijection with inertial classes of supercuspidal supports for $\GL_{2}(F_{v})$, and  for the class $\sigma=\sigma_i $ of $\frak X_i$ there are three possibilities:
\begin{itemize}
\item $\sigma$ corresponds to  the class of a  supercuspidal representation  $\sigma_0$ of $\GL_{2}(F_{v})$ over $\baar{\Q}_{p}$. In this case, there is an unramified character $\omega\colon F_{v}^{\times}\to \OO(\frak{X}_i)^{\times}$ such that every $y'\in \mathfrak{X}_{i}$ corresponds to $\sigma_0\otimes\omega_{y'}$. Hence for  every classical  $x'\in  \X_{i}$, we have $$\pi_{x',v}\cong {\rm JL}_{v}( \sigma_0\otimes\omega_{f_{i}(x')})={\rm JL}_{v}( \sigma_0)\otimes\omega_{f_{i}(x')}\cong \pi_{x,v}\otimes \omega_{f_{i}(x')}.$$
 As $\omega_{f(x')}$ is unramified, it follows that  $\pi_{x',v}^{U_{v}}=\pi_{x',v}$.
\item $\sigma$ corresponds to the class of the supercuspidal support of ${\rm St}\otimes\omega_{0}$, where ${\rm St}$ is the Steinberg representation and $\omega_{0}\colon F_{v}^{\times}\to \baar{\Q}_{p}^{\times}$ is a character. Then there exist a closed subset $\mathfrak{X}_{i}'\subset \mathfrak{X}_{i}$ and an unramified character $\omega\colon F_{v}^{\times}\to \OO(\mathfrak{X}_{i})^{\times}$ such that  every $y'\in \mathfrak{X}_{i}'$ corresponds to the supercuspidal support of ${\rm St} \otimes\omega_{0}\omega_{y'}$, and such that every $y'\in \mathfrak{X}_{i}-\mathfrak{X}_{i}'$ corresponds to the support of an irreducible principal series representation. It follows that for  every classical $x'\in \X_{i}$,  the image $f_{i}(x')\in \mathfrak{X}_{i}'$ (since $\pi_{x,v}'\otimes \baar{\Q}_{p}$ is in the domain of ${\rm JL}_{v}$),  and that $\pi_{x',v} ={\rm JL}_{v}({\rm St}\otimes\omega_{0}\omega_{f_{i}(x')})=\omega_{0}\omega_{f_{i}(x')}\circ {\rm Nm}$. We conclude as above.

\item no element of the inertial class $\sigma$ is the supercuspidal support of a special or supercuspidal representation. This case is excluded as only those representations are in the image of the Jacquet--Langlands correspondence.
\end{itemize}
\end{proof}

\subsubsection{Galois representation from geometry}
Let $U^{p}\subset \G(\A^{p\infty})$, $V^{p}\subset \H(\A^{p\infty})$ be compact open subgroups. We will consider various  compact open subgroups $U^{p}\subset U^{p}_{*} \subset \G(\A^{p\infty})$, and will correspondingly denote by   $K^{p}_{*}$ be the image of $U^{p}_{*}\times V^{p}$ in $(\G\times \H)'(\A^{p\infty})$. Let $\X$ be an irreducible component of $\cE_{K_{}^{p}}^{\ord}\subset \cE_{\G, U^{p}}^{\ord}\times \cE_{\H, V^{p}}$, and let $\X_{\G}\subset\cE_{\G, U^{p}}^{\ord}$ be the irreducible component such that $\X\subset \X_{\G\ts \H}:=\X_{\G}\ts  \cE_{\H, V^{p}}$.

Suppose from now on that $\X$ is locally distinguished by $\H'$ (Definition \ref{hida loc dist}).
Let  $\cV_{\G}$ be the $\OO_{\X_{\G}''}[G_{F, S}]$-module   constructed in Proposition \ref{galois2} and Proposition \ref{schur}, and let $\cV_{\H}$ be the universal character $\chi_{\rm univ}$ of $G_{E, S}$   from \eqref{chiuniv}. 
Let 
$$\X^{(0)}:= \X\cap (\X_{\G}'' \ts  \cE_{\H, V^{p}})$$ an open subset,  and consider the $G_{F, E, S}$-representation 
$$\cV':=(\cV_{\G}\boxtimes \cV_{\H})_{|\X^{(0)}}$$

We  define another  sheaf $\cV$  with $G_{F, E, S}$-action,  that will provide a more  convenient and concrete substitute for $\cV'$ on (an open subset of) $\X^{(0)}$.

 Let 
$$U_{0}^{p}{}'=U^{\Sigma p}\prod_{v\in \Sigma } U_{v}',$$ with $U_{v}'$ as in Lemma \ref{levSig}.  Let $K_{0}^{p}{}'=(U_{0}^{p}{}'\times V^{p})F_{\A^{p\infty}}^{\times}/F_{\A^{p\infty}}^{\times}$, and let
$$\cV:=\cM_{K_{0}^{p}{}'}^{{H_{\Sg}'}},$$ viewed a sheaf over 
$\X$.

\begin{lemm}\label{dirsum}
 The sheaf $\cV$ is a direct summand of $\cM_{K_{0}^{p}{}'}$.
 \end{lemm}
 \begin{proof} The  group ${H_{\Sg}'}=\prod_{v\in \Sigma} E_{v}^{\times}/F_{v}^{\ts}$ acts on the locally free sheaf  $\cM_{K^{p}_{0}{}'}$
through a quotient by an open subgroup. Since $H_{\Sg}'$ is compact, such a quotient is  finite; 
 therefore the inclusion $\cV\subset \cM_{K_{0}^{p}{}'}$ splits. 
 \end{proof}

\begin{prop}\label{Xad*} 
  There is an open subset $\X^{(1)}\subset \X$ containing all classical points  such that $\cV$ is locally free of rank~$2$ 
  along $\X^{(1)}$.  For every  $z=(x,y)\in \X^{ \cl}$  we have
$$\cV_{z}\cong V_{{x}}\otimes \chi_{{y}}$$
as a $G_{F, E, S}$-representation.  
\end{prop}
\begin{proof} 
 By 
Corollary \ref{fibreMad},
for $z=(x,y)\in \X^{\cl}$ we have
$$\cM_{K_{0}^{p},(x,y)}\cong(\pi_{x}^{\vee, p, U_{0}^{ p}}\otimes \chi_{y}^{-1, p})\otimes (V_{x}\otimes \chi_{y})$$
(where the first pair of factors is a representations of $\G\times \H(\A^{p\infty})$ and the second one is a a representation of $G_{E}$).
By Lemma \ref{dirsum}, taking ${H_{\Sg}'}$-invariants commutes with specialisation, and we find that 
\begin{align}\label{old}
\cV_{z}\cong  (\pi_{x}^{\vee, \Sigma p, U_{0}^{ \Sigma p}}\otimes\chi_{y}^{-1,  \Sigma p})\otimes (\pi_{x, \Sigma}^{\vee}\otimes \chi_{y,\Sigma}^{-1})^{{H_{\Sg}'}} \otimes( V_{x}\otimes \chi_{y}).\end{align}
The first factor is $1$-dimensional by the theory of local newforms, and the second factor is $1$-dimensional by assumption $(\eps_{v})'$.

Since the fibre-rank of $\cV$ is $2$ in the dense set $\X^{\cl}$, there is an open neighbourhood  of this set over which $\cV$ is locally free of rank $2$. 
\end{proof}

\begin{coro} There exist:
 an open subset  $\X^{(2)}\subset \X^{(0)}\cap \X^{(1)}$ containing $\X^{\cl}$ such that 
\beqq \End_{\OO_{\X^{(2)}}[G_{F, E,S}]}(\cV)=\OO_{\X^{(2)}},
\eeqq
 an invertible sheaf $\calL$ over $\X^{(2)}$ with trivial Galois action, and a $G_{F, E, S}$-equivariant isomorphism of sheaves on $\X^{(2)}$
$$\cV\cong \calL\otimes \cV'.$$
\end{coro}
\begin{proof} By Proposition \ref{Xad*} and the construction of $\cV'$, the representations $\cV$, $\cV'$ have a common  trace $\mathscr{T}\colon G_{F, E, S}\to \OO(\X^{(1)}) $. Since this is an irreducible pseudocharacter, the assertions follow
 from Lemma \ref{uniq} and (the argument of) Proposition \ref{schur}.
\end{proof}

\subsubsection{The universal ordinary representation} In what follows, all sheaves $\cM_{K^{p}_{*}}$ will be considered as sheaves over $\X$ (or open subsets of $\X$). Note that, as the action of $H'_{\Sg}$ on $\cM_{\G}$, $ \cM_{\H}$ commutes with the Galois action, the sheaves $\cM_{K^{p}_{*}}$ retain an action of $G_{F, E, S}$.  

We will use the following well-known fact. 
\begin{lemm}\lb{detT} Let $R$ be a ring and let $T\colon M\to N$ be a map of free $R$-modules of the same rank. The set of those $x\in \Spec R$ such that $T\otimes R/\frak{p}_{x}$ is an isomorphism is open in $\Spec R$.
\end{lemm}
\begin{proof} The locus is the complement of $V(\det T)$. 
\end{proof}

In what follows, similarly to \S~\ref{ssec cong}, if `$?$' is any decoration,  $U_{?}^{p}$ is a subgroup of $\G(\A^{p\infty})$, and $V^{p}\subset \H(\A^{p\infty})$ is a fixed subgroup,  we denote by $K_{?}^{p}\subset (\G\ts\H)'(\A^{p\infty})$ the image of $U_{?}^{p}\ts V^{p}$. 
\begin{prop}\label{cofinal} Fix a finite set of primes $\Sg'$ disjoint from $\Sigma \cup S_{p}$, such that $U_{v}$ is maximal for all $v\notin {\Sg'}\cup \Sigma\cup S_{p}$, and consider the set $\mathscr{U}$ of subgroups   $U^{p}{}'=\prod_{v\nmid p}U_{v}'\subset U^{p}$  with $U'_v$  as in Lemma \ref{levSig} for all $v\in \Sigma$, and $U_{v}'=U_{v}$ for all $v\notin {\Sg'}\cup \Sigma \cup S_{p}$. (In particular $U_{0}'\in \mathscr{U}$.)
 \begin{enumerate}
\item
There exists a cofinal   sequence $(U^{p}_{i}{}')_{i\geq 0}\subset \mathscr{U} $, and open subsets $\X_{i}\subset \X^{(2)}\subset \X$ containing $\X^{\cl}$ such that $\X_{j}\subset \X_{i}$ for $i\leq j$, satisfying the following: there are  
 integers $r_{i}$ and $G_{F ,E}$-equivariant  maps
$$T_{i}\colon \cV^{\oplus r_{i}}= (\cM_{ K^{p}_{0}{}'}^{{H_{\Sg}'}} )^{\oplus r_{i}}   \to 
\cM_{K^{p}_{i}{}'}^{{H_{\Sg}'}}$$
that 
 are isomorphisms over $\X_{i}$.
\item
For each 
 $U^{p}{}'\in \mathscr{U}$, there is an open subset $\X_{U^{p}{}'}\subset \X^{(2)}$ containing $\X^{\cl}$ such that 
   the restriction to $\X_{U^{p}{}'}$ of 
\beq\label{piku}
\Pi^{K^{p}{}', \ord}_{{H_{\Sg}'}}:=\underline{\Hom}_{\OO_{\X}[G_{F, E, S}]}(\cM_{K^{p}{}'}^{{H_{\Sg}'}},\cV) \eeq
is a locally free
$\OO_{\X_{U^{p}{}'}}$-module, and we  have  an isomorphism of locally free sheaves with Hecke- and Galois- actions
 $$\cM_{K^{p}{}'}^{{H_{\Sg}'}} \cong 
( \Pi^{K^{p}{}', \ord}_{{H_{\Sg}'}} )^{\vee} \otimes \cV.$$
Moreover $\Pi^{K^{p}{}', \ord}_{{H_{\Sg}'}}\subset \Pi^{K^{p}{}'', \ord}_{{H_{\Sg}'}} $  for $U^{p}{}''\subset U^{p}{}'$
 via the natural projections  $\cM_{ K^{p}{}''}^{{H_{\Sg}'}}\to\cM_{ K^{p}{}'}^{{H_{\Sg}'}} $.
\item  The $ \cH_{\G\times \H, {\Sg'}}^{K^{p}{}'}$-module  $\Pi^{K^{p}{}', \ord}_{{H_{\Sg}'}}$ is generated by  $\Pi^{K_{0}^{p}{}', \ord}_{{H_{\Sg}'}} $
over $\X_{U^{p}{}'}$.
\item\label{Piord}  For each $z=(x,y)\in\X^{\cl}$, we have  
$$(\Pi^{K^{p}{}', \ord}_{{H_{\Sg}'}})_{z}\cong (\pi_{x}^{U^{p}{}', \ord})_{{H_{\Sg}'}}\otimes \chi_{y}, $$
with the notation of \eqref{piord}.
\end{enumerate}
\end{prop}
\begin{proof} 
It suffices to prove part 1 for a sequence of subgroups $U_{i}^{p}{}'=\prod_{v\nmid p}U_{i,v}'$ that are $\B_{S}^{\times}$-\emph{conjugate} to a cofinal sequence (if $U_{i}^{p}{}''=g_{i}U_{i}^{p}{}'g_{i}^{-1}$ is cofinal and $(U_{i}^{p}{}', T_{i})$ satisfies the desired condition, then $(U_{i}^{p}{}'', g_{i}^{-1}\circ T_{i} )$ also  satisfies the desired condition).
We thus take any sequence with   $U_{i,v}'=U_{1}(\vpi^{m_{i,v}})$ for $v\notin  {\Sg'}\cup\Sigma \cup S_{p}$, with $m_{i,v} \geq m_{v}$ and such that $\min_{v\in {\Sg'}}m_{i,v}\to \infty$.

Let $r_{i}=\prod_{v}(1+m_{i,v}-m_{v})$. By the local theory of oldforms of   \cite{casselman} and the isomorphisms \eqref{old} and
\begin{align}\label{old2}
 (\cM_{ K^{p}_{i}{}'})^{ {H_{\Sg}'}}_{z}\cong  (\pi_{x}^{\vee, \Sigma p, U_{i}^{ \Sigma p}{}'}\otimes\chi_{y}^{-1,  \Sigma p})\otimes (\pi_{x, \Sigma}^{\vee}\otimes \chi_{y,\Sigma}^{-1})^{{H_{\Sg}'}} \otimes( V_{x}\otimes \chi_{y}),
\end{align}
 there are Hecke operators $$T_{v, j_{v}} \colon \cM_{ K^{p}_{0}{}'}^{{H_{\Sg}'}}     \to 
\cM_{K^{p}_{i}{}'}^{{H_{\Sg}'}}$$ such that the map $T_{i}:=\prod_{v\in S} \oplus_{j_{v}} T_{i, v, j_{v}} $ is an isomorphism after specialisation at any $z$ in the dense set $ \X^{\rm  cl}$. Hence $T_{i}$ is an isomorphism in an open neighbourhood $\X_{i}$ of $ \X^{\cl}$ (which we possibly shrink to make sure it is contained in $\X^{(2)}$).   Together with Lemma \ref{detT}, this concludes the proof of part 1. Part 2 is a consequence of part 1 and the absolute irreducibility of $\cV$, in the special case  $U^{p}{}'=U^{p}_{i}{}'$, with $\X_{U^{p}{}'}=\X_{i}$. The general case is deduced from the special case: if $U^{p}{}'\subset U_{i}^{p}{}'$, let $\X_{U^{p}{}'}:=\X_{i}$ and take  on both sides the locally free summands  consisting of $U^{p}{}'$-invariants (for the first assertion) or coinvariants (for the second assertion). For part 3,  we may again reduce to  the special case $U^{p}{}'=U^{p}_{i}{}'$; then the space $\Pi^{U_{i}^{p}{}'}$  is generated by the images of the transposes of various ``oldforms'' degeneracy maps $T_{i}$ from part 1, that are elements of the Hecke algebra $\cH_{\G, S}^{U^{p}{}'}$. Finally, part \ref{Piord} follows from \eqref{old2}. 
\end{proof}

\begin{defi}[Universal ordinary representation] 
Let $\mathscr{U}$ be as in Proposition \ref{cofinal}, and fix an arbitrary 
$ U' \in \mathscr{U}$. Let 
$$\X^{(3)}:=\X_{U^{p}{}'}$$
be as in Proposition \ref{cofinal}, and 
let 
$$\Pi^{K^{p}{}', \ord}_{{H_{\Sg}'}}:=\underline{\Hom}_{\OO_{\X^{(3)}}[G_{F, E, S}]}(\cM_{K^{p}{}'}^{{H_{\Sg}'}},\cV) $$ as in \eqref{piku}.
The \emph{universal ordinary representation}
 $$\Pi^{K^{Sp}, \ord}_{{{H_{\Sg}'}}}\subset
 {}' \Pi^{K^{Sp}, \ord}_{{H_{\Sg}'}}:=\varinjlim_{U^{p}{}''\in \mathscr{U}} \Pi^{K^{p}{}'', \ord},$$
 is the $\OO_{\X^{(3)}}[(\B_{{\Sg'}}^{\times}\times E_{{\Sg'}}^{\times})/F_{{\Sg'}}^{\ts}]$-submodule
  generated by $\Pi_{{H_{\Sg}'}}^{K^{p}{}', \ord}$.  
  \end{defi}

\subsubsection{Local-global compatibility}
The next theorem describes $\Pi^{K^{{S}p}{}', \ord}_{{{H_{\Sg}'}}}$, as a sheaf with an action by $\B_{{\Sg'}}^{\times}\times E^{\Sigma\infty,\times}$,  in terms of the local Langlands correspondence in  families of \cite{LLC}, denoted  by 
$$\cV_{\G}\mapsto \pi_{\G, {\Sg'}}(\cV_{\G}) .$$
This correspondence attaches, to any  family $\cV_{\G}$ of representations of $\prod_{v\in {\Sg'} } G_{F_{v}}$ on  a rank-$2$ locally free sheaf over a Noetherian  scheme   $\Y/\Q$, a   family of representations of $\GL_{2}(F_{{\Sg'}})$ on a torsion-free sheaf over $\Y$.  The latter  representation is \emph{co-Whittaker} in the sense of \cite[Definition 4.2.2]{LLC}; in particular it admits a unique Whittaker model.

\begin{theo}[Local-global compatibility]\label{LGC} 
Let 
$$\pi_{\G,{\Sg'}}(\cV_{\G})$$
be the representation of $\GL_{2}(F_{\Sg'})$ over $\X^{(3)}$ associated with $\cV_{\G}$ by the local Langlands correspondence in  families for $\GL_{2}(F_{{\Sg'}})$ of \cite{LLC};  let $\chi_{\H, {\rm univ}, {\Sg'}}$ be
 the pullback to $\X^{(3)}$ of the sheaf  $\chi_{\H, \rm univ}$  of \eqref{chiHuniv}, with the $\H(\A^{\infty})$-action restricted to $E_{{\Sg'}}^{\times}$.

Then there  exist an open subset $\X^{(4)}\subset \X^{ (3)}\subset \X$ containing $\X^{\cl}$,  a line bundle $ \Pi^{K^{Sp}{}',S,  \ord}_{{H_{\Sg}'}}$ over $\X^{ (4)}$, and an isomorphism of $\OO_{\X^{(4)}}[\GL_{2}(F_{{\Sg'}})\times E_{{\Sg'}}^{\times}]$-modules
$$
\Pi^{K^{{S}p}{}',  \ord}_{{{H_{\Sg}'}}}\cong 
(
\pi_{\G,{\Sg'}}(\cV_{\G})\otimes_{\OO_{\X^{(3)}}} \chi_{\H, {\rm univ}, {\Sg'}}
)
  \otimes 
  \Pi^{K^{{S}p}{}',{S},  \ord}_{{H_{\Sg}'}}. $$
\end{theo}
\begin{proof}
For $*=\emptyset, '$,  consider
$${}^{*}\pi_{\G, {\Sg'}}':=\underline{\Hom}_{\OO_{\X^{(3)}}[E_{S}^{ \times}]}  (\chi_{\H, {\rm univ},S}, {}^{*}\Pi^{K^{Sp}, \ord}_{{{H_{\Sg}'}}}),$$
a torsion-free sheaf over $\X^{(3)}$ with action by $\B_{{\Sg'}}^{\times}=\GL_{2}(F_{{\Sg'}})$. There is an obvious isomorphism
\begin{align}\label{subst}
{}^{*}\Pi^{K^{{S}p}, \ord}_{{H_{\Sg}'}} \cong {}^{*}\pi_{\G,{\Sg'}}'\otimes
 \chi_{\H, {\rm univ}, {\Sg'}}.
\end{align}

 By 
 Proposition \ref{cofinal}.\ref{Piord}
 and the local freeness of  each
  ${}'\Pi^{K^{p}{}''}_{{H_{\Sg}'}}$ near $\X^{\cl}$, the fibre of  
  $({}'\pi_{\G, S}')^{U_{S}'}$ at any $z=(x,y)\in \X^{\cl}$ equals   $\pi_{\G,S}(\cV_{\G,x})^{U_{S}''}$;    by Proposition \ref{cofinal}.3 the same is true if one replaces $({}^{}\pi'_{\G, {\Sg'}})^{U_{{\Sg'}}''}$ by the submodule $(\pi'_{\G, {\Sg'}})^{U_{S}''}$.
In conclusion, taking the limit over $U^{p}{}''\in \mathscr{U} $ we find that the smooth, finitely generated, admissible  $\OO_{\X^{(3)}}[\GL_{2}(F_{{\Sg'}})]$-module $\pi'_{\G,{\Sg'}}$ satisfies 
$$\pi'_{\G,{\Sg'}, (x,y)}\cong \pi_{\G, {\Sg'}}(\cV_{\G,x}).
$$
for all $(x, y)\in \X^{\cl}$. 
Then by \cite[Theorem  4.4.3]{LLC},  there exist an open subset $\X^{(4)}\subset \X^{(3)}$ containing $\X^{\cl} $ and  a line bundle that we denote by $ \Pi^{K^{Sp},S, \ord }_{{H_{\Sg}'}}$, such that 
$$\pi'_{\G, {\Sg'}}\cong  \pi_{\G, {\Sg'}}(\cV_{\G})\otimes \Pi^{K^{Sp},S, \ord}_{{H_{\Sg}'}}.$$
Substituting in \eqref{subst} gives the desired result.
\end{proof}

\section{Pairings}
\subsection{Global dualities}  \label{sec: duality} We construct Hecke- and/or Galois-equivariant duality pairings on the sheaves constructed in the previous section. The results of this somewhat technical subsection are summarised in Propositions \ref{duality}, \ref{abovco}. 

\subsubsection{Pairings, symmetries and involutions}\lb{skew def}
If $\epsilon \in \{\pm 1\}$,  $R$ is a ring or sheaf of rings, and $M$, $J$ are $R$-modules, an $R$-bilinear pairing $(\ ,  \ )\colon M\ot M\to J$ is said to be \emph{$\epsilon$-symmetric}  if it satisfies $(m, m')=\epsilon\cd (m', m)$. If $R$ is equipped with an involution $\iota$,  denote by  $(\cd)^{\iota}$ both the functor $\ot_{R, \iota}R$ and the maps $m\mapsto m\ot1$;  an $(R, \iota)$-sesquilinear pairing $(\ ,  \ )\colon M\ot M^{\iota}\to J$ is said to be \emph{$\epsilon$-hermitian} if it satisfies $  (m, n)= \epsilon \cd \iota ((n^{\iota}, m^{\iota})^{\iota})$.   We will also use the prefix `skew-' (respectively no prefix) if $\epsilon=-1$ (respectively $+1$).

\subsubsection{Involutions}
We denote by  the same name $\iota$  the involutions on $\cH_{\G_{*}}^{\rm sph}$, $\cH_{\G_{*}}^{\rm sph, \ord}$ $\Lambda_{\G_{*}, U^{p}_{*}} $,  $\cE_{\G_{*}}^{\ord}$ deduced from those of  \S~\ref{ssec invo}. 
If $M$ is a module over any of the above rings (or sheaf of modules over any of the above spaces), we let 
$M^{\iota}= \iota^{*} M$.  

\begin{lemm}\label{iota and dual} Let $W$ be an irreducible   algebraic representation of $\G_{*}$ over $L$.
\begin{enumerate}
\item We have $\sigma_{W^{\vee}}(t)=\sigma_{W}(t^{\iota})$ for all $t\in T_{\G_{*}}$.
\item
If $\pi^{\ord}$ is the ordinary part of  an automorphic representation of $\G_{*}(\A^{\infty})$ over $L$ of weight $W$, unramified of level $U_{*}^{S}$ outside of a finite set of primes $S$, then there is an isomorphism of $\cH_{\G_{*}, U^{S}}^{{\rm sph}, \ord}$-modules $\pi^{\vee,U_{*}^{S}, \ord}\cong (\pi^{U_{*}^{S},\ord})^{\iota}$.
  \item There is an identification 
$$(\cE_{G_{*}}^{\cl, W})^{\iota}=\cE_{\G_{*}}^{\cl, W^{\vee}}$$
such that $\pi^{\ord}_{\iota(x)}=(\pi_{x})^{\vee, \ord}$. 
\end{enumerate}
\end{lemm}
\begin{proof} All results can be reduced to the case $\G_{*}=\H$, which is trivial, or $\G_{*}=\G$, that we address. Part 1 follows from the explicit description of $\sigma_{W}$ in \eqref{alg char} and $W_{\G, (w; (w_{\sigma}))}^{\vee}\cong W_{\G,(-w; (w_{\sigma}))}$ (see \eqref{inv pair W} below for an explicit duality). 

For part 2, we use $\pi^{\vee}=\pi\otimes \omega^{-1}$ where $\omega$ is the central character of $\pi$, and verify the statement separately for the spherical Hecke algebra and for the operators $\Up_{t}$. For the former, it is well known that the spherical Hecke algebra is generated by operators $T(z)$ and $T(\smalltwomat x{}{}1)=T(\smalltwomat 1{}{}x)$ for $z,x\in F_{S}^{\times}$; denoting by $\lambda_{\pi^{?}}(\cdot) $ the eigenvalue of $T(\cdot) $ on $(\pi^{?})^{U^{S}}$, we then have 
$\lambda_{\pi}(z^{\iota})=\lambda_{\pi}(z^{-1})=\omega(z)^{-1}=\lambda_{\pi^{\vee}}(z)$, and $$\lambda_{\pi}((\smalltwomat x{}{}1^{\iota})= \lambda_{\pi}(x^{-1}(\smalltwomat 1{}{}x)=\omega(x)^{-1}\lambda_{\pi}(\smalltwomat x{}{}1)=\lambda_{\pi^{\vee }}(\smalltwomat x{}{}1)$$ as desired. 

For the operators $\Up_{t}$, we verify that if $\pi$ is ordinary at $v$ with unit character $\alpha_{v}^{\circ}=\alpha_{v}\sigma_{W}^{-1}$ (as a character of $T_{\G, v}^{+}$),  then $\pi^{\vee}$ is ordinary at $v$ with unit character 
$$\alpha_{v}^{\circ,\iota}\colon t\mapsto \alpha_{v}^{\circ} (t^{\iota}).$$
This follows from observing 
\beqq
 \pi_{v}^{\vee} & \cong {\rm Ind} ( \alpha_{v}^{ \iota}\cdot (|\  |_{v},|\  |_{v}^{-1})),\\
\alpha_{v}^{\circ}(t^{\iota}) & =\alpha_{v}^{\circ}(t)  \alpha_{v}^{\circ}(\nu(t))^{-1}\in \OO_{F}^{\ts}.
\eeqq

Finally, part 3 follows from parts 1 and 2. 
\end{proof}

\subsubsection{Homological and cohomological  dualities} 
We shall define various pairings $\la\ , \  \ra_{?}$ in the (ordinary, completed) homology of Shimura varieties, starting from the Poincar\'e duality pairings. Then we will use them to construct corresponding pairings $(\,  \ )_{?}$ on spaces of representations, as follows.

\begin{enonce}{Construction}\label{constr pairing} 
Let $A$ be a ring, $G$ a group,  and let $M_{1}, M_{2}, V_{1}, V_{2}, A(d)$ be $A[G]$-modules, projective and  of finite type over $A$; denote 
$$V^{D}:= \Hom(V, A(d)).$$
Let $f_{i}\in \pi_{i}:=\Hom( M_{i}, V_{i})$ be  $A[G]$-maps,  suppose we have fixed an identification $V_{2}^{D}\cong V_{1}$;  let $\la\ , \ \ra$ be a perfect pairing $M_{1}\times M_{2}  \to A(d)$, inducing $u_{\la ,\ra} \colon M_{2}^{D}\to M_{1}$. Let $f_{2}^{D}\colon V_{2}^{D}\cong V_{1}\to M_{2}^{D}$ be the dual map, then 
 we define a pairing on $\pi_{1}\times \pi_{2}$ by
\beq\label{pairinggg}
(f_{1}, f_{2})_{\la, \ra}= f_{1}\circ u_{\la , \ra}(f_{2}^{D}) \in \End_{A[G]}(V_{1}).\eeq
\end{enonce}

\subsubsection{Homological dualities / 1} 
Fix lattices $W^{\circ}$ and $W^{\vee, \circ}$ on any  right algebraic representation of $\G_{*}$ over $L$, and denoted by $\la\ , \ \ra^{W}\colon W\ot W^{\vee}\to L$ the natural invariant pairing. This may not preserve the lattices but it does so up to a bounded denominator which we denote by $p^{-|W|}$.\footnote{With  respect to the model in \eqref{inv pair W}, we have $|W|={\rm ord}_{p}\left({k-2\choose (k-2+l)/2}\right)$ for the representation \eqref{Wwkl} of $\G$.}

We may then consider the  Poincar\'e duality pairings
\beq\label{poincare} \la , \ra_{U_{*},W}\colon H_{d}(\baar{X}_{*, U_{*}}, \cW)\times H_{d}(\baar{X}_{*, U_{*}}, \cW^{\vee})\to H_{0}(\baar{X}_{*, U_{*}}, \cW\otimes \cW^{\vee})\to L(d),
\eeq
where the second map is induced by $\la, \ra^{W}$ and summation over the connected components of $\baar{X}_{*, U_{*}}$.
  These pairings are integral up to  a bounded denominator $p^{-|W|}$ and  satisfy 
$$\la xT, y\ra_{U_{*}, W}= \la x, yT^{\iota}\ra_{U_{*}, W} $$
for any $T$ in $\cH_{\G_{*}, U_{*}}^{p}$, as well as the projection formula
\beq\label{proj formula}
\la {\rm p}_{U'_{*}/U_{*}, *}(x), y\ra_{W, U_{*}} = \la x,  {\rm p}_{U'_{*}/U_{*}}^{*}(y)\ra_{W, U_{*}'} 
\eeq
for all pairs of levels $U_{*}'\subset U_{*}$; here ${\rm p}_{U'_{*}/U_{*}}\colon X_{U'_{*}}\to X_{U_{*}}$ is the projection.

\subsubsection{Homological dualities / 2} We start to promote and modify the Poincar\'e duality pairings. The following lemma is clear.

\begin{lemm} \label{et trace}
 Let $R$ be a ring, $S$  a finite $R$-algebra,   $M$ a finite $S$-module. 
\begin{enumerate}
\item Suppose that $S$ is \'etale over $R$. Then there is a natural isomorphism
 $$\alpha\colon \Hom_{R}(M, R)\to \Hom_{S}(M, \Hom_{R}(S, R)) \to \Hom_{S}(M, S)$$
 where the first map is $\lambda\mapsto (m\mapsto (s\mapsto \lambda(s m)).$
and the  second one comes from the isomorphism $S\cong \Hom_{R}(S, R)$ induced by the relative trace map.
 \item Suppose that $S=R[T]$ for a finite abelian group $T$, then there is an isomorphism
$\beta\colon \Hom_{R}(M, R)\to \Hom_{R[T]}(M, R[T])$
given by $\lambda\mapsto (m\mapsto \sum_{t}\lambda(tm) [t^{-1}])$.
\end{enumerate}
If $S=R[T]$ is \'etale over $R$ then we have $\alpha(\lambda)=|T|^{-1}\beta(\lambda)$.
\end{lemm}

If $S=R[T]$ for a finite abelian group $T$, one verifies that the   isomorphism of the lemma is given by
$$\la\ ,  \  \ra\mapsto \la\la\ ,\ \ra \ra, \quad \la\la x, y\ra\ra:= \sum_{t\in T} \la x, ty\ra [t^{-1}].$$

We may apply case 2 of the  lemma to  
$$M=M_{\H,V^{p}, r,W}\otimes M_{\H, V^{p},r,W^{\vee}}, 
\quad R=L,\quad  S= \Lambda_{\H, V^{p},r}:=  \Lambda_{\H, V^{p},r}^{\circ}\ot_{\OO_{L}}L \cong L[{\baar{T}_{\H, 0}/ \baar{T}_{\H, r}}]$$ (with the isomorphism of \eqref{lambda r}). 
We obtain, from the pairings $\la\ , \ra_{V_{r}, W}$, pairings
$$\la \la \ , \ \ra\ra_{V^{p},W, r}\colon M_{\H,V^{p},  W, r}\otimes_{\Lambda_{\H, W, r}} M_{\H, V^{p},  W^{\vee}, r}^{\iota} \to \Lambda_{\H, V^{p}, W, r}\otimes L,$$
and thanks to an easily verified compatibility, a well-defined pairing 
\beq\label{pair EH}\la \la \ , \ \ra\ra_{V^{p},W} \colon M_{\H,V^{p},W}\otimes_{\Lambda_{H}} M_{\H,V^{p}, W^{\vee}}^{\iota} &\to \Lambda_{\H}\otimes L=\OO_{\cE_{\H}}\otimes L\\
x\otimes y&\mapsto \lim_{r }  \la \la x_{r}, y_{r} \ra\ra_{V^{p},W, r}.
   \eeq

\subsubsection{Automorphic inner products}

Let 
$${\rm v}(U_{*}):= \vol (X_{*, U_{*}}(\bC))$$ where `$\vol$' denotes the volume with respect to the metric deduced from  the hyperbolic metric $dxdy/2\pi y^{2}$ (using the complex uniformisation \eqref{z quot}),
when $\G_{*}=\G$, and  the counting metric, when $\G_{*}=\H$. By \cite[Lemma 3.1]{yzz}, ${\rm v}(U_{*})\in \Q^{\times}$ and, when $d=\dim X_{*}=1$, it equals the degree of the Hodge bundle $L_{U_{*}}$ defined  as in \emph{loc.cit}. We have
\beq \label{deg vol}
\deg {\rm p}_{U_{*}' ,U_{*} } ={\rm v}(U'_{*})/{\rm v}(U_{*}) = {\rm Z}_{\G_{*}}(\Q)\cap U_{*}\bks U_{*}/U_{*}',
\eeq
where the last equality can be easily seen e.g. from the complex uniformisation \eqref{z quot}.
We set for any $r\geq 1$
\beq\lb{vUpp}{\rm v}(U_{*}^{p}):={{\rm v}(U_{*}^{p}U_{*,0}(p^{r})_{p})\over p^{dr[F:\Q]}}\eeq
 where $U_{*,0}(p^{r})_{p}\subset \G_{*}(\Q_{p})$ is a maximal compact subgroup if $\G_{*}=\H, \H'$, it is the group of those matrices that are upper triangular modulo $p^{r}$ if $\G_{*}=\G$, and it is deduced from those by product and quotient if $\G_{*}=\G\ts\H, (\G\ts\H)'$.  The right hand side of \eqref{vUpp} is independent of $r\geq 1$.

Let $\pi$ be an automorphic representation of $\G_{*}(\A^{\infty})$ of weight $W^{*}$ over $L$, $V_{\pi}$ the corresponding $G_{E_{*}}$-representation. Then we have an isomorphism $V_{\pi^{\vee}}\cong V_{\pi}^{*}(1)$, hence we may use   Construction \ref{constr pairing}
 with $A=L$, $G=G_{E_{*}}$,  $M_{1}=  \H_{d}(\baar{X}_{*,U_{*}}, \cW)$, $M_{2}=  \H_{d}(\baar{X}_{*,U_{*}}, \cW^{\vee})$, $V_{1}=V_{\pi}$, $V_{2}=V_{\pi^{\vee}}$ and the pairings \eqref{poincare}.
  Using \eqref{carayol iso 2}, we obtain 
 $$( \ , \  )_{\pi, U_{*}}:=(\ ,  \ )_{\la, \ra_{U_{*},W}}\colon \pi^{U_{*}}\times \pi^{\vee,U_{*}}\to L.$$ 
One verifies thanks to \eqref{proj formula} and  \eqref{deg vol} that  the pairing 
\beq\label{pair pi}
(  \ , \ )_{\pi}:= \lim_{U_{*} }\ (\dim W^{} \cdot {\rm v}(U_{*}))^{-1}\cdot  ( \ , \  )_{\pi,{U_{*}}}\colon \pi\times \pi^{\vee}\to L.
\eeq
is well-defined.

When  $\G_{*}=\H$, denoting $\pi=\chi_{\H}$, we may alternatively apply Construction \ref{constr pairing} to $A, M_{1}, M_{2}, V_{1}, V_{2}$ as above and the image of the pairings $\la \la \ ,  \ra\ra_{V^{p},r W} $ under the map $\Lambda_{\H, r, W}\to L$ given by $[t]\mapsto \chi_{\H}(t)$, and denote the resulting pairings on $\chi_{\H}\times \chi_{\H}^{-1}$ by $(\ , \ )_{\la\la, \ra\ra_{\chi_{\H},V^{p},r}}$. As $|\baar{T}_{\H, 0}/\baar{T}_{\H, r}|\cdot {\rm v}(V^{p})={\rm v}(V^{p}V_{p,r})$ by \eqref{deg vol}, we have 
$$( \ ,  \ )_{\chi_{\H}}  ={\rm v}(V^{p})^{-1} (\ , \ )_{\la\la, \ra\ra_{\chi_{\H},V^{p},r}},$$
and in particular the right-hand side is independent of $V^{p}$.

Assume for the rest of this subsection that $\G_{*}=\G,\G\times \H, (\G\times\H)'$. Then we need  a twist  in order to isolate the toric action and to obtain the $\iota$-equivariance  of the  pairings under the action of the $\Up_{p\infty}$-operators.

Let $\pi=\pi^{\infty}\ot W$ be an ordinary representation of $\G_{*}(\A)$. Using the transformation
$w_{\rm a}^{\ord}$
defined in Proposition \ref{w ord},  we define a pairing
\beq\label{()'} (f_{1} , f_{2} )_{\pi}^{\ord}:= \dim W^{}  \cdot (w_{\rm a}^{\ord}f_{1}, f_{2})_{\pi}
 \colon \pi^{\ord}\times \pi^{\vee, \ord} \to L.
\eeq
See Lemma  \ref{cor w ord} for its nondegeneracy.

\subsubsection{Homological dualities / 3} Analogously to the previous paragraph, we define a twisted Poincar\'e pairing
\beq\label{lara'}
 \H_{1}(\baar{X}_{*,U_{*}^{p},r}, \cW)^{\ord}\otimes  \H_{1}(\baar{X}_{*,U_{*}^{p}, r}, \cW^{\vee})^{\ord} &\to L(1)\\
\la x, y\ra_{U_{*}^{p},W, r}^{\ord} &:= \la x, y  w_{\rm a}^{\ord}\ra_{U_{*}^{p}U_{*,p,r},W},
\eeq
of which we will especially consider the restriction to the ordinary parts of homology.

\begin{lemm}
Let $\pi$ be an ordinary representation of $\G_{*}(\A)$, and identify $$\pi^{\ord} = \Hom_{L[G_{E_{*}}]}(\H_{1}(\baar{X}_{*,U_{*}^{p},r},\cW)^{\ord}, V_{\pi})$$ for sufficiently large $r$ similarly to   Proposition \ref{fibreM2}. Then Construction \ref{constr pairing} provides a pairing $(,)_{\la, \ra_{W, r}^{\ord}}$ on $\pi^{\ord}\times \pi^{\vee,\ord}$; it is related to 
 \eqref{()'} by
\beq\label{compare'}
(,)^{\ord}_{\pi} =   
{ {\rm v}(U_{*})}^{-1}\cdot
  (, )_{\la, \ra^{\ord}_{U*, W, r}}.\eeq
\end{lemm}
\begin{proof} This follows by chasing the definitions.\end{proof}

By applying case 2 of Lemma \ref{et trace} as in \eqref{pair EH}, corrected by a factor $p^{r[F:\Q]}$,\footnote{This factor accounts for the `$K_{0}(p^{r})$'-part of the level.}
 we obtain from \eqref{lara'}
 pairings
\beq\label{larara}
 \la \la , \ra\ra_{U^{p}_{*}, W, r}\colon \H_{1}(\baar{X}_{*,U_{*}^{p},r}, \cW)\otimes_{\Z_{p}}  \H_{1}(\baar{X}_{*,U_{*}^{p}, r}, \cW^{\vee}) &\to \Lambda_{\G_{*},U_{*}^{p}, r}(1) \\
x\ot y& \mapsto p^{r[F:\Q]} \sum_{t\in \baar{T}_{\G_{*,0}}/\baar{T}_{\G_{*, r}}} \la x, y\ra_{U^{p}_{*}, W, r}^{\ord}
\eeq

\begin{lemm} \lb{Tequiv}
The parings \eqref{larara} satisfy $\la \la x_{r}T, y_{r}\ra\ra_{W, r}=\la \la x_{r}, y_{r}T^{\iota}\ra\ra_{U^{p}_{*},W, r}$ for all $T\in \cH_{\G, r}^{\ord}$ and all $x_{r}\in  H_{d}(\baar{X}_{*,U_{*}^{p},r}, \cW)$, $ y_{r}\in H_{d}(\baar{X}_{*,U_{*}^{p},r}, \cW^{\vee})$. 

For $z\in M_{\G_{*}, \cW}$, denote by   $z_{r}$ its image in $M_{\G_{*}, \cW, r}:=H_{d}(\baar{X}_{*,U_{*}^{p},r}, \cW)^{\ord}$. The pairing
\beq\label{pair lam}
\la \la \ , \ \ra\ra_{\Lambda, U^{p}_{*},W} \colon  M_{\G_{*},U_{*}^{p},\cW}\otimes_{\cH_{\G}^{\ord}}  M_{\G_{*},U_{*}^{p},\cW^{\vee}}^{\iota} & \to \Lambda_{\G_{*}, U_{*}^{p}}(d)\otimes L\\
\la\la x , y^{\iota} \ra\ra_{\Lambda,U^{p}_{*}, W}&:=\lim_{r}\la\la x_{r}, y^{\iota}_{r}\ra\ra_{U^{p}_{*},W, r}
\eeq
is well-defined.
\end{lemm}
The above construction is a  minor variation on the one of  \cite[\S~ 2.2.4]{fouquet}, to which we refer for the proof of the lemma.
As usual, when $W=\Q_{p}$ we shall omit it from the notation.

\begin{lemm}\label{pair indep} The diagram
\begin{equation*}
\xymatrix{
M_{\G_{*},U_{*}^{p}, \cW}\otimes_{\cH_{\G_{*}}^{\circ}} M_{\G_{*},U_{*}^{p}, \cW^{\vee}}^{\iota}\ar[r]^{{\la\la\, ,\,  \ra\ra}_{\Lambda,U^{p}_{*}, W}} &\Lambda_{\G_{*}, U^{p}_{*}}(d)\otimes L\\
M_{\G_{*}, U_{*}^{p}}\otimes_{\cH_{\G_{*}}^{\circ}} M_{\G_{*}, U_{*}^{p}}^{\iota}\ar[r]^{\la\la\, , \, \ra\ra_{\Lambda, U^{p}_{*}}} \ar[u]^{j_{W}\otimes j_{W^{\vee}}} &\Lambda_{\G_{*}, U^{p}_{*}}(d)\otimes L\ar[u]^{\cong},
}
\end{equation*}
where the left vertical map comes from Proposition \ref{free indep}.2 and the right vertical map is  $[t]\mapsto \sigma_{W}^{-1}(t)[t]$, is commutative.
\end{lemm}
\begin{proof} For simplicity we write down the proof for the group $\G_{*}=\G$ and we drop the subscripts $U^{p}$.
Poincar\'e duality and  the  pairings $ \la\ , \ \ra^{W} $  preserve integral structures up to $p^{-|W|}$.  Then  by construction it suffices to show the identity
$$\la j_{W, r}(x), j_{W^{\vee},r}(y)\ra_{W}^{\ord} \equiv  \la x, y \ra^{\ord} \pmod{p^{r-|W|}\OO_{L}}$$
for all $r\geq 1$ and  $x$, $y$ in $ H_{d}(\baar{X}_{ r},\Z_{p})$. 

By definition in \eqref{jWr}, we have
$$j_{W, r}(x)=x\ot \zeta\ot \zeta^{*}$$ where $\zeta_{r}\in W^{\circ, N_{0}}/p^{r}$ and $\zeta_{r}^{*}\in W^{\circ}_{N_{0}}/p^{r}$ are elements pairing to $1$; we denote by $\zeta_{r}^{\vee}$, $\zeta_{r}^{\vee, *}$ the analogous elements for $j_{W^{\vee}, r}$. 
 Then we need to show that 
$$\la( x\ot \zeta_{r}\ot \zeta_{r}^{*}) w_{\rm a}^{\ord}, y\ot \zeta_{r}^{\vee}\ot \zeta_{r}^{\vee*} \ra= \la x, y\ra.$$
By the definition of $w_{\rm a}^{\ord}$ in Proposition \ref{w ord}, this reduces to the identity
$$\la \zeta_{r} w_{0}, \zeta_{r}^{\vee}\ra \cdot \la \zeta_{r}^{*} w_{0}, \zeta_{r}^{\vee, *}\ra = \la \zeta_{r} , \zeta_{r}^{*}\ra\cdot \la \zeta_{r}^{\vee} , \zeta_{r}^{\vee,*}\ra=1, $$
which can be  immediately  verified using an explicit model for the pairing such as given in \eqref{inv pair W}. 
\end{proof}

\subsubsection{Dualities over Hida families} 

Let $\X$ be an irreducible component of $\cE_{K^{p}}^{\ord}$. By Proposition \ref{etale}, the map $\cE^{\ord}_{K^{p}}\to \Spec \Lambda_{\Q_{p}}$ is \'etale  in a neighbourhood $\X'$ of $\X^{\cl}$, hence we may  apply case 1 of  Lemma \ref{et trace} to deduce  from \eqref{pair lam} a pairing 
\beq
\label{pair E}
\la \la \ , \ \ra\ra_{K^{p}{}^{\ord}} \colon \cM_{K^{p}{}^{\ord}}\otimes_{\OO_{\X'}} \cM_{K^{p}{}'}^{\iota} \to \OO_{\X'}(1).\eeq

We summarise the situation.

\begin{prop}[Duality]\label{duality}
Let $\X^{(5)}\supset \X^{\cl}$ be the intersection of the subset $\X^{(4)} $ of Theorem \ref{LGC} with the locus where the map $\X\to \Spec\Lambda_{\Q_{p}} $ is \'etale. There exist
\begin{itemize}
\item  a perfect, $G_{E}$-equivariant, skew-hermtian pairing 
\begin{align}
\cM_{K^{p}{}'}^{{H_{\Sg}'}}\otimes_{\OO_{\X^{(5)}}} \cM_{K^{p}{}'}^{{H_{\Sg}'},\iota}\to \OO_{\X^{(5)}}(1).\label{iso 1}
\end{align}
induced from \eqref{pair E};
\item  a perfect, $G_{E}$-equivariant, skew-hermtian pairing 
\beq 
\cV \ot_{\X^{(5)}} \cV^{\iota} \to \OO_{\X^{(5)}}(1).\label{iso 2}.
\eeq
\item a perfect pairing 
\begin{align}\label{iso5}
((\ , \  )):= {\rm v}(K^{p}{}')^{-1}\cdot ( \ , \  )_{\la \la, \ra\ra_{K^{p}{}'}} \colon \Pi_{{H_{\Sg}'}}^{K^{p}{}'}\otimes_{\OO_{\X^{(5)}}} (\Pi_{{H_{\Sg}'}}^{K^{p}{}'})^{\iota}  \to \OO_{\X^{(5)}},\end{align}
where $(\ , )_{\la \la, \ra\ra_{K^{p}{}'}}$ is deduced from \eqref{iso 1}, \eqref{iso 2} and the isomorphism of  Proposition \ref{cofinal}.2 via Construction \ref{constr pairing}. 
\end{itemize}
\end{prop}
\begin{proof} 
Observe that the natural map
$$ (\cM_{K^{p}{}'}^{{H_{\Sg}'}})^{*}\to ((\cM_{K^{p}{}'})_{{H_{\Sg}'}})^{*}  =(\cM_{K^{p}{}'}^{*})^{{H_{\Sg}'}}$$
(where ${}^{*}$ denotes $\OO_{\X^{(5)}}$-dual) is an isomorphism; as \eqref{pair E} is equivariant for the action of the full Hecke algebra, this implies that its restriction \eqref{iso 1} is perfect.  It is skew-hermitian by Lemma \ref{Tequiv} and the fact that the Poincar\'e pairing  \eqref{poincare}, when $W$ is trivial, is skew-symmetric.

 We find the pairing \eqref{iso 2} by specialising \eqref{iso 1} to   $K^{p}{}'=K^{p}$, and the pairing \eqref{iso5} as described.
\end{proof}

\subsubsection{Specialisations} We describe the specialisation of the pairing $((, ))$ just constructed. 

For each algebraic representation $W $ of $(\G\times \H)$, denote by
$$\X_{r}^{\cl,W}:=\X\cap \cE_{K^{p},r}^{\cl , W},$$
the set of classical points of weight $W$ (omitted from the notation when $W=\Q_{p}$) and level $r$. Denote by a subscript `$W,r$' the pullbacks of sheaves or global sections from $\X^{(5)}$ to $\X_{r}^{\cl, W}$ (which is a finite \'etale scheme  over $\Q_{p}$). We let 
\begin{gather*}
V_{W,r}:=\Gamma(\X_{r}^{\cl, W}, \cV_{W, r}), \qquad
\Pi_{{H_{\Sg}'}, W, r}^{K^{p}{}', \ord}:=\Gamma(X_{r}^{\cl, W}, \Pi_{{H_{\Sg}'}}^{K^{p}{}', \ord}), \\
M_{K^{p}{}', W, r}^{{H_{\Sg}'}}:= \Gamma(\X_{r}^{\cl, W},  \cM_{K^{p}{}'}^{{H_{\Sg}'}})=  H_{1}(\baar{Z}_{K^{p}{}',r} , \cW)^{\ord,{H_{\Sg}'}},
\end{gather*}
where the last equality is by Proposition \ref{free indep}.3. 
We  denote
$$((\ , \ ))_{W, r}:={\rm v}(K^{p}{}')^{-1} \cdot (\ , \ )_{\la \la, \ra\ra_{K^{p}{}', W, r}}\colon \Pi_{{H_{\Sg}'}, W, r}^{K^{p}{}', \ord}\times \Pi_{{H_{\Sg}'}, W, r}^{K^{p}{}', \ord, \iota}\to \OO(\X_{r}^{\cl, W}).$$

\begin{prop} \label{abovco} 
Let $f_{1}$, respectively $f_{2}$ be global\footnote{The same statements hold with some extra notational burden if $f_{1}, f_{2}$ are only defined over an open subset of $\X^{(5)}$.} sections of $\Pi^{K^{p}{}', \ord}_{{H_{\Sg}'}}=\underline{\Hom}_{\OO_{\X^{(5)}}[G_{F, E}]}(\cM_{K^{p}{}'}^{{H_{\Sg}'}},\cV) $, respectively $(\Pi^{K^{p}{}', \ord}_{{H_{\Sg}'}})^{\iota}$. 
Let $$f_{1, W, r}\colon  M_{K^{p}{}', W, r}^{{H_{\Sg}'}} 
 \to V_{W,r}, \qquad f_{2, W, r}\colon 
( M_{K^{p}{}', W, r}^{{H_{\Sg}'}})^{\iota}\to V_{W,r}^{\iota}
$$ be $\OO_{\X^{\cl, W}_{r}}[G_{F, E}]$-linear maps. 

Let $x\in \X^{\cl, W}_{r}$, $\pi:= \pi(x) $ and let 
$$ f_{1, x}\colon  \H_{1}(\baar{Z}_{K^{p}{}',r} , \cW)\to V_{\pi},\qquad f_{2, x}\colon \H_{1}(\baar{Z}_{K^{p}{}',r} , \cW^{\vee})\to V_{\pi^{\vee}}$$  
be $\Q_{p}(x)[G_{F, E}]$-linear maps.

The following hold.
\begin{enumerate}
\item Suppose that for $i=1,2$,   the map $f_{i, W, r}$ arises as the specialisation of $f_{i}$. Then 
 $${((f_{1}, f_{2}))}|_{\X_{r}^{\cl, W}}= ((f_{1, W, r}, f_{2, W, r}))_{{W,r}}\qquad \text{in } \OO(\X_{r}^{\cl, W}).$$
 \item Suppose that  for $i=1, 2$, the map $f_{i, x}$ factors through the projection
 $${\rm p}^{?}\colon  \H_{1}(\baar{Z}_{K^{p}{}',r} , \cW^{?})\to  \H_{1}(\baar{Z}_{K^{p}{}',r} , \cW^{?})^{\ord}_{{H_{\Sg}'}} \cong H_{1}(\baar{Z}_{K^{p}{}',r} , \cW)^{\ord,{H_{\Sg}'}} = M_{K^{p}{}', W, r}^{{H_{\Sg}'}},$$
 where $?=\emptyset$ if $i=1$, $?=\vee$ if $i=2$; 
 and that $f_{i,x}$ coincides with the specialisation of $f_{i, W, r}$ at $x$. Then 
 \beq\label{eq specialise}
((f_{1}, f_{2}))_{W, r }(x)
= (f_{1,x}, f_{2,x})^{\ord}_{\pi} =\dim W\cdot (w_{\rm a}^{\ord} f_{1}, f_{2})_{\pi}
 \qquad \text{in } \Q_{p}(x).
\eeq
\end{enumerate}
\end{prop}

\begin{proof}
We simplify the notation by omitting the superscripts $ {{H_{\Sg}'}}$ and subscripts $K^{p}{}'$; moreover we ignore the normalisations ${\rm v}(K^{p}{}')^{-1}$ that are present in all of the pairings to be compared.

Part 1 follows from the definition \eqref{pair lam} if $W=\Q_{p}$, and similarly we can also identify $((\ ,  \ ))_{W,r}$ with  the restriction to $\X_{r}^{\cl, W}$ of the pairing on functions on $M_{W}$ deduced from $\la\la\ , \ \ra\ra_{W}$ via 
Construction \ref{constr pairing}. By Lemma \ref{pair indep} this implies that the desired statement holds for all $W$.

For Part 2, let $H_{W,r}:=\OO(\X_{r}^{\cl , W})$.  
First  notice that, by the construction of case 1 of Lemma \ref{et trace}, the diagram  of $H_{W, r}$-modules
\beqq
\xymatrix{
 \Hom_{H_{W,r}}(M_{W,r}^{\iota}, H_{W,r}(1))_{\baar{\Q}_{p}}\ar[d]\ar[r]^{\ \ \ \  \ \  \ u_{\la\la, \ra\ra_{W, r}}}  & M_{W,r}\ar[d]\\
  \Hom(M_{W, r, \Lambda_{W, r}}^{\iota}, \Lambda_{W, r}(1))\ar[r]^{\ \   \ \  \ \ \ u_{\la \la, \ra\ra_{\Lambda, W, r}}}  & M_{W,r}  
} 
\eeqq
is commutative. 

On the other hand, let $x\in \X_{\cl}^{W, r}$ and let $\alpha_{x}$ be the associated character of $T^{+}$. By  definition in \eqref{larara},
the pairing $\la  \la, \ra\ra_{\Lambda, W}$ specialises, on $M_{W,r|x}\otimes M_{W, r|x}^{\iota}$, to 
$$(x, y)\mapsto \sum_{t\in \baar{T}_{0}/\baar{T}_{r}} \la x, ty\ra_{U_{*}^{p},W, r}^{\ord} [t^{-1}](x)= \sum_{t\in \baar{T}_{0}/\baar{T}_{r}} \alpha_{x}(t) \la x, y\ra_{U_{*}^{p},W, r}^{\ord} \alpha^{-1}_{x}(t) =p^{r[F:\Q]} |\baar{T}_{0}/\baar{T}_{r}|\cdot \la x, y\ra_{U_{*}^{p},W, r}^{\ord}.$$
It follows that $u_{\la\la, \ra\ra_{W, r}}$ specialises at $x$ to
 $p^{-r[F:\Q]}|\baar{T}_{0}/\baar{T}_{r}|^{-1}u_{\la, \ra'_{W,r}}$, hence that the specialisation of $((, ))(x)={\rm v}({K^{p}{}'})^{-1}(,)_{\la\la,\ra\ra_{W,r}}(x)$ is
$${  (,)_{\la,\ra_{W,r}^{\ord}} 
\over p^{r[F:\Q]} |\baar{T}_{0}/\baar{T}_{r}|  {\rm v}({K^{p}{}'})}
 =
{p^{r[F:\Q]}
\cdot
{ {\rm v}(K^{p}{}' K_{1}^{1}(p^{r})_{p})}  (, )_{\pi}^{\ord}
\over {p^{r[F:\Q]} [{\rm v}(K^{p}{}' K_{1}^{1}(p^{r})_{p })/ {\rm v}(K^{p}{}' K_{0}(p)^{r}_{p }) ]   {\rm v}(K^{p}{}' K_{0}(p^{r})_{p})}}
= (, )_{\pi}^{\ord},$$
where we have used $|\baar{T}_{0}/\baar{T}_{r}|= {\rm v}(K^{p}{}' K_{0}(p^{r})_{p }) / {\rm v}(K^{p}{}' K_{1}^{1}(p)^{r}_{p }  )$ (by \eqref{deg vol}) and \eqref{compare'}. 

This establishes the first equality of \eqref{eq specialise}; the second one is just a reminder of \eqref{()'}.
\end{proof}

\subsection{Local toric pairings} \lb{sec 42}
Let $F$ be a non-archimedean local field, $E$ a quadratic \'etale algebra over $F$ with associated character $\eta\colon F^{\times}\to\{\pm 1\}$, $B$ a quaternion algebra over $F$,  $G=B^{\times}$, $H=E^{\times}$, $H'=H/F^{\times}$, and suppose given an embedding $ H\into G$ 

\subsubsection{Definition of the pairing}
Let $\pi$ be a smooth irreducible representation of $G$ over a finite extension  $ L$  of $\Q_{p}$, with a central character $\omega\colon F^{\times}\to L^{\times}$. Let $\chi\colon E^{\times}\to L^{\times}$ a character such that $\chi|_{F^{\times}}\cdot \omega=\one$. We identify $\chi$ with a representation  $L\chi$ of $E^{\times} $ on $L$, and when more precision is needed we denote by $e_{\chi}$ the basis element corresponding to the character $\chi$ in $L\chi$.  Let  $\Pi:=\pi\otimes\chi$, a representation of $(G\ts H)'=(G\times H)/F^{\times}$ over $L$. We assume that $\pi$ is essentially unitarisable, that is  that for any embedding $\iota\colon L\into \bC$,  a twist of $\iota\pi$ is isomorphic to the space of smooth vectors of a unitary representations. (This holds automatically if $\pi $ arises as the local component of a cuspidal automorphic representation over $L$.)
Let $\pi^{\vee}$ be the smooth dual, $\Pi^{\vee}:=\pi^{\vee}\otimes \chi^{-1}$

Assume from now on  that the modified local sign  $\eps(\Pi)=\eqref{eps v}$ equals $+1$.
 Then, by the result of Tunnell and Saito mentioned in the introduction, 
the space
$$\Pi^{*,H'}:={\Hom}_{H'} (\Pi, L).$$
has dimension   $1$ over $L$. Moreover the choices of an invariant pairing $(\ , \ )$ on $\Pi \otimes \Pi^{\vee}$  and a Haar measure $dt$ on $\H'$ 
give a   generator  
$$Q=Q_{(\ , \ ), dt}\in \Pi^{*,H'}\otimes_{L} (\Pi^{\vee})^{*,H'}$$
defined by the absolutely convergent integral
\begin{gather}\label{Qpair}
 Q_{(, )}(f_{1}, f_{2}):=
 \calL(V_{v}, 0)^{-1}
 \cdot
\iota^{-1} \int_{E^{\times}/F^{\times}} (\iota\Pi(t) f_{1}, \iota f_{2})\, dt ;
 \end{gather}
 for any  $\iota\colon L\into \bC$; here $\calL(V_{v}, 0)=\eqref{calLv}$. 
 
 Recall also from the introduction  \eqref{Qv intro} that 
\beq\label{Qv sec4}
 Q_{dt}\left({f_{1}\ot f_{2}\over f_{3}\ot f_{4} }\right) := {Q_{( , ), dt}(f_{1}, f_{2})\over ( f_{3}, f_{4})},
\eeq
is independent of $(\ ,\ )$ whenever it is defined.

We study the pairing, or some of its variations, in a few different contexts.

\subsubsection{Interpretation in the case $E=F \oplus F$} In this case   $G=\GL_{2}(F)$, and  the  integral \eqref{Qpair} has an interpretation as product of zeta integrals.  Let $\cK(\pi)$ and $\cK(\pi^{\vee})$ be Kirillov models over $L$ as in \cite[\S~ 2.3]{dd-pyzz}.
By \cite[Lemma 2.3.2]{dd-pyzz}, the $L$-line of invariant pairings on $\cK(\pi)\times\cK(\pi^{\vee})$  is generated 
by an element  $(\ , \ )$ such that, for each $\iota\colon L\into \bC$, 
we have 
\beq \label{kir pair}
\iota ( f, f^{\vee})={\zeta(2)\over L(1, \pi\times \pi^{\vee})}\cdot \int_{F^{\times}} \iota f(y) \iota f^{\vee}(y) d^{\times }y,\eeq
where
 the integral is absolutely convergent (as $\iota\pi$ is essentially unitarisable) and $d^{\ts}y$ is any $L$-valued Haar measure. Identify $E^{\times}$ with the diagonal torus in $\GL_{2}(F)$ and write $\chi=(\chi_{1}, \chi_{2})$ according to the decomposition $E=F\oplus F$; noting that $\chi_{2}=\omega^{-1}\chi_{1}$ and $\pi=\pi^{\vee}\otimes \omega^{-1}$, we indentify   $Q_{(, ), dt}$ with
\begin{gather}\label{Qv reg}
\iota Q_{(, ), dt}(f\otimes e_{\chi}, f^{\vee}\otimes e_{\chi^{-1}})\doteq L(1/2, \iota \pi_{E}\otimes \iota\chi)^{-1} \int_{E^{\times}/F^{\times}} \iota \chi(t) (\iota\pi(t) f,  \iota f^{\vee} )|t|^{s}\, dt|_{s=0}\\
\nonumber
=
 L(1/2, \iota \pi\otimes \iota\chi_{1})^{-1} \int_{F^{\times}} \iota f^{\vee}(t)\iota\chi_{1}(t) |t|^{s}d^{\times }t|_{s=0}
\cdot
L(1/2, \iota \pi\otimes \iota\chi_{1})^{-1} \int_{F^{\times}} \iota f^{\vee}{}(y)\iota\chi_{1}(y) |y|^{s}d^{\times }y|_{s=0}\\
\nonumber
= (L(1/2,\iota \pi\otimes \iota \chi_{1})^{-1} \cdot I(\iota f, \iota\chi_{1}, 1/2))\cdot  (L(1/2, \pi\otimes \iota\chi_{1})^{-1} \cdot I(\iota f^{\vee} , \iota\chi_{1} , 1/2)),
\end{gather}
where $I(\cdot , \cdot , 1/2)$ is the zeta integral of \cite[\S~5.2]{LLC} 
 for $\GL_{2}(F)\times \GL_{1}(F)$, and $\doteq$ denotes an equality up to constants in $L^{\ts}$ depending on the choices of measures.

\subsubsection{Special line in the unramified case} We study the first one in a short list of special cases in which there are `canonical' lines in $\Pi$, $\Pi^{\vee}$, on which  the value of the pairings $Q $ can be explicitly computed. 
\begin{lemm}[{\cite[Lemme 14]{wald}}]\label{unrQ} Suppose that $B$ is split, $E/F$ is unramified, and both $\pi$ and $\chi$ are unramified. Let $K\subset (G\times E^{\times})/F^{\times}$ be   a maximal compact subgroup. Then 
$$Q_{( \ ,\  ), dt}(v, w)=\vol(\OO_{E}^{\ts}/\OO_{F}^{\ts}, dt)\cdot  (v, w)$$
 for all $v$, respectively $w$, in the lines $\Pi^{K}$, respectively $(\Pi^{\vee})^{K}$.
\end{lemm}

\subsubsection{Special line when   $B$ is nonsplit} Suppose now that $B$ is nonsplit and that $\Pi$ is an irreducible representation of $(G\times E^{\times})/F^{\times}$ as above. Note that $\Pi$ is finite-dimensional and $H'$ is compact, so that $\Pi^{\vee}=\Pi^{*}$ and  the natural maps $\Pi^{H'}\to \Pi_{H'}$ ($=H'$-coinvariants) and $\Pi^{*,H'} \to (\Pi^{H'})^{*}$ are isomorphisms. Moreover the non-degenerate pairing $(\ , \ )$ restricts to a non-degenerate pairing on $\Pi_{H'}\otimes \Pi^{\vee}_{H'}$.
Then we may compare the restrictions of the pairings $Q_{(\ ,  \ )}$   of $(\ , \ )$  to the line $\Pi^{H'}\otimes \Pi^{\vee,H'}$.
\begin{lemm}\label{ramifQ} In the situation of the previous paragraph, we have 
\beqq
Q_{(\ , \ ), dt}= \calL(V_{(\pi , \chi), v}, 0)^{-1}
 \cdot \vol(E^{\times}/F^{\times}, dt)\cdot (\ , \ )
\eeqq
as elements  of $(\Pi^{H'})^{*}\otimes (\Pi^{\vee,H'})^{*}$. 
\end{lemm}
\begin{proof} 
  This follows from the definition in \eqref{Qpair}, since in this case the integration over the compact set $E^{\times}/F^{\times}$ converges.
\end{proof}

\subsection{Ordinary toric pairings} \lb{sec 43}
We define a variant  for ordinary forms of the pairing $Q$.

\subsubsection{Definition of the ordinary paring} Let $\Pi=\pi\otimes \chi$ be an ordinary automorphic representation of $(\G\times \H)'(\A^{})$ over $L$.
When referring to  local objects considered in the previous paragraphs or products thereof, we append subscripts as appropriate.

For each $v\vert p$, let $$\mu_{v}^{+}\colon E_{v}^{\ts}\to L^{\ts}$$ be the character by which $E_{v}^{\ts}$ (or equivalently $\prod_{w\vert v} G_{E_{w}}^{\rm ab}$) acts on $V_{\pi,v}^{+} \ot\chi_{v}$, and let ${\rm j}_{v}\in E_{v}$ be the purely imaginary element fixed in \eqref{jv}. Define 
$$\mu^{+}({\rm j}):=\prod_{v\vert p} \mu_{v}^{+}({\rm j}_{v}).$$
For measures $dt_{v}=dt_{v, p}$ on $H'_{v}$, $dt_{v,\infty}$ on $H'_{v, \infty}$ (the latter a merely   formal notion as in the introduction), define
\beq\lb{vol circ} {\vol^{\circ}(H'_{v}, dt_{v})} &:=
 {\vol(\OO_{E, v}^{\ts}/\OO_{F, v}^{\ts}, dt_{v})\over e_{v}
   L(1, \eta_{v})^{-1}},
  \qquad { \vol^{\circ}(H'_{v, \infty}, dt_{v, \infty})} := {\vol (H'_{v, \infty}, dt_{v, \infty}) \over 2^{[F_{v}:\Q_{p}]}}, \\
   {\vol^{\circ}(H'_{p\infty}, dt_{p\infty})}&:= \prod_{v\vert p } {\vol^{\circ}(H'_{v}, dt_{v})} \cdot { \vol^{\circ}(H'_{v, \infty}, dt_{v, \infty})}.
\eeq
The denominators in the right-hand sides are the volumes of $\vol(\OO_{E, v}^{\ts}/\OO_{F, v}^{\ts})$, respectively $\bC^{\ts}/\R^{\ts}$, for the ratio of  (rational normalisations of) selfdual measures, cf. \cite[\S~ 1.6.2]{yzz} and the proof of Proposition \ref{toric period}.

\begin{defi}\lb{Q ord def}
Let $dt=dt^{p\infty}dt_{p\infty}$ be a decomposition of the ad\`elic measure $dt$ specified in \eqref{dt norm}. 
Then we define:
\begin{itemize}
\item
 for each $f_{1, p \infty}, f_{3,p \infty}\in \Pi_{p \infty}^{\ord}$, $f_{2, p \infty}, f_{4,p \infty}\in \Pi_{p \infty}^{\vee, \ord}$ with $f_{3,p \infty}\ot f_{4,p \infty}^{\ord}\neq 0$, 
\beq 
\label{Q orddd loc} 
Q^{\ord}_{dt_{p\infty}}\left( {f_{1,p\infty}\ot f_{2,p\infty}\over f_{3,p\infty}\ot f_{4,p\infty}}\right) :=\mu^{+}({\rm j})^{} 
 {\vol^{\circ}(H'_{p\infty}, dt_{p\infty})} \cdot {f_{1, p\infty}\ot f_{2, p\infty}\over f_{3, p\infty}\ot f_{4, p\infty}}.
\eeq
\item
 for each $f_{1}, f_{3}\in \Pi^{\ord}$, $f_{2}, f_{4}\in \Pi^{\vee, \ord}$ with $(f_{3}, f_{4})^{\ord}\neq 0$, 
\beq 
\lb{Q orddd} 
Q^{\ord} \left( {f_{1}\ot f_{2}\over f_{3}\ot f_{4}}\right)  
:=Q^{p\infty}_{dt^{p\infty}}\left( {f_{1}^{p\infty}\ot f_{2}^{p\infty}\over f_{3}^{p\infty}\ot f_{4}^{p\infty}}\right)
\cdot 
Q^{\ord}_{dt_{p\infty}}\left( {f_{1,p\infty}\ot f_{2,p\infty}\over f_{3,p\infty}\ot f_{4,p\infty}}\right).
\eeq
\end{itemize}
\end{defi}
The normalisation at $p\infty$ is justified by the clean formula of Proposition \ref{compare Qs} below.
\begin{rema} \lb{Qord nonzero}Suppose  that $\Pi$ is locally distinguished, so that as explained in the introduction the functional $Q_{dt}$ is nonzero. Then the functional $Q^{\ord}_{dt}$ is also nonzero.
\end{rema}

\subsubsection{Decomposition}
Fix a decomposition $dt= \prod_{v\nmid p \infty } dt_{v} dt_{p\infty}$ such that for all but finitely many $v$, $\vol(\OO_{E,v}^{\ts}/\OO_{F, v}^{\ts})=1$. 
Let $\Sg'$ be a finite set of finite  places of $F$  disjoint from $\Sg $ and $S_{p}$ and containing the other places of ramification  of $\Pi$, and those such that $\vol(\OO_{E,v}^{\ts}/\OO_{F, v}^{\ts})\neq 1$. Let $K^{p}\subset (\G\times \H)'(\A^{p\infty})$ be an open compact subgroup that is maximal away from $S:=\Sg\cup \Sg'$ and such that $\Pi_{v}^{K_{v}}=\Pi_{v}$ for $v\in \Sigma$. 
\begin{lemm} For all $f_{1}, f_{3} \in \Pi^{K^{p}, \ord}_{H_{\Sg}'}$, $f_{2}, f_{4} \in \Pi^{\vee, K^{p}, \ord}_{H_{\Sg}'}$ with $(f_{3}, f_{4})^{\ord}\neq 0$, we have
\beq\label{dec Qord} Q^{\ord}\left(   {f_{1} \ot f_{2}\over  f_{3}\ot f_{4} } \right)
&=
\prod_{v\in \Sg'}Q_{v, dt_{v}}\left({  f_{1, v} \ot  f_{2, v}\over f_{3,v}\ot f_{4, v}}\right)
 \cdot
 \prod_{v\in \Sg} \vol(E_{v}^{\times}/F_{v}^{\times}, dt)   \calL(V_{(\pi , \chi), v}, 0)^{-1}   {  f_{1, v} \ot  f_{2, v}\over f_{3,v}\ot f_{4, v}}\\
&\cdot 
  {f^{Sp\infty}_{1}\otimes f_{2}^{Sp\infty}\over f^{Sp\infty}_{3}\otimes f^{Sp\infty}_{4}} \cdot Q^{\ord}_{p\infty,dt_{p\infty}}\left( {f_{1,p\infty}\ot f_{2,p\infty}\over f_{3,p\infty}\ot f_{4,p\infty}}\right).
 \eeq
\end{lemm}

\begin{proof} This follows from the definitions and the results of \S~\ref{sec 42}. 
\end{proof}
\subsubsection{Relation to between the toric pairing and its ordinary variant} We  gather the conclusion of the computations from the appendix.

\begin{prop} \lb{compare Qs} Let $\Pi=\pi\ot \chi= \Pi^{\infty} \ot W$ be an ordinary representation of $(\G\ts \H)'(\A)$.  
Let  $w_{\rm a}^{\ord}$ and $\tord$ be the operators defined in Propositions \ref{w ord} and \ref{gHo}. Let $e_{p}(V_{(\pi, \chi)})=  \eqref{epinf L}$ be the interpolation factor of the $p$-adic $L$-function.  For all  $f_{1}$, $f_{3}\in \Pi^{\ord}$, $f_{2}$, $f_{4}\in \Pi^{\ord}$ with $(f_{3}, f_{4})^{\ord}\neq0$, we have
\beqq Q^{}\left(   {\tord(f_{1}) \ot \tord (f_{2})\over  w_{\rm a}^{\ord}(f_{3})\ot f_{4} } \right) =   e_{p}(V_{(\pi, \chi)})\cdot \dim W \cdot  Q^{\ord}\left(   {f_{1} \ot f_{2}\over  f_{3}\ot f_{4} } \right). \eeqq
\end{prop}
\begin{proof}
There is a   decomposition $Q_{p\infty , dt_{p\infty}}^{\ord} = \prod_{v\vert p}Q_{v,dt_{v}}^{\ord}\cdot  \prod_{v\vert p}Q_{v, \infty, dt_{v, \infty}}^{\ord}$, whose terms are defined in \S\S~\ref{A3}-\ref{A4}. The only point worth stressing is  that if $\mu_{v}^{+}$, respectively $\mu_{v, \infty}^{+}$ is the character defined in \S~ \ref{A3mu},\footnote{Note that despite the similar notation, the character $\mu_{v}$ is defined using the Weil--Deligne  representations rather than the continuous Galois representations.} respectively \S~\ref{A4mu}, then the decomposition $\mu^{+}=\mu^{+, {\rm sm}}\mu^{+, {\rm alg}}$  of $\mu^{+}$ into a product of a smooth and an algebraic character is given by 
$\mu^{+, {\rm sm}}=\prod_{v\vert p }\mu_{v}^{+}$, $\mu^{+, {\rm alg}}= \prod_{v\vert p}\mu_{v, \infty}^{+}$. 

Then the result follows from Propositions \ref{toric period} and \ref{compare toric inf}.
\end{proof}

\subsection{Interpolation of the  toric pairings}  \lb{sec 44} We interpolate the pairings $Q^{\ord}_{dt}$ along Hida families
\subsubsection{Interpolation of  the local pairings} We use the same notation $F, E$ of \S~\ref{sec 42}.

\begin{lemm}\label{interp Lv} Let $\X$ be a scheme over $\Q$ and let   $r'=(r, N)$ be a Weil--Deligne representation  of $W_{F}$
on a rank-$2$ locally free sheaf over $\X$.  Suppose that $\X$ contains a  dense subset $\X^{\cl}$ such that $r'_{x}$ is pure for all $x\in \X^{\cl}$. Let $\ad(r')$ be the rank-$3$ adjoint representation. Then there exist an open subset $\X''\subset \X$ containing $\X^{\cl }$ and functions 
$$L(0, r')^{-1}, \qquad L(1, r', \ad)\in \OO(\X'')$$
such that for every $x\in \X'$ we have $L(0, r')^{-1}(x)=L(0, r'_{x})^{-1}$ and $L(1, r', \ad)(x)=L(1, \ad(r'_{x}))$.
\end{lemm}
\begin{proof}  
By  \cite[\S~5.1]{LLC},
 there exist an open set $\X'''\subset \X$ containing $\X^{\cl}$ and functions
 $L(0, r')^{-1}$, respectively $L(1, r', \ad)^{-1}$, in $\OO(\X''')$ interpolating $L(0, r'_{x})^{-1}$, respectively $L(1, \ad(r'_{x}))^{-1}$, for all $x\in \X'$.  By purity, $L(1,r, \ad)^{-1}$ does not vanish on $\X^{\cl}$, hence it is invertible in an open neighbourhood $\X''$ of $\X^{\cl}$ in $\X'''$.
\end{proof}

 Let $\X$ be an integral scheme,  $\cF^{\ts}$ be a $\cK_{\X}^{\times}$-module, then we define $\cF^{\ts,-1}$ to be the $\cK_{\X}$-module such that for each open $\cU\subset \X$, 
\beqq
\cF^{\ts,-1}(\cU):= \{ f^{-1}\, |\ \ f\in \cF^{\ts}(\cU)\}\eeqq
  with $\cK_{\X}^{\times}$-action given by $a\cdot f^{-1}=(a^{-1}f)^{-1}$.

\begin{prop} \label{interp QcircS}
Consider the situation of Lemma \ref{interp Lv}. Let 
$$\pi=\pi(r')$$
 be the $\OO_{\X}[\GL_{2}(F)]$-module attached to $r'$  by   the local Langlands   correspondence in  families of \cite{LLC}, let $\omega\colon F^{\times}\to \OO(\X)^{\times}$ be its central character, and  let $\chi\colon E^{\times}\to \OO(\X)^{\times}$ be a  character  such that $\omega\cdot \chi|_{F^{\times}}=\one$. 
Let $\pi^{\vee}:=\pi(\rho^{*}(1))$ and let $\Pi=\pi \otimes \chi$, $\Pi^{\vee}=\pi^{\vee}\otimes \chi^{-1}$. 
Let $(\Pi\ot_{\OO_{\X}^{\ts}}\Pi^{\vee})^{\ts}$ be the $\OO_{\X}^{\ts}$-submodule of those $f_{3}\ot f_{4}$ such that $(f_{3}, f_{4})\neq 0$. 

 Then there exist: an open subset $\X'\subset \X$ containing $\X^{\cl}$; letting  $\OO:=\OO_{\X'}$, $\cK:=\cK_{\X'} $, an $\OO^{\ts}$-submodule
  $(\Pi\ot_{\OO_{\X}^{\ts}}\Pi^{\vee})^{\ts}$ specialising at all $z\in \X^{cl}$ to the space of $f_{3,z}\ot f_{4,z}$ such that $(f_{3, z}, f_{4, z})_{z}\neq 0$;  and  a map of $\OO$-modules
$$\cQ_{dt}^{} \colon 
(\Pi\otimes_{\OO}\Pi^{\vee})\otimes_{\OO^{\times}}  ({\Pi} \otimes_{\OO^{\times}} {\Pi}^{\vee})^{\ts, -1}
 \to \mathscr{K},$$
satisfying the following properties.
\begin{enumerate}\item
For all $t_{1}, t_{2}\in E^{\times}/F^{\ts}$, $g\in (\GL_{2}(F)\times E^{\times})/F^{}$,
$$\cQ_{dt}\left({ \Pi(t_{1})f_{1}\otimes \Pi^{\vee}(t_{2})f_{2})\over \Pi( g)f_{3}\otimes f_{4}}\right)=
\cQ_{dt}\left( 
{f_{1}\otimes f_{2}\over  f_{3}\otimes \Pi^{\vee}(g^{-1})f_{4}}\right);$$
\item For all $x\in \X^{\cl}$,  
$$\cQ^{}_{dt | x} =Q_{dt}, $$ where $Q_{dt}$ is  the paring on $\Pi_{x}\ot \Pi_{x}^{\vee}$ of   \eqref{Qv sec4}.
\end{enumerate}
\end{prop}
\begin{proof} 
For each $x\in \X^{\rm cl}$, $\pi_{x}$ corresponds to a pure Weil--Deligne representation under local Langlands, hence it is essentially unitarisable (and in fact tempered, see \cite[Lemma 1.4 (3)]{TY}).
Then by \cite[Lemma 5.2.5]{LLC} 
there is an open neighbourhood $\X'$ 
of $\X^{\cl}$ in $\X$ and an invariant  pairing over $\X'$ 
\beq
\label{inv pair family}
(\ , \, )\colon \pi\otimes \pi^{\vee}\to \OO_{\X'}
\eeq
specialising to   the pairing $(\ , \ )_{x}$ defined by  \eqref{kir pair} at all $x\in \X^{\cl}$.  It induces an invariant pairing $ \Pi\otimes \Pi^{\vee}\to \OO_{\X'}$ still denoted by $(\ , \, )$.

By Lemma \ref{interp Lv}, up to possibly shrinking $\X'$, we have regular functions on $\X'$ interpolating   $z\mapsto L(1/2, \pi_{z,E}\otimes\chi_{z})^{-1}=L(0, r_{z}|_{W_{E}'} \otimes \chi_{z})^{-1}$ and $x\mapsto L(1, \pi_{x}, \ad)=L(1, r_{x}', \ad)$.

If $E/F$ is split,   \cite[Proposition 5.2.4]{LLC} 
 applied to \eqref{Qv reg} gives an element $\cQ_{(, ),dt}\colon \Pi_{E^{\times}}\otimes \Pi_{E^{\times}}^{\vee}\to \OO_{\X' }$ interpolating $Q_{(, )_{x}}$ for $x\in \X^{\cl}$, and we define
\beq\label{qcircdef}
\cQ_{dt}\left({f_{1}\ot f_{2}\over f_{3}\ot f_{4} }\right) := {\cQ_{( , ), dt}(f_{1}, f_{2})\over ( f_{3}, f_{4})},
\eeq

If $E/F$ is nonsplit, by the previous discussion we can interpolate all terms occurring in the definition  \eqref{Qpair} (note that the integral there is just a finite sum), to obtain a pairing  $\cQ_{(,)}$ over $\X'$ interpolating $\cQ_{(,)_{x}}$ for $x\in \X^{\cl}$. Then we again define $\cQ^{}$ by \eqref{qcircdef}.
\end{proof}

\subsubsection{Product of local pairings} We  consider the global situation, 
 resuming with the setup of \S\S~\ref{3.2}-\ref{sec: duality}.

Let $\Pi:=\Pi^{K^{p}{}',  \ord}_{{{H_{\Sg}'}}}$ over $\X^{(5)}$.
Recall that we have a decomposition 
\beq\label{deco}
\Pi
\cong (\pi_{\G,\Sg'}(\cV_{\G})\otimes \chi_{\H, {\rm univ},{\Sg'}})
\otimes_{\OO_{\X^{(5)}}}
  \Pi^{K^{p}{}',S,  \ord}_{{H_{\Sg}'}} 
  \eeq
from Theorem \ref{LGC}.

 Let 
 $
 (\Pi\otimes_{\cK^{\times}_{\X^{(5)}}}\Pi^{\iota})^{\ts}
 \subset  \Pi\otimes_{\cK^{\times}_{\X^{(5)}}}\Pi^{\iota}$ be the 
  $\OO_{\X^{(5)}}^{\times}$-submodule of sections $f_{3}\otimes_{\OO_{\X^{(5)}}^{\times}} f_{4}$ such that $f_{3}\otimes f_{4}\neq 0$ and $(f_{3,v},f_{4,v})_{v}\neq 0$ for each  the pairings $(, )_{v}=$ \eqref{inv pair family}, $v\in \Sg'$.

 \begin{theo} \label{X6} Let $\Pi:=\Pi_{{H_{\Sg}'}}^{K^{p}{}', \ord}$ and $\X^{(5)}$ be as in Proposition \ref{duality}, and let 
 $
 (\Pi\otimes_{\cK_{\X^{(5)}}}\Pi^{\iota})^{\ts}
 \subset  \Pi\otimes_{\cK_{\X^{(5)}}}\Pi^{\iota}$ be the submodule defined above. 
 Then there exist an open subset $\X^{(6)}\subset \X^{(5)}$ containing $\X^{\cl}$ and, letting  $\OO=\OO_{\X^{(6)}}$, $\cK:=\cK_{\X^{(6)}}$, a map of $\OO^{\ts}$-modules
$$\cQ^{} \colon 
(\Pi\otimes_{\OO}\Pi^{\iota})\otimes_{\OO^{\times}}  ({\Pi} \otimes_{\OO^{\times}} {\Pi}^{\iota})^{\ts, -1}
 \to \mathscr{K}_{\X}$$
satisfying:
\begin{enumerate}\item
For any $t_{1}, t_{2}\in E_{\Sg'}^{\times}/F_{\Sg'}^{\ts}\subset (\GL_{2}(F_{\Sg'})\times E_{\Sg'}^{\times})/F_{\Sg'}^{\ts}$, any $h\in \cH_{S, K_{\Sg'}}$  
and any section 
$$(f_{1}\otimes f_{2}) \otimes  (f_{3}\otimes f_{4})^{-1} \quad 
\text{ of }
\quad
( \Pi\otimes_{\cK} \Pi^{\iota})\otimes_{\cK^{\times}}  ({\Pi} \otimes_{\cK^{\times}} {\Pi}^{\iota})^{\ts, -1}, $$
 we have
$$\cQ^{}\left({\Pi_{\Sg'}(t_{1})f_{1} \otimes \Pi_{\Sg'}^{\iota}(t_{2})f_{2} \over \Pi( h)f_{3} \otimes f_{4}}\right)=
  \cQ^{}\left({ f_{1}\otimes f_{2} \over f_{3} \otimes \Pi^{\iota}(h)f_{4}}\right);$$
 in the left-hand side, $\Pi_{\Sg'}$, respectively $\Pi_{\Sg'}^{\iota}$ denote the actions of the Hecke algebras at $S$ on  $\Pi$, respectively $ \Pi^{\iota}$.
\item For all $x\in \X^{\cl}$,  $$\cQ^{}_{dt |x}=Q^{\ord}, $$  where $Q_{}^{\ord}$  is the restriction of the pairing on $\Pi_{x}^{\vee, \ord}\ot \Pi_{x}^{\vee,\ord}$ of Definition \ref{Q ord def}.
\end{enumerate}
\end{theo}

\begin{proof} By \eqref{Q orddd loc}, \eqref{Q orddd}, \eqref{dec Qord}, and \eqref{deco}, we need to interpolate:
\begin{itemize}
\item the terms $\calL(V_{(\pi, \chi), v}$ for $v\in \Sg$: this is Lemma \ref{interp Lv};
\item the characters $\mu_{v}^{+}$ for $v\vert p$:  this follows form the existence of the filtration  \eqref{decatv} over  an open subset of $\X$.
\item the term $Q_{dt, \Sg'}:=\prod_{v\in \Sg'}Q_{dt, v}$,
 According to the proof of \cite[Theorem 4.4.1]{LLC}, 
the representation $\pi_{\G, {\Sg'}}(\cV_{\G})$ is the maximal torsion-free quotient of $\otimes_{v\in {\Sg'}}\pi_{\G,v}(\cV_{\G})$.  For  sections $f_{i,{\Sg'}}$ that are  images of $\otimes_{v\in {\Sg'}}f_{i,v}$, with $f_{i,v}$ sections of $ \pi_{\G,v}(\cV_{\G})\otimes\chi_{\H, {\rm univ},v}$ if $i=1,3$, or of  $ \pi_{\G,v}(\cV_{\G}^{\iota})\otimes\chi^{-1}_{\H, {\rm univ},v}$ if $i=2,4$, let
$$\cQ^{}_{{\Sg'}}\left({ f_{1, {\Sg'}}\otimes f_{2, {\Sg'}}\over f_{3, {\Sg'}}\otimes f_{4, {\Sg'}}}\right):=
\prod_{v\in {\Sg'}} \cQ_{v}\left( {f_{1,v}\otimes f_{2,v}\over f_{3,v}\otimes  f_{4,v}}\right),$$
where the factors in the right-hand side are provided by Proposition \ref{interp QcircS}. This is
 well-defined  independently of the choices of $f_{i,v}$ as $\mathscr{K}$ is torsion-free.   
\end{itemize}

This completes the interpolation of \eqref{Q orddd} into a function $\cQ$, that satisfies properties 1 and 2 by  construction and the corresponding properties from Proposition \ref{interp QcircS}.
\end{proof}

\section{Selmer sheaves and $p$-adic heights}\label{big sec: sel}
In this section we present the   theory of Selmer complexes and $p$-adic heights  needed in the rest of the paper. The foundational material  is  taken from the   book of \nek\  \cite{nek-selmer}.
\subsection{Continuous cohomology} 
Let
  $(R^{\circ}, \mathfrak{m})$ be a complete Noetherian local ring,  let $G$ be a topological group. 
  \subsubsection{Continuous cochains for (ind-) admissible $R[G]$-modules} \lb{ssec adm}
Let   $M$  be an $R^{\circ}[G]$-module.
  We say that $M$ is  \emph{admissible of finite type }   if it is of finite type as an $R^{\circ}$-module  and  
   the action $G\times M\to M$ is continuous  (when $M$ is given the $\mathfrak{m}$-adic topology). We say that $M$ is \emph{ind-admissible} if $M=\bigcup_{\alpha}M_{\alpha}$ where $\{M_{\alpha}\}$ is the set of   finite-type admissible $R^{\circ}[G]$-submodules of $M$. 

The complex of continuous cochains of $M$ is denoted by $C^{\bullet}_{\rm cont}(G, M)$; it is defined in the usual  way \cite[(3.4.1)]{nek-selmer} when $M$ is admissible of finite type,  and by  $C^{i}_{\rm cont}(G, M):= \varinjlim_{\alpha} C^{i}_{\rm cont}(G, M_{\alpha})$ when we have a presentation $M=\bigcup_{\alpha}M_{\alpha}$ as above.   The image of  $C^{\bullet}_{\rm cont}(G, M)$ in the derived category of   $D({}_{R}{\rm Mod})$ of $R^{\circ}$-modules is denoted by 
$$ {\rm R}\Gamma(G, M)$$
and its cohomology groups by $$H^{i}(G, M)$$
(we omit the subscript `cont' as we will only be working with continuous cohomology). 
 
\subsubsection{Localisation} Let  
$$R=R^{\circ}[\mathcal{S}^{-1}]$$
 for some multiplicative subset $\mathcal{S}\subset R^{\circ}$, and let $M$ be an $R[G]$-module. We say that $M$ is ind-admissible if it is ind-admissible as an $R^{\circ}[G]$-module, and that it is of finite type if it is of finite type as an $R$-module. Suppose that   $M:=M^{\circ}\otimes_{R^{\circ}} R$ for an ind-admissible $R^{\circ}[G]$-module $M^{\circ}$. Then $M$ is ind-admissible  as an $R^{\circ}[G]$-module and   there is a canonical isomorphism
\beq\label{localisation}
C^{\bullet}_{\rm cont}(G, M)\cong C^{\bullet}_{\rm cont}(G, M^{\circ})\otimes_{R^{\circ}}R
\eeq
(\cite[(3.7.4)]{nek-selmer}).

\begin{rema} \lb{calCR}Let 
$$\mathscr{C}= \mathscr{C}_{R^{\circ}}$$ 
be the category of schemes isomorphic to open subschemes of $\Spec R^{\circ}$, with maps being open immersions. It follows  from  the previous paragraph that,  for any object $X$ of $\mathscr{C}$,  the condition of ind-admissibility is defined for all quasicoherent $\OO_{X}[G]$-modules, and  the functors ${\rm R}\Gamma(G,- )$  are well-defined on ind-admissible  $\OO_{X}[G]$-modules. Moreover, both the ind-admissibility condition and the functors ${\rm R}\Gamma(G, - )$  are compatible with restriction along open immersions in $\mathscr{C}$. 
\end{rema}

In  the  following, we will not further comment on the generalisation indicated in the previous remark when referring to sources only considering $R^{\circ}[G]$-modules.

\subsubsection{Completed product} \lb{compl product} For $i=1,2$, let $R_{i}^{\circ}$ be complete noetherian local rings, and let $R^{\circ}:= R_{1}^{\circ}\hat{\ot} R_{2}^{\circ}$. We have  a functor 
\beq\lb{hatprod}\hat{\ts}\colon \mathscr{C}_{R_{1}^{\circ}}\ts \mathscr{C}_{R_{2}^{\circ}} \to  \mathscr{C}_{R^{\circ} }\eeq
defined on objects 
by $\Spec R_{1}^{\circ}[1/f_{1}] \hat{\ts} \Spec R_{1}^{\circ}[1/f_{2}]:= \Spec R_{1}^{\circ}\hat{\ot}R_{2}^{\circ}[1/f_{1}\ot1, 1/1\ot f_{2}]$ and glueing.

\subsubsection{Notation} 
Throughout the rest of this section, 
$X$ will denote an object of $\mathscr{C}_{R^{\circ}}$.  If $\calA=\OO_{X}, \OO_{X}[G]$, we denote  by $D({}_{\calA}{\rm Mod})$ the derived category of $\calA$-moduels. We use sub- or superscripts $$\text{ft,  {ind-adm}, $+$, $-$, {b}, $[a, b]$,  {perf},} $$ to denote the full subcategory of objects quasi-isomorphic to complexes of $\calA$-modules that are respectively termwise of finite type, termwise ind-admissible, bounded below, bounded above, bounded, concentrated in degrees $[a,b]$, bounded, perfect (= bounded and termwise projective and of finite type).

\begin{prop}[{ \cite[(3.5.6)]{nek-selmer}}]
The functor ${\rm R}\Gamma(G, -)$ can be extended  to a functor on the category of bounded-below complexes of ind-admissible $\OO_{X}[G]$-modules, with values in bounded-below complexes of $\OO_{X}$-modules \cite[(3.4.1.3), (3.5.1.1)]{nek-selmer}. It descends to an exact   functor  
$${\rm R}\Gamma(G, -)\colon D^{+}({}^{\textup{ind-adm}}_{\OO_{X}[G]}{\rm Mod})\to D^{+}({}_{\OO_{X}}{\rm Mod}).$$
\end{prop}

\subsubsection{Base-change} Suppose that $R\twoheadrightarrow R'=R/I$ is a surjective map of rings. Let $j\colon R^{\circ} \to R= R^{\circ}[\mathscr{S}^{-1}]$ be the natural map and let  $I^{\circ}:= j^{-1}(I)$. Then $R^{\circ}{}':= R^{\circ}/I^{\circ}$ is also complete local Noetherian, and  we may write $R'= R^{\circ}{}'[\mathscr{S}']^{-1}$ where  $\mathscr{S}'$ is the image of $\mathscr{S}$ in $R^{\circ}{}'$. Let $M'$ be an ind-admissible $R'$-module, then $C^{\bullet}_{\rm cont}(G, M')$ is the same whether we consider $M'$ as an $R'$-module or as an $R$-module: in the special case $R=R^{\circ}$ this follows from the fact that the maximal ideal of $R^{\circ}{}'$ is the image of the maximal ideal of $R^{\circ}$, so that the $\mathfrak{m}$-adic and $\mathfrak{m}'$-adic topologies on finitely generated $R^{\circ}{}'$-modules coincide; the general case   follows from  the special case by localisation \eqref{localisation}. 

More generally, if $Y\subset X $ is a closed subset, the functor ${\rm R}\Gamma(G, -)$ on $\OO_{Y}[G]$-modules coincides with the restriction of the functor on $\OO_{X}[G]$-modules of the same name.

 \begin{prop}\label{base change} Let $M$ be an ind-admissible $\OO_{X}[G]$-module and let $N$ be an $\OO_{X}$-module of finite projective dimension. Then  there is a natural isomorphism  in $D^{\rm b}({}_{\OO_{X}}{\rm Mod})$
 $$  {\rm R}\Gamma(G, M) \stackrel{\rm L}\otimes_{\OO_{X}} N\cong  {\rm R}\Gamma(G, M \stackrel{\rm L}\otimes_{\OO_{X}} N). $$
  \end{prop}
 \begin{proof} Let $P^{\bullet}$ be a finite projective resolution of $N$.   The natural map  of complexes of $\OO_{X}$-modules
$$C_{\rm cont}^{\bullet}(G, M)\otimes_{\OO_{X} }P^{\bullet}\to  C_{\rm cont}^{\bullet}(G, M\otimes_{\OO_{X}}P^{\bullet})$$
 is an isomorphism by   \cite[(3.4.4)]{nek-selmer}.\footnote{In \emph{loc. cit.}, the ring denoted by $R$ is our $R^{\circ}$, but as our $X$ is open in $\Spec R^{\circ}$, the $\OO_{X}$-modules $P^{n}$ are also flat as $R^{\circ}$-modules and the  cited result applies.} The desired result follows from the definition of derived tensor product. 
 \end{proof}
The proposition applies when $N=\OO_{Y}$ with $Y\subset X$  a local complete intersection, 
or when  $X$ is regular and $N$ is any coherent $\OO_{X}$-module.
 We highlight the following case. 
 \begin{coro}   \label{base change cor} 
 Let $M$ be an ind-admissible $\OO_{X}[G]$-module that is flat as an $\OO_{X}$-module, and let $x\in X$ be a nonsingular point. Then there is an isomorphism in $D^{\rm b}({}_{\kappa(x)}{\rm Mod})$
 $$ {\rm R}\Gamma(G, M) \stackrel{\rm L}\otimes_{\OO_{X}} \kappa(x)\cong 
 {\rm R}\Gamma(G, M \otimes_{\OO_{X}} \kappa(x)) $$
 hence a second-quadrant spectral sequence 
 $${\rm Tor}_{-p}(H^{q}(G, M), \kappa(x)) \Rightarrow  H^{q-p}(G , M\otimes_{\OO_{X}}\kappa(x)).$$
 \end{coro} 
\begin{proof} After possibly localising at $x$, we may assume that $X=\Spec R$ is the spectrum of a  local ring, which by assumption will be regular. Then  
 $\kappa(x)$   has finite projective dimension over $R$, and the result follows from the previous proposition. 
\end{proof}
 
\subsubsection{Continuous cohomology as derived functor} For $i=0, 1$, the functors $M\mapsto H^{i}(G,M)$ on the category of ind-admissible $R$-modules  coincide with the $i^{\rm th}$ derived functors of $M\mapsto M^{G}$ (\cite[(3.6.2)(v)]{nek-selmer}).

\subsection{Specialisations}\label{(F)} 
From here on we further  assume that $R^{\circ}$ has finite residue field of characteristic~$p$. 

\subsubsection{Finiteness conditions}
Let $G$ be a profinite group. We consider the condition 
$$\text{(F)}\qquad H^{i}(G, M) \text{ is finite for all finite discrete ${\bf F}_{p}[G]$-modules and all $i\geq 0$}$$
and define the $p$-cohomological dimension of $G$ to be
$${\rm cd}_{p}(G):= \sup\, \{i \, :\, \exists \text{ a finite discrete ${\bf F}_{p}[G]$-module $M$ with } H^{i}(G, M)\neq 0\}.$$
 \begin{lemm}\label{finite coh}
 If $G$ satisfies (F) then the cohomology groups of ind-admissible $\OO_{X}[G]$-modules of finite type are $\OO_{X}$-modules of finite type (\cite[(4.2.5), (4.2.10)]{nek-selmer}).
  The cohomology of any ind-admissible $\OO_{X}[G]$-module vanishes in degrees $>{\rm cd}_{p}(G)$ (\cite[(4.26)]{nek-selmer}). 
 \end{lemm}
 
 When $E$ is a number field, $S$ is a finite set of places of  $E$ and $G=G_{E, S}$,
  condition (F) is satisfied and ${\rm cd}_{p}(G)= 3$. When $E_{w}$ is a local field and $G=G_{E_{w}}$,
   condition (F) is satisfied and ${\rm cd}_{p}(G)= 2 $. In the latter case we use the notation $H^{i}(E_{w}, M) $ for $H^{i}(G, M)$. 

\subsubsection{Projective limits, specialisations} We give two  results on the compatibility of  $G$-cohomology  with other functors.

\begin{lemm}\label{projlim ok} Let $ G$ be a profinite group satisfying  (F) and let $M=\varprojlim_{n} M_{n}$ be the limit of  a countable projective system of admissible $R^{\circ}$-modules of finite type. Then for all $i$ the natural map
$$ H^{i}(G, M) \to  \varprojlim_{n} H^{i}(G, M_{n})$$
is an isomorphism.
\end{lemm}
\begin{proof}
In the special case $M_{n}=M/\mathfrak{m}^{n}M$, it is shown in \cite[Corollary 4.1.3]{nek-selmer} that the map under consideration is surjective with kernel $\lim^{(1)}_{n}  H^{i-1}(G, M_{n})$; this vanishes since by (F) those cohomology groups are finite, hence the projective system they form satisfies the Mittag-Leffler condition.  The general case follows from applying the special case to $M$ and the $M_{n}=\varprojlim_{r}M_{n}/\mathfrak{m}^{r}M_{n}$. 
\end{proof}

\begin{prop}\label{torss} Let $G$ be a profinite group satisfying (F) and ${\rm cd}_{p}(G)=e<\infty$. Let $M$  be an ind-admissible $\OO_{X}[G]$-module of finite type. Let  $x\in X$ be a nonsingular point, let $i_{0}\geq 0$  and suppose that   
$$H^{i}(G, M\otimes_{R}\kappa(x))=0$$
for all $i\geq i_{0}+1$. 
\begin{enumerate}
\item For all $i\geq i_{0}+1$, the support of the finitely generated $R$-module $H^{i}(G, M)$ is a proper closed subset not containing $x$.
\item The natural map
$$H^{i_{0}}(G, M)\otimes_{\OO_{X}}\kappa(x)\to H^{i_{0}}(G, M\otimes_{R} \kappa(x))$$
is an isomorphism.
\item\lb{torss3} Suppose further  that $i_{0}=1$,  and that for $y$ in some dense open subset of $X$,  $\dim _{\kappa(y)} H^{1}(G, M\ot\kappa(y))= \dim _{\kappa(x)} H^{1}(G, M\ot\kappa(x))$. Then the natural map
$$H^{{0}}(G, M)\otimes_{R}\kappa(x)\to H^{{0}}(G, M\otimes_{\OO_{X}} \kappa(x))$$
is an isomorphism.
\end{enumerate}
\end{prop}
\begin{proof}  By Nakayama's lemma  and the vanishing assumption, the first statement is equivalent to  
\beq \label{dhi} 
H^{i}(G, M)\otimes_{\OO_{X}}\kappa(x)\cong H^{i}(G, M\otimes_{\OO_{X}} \kappa(x)).
\eeq
Therefore, for the proof of the first and second statements  it is enough to prove  \eqref{dhi}  for all $i\geq i_{0}$, which we do by decreasing induction on $i$.

 For $i\geq e+1$ the result is automatic. In general, Corollary  \ref{base change cor} gives a second-quadrant spectral sequence 
\beq \lb{E2pq}
E_{2}^{p,q}={\rm Tor}_{-p}^{\OO_{X}}(H^{q}(G, M),  \kappa(x))\Rightarrow H^{q-p}(G, M\otimes_{\OO_{X}} \kappa(x)).\eeq
By induction hypothesis, all terms on the diagonal  $q-p=i$ vanish except possibly  the one with $p=0$, and the differentials with source and target such term are $0$. It follows that $H^{i}(G, M\otimes_{\OO_{X} }\kappa(x))=E_{\infty }^{0, i}=E_{2}^{0, i}=H^{i}(G, M)\otimes_{\OO_{X}}\kappa(x) $.

Finally, under the assumptions of part 3, the finitely generated $R$-module $H^{1}(G, M)$ is locally free of constant rank in a neighbourhood of $x$. Hence in the exact sequence 
$$0\to H^{{0}}(G, M)\otimes_{R}\kappa(x)\to H^{0}(G, M\otimes_{\OO_{X}} \kappa(x)) \to {\rm Tor}_{1}^{\OO_{X}}(H^{1}(G, M), \kappa(x))$$
deduced from \eqref{E2pq}, the last term vanishes.
\end{proof}

\subsection{Selmer complexes and height pairings}\label{sec: sel} As in the preceding subsection, let $R^{\circ}$ be a Noetherian local ring  with finite residue field  of characteristic $p$,  $X$ an object of $\mathscr{C}_{R}$. 

When $E_{w}$ is a local field, we write ${\rm R}\Gamma(E_{w}, -):= {\rm R}\Gamma(G_{E_{w}}, -)$ and similarly for its cohomology groups. For number fields, we will only use the analogous shortened notation for Selmer groups.
\subsubsection{Greenberg data} Let $E$ be a number field, $Sp$  a finite set of finite places of $E$ containing those above $p$. 
 Fix for every $w\vert p$ an embedding $\baar{E}\into\baar{ E}_{{w}}$ inducing an embedding $G_{w}:=G_{E_{w}}\into G_{E, Sp}$.
 If $M$ is an $\OO_{X}[G_{E, Sp}]$-module, we   denote by $M_{w}$ the module $M$ considered as an $\OO_{X}[G_{w}]$-module.
 
 \begin{defi}\label{grbg datum} A  \emph{Greenberg datum}   $ ( M, (M_{w}^{+})_{w\in  Sp})$ (often abusively abbreviated by $M$ in what follows)  over $X$ consists of
 \begin{itemize}
 \item an ind-admissible  $\OO_{X}[G_{E, Sp}]$-module $M$,   finite and locally free  as an $\OO_{X}$-module;
 \item  for every $w\in Sp$ a \emph{Greenberg local condition}, that is  a short exact sequence $$ 0 \to M_{w}^{+}\stackrel{i_{w}^{+}}{\to} M_{w}\to M_{w}^{-}\to 0$$
of ind-admissible $\OO_{X}[G_{w}]$-modules,  finite and locally free as $\OO_{X}$-modules. 
 \end{itemize}
 \end{defi}
 In this paper, at places $w \nmid p$ we will only consider the \emph{strict} Greenberg conditions $M_{w}^{+}=0$.

\subsubsection{Selmer complexes} Given a Greenberg  datum  $M= (M, (M_{w}^{+})_{w\in Sp})$,
  the \emph{Selmer complex}  
$$\wtil{{\rm R}{\Gamma}}_{f}(E, M)$$ is the image of the complex
$$ {\rm Cone}\left(     C^{\bullet}_{\rm cont}(G_{E, Sp},M)\oplus\bigoplus_{w\in Sp} C^{\bullet}_{\rm cont}(E_{w}, M_{w}^{+})\stackrel{ \oplus_{w}{\rm res}_{w}-i_{w, *}^{+} }{\longrightarrow}    \bigoplus_{w\in Sp} C^{\bullet}_{\rm cont}(E_{w}, M_{w}) \right)[-1]  $$
  in $D({}_{\OO_{X}}^{\rm ft}{\rm Mod})$. Its cohomology groups are denoted by $\wtil{H}^{i}_{f}(E, M)$. 
We   have an exact triangle 
\beq\label{ex tri}
\wtil{{\rm R}{\Gamma}}_{f}(E, M)\to {\rm R}{\Gamma}_{}(E, M)\to \oplus_{w\in Sp} {\rm R}\Gamma(E_{w}, M^{-}_{w}). 
 \eeq

\begin{prop} The Selmer complex $\wtil{{\rm R}{\Gamma}}_{f}(E, M)$  and all terms of  \eqref{ex tri} belong to $D_{\rm perf}^{[0,3]}({}_{\OO_{X}}{\rm Mod})$. 
\end{prop}
\begin{proof} As in \cite[Proposition 9.7.2 (ii)]{nek-selmer}. 
\end{proof}

From the triangle  \eqref{ex tri} we extract an exact sequence
\beq\label{tilde ses}
0 \to H^{0}(G_{E, Sp}, M)\to \bigoplus_{w\in Sp} H^{0}(E_{w}, M^{-}_{w}) \to \wtil{H}^{1}_{f}(E, M)\to H^{1}_{f} (E, M)\to 0
\eeq
where the last term is the (Greenberg) Selmer group 
\beq\label{sel gr} H^{1}_{f} (E, M):= \Ker \left( H^{1}_{}(G_{E, Sp},M)\to \bigoplus_{w\in Sp}  H^{1} (E_{w}, M_{w}^{-})\right). 
\eeq

\subsubsection{Height pairings} \label{sec: ht}
For $?=\emptyset,\iota$, let  $M^{?}=(M^{?}, (M_{w}^{?,+}))$ be a strict Greenberg datum for $G_{E, Sp} $ over $X$. Suppose given a perfect  pairing of $\OO_{X}[G_{E, Sp}]$-modules
\beq \lb{skewp}
M\ot_{\OO_{X}}M^{\iota}\to \OO_{X}(1) \eeq
such that $M^{+}_{w}$ and $M^{+, \iota}_{w}$ are exact orthogonal of each other. Let $\Gamma_{F}$ be a profinite abelian group.

For every pair of  Greenberg data $M$, $M^{\iota}$ as above, there  is a   height pairing 
\beq \lb{hhh}
h_{M}\colon \wtil{H}^{1}_{f}(E, M)\ot_{\OO_{X}}  \wtil{H}^{1}_{f}(E, M^{\iota})
\to \OO_{X}\hat{\ot}\Gamma_{F}\eeq
constructed  in  \cite[\S11.1]{nek-selmer}. The following is a special case of \cite[Appendix C, Lemma 0.16]{venerucci-thesis}. 
\begin{prop} \lb{prop ht} 
For each regular point $x\in X$ and $P_{1}\ot P_{2}\in \wtil{H}^{1}_{f}(E,  M)\ot_{\OO_{X}}  \wtil{H}^{1}_{f}(E, M^{\iota})$, we have 
$$h_{M\ot \kappa(x)} (P_{1,x}, P_{2, x})=(h_{M}(P_{1},  P_{2}))(x).$$
\end{prop}

 Venerucci has defined height pairings in an even more general context. Let $M_{X}$ be a strict Greenberg datum over $X$ as above, let $Y\subset X$ be a  local complete intersection, and let $M_{Y}^{?}$ be the restriction of $M_{X}^{?}$.  Let $\mathscr{N}_{Y/X}^{*}$ be the conormal sheaf of $Y\to X$. 
Then there is a height pairing
\beq \lb{hnv} h_{M_{Y}/M_{X}}\colon  \wtil{H}^{1}_{f}(E, M_{Y})\ot_{\OO_{Y}}  \wtil{H}^{1}_{f}(E, M_{Y}^{\iota})\to \mathscr{N}_{Y/X}^{*}, \eeq
constructed in \cite[Appendix C, \S~0.21]{venerucci-thesis}. 

We  note its relation to \eqref{hhh} in a special case, and its symmetry properties in a conjugate-self-dual case.

\begin{prop} \lb{prop ht2} The pairing \eqref{hnv}
satisfies the following properties. 
\begin{enumerate}
\item
Let $\Gamma_{F}$ be a profinite abelian quotient of $G_{E, Sp}$, let $X=Y\hat{\ts}_{\Spec \Q_{p}}\Spec \Z_{p}\llb \Gamma_{F}\rrb_{\Q_{p}}$ (where $\hat{\ts}= \eqref{hatprod}$), and assume that $M_{X}^{?}=M_{Y}^{?}\ot_{\Z_{p}} \Z_{p}\llb \Gamma_{F} \rrb$ for $?=\emptyset, \iota$, where if $?=\emptyset $  (respectively $?=\iota$) then $G_{E, Sp}$ acts on $\Gamma_{F}$ through the tautological character (respectively its inverse). 
Then $$h_{M_{Y}/M_{X}} = h_{M_{Y}}=\eqref{hhh}.$$
\item Suppose that there is an involution $\iota\colon X\to X$ stabilising $Y$ and such that $M_{X}^{\iota}= M_{X}\ot_{\OO_{X}, \iota}\OO_{X}$, $M_{X,w}^{+,\iota}= M_{X, w}^{+}\ot_{\OO_{X}, \iota}\OO_{X}$. 

 Let $\epsilon, \epsilon'\in\{\pm 1\}$. Assume that the pairing \eqref{skewp} is $\epsilon$-hermitian (\S~\ref{skew def}),
that  ${\rm d}_{Y/X}\iota= \epsilon'\id$  on $\mathscr{N}_{Y/X}^{*}$, and that there is an  $\OO_{Y}$-linear isomorphism 
$${\rm c}\colon \wtil{H}^{1}_{f}(E, M_{Y}^{\iota}) =  \wtil{H}^{1}_{f}(E, M_{Y})^{\iota} \to  \wtil{H}^{1}_{f}(E, M_{Y}) .$$
Then the pairing 
\beq\lb{skewsym}
h_{M_{Y}/M_{X}}^{\Box} \colon  \wtil{H}^{1}_{f}(E, M_{Y})\ot_{\OO_{Y}}  \wtil{H}^{1}_{f}(E, M_{Y}) &\to \mathscr{N}_{Y/X}^{*}\\
(z, z') &\mapsto h_{M_{Y}/M_{X}}(z, {\rm c} z')
 \eeq
 is $\epsilon\epsilon'$-symmetric. 
\end{enumerate}
\end{prop}
In our main application  in Theorem \ref{ugz thm}, we have  $Y=\X$ (or an open subset), the Hida family for $(\G\ts\H)'$; $X=\X^{\sharp}$, the Hida family for $\G\ts \H$ containing $X$; and  $M_{Y}=\cV$, $M_{X}=\cV^{\sharp}$, the corresponding universal $G_{E}$-representations. In that case, the height pairing $h_{\cV/\cV^{\sharp}}$ is simply $1/2 $ of the pairing $h_{\cV}$ of Proposition \ref{prop ht}.

\begin{proof} Part 1 follows from the construction. (We omit further details since, by the remark preceding the present proof, we do not actually need it in this paper). 
We prove the symmetry properties. Let $I$ be the ideal sheaf of $Y$; up to restricting to some open subset of $X$ we may assume that $I$ is generated by a regular sequence $x=(x_{1}, \ldots, x_{r})$. Then $\calN_{Y/X}^* = I/I^2$ is finite locally free generated by $([x_1], \ldots, [x_r])$. Let $\partial_i \in N = (I/I^2)^\vee$ be the map $\partial_i([x_j])=\delta_{ij}$. It suffices to show that $h_i:=\partial_i \circ h$ is   $\epsilon \epsilon'$-hermitian 
 for all $i$. By \cite[Appendix C, Proposition 0.5]{venerucci-thesis}, $h_i$ is identified with $ h_{M_{Y}/M_{X_{i}}}$,  where $X_i=V_{X}((x_j)_{j\neq I})$ so that $Y=V_{X_{i}}(x_{i})$. Hence it suffices to prove the claim for $r=1$.

We argue similarly to \cite[Proof of Corollary 10.10]{venerucci-thesis}.  Assume thus $r=1$, write $x $ in place of $x_{1}$, and let $K$ be the fraction field of $X$. By  \cite{nek-selmer} and \cite[Appendix A, \S~ 0.7]{venerucci-thesis} we have an $\epsilon$-hermitian Cassels--Tate pairing 
$$\cup \colon \wtil{H}^{2}_{f}(E, M_{X})_{\text{$\OO_{X}$-tors}}
 \ot_{\OO_{X}} \wtil{H}^{2}_{f}(E, M_{X}^{\iota})_{\text{$\OO_{X}$-tors}} \to K/\OO_{X},$$
and by \cite[Appendix C, Proposition 0.17]{venerucci-thesis}, we have a map
 $$ i_{x}\colon \wtil{H}^{1}_{f}(E, M_{Y})\to  \wtil{H}^{2}_{f}(E, M_{X})[x]$$
 such that $h_{M_{Y}/M_{X}}$ coincides with 
$$ \wtil{H}^{1}_{f}(E, M_{Y}) \ot  \wtil{H}^{1}_{f}(E, M^{\iota}_{Y}) 
\stackrel{i_{x}\ot i_{x}^{\iota}}{\longrightarrow} 
 \wtil{H}^{2}_{f}(E, M_{X})[x]\ot  \wtil{H}^{2}_{f}(E, M^{\iota}_{X})[x]  
 \stackrel{\cup}{\longrightarrow}  x^{-1}\OO_{X}/\OO_{X}  \stackrel{[\cd x^{2}]}{\longrightarrow}  I/I^{2}= \calN^{*}_{Y/X}.$$
Since all the above maps are $\iota$-equivariant, we find that $h_{M_{Y}/M_{X}}$ is $\epsilon$-hermitian as well. The desired assertion follows from this and the fact that $\iota $ acts by $\eps'$ on $\calN^{*}_{Y/X}$.
 \end{proof}

\section{Universal Heegner class}

\subsection{Tate cycles and Abel--Jacobi maps} Let  $X/E$ be an algebraic variety over a number field, and let  $R$ 
be a finite extension of $\Q_{p}$, or its ring of integers, or a finite quotient of its ring of integers.
\subsubsection{Tate cycles}
If $\cW $ is an \'etale local system of  $R[G_{E}]$-modules on $X$,   the $R$-module of  \emph{Tate ($0$)-cycles} is the space 
$$\mathscr{Z}_{0}(X , \cW):=\bigoplus_{x\in X} H^{0}(x, \cW)$$
where the sum runs over the closed points of $X$ and, if $x\in X$ and $\baar{x}:= x\times_{\Spec E}\Spec \baar{E}$, we define $H^{0}(x, \cW):= H^{0}(\baar{x}, \cW)^{G_{E}}$.  Elements of the latter space are written $\sum_{\baar{x}'} [\baar{x}' ]\otimes \xi_{\baar{x}'}$, where $\baar{x}'$ runs through the  points of $\baar{x}$.  
When $\cW=R$, the module $\mathscr{Z}_{0}(X, R)$ is simply the usual $R$-module of $0$-cycles with coefficients in $R$. Its quotient by the relation of rational equivalence is denoted ${\rm CH}_{0}(X, R)$.

When $X$ has dimension $0$, its \emph{fundamental class} is the Tate cycle with trivial coefficients 
 $$[X]:= \sum_{\baar{x}'\in Z(\baar{E}} [\baar{x}' ]  \otimes 1 \in \mathscr{Z}_{0}(X, \Z_{p}).$$ 
 If $a=\sum_{\baar{x}'}[\baar{x}'] \otimes \xi_{\baar{x}'}\in \mathscr{Z}_{0}(X, \cW)$ its support $|a|\subset \baar{X}$ is the support of the divisor $\sum [\baar{x}']$, where the sum extends to those $\baar{x}' $ such that $\xi_{\baar{x}'}\neq 0$. 

\subsubsection{Abel--Jacobi map}  A Tate cycle $a\in \mathscr{Z}_{0}(X, \cW)$ yields a map $R\to H^{0}(|a|, \cW)^{G_{E}}$ and, if $X$ has dimension $1$, the latter cohomology group maps to $H^{2}_{|a|}(\baar{X}, \cW(1))
$. The image of $1\in R$ under the compostion
$$R\to H^{0}(|a|, \cW)\to H^{2}_{|a|}(\baar{X}, \cW(1))\to H^{2}(\baar{X}, \cW(1)),$$ is denoted by $\baar{\rm cl}(a)$. 
Consider the exact sequence 
\beq\label{ex pair}
0\to H^{1}(\baar{X}, \cW(1))\to H^{1}(\baar{X}-|a|, \cW(1))\to H^{2}_{|a|}(\baar{X}, \cW(1))\to H^{2}(\baar{X}, \cW(1)).
\eeq
 Let $e$ be a Galois-equivariant idempotent acting on the right on $H^{*}(\baar{X}, \cW(1) )$, such that $\baar{\rm cl}(a)e=0$. Then we may apply the idempotent $e$ to \eqref{ex pair} and pull back the resulting exact sequence via the map $R\to H^{2}_{|a|}(\baar{X}, \cW(1))$ given by $a$, obtaining an extension 
\beq \label{AJ extn}
0\to H^{1}(\baar{X}, \cW(1))e\to E_{a} \to R\to 0
\eeq
in the category of $G_{E}$-representations over $R$.  The map sending $a$ to the class ${\rm AJ}(a)e$ of this extension is called the $e$-Abel--Jacobi map, 
$${\rm AJ}e\colon \mathscr{Z}_{0}(X, \cW)\to H^{1}(G_{E}, H^{1}(\baar{X}, \cW(1))e)=  H^{1}(G_{E}, H_{1}(\baar{X}, \cW)e),$$
where the last equality is just a reminder of our notational conventions.  When $e={\rm id}$, it is omitted from the notation.
When  $\cW=R$ and $e$ acts via correspondences, the map ${\rm AJ}e$ factors through ${\rm CH}_{0}(X, R)e$.

\subsection{Heegner cycles} 
We use the notation from \S~\ref{sec: not} for compact subgroups $U_{*,p, \ur}\subset  U_{*,p}(p^{\ur})\subset \G_{*}(\Q_{p})$ and let $X_{*, U^{p}_{*}{}',\ur}\rightarrow X_{*, U^{p}{}'}(p^{\ur})$ be the associated Shimura varieties; the level $U_{*}^{p}{}'$ will be fixed 
 and often omitted from the notation. If $p\OO_{F,p}=\prod_{v\vert p }\vpi_{v}^{e_{v}}\OO_{F, v}$ we use  $r$ as a shorthand for  $\ur=(e_{v}r)_{v\vert p}$.

\subsubsection{Embeddings of Shimura varieties}

For any pair of subgroups $V'\subset \H'(\A^{\infty})$, $K\subset (\G\times\H)'(\A^{\infty})$ such that $K \cap \H'(\A^{\infty}) \supset V$, we define the diagonal embedding
\beqq {\rm e'}={\rm e}_{V', K}' \colon Y'_{V'} &\to Z_{K}\\
y& \mapsto [({\rm e}(\tilde{y}), \tilde{y})]
\eeqq
if $\tilde{y}$ is any lift of $y$ to $Y_{V}$ for some $V\subset \H(\A^{\infty})$ such that $VF_{\A^{\infty}}^{\times}\subset V'$.  
   
Let $W=W_{\G}\otimes W_{\H}$ be an irreducible right  algebraic representation of $(\G\times \H)'$ over $L\supset \Q_{p}$.  If $W$ satisfies (wt), the space $ W^{H'}$ is $1$-dimensional over $L$.
Let $\cW $ be the \'etale sheaf on the Shimura tower $Z$  associated with $W$;  any $\xi\in W^{H'}$ induces a map $ \Q_{p}\to {\rm e}'^{*}\cW$ of \'etale sheaves on the tower  $Y'$; by adjunction we obtain a canonical map $ \Q_{p}\to {\rm e}'^{*}\cW \ot W^{\vee}_{H'}$ where the second factor is simply an $L$-line. 

We let
\beqq
{\rm e'}_{W, K, V',*} &\colon 
 \mathscr{Z}_{0}(Y_{V'}', \Z_{p}) \to \mathscr{Z}_{0}(Y_{V'}'  ,  {\rm e}'^{*}\cW^{\circ})  \ot W^{\vee}_{H'}
   \to \mathscr{Z}_{0}(Z_{K}, \cW^{\circ}) \ot W^{\vee}_{H'}
\\
{\rm e}'_{W,\ur, *}&\colon
 \mathscr{Z}_{0}(Y'_{\ur}, \Z_{p})\to \mathscr{Z}_{0}(Y_{\ur}',  {\rm e}'^{*}\cW^{\circ}) \ot W^{\vee}_{H'}
   \to \mathscr{Z}_{0}(Z(p^{\ur}), \cW^{\circ})  \ot W^{\vee}_{H'} \to \mathscr{Z}_{0}(Z_{\ur}, \cW^{\circ})  \ot W^{\vee}_{H'}
\eeqq
be the compositions of the maps described above and, respectively, ${\rm e}'_{W, *}$  or ${\rm e}'_{\ur, *}$.

\subsubsection{CM cycles}
Let  $[Y_{V'}']\in \mathscr{Z}_{0}(Y'_{V'}, \Z_{p})$ be the  fundamental class. 
For any pair of levels $K, V$  such that ${\rm e}_{W,(K, V')} $ is defined, 
let
 \begin{align}
\nonumber
\Delta_{W, (K, V')}&:=  {\rm e}_{W,K, V, *}'[Y'_{V'}]  \in  \mathscr{Z}_{0}(Z_{K}, \cW^{\circ}), 
\end{align}

When $ W\neq\Q_{p}$, we consider the elements
\beqq \Delta_{W,  (K, V')}^{\circ}:=
 {1\over |Y'_{V'}(\baar{E})|} \cdot \Delta_{ W,(K, V')}  \in  \mathscr{Z}_{0}(Z_{K}, \cW^{\circ})
 \eeqq
 When  $W=\Q_{p}$, we consider the modification
\beq\label{modif hodge}
\Delta_{ (K, V')}^{\circ}:= {1\over |Y'_{V'}(\baar{E})|} \cdot( \Delta_{(K, V')} -\deg(\Delta_{(K, V')})\cdot\xi_{\text{Hodge}} ) \in {\rm CH}_{0}(Z_{K})_{\Q_{p}}, 
\eeq
where $\xi_{\text{Hodge}} $ is the {Hodge class} of  \cite[\S3.1.3]{yzz}, whose introduction is motivated by the following lemma.
\begin{lemm} The image under pushforward of $\Delta^{\circ}_{W,(K'', V'')} $ in $ \mathscr{Z}_{0}(Z_{K}, \cW)$ (if $W\neq \Q_{p}$) or ${\rm CH}_{0}(Z_{K})_{\Q_{p}}$ (if $W=\Q_{p})$  is independent of $V'', K''$ such that 
$V''\subset K''\cap \H'(\A^{\infty})  $ and $K'\subset K$. We have 
$$\baar{\rm cl}(\Delta^{\circ}_{W,(K, V')})=0 \quad \text{in } H^{2}(\baar{Z}_{K}, \cW(1)).$$
\end{lemm}
\begin{proof}
If $W\neq \Q_{p}$, the first assertion is clear; the second one is automatic as  $H^{2}(\baar{Z}_{K}, \cW(1))=0$ (see the argument in \cite[bottom of p. 1089]{saito}). If $W=\Q_{p}$, the  assertions amount, respectively, to the compatibility of the Hodge classes under pushforward and the fact that, by construction,  the $0$-cycle $\Delta_{K, V}^{\circ}$ has degree zero; both facts are explained in \cite[\S3.1.3]{yzz}.
\end{proof}
\subsubsection{Cycles, Selmer classes,  and functionals} Let
\beq \label{PWKV}
P_{W,( K, V')}:= {\rm AJ}(\Delta_{W,  (K, V')}^{\circ})\in H^{1}(G_{E, Sp}, H_{1}(\baar{Z}_{K}, \cW)).
\eeq
The classes $P_{W,( K, V')}$ are also compatible under pushforward and yields  elements
\beqq
P_{W} &:= \lim_{K \cap \H'(\A^{\infty}) \supset V'} P_{W, (K, V')} \in
\varprojlim_{K}   H^{1}(G_{E, Sp}, H_{1}(\baar{Z}_{K}, \cW)).
\eeqq
The space in the right-hand side has a right  action by $(\G\times\H)'(\A^{\infty})$, and $P_{W}$ is invariant under $\H'(\A^{\infty})$. Via \eqref{carayol iso 2} and the biduality $W^{\vee\vee}=W$, $P_{W}$ yields, for each ordinary  representation $\Pi$ of weight $W$, a map
$$P_{\Pi}\colon \Pi\to H^{1}(G_{E, Sp}, V_{\Pi}).$$

Using  the map $\gamma_{H'}^{\ord}\colon \Pi^{\ord}\to \Pi_{H'} $ from Proposition \ref{gHo}, we also obtain a  map
\beq\lb{Pord text}
P_{\Pi}^{\ord}:=   P_{\Pi} \gamma_{H'}^{\ord}\colon \Pi^{\ord}\to H^{1}(G_{E, Sp}, V_{\Pi}).\eeq

\begin{rema}\lb{rem-P in H1f}  We conjecture that (i) there exist an algebraic variety $N_{W,(K, V')}/E$ of odd dimension $2d_{W}+1$, a homologically trivial cycle $\mathfrak{Z}_{W, (K, V')}\in {\rm CH}_{d_{W}}(N_{W, (K, V')})_{0}$, and a map $$\lm\colon H^{2d_{W}+1}(\baar{N}_{W, (K, V')}, \Q_{p}(d+1)) \to H_{1}(\baar{Z}_{K}, \cW)$$ such that $P_{W, (K, V')} =\lm({\rm AJ}(\mathfrak{Z}_{W, (, V')}))$; (ii) the elements $P_{W, (K, V')}$  belong to  $ H^{1}_{f}(E, H_{1}(\baar{Z}_{K}, \cW))$, so that the maps $P_{\Pi}$ take values in $H^{1}_{f}(E, V_{\Pi})$. 

When $\G=\GL_{2/\Q}$, one can prove (i) with $N_{W, (K, V')}$ a Kuga--Sato variety for $Z_{K}$,  generalising \cite[Proposition II.2.4]{nek-heeg}. 
The (probably not insurmountable) difficulty  in the general case is that, if $F\neq \Q$, the Shimura variety $Z $ is not of PEL-type.  Part (ii) should essentially be a consequence of either  (a)  part (i),  via \cite{nek-syn, nek-niz}, or (b) granted a generalisation of the theory of \emph{locc. citt.}  to nontrivial coefficients system,  of the  weaker assertion that, for a finite place $w$ of $E$, the image of $\Delta_{W, (K, V')}$ in $H_{0}({Z}_{K, E_{w}},\cW)$  comes from a corresponding  class in  the syntomic cohomology of $Z_{K, E_{w}}$ with coefficients in $\cW$. 
\end{rema} 

\subsection{Universal Heegner class} \lb{sec: 6.3}
We use the local construction described in  \S~\ref{ssec tord} to turn the $\H'(\A)$-invariant class $P_{W}$ into  an $\H'(\A^{p\infty})$-invariant  class $\cP_{W}$ with values in the ordinary completed homology.
Then we show that $\cP_{W}$ is independent of $W$ and it interpolates $P^{\circ}_{\Pi}$ at all representations $\Pi$ satisfying (ord), (n-exc). 
\subsubsection{Construction}
Let $d_{\ur}:=|Y'_{\ur}(\baar{E})|$ and let 
$ d^{\circ} =  d_{\ur}\prod_{v\vert p}q_{v}^{-r_{v}} \in \Z_{\geq 1}$, which is the limit of an eventually constant sequence. 
Recall that for  the tame level $K^{p}{}'\subset( \G\times \H)'(\A^{p\infty}), $
 we denote  $M^{\circ}_{K^{p}{}', W}:=  \varprojlim_{r}\H_{1} (\baar{Z}_{W,r}, \cW^{\circ})^{\ord}$. 

\begin{defi} The \emph{universal Heegner point} of weight $W$ is the element 
\beq\label{cPW-def}
\cP_{W}:= P_{W}\gamma_{H'}^{\ord}  \in d^{\circ, -1} H^{1}(G_{E, Sp}, M^{\circ}_{ K^{p}, W})
\eeq
where we still   denote by $\gamma_{H'}^{\ord}$ the map induced by the map 
$$\gamma_{H'}^{\ord}\colon \varprojlim_{K_{p}} \H_{1}(\baar{Z}_{K^{p}{}'K_{p}}, \cW)^{\H'(\A)} \to M_{ K^{p}{}', W}$$
of Proposition \ref{gHo}.
As usual, we simply write $\cP:= \cP_{\Q_{p}}$. When we want to emphasise the choice of $K^{p}{}'$ we write $\cP_{K^{p}{}', W}$ instead of $\cP_W$.
\end{defi}

\subsubsection{Independence of weight}  The class $\cP_{W}$ does not depend on $W$. 
\begin{prop}\label{indep of wt}
Under the identification 
\beq 
\lb{iddd} H^{1}(E, M^{\circ}_{K^{p}{}'}\otimes_{\Z_{p}}\OO_{L})\stackrel{j_{W,*}}\cong H^{1}(E,  M^{\circ}_{K^{p}{}',W})  \eeq
induced from the isomorphism $j_{W}$ of  Proposition \ref{free indep}.2, we have
$$j_{W, *}(\cP)= \cP_{W}.$$
\end{prop}
\begin{proof}
We show that the difference $j_{W, *}(\cP)-  \cP_{W}$ is $p$-divisible. Since  $H^{1}(G_{E, Sp},  M^{\circ}_{K^{p}{}',W})  $ is a finitely generated module over the ring $\Lambda^{\circ}_{K^{p}{}'}$ by  Lemma \ref{finite coh}, any $p$-divisible element is zero. We will use some of the notation and results of the appendix, in particular the matrices $\gamma$ defined in \S~\ref{A1}, the involution $\iota=(-)^{\rm T, -1}$ on $\GL_{2}$,  and the operator $\tord$ of Proposition \ref{gHo}. 

We  tacitly multiply both sides by  $d^{\circ}$, so that they belong to the lattices \eqref{iddd}. By the definitions of $\cP_{W}$ and $j_{W}$, we need to show the following. Denote by $[ -]_{r}$ the reduction modulo $p^{r}$, and by $c(W)$ the constant \eqref{cW}; then we should have
$$[p^{r[F:\Q]}\Delta_{\Q_{p}, \ur} \gamma_{r, p}\Up_{p}^{-r} \gamma^{\iota}_{0,\infty} ]_{r}\mapsto [c(W)^{-1}p^{r[F:\Q]}\Delta_{W, \ur}  \gamma_{r, p}\Up_{p}^{-r} \gamma^{\iota}_{0,\infty} ]_{r},$$
under the map 
\beq \lb{opla} j_{W}'\colon H^{1}(G_{E, Sp}, H_{1}(\baar{Z}_{K^{p}K_{p}(p^{r})}, \Z/p^{r})) &\to     H^{1}(G_{E, Sp}, H_{1}(\baar{Z}_{K^{p},K_{p}(p^{r})}, \Z/p^{r}) \ot_{\Z/p^{r}} (W^{\circ}/p^{r})^{N_{0,r}} \ot_{\OO_{L}/p^{r}} (W^{\vee, \circ}/p^{r})_{N_{0,r}}  \\
c& \mapsto c\ot \zeta_{r}\ot \zeta_{r}^{\vee},
\eeq
where $\zeta_{r}\ot\zeta_{r}^{\vee}$ is the unique element pairing to $1$.

As  the local system $\cW^{\circ}/p^{r}\cW^{\circ}$ is trivial on $\baar{Z}_{K^{p}K_{p}(p^{r})}$, we have 
$$[p^{r[F:\Q]}\Delta_{W, \ur}]_{r} = [p^{r[F:\Q]}\Delta_{\Q_{p}, \ur} \ot \xi \ot \xi^{\vee}]_{r} $$
in 
$$
  H^{1}(G_{E, Sp}, H_{1}(\baar{Z}_{K^{p},K_{p}(p^{r})}, \Z/p^{r})) \ot_{\Z/p^{r}} (W^{\circ}/p^{r}W^{\circ})^{H'} \ot_{\OO_{L}/p^{r}} (W^{\vee, \circ}/p^{r}W^{\vee,\circ})_{H'}  ,$$
  where $\xi\ot\xi^{\vee}$ is the unique element pairing to $1$.  Note first  that the  image of $  [p^{r[F:\Q]}\Delta_{W, \ur}]_{r}$ under  $\gamma_{r, p}\Up_{p}^{-r}\gamma_{0, \infty}^{\iota}$ belongs to the right-hand side of \eqref{opla}: indeed it  suffices to show that for any $\xi \in W^{\circ}$,  the class $[\xi \gamma_{r}]_{r}$ is fixed by $N_{0, r}$, which follows from the congruence 
  $$ \gamma_{r} n - \gamma_{r}\equiv 0 \pmod{p^{r} M_{2}(\Z_{p})}$$
valid    for any  $n\in N_{0, r}$. 
  
It remains to see that if $\xi\ot \xi^{\vee}$ pairs to $1$, then so does $c(W)^{-1}\cdot \xi\gamma_{r, p}\ot \xi^{\vee}\gamma_{0, \infty}^{\iota}$ in the limit  $r\to \infty$.  This is proved in Lemma \ref{unitary}.
\end{proof}

\subsection{Local properties of the  universal Heegner class}
\lb{sec: 6.4}
Recall that $\X$ is an irreducible component of $\cE_{K^{p}}$ hence of the form $\Spec R$ with $R=R^{\circ}[1/p ]$ and  $R^{\circ}= {\bf T}_{(\G\times \H)', K^{p}, \mathfrak{m}}^{{\rm sph}, \ord}/\mathfrak{a}$ for some maximal and minimal ideals $\mathfrak{m}\subset \mathfrak{a}\subset {\bf T}_{(\G\times \H)', K^{p}}$. The ring $R^{\circ}$ satisfies the assumptions of \S~\ref{sec: sel}, hence Greenberg data over open subsets of $\X$ give rise to sheaves of Selmer complexes. 

Let $\X^{(i)}\subset \X\subset \cE_{K^{p}{}'}$ be the open sets defined in \S~\ref{section 3}. 
Proposition \ref{galois2} provides  a strict Greenberg datum  $(\cV, (\cV_{w}^{+})_{w\vert p}, (0)_{w\in S})$ over $\X$. Via Proposition \ref{cofinal}.2 we obtain  a strict Greenberg datum $(\cM_{K^{p}{}'}^{{H_{\Sg}'}}, (\cM_{K^{p}{}', w}^{{H_{\Sg}'}, +})_{w\vert p}, (0)_{w\in S}) $ over $\X^{(3)}$ with
$$ \cM_{K^{p}{}', w}^{{H_{\Sg}'},\pm}= \cV_{w}^{\pm} \otimes ( \Pi^{K^{p}{}', \ord}_{{H_{\Sg}'}} )^{\vee} .$$

 We begin the study of the Selmer complexes attached to the above Greenberg data, with the goal to promote $\cP$ to a section  of $\wtil{H}^{1}_{f}({E}, \cM_{K^{p}{}'}^{{H_{\Sg}'}})$ over a suitable open subset of $\X$. 

\subsubsection{Comparison of  Bloch--Kato and Greenberg Selmer groups}
 Let $z\in \X^{\cl}$ and let $V=\cV_{|z}$. We compare  two notions of Selmer groups for $V$.

\begin{lemm}\label{wt-mon} Let $w\nmid p$ be a place of $E$. Then, for all $i$,
$$H^{i}(E_{w},V)=0.$$
\end{lemm}
\begin{proof} As observed in \cite[Proposition 2.5]{nek-AJ}, this is implied by the prediction from the weight-monodromy conjecture that the monodromy filtration on $\cV_{z}$ is pure of weight $-1$. Writing $z=(x, y)\in \cE^{\ord, \cl}\subset \cE_{\G}^{\ord ,\cl}\times \cE_{\H}^{\ord, \cl}$, the weight-mondromy conjecture for $\cV_{z}$ follows from the corresponding statement for $\cV_{\G,x}$, that is Theorem \ref{galois for hilbert}.2.
\end{proof}
Let $H^{1}_{f, {\rm Gr}}({E}, V)$ be the Greenberg Selmer group.
  Bloch and Kato \cite{bk} have defined subspaces $H^{1}_{f}(E_{w}, V)\subset H^{1}(E_{w}, V)$ and a Selmer group
  $$H^{1}_{f, {\rm BK}}(E, V):= \{ s\in H^{1}_{}(E, V)\, :\, \forall w\in Sp, \ {\rm loc}_{w}(s)\in H^{1}_{f}(E_{w}, V)\}.
  $$
  \begin{lemm} Suppose that $\Pi_{z}$ satisfies $(\textup{wt})$. We have 
  $$H^{1}_{f, {\rm BK}}(E, V) = H^{1}_{f, {\rm Gr}}(E, V),$$
  where the right-hand side  is the Greenberg Selmer group as in \eqref{sel gr}.
  \end{lemm}
\begin{proof} We need to show that for all $w\in Sp$, $ H^{1}_{f}(E_{w}, V)= \Ker \left(  H^{1}(E_{w}, V)\to  H^{1}_{f}(E_{w}, V_{w}^{-})\right)$. This is automatic for $w\nmid p$ by Lemma \ref{wt-mon}. For $w\vert p$ this is \cite[(12.5.8)]{nek-selmer}: the context of \emph{loc. cit} is more restricted but the proof still applies, the key point being that (12.5.7)(1)(i) \emph{ibid.} still holds for all $w$ under the weight condition (wt).
\end{proof}

\begin{lemm}\label{sel away from p} Let $w\nmid p$ be a finite place of $E$.  Then  $H^{1}(E_{w},\cV)$ and  $H^{1}(E_{w}, \cM_{K^{p}{}'}^{{H_{\Sg}'}}) $ are supported in a closed subset of $\X$ (respectively $\X^{(3)}$) disjoint from $\X^{ \cl}$. 
\end{lemm}
\begin{proof}
This follows from Proposition  \ref{torss} and Lemma \ref{wt-mon}.
\end{proof}

\subsubsection{Local Selmer properties  of $\cP$}  \lb{6777}
Let $w\nmid p$ be a place of $E$ as above.
As $$  \cM_{K^{p}{}'}^{{H_{\Sg}'}}= \cV\otimes ( \Pi^{K^{p}{}', \ord}_{{H_{\Sg}'}} )^{\vee} $$ over $\X^{(3)}$, the support of $H^{1}(E_{w}, \cM_{K^{p}{}'}^{{H_{\Sg}'}}) $ is in fact the intersection of $\X^{(3)}$ and of the support of  $H^{1}(E_{w},\cV)$. 
We denote by 
\beq\label{open sel away from p}
\X^{(3, w)}\supset \X^{\cl}
\eeq the open complement in $\X^{(3)}$ of the support of $H^{1}(E_{w},\cV)$ .

\begin{lemm}\label{sel at p} Let $w\vert p$ be a place of $E$, with underlying place $v$ of $F$.  The
     image ${\rm loc}_{w}^{-}(\cP)$ of $\cP$ in
$$H^{1}(E_{w}, \cM_{K^{p}{}', w}^{{H_{\Sg}'}, -})$$
vanishes over $\X^{(3)}$. 
\end{lemm}
\begin{proof} 
We lighten the notation by dropping form the notation the superscript `$(3)$'  and all decorations from $\cM$, $\cM_{w}^{-}$. Let  $\tilde{\X} := \Spec_{\X} \OO_{\X} [([\sqrt{z}])_{z\in A}]$, where $A$ is a (finite) set of topological generators for $F^{\ts}\bks \A_{F}^{\infty\ts}/(K^{p}{}'\cap {\rm Z}(\A_{F}^{p\infty\ts}))$. As $\tilde{\X}\to \X$ is faithfully flat, we may prove the statement after a base-change to $\tilde{\X}$; we denote base-changed sheaves and sections thereof with a tilde $\tilde{\Box}$.

Let $\chi\colon E^{\ts} \bks \A_{E}^{\infty\ts}\to G_{E}^{\rm ab}\to \OO(\X)^{\ts}$ be the universal character, and let  $\omega =\chi_{|F^{\ts}\bks \A_{F}^{\infty\ts}}$, so that $\det \cV_{\G|\X} (-1)=\omega$. Let $\omega^{1/2} \colon F^{\ts}\bks \A_{F}^{\infty\ts} \to \OO(\tilde{\X})^{\ts}$ be a square root of $\omega$.
  We may write
   $$\tilde{\cV} =(\tilde{\cV}_{\G}\ot \omega^{-1/2})_{|G_{E}} \ot \tilde{\chi}' , 
  \qquad  \tilde{\chi}' :=\tilde{\chi} \ot \omega^{-1/2}_{|G_{E}},$$
where now   $\tilde{\chi}'\colon G_{E} \to \Gal(E_{\infty}/E)$  is the projection for the abelian extension $E_{\infty}/E$ such that  $\Gal(E_{\infty}/E) $ is the maximal pro-$p$ quotient of $E^{\ts}\A_{F}^{\infty\ts} \bks \A_{E}^{\infty\ts} /V^{p}{}'$.

Write $E_{\infty}=\bigcup_{n\geq0} E_{n}$ as an increasing union of finite extensions, where $E_{0}=E$ and  eventually $E_{n+1}/E_{n}$ is totally ramified at each prime above $p$, and let   
$\tilde{\chi}'_{n}\colon G_{E}\to \Gal(E_{n}/E)$ be the natural projection. 
  Let $\alpha^{\circ}_{v}$ be the character giving the $G_{F_{v}}$-action on $\cV_{\G}^{+}(-1)$, 
  so that $$\tilde{\cV}^{-}_{w}= \omega_{w}\alpha_{w}^{\circ, -1}\ot\tilde{\chi}_{w}= \omega^{1/2}_{w}\alpha^{\circ, -1}_{w} \tilde{\chi}'_{w},$$ 
where for a character $\omega'_{v}$ of $G_{F_{v}}$, we denote $\omega'_{w}:=\omega'_{v|\G_{E_{w}}}$. Let
$$\tilde{\cV}^{-}_{n,w} :=\omega_{w}^{1/2}\alpha_{w}^{\circ, -1} \tilde{\chi}'_{n,w}, \qquad 
\tilde{\cM}^{-}_{n,w}:= \tilde{\cV}^{-}_{n,w}\ot (\Pi^{K^{p}{}', \ord}_{{H_{\Sg}'}} )^{\vee}\subset \tilde\cM_{w}.$$

 Then the same argument as in \cite[proof of Proposition 2.4.5,  primes $v\vert p$]{howbig} shows that
the image  ${\rm loc}_{w}^{-}(\tilde{\cP})$ vanishes in $H^{1}(E_{w},\cM^{-}_{n,w})= \prod_{w'\vert w} H^{1}(E_{n},\tilde{\cM}^{-}_{0,w})$ for each $n$; here $w'$ runs through the (eventually constant) set of primes of $E_{n}$ above $w$. Since $H^{1}(E_{w},\tilde{\cM}^{-}_{w}) =\varprojlim_{n} H^{1}(E_{w}, \tilde\cM^{-}_{n, w})$ by Proposition \ref{projlim ok},  the lemma is proved. 
  \end{proof}

\begin{coro}\label{X3f} Let $\X^{(3, f)}:=\bigcap_{w\in S} \X^{(3,w)}\supset \X^{\cl}$, where the sets $\X^{(3, w)}$ are as defined  in \eqref{open sel away from p}. Then $\cP$ defines a section
$$
\cP\in \wtil{H}^{1}_{f}(E,  \cM_{K^{p}{}'}^{{H_{\Sg}'}}) (\X^{(3, f)})=
H^{1}_{f}({E},  \cM_{K^{p}{}'}^{{H_{\Sg}'}})  (\X^{(3, f)}).$$
\end{coro}

\begin{proof} This follows from Lemmas \ref{sel away from p} and  \ref{sel at p}. The displayed equality is a consequence of  \eqref{tilde ses}.
\end{proof}

\subsubsection{Proof of Theorem \ref{theo heeg univ}} \lb{645}
Via Proposition \ref{cofinal}.2,  we may view the class $\cP=\cP_{K^{p}}=\cP_{K^{p}, \Q_{p}}$  (Definition \ref{cPW-def}) as an $\cH_{K}^{p\Sg}$-equivariant  functional
\beq \label{cP as functional} \cP_{K^{p}{}'}\colon \Pi^{K^{p}{}',\ord}_{\H_{\Sg'}}\to {H}^{1}_{}(E, \cV)(\X^{(3,f)}).
\eeq

By the  results of \S~\ref{6777},   $\cP$ takes values in  the Selmer group $\wtil{H}^{1}_{f}(E, \cV)(\X^{(3,f)})$.
  It satisfies the asserted interpolation properties by the definitions of the classes $\cP_{W}$ in \S~\ref{sec: 6.3} and Proposition \ref{indep of wt}.

\subsubsection{Exceptional locus of $\X$} \lb{ssec exc}
  Let $w|v$ be places of $E$ and $F$ above $p$, and let $\mu_{w}^{\pm}\colon E_{w}^{\times}\to \OO(\X)^{\times}$ be the characters giving the Galois action on $\cV^{\pm}$. Let $\X^{{\rm exc},v }\subset \X$ be the closed subset defined by $\mu_{w}^{-}=1$ for some (hence automatically all) places $w\vert v$ of $E$. We let 
\beq\label{exc sets}
\X^{\exc}:=\bigcup_{v\vert p} \X^{\exc, v}, \quad \X^{\cl, \exc, (v)}:=\X^{\cl}\cap \X^{\exc, (v)}, \quad  \X^{\cl,{\text{n-exc}} , (v)}:=\X^{\cl} - \X^{\cl, \exc, (v)}.
\eeq

We say that an ordinary  automorphic representation of $\Pi=\Pi_{|z}$ over a $p$-adic field  is \emph{exceptional at the place $v\vert p$} if $z\in \X^{\exc, v}$.

We may characterise the exceptional representations, and seize the opportunity to collect some useful results; see also Lemma \ref{exc vanishing}.

\begin{lemm}\lb{6444} Let $\Pi=\pi\otimes\chi$ be an ordinary automorphic representation of $(\G\ts\H)'(\A)$ over a $p$-adic field $L$,  of numerical weights $\uw$, $\ul$.
\begin{enumerate}
\item  Let $v\vert p$ be a place of $F$. The  following are equivalent:
\begin{enumerate}
\item
the representation $\Pi$ is exceptional  at $v$;
\item $e_{v}(V_{(\pi, \chi)})= \eqref{eVv}=0$;
\item the following conditions hold:
\begin{itemize}
 \item  the smooth representation $\pi_{v}$ of $G_{v}$ is special of the form ${\rm St}\ot\alpha_{v}$;
 \item for some (equivalently all)  places $w\vert v $ of $E$, we have $\chi_{w} \cdot \alpha_{v}\circ N_{E_{w}/F_{v}}=\one$.
 \item $w_{\sg}=2$ and $l_{\sg}=0$ for all $\sg\colon F\into \overline{L}$ inducing the place $v$.
\end{itemize}\end{enumerate}
\item Let $S_{p}^{\rm exc}= S_{p}^{\rm exc,\, s} \cup S_{p}^{\rm exc, \, ns}$ be the set of places $v\vert p$ (respectively those that  moreover are split, nonsplit in $E$) where $\Pi$ is exceptional.
\begin{enumerate}
\item The kernel of the natural surjective map $$\wtil{H}^{1}_{f}(E, V_{\Pi}) \to{H}^{1}_{f}(E, V_{\Pi})$$
has dimension 
$2 |S_{p}^{\rm exc, \, s}| + |S_{p}^{\rm exc, \, ns}|$. 
\item  Assume that $\G$ is associated with a quaternion algebra  $\B$ that is split at all places $v\vert p$.  Then 
$$\eps_{v}^{\G}(V_{\Pi}) = -1\qquad  \Longleftrightarrow  \qquad v\in S_{p}^{\rm exc, \, ns}.$$
\end{enumerate}
\end{enumerate}
\end{lemm}
\begin{proof} Consider part 1. We first prove the equivalence of the first two conditions. The adjoint gamma factor in the denominator of each $e_{v}(V_{(\pi, \chi)})$ is always defined and nonzero, whereas the gamma factor in the numerator is never zero and it has a pole if and only if, for some $w\vert v$, $V_{w}^{+}$  is the cyclotomic character of $E_{w}^{\ts}$. This happens precisely when, for some $w\vert v$,  $V_{w}^{-}$ is the trivial character -- that is, when $\Pi$ is exceptional at $v$. 

Now let us prove the equivalence to (c). Let $V=V_{\Pi}=V_{\pi|G_{E}}\ot V_{\chi}$. By the weight-monodromy conjecture (Theorem \ref{galois for hilbert}.2), the $1$-dimensional representations  $V^{\pm}_{\pi,v}$ are both  of motivic weight $-1$, thus have no $G_{E_{w}}$-invariants for any $w\vert v$,  unless $\pi_{ v}$ is a special representation.  In the latter case $V^{+}_{w|z}$ (respectively $V_{w}^{-})$ is of weight $-2$ (respectively $0$).  
This is compatible with the ordinariness requirement only when the weight $\uw$ is~$2$ at $v$ as in the statement of the lemma.  The second condition in (c) is immediate from the definition of (a). 

Consider now part 2. The first statement follows directly from  \eqref{tilde ses}. Let us prove the second one. By the results recalled in \S~\ref{sec m1} and \cite[Lemme 10]{wald}, the  condition $ \eps_{v}^{\G}(V_{\Pi}) = -1$ is equivalent to the vanishing of  the functional $Q=Q_{\Pi_{v}}=\eqref{Qpair}$ and of the space $\Pi_{v}^{*, H'_{v}}$. These conditions are never met if $v$ splits in $E$ or $\pi_{v}$ is a principal series, and otherwise they are equivalent  to the nonvanishing of $(\Pi_{v}')^{*, H_{v}'}$, where $\Pi'_{v}=\pi_{v}'\otimes \chi_{v}$ and $\pi_{v}'$ denotes the Jacquet--Langlands transfer of $\pi_{v}$ to the nonsplit quaternion algebra $B_{v}'^{\ts}$ over $F_{v}$. 

Assume that $v$ is nonsplit in $E$. If $\pi_{v}$ is exceptional, then by part 1 we have $\pi_{v}' =  \chi_{v}\circ {\rm Nm}$, where ${\rm Nm} $ is the reduced norm of $B_{v}'$, so that obviously $(\Pi_{v}')^{*, H_{v}'}\neq 0$. If $\pi_{v}$ is not exceptional, then by the explicit computation of Proposition \ref{toric period} we have $Q_{\Pi_{v}'}\neq 0$ (see also \cite[Corollary A.2.3]{nonsplit} for a variant of the last argument).
\end{proof}

\subsubsection{Heegner classes belong to the Bloch--Kato Selmer group}
 We can now prove the first assertion of Theorem \ref{GZ thm}.

\begin{prop}\lb{in H1f}  If $\Pi$ is not exceptional or has trivial weight,  the map $P_{\Pi}$ of \eqref{cP Pi 0} takes values in $H^{1}_{f}(E, V_{\Pi})\subset H^{1}(G_{E, Sp},V_{\Pi})$.
\end{prop}
\begin{proof} If $\Pi$ has trivial weight this is clear. Assume  that $\Pi$ is not exceptional.
 Let $\partial\colon H^{1}(E, V)/H^{1}_{f}(E , V)\to L$ be any linear map. Then we need to show that the $\H'(\A)$-invariant map $\partial P_{\Pi}\colon \Pi \to L$ is zero.
By Corollary \ref{X3f} and  Theorem \ref{theo heeg univ}, whose proof we have just completed,  the map $P_{\Pi}^{\ord}$ takes values in $H^{1}_{f}(E, V)$; equivalently, $\partial  P_{\Pi}\tord =0$. Since $\Pi$ is not exceptional,  by Lemma \ref{exc vanishing} this means that $\partial  P_{\Pi}=0$. 
\end{proof}

\subsubsection{Enhanced ordinary Heegner classes for exceptional representations}\lb{ss 648}
For any $z=(x, y)\in \X^{\cl}$ corresponding to a representation $\Pi$, define the \emph{enhanced Heegner class}
\beq \lb{tilPo} \wtil{P}_{\Pi}^{\ord}:= \cP_{|z}\in \wtil{H}^{1}_{f}(E, V_{\Pi}). \eeq
By the results established so far,  $ \wtil{P}_{\Pi}^{\ord}$ has image $P_{\Pi}^{\ord}$ under the natural map $ \wtil{H}^{1}_{f}(E, V_{\Pi})\to {H}^{1}_{f}(E, V_{\Pi})$; as  noted in Lemma \ref{6444}, this map fails to be an isomorphism precisely when $\Pi $ is exceptional.

\section{The main theorems, and a conjecture}
In this section, we prove our main theorems (\S~\ref{sec: final}, or \S~\ref{sec G} for  Theorem \ref{nv exc}), as well as a universal Waldspurger formula for  families of `sign $+1$' (\S~\ref{sec: wald}). Then, we discuss a conjecture on the leading terms of universal Heegner points (and toric periods) at classical points (\S~\ref{sec bd conj}).

\subsection{Proofs of the main theorems}\lb{sec: final}
Both of our central  theorems (Theorems \ref{GZ thm} and \ref{ugz thm}) ultimately follow from \cite{dd-pyzz, nonsplit}, where Theorem \ref{GZ thm} is established when $W$ is trivial, by an argument combining interpolation and multiplicity-one principles.

\subsubsection{$p$-adic Gross--Zagier formula for ordinary forms} \lb{B ord}
We start by stating a  variant of  Theorem \ref{GZ thm}, valid under  the same assumptions. 

  \begin{customthm}{$\text{B}^{\, \ord}$}\label{GZ thm'} Let $\Pi=\pi\ot \chi$ be an 
ordinary, locally distinguished automorphic  
 representation of $(\G\ts\H)'(\A)$ over $L$.
Let $V=V_{\Pi}$, and let $\wtil{P}_{\Pi}^{\ord} \in \wtil{H}^{1}_{f}(E, V_{\Pi})$  be the enhanced Heegner class defined in \eqref{tilPo}.

 Then 
for all $f_{1}^{}\in \Pi_{H'_{\infty}}^{\ord}$, $f_{2}^{}\in\Pi^{\vee, \ord}_{H'_{{\infty}}}$, $f_{3}^{}\in \Pi^{\ord}$, $f_{4}^{}\in\Pi^{\vee, \ord}$ with $(f_{3}, f_{4})^{\ord}\neq 0$, we have
\beq\lb{formula B'}   {  h_{V}(\wtil P_{\Pi}^{\ord}(f_{1}^{}),\wtil P_{ \Pi^{\vee}}^{\ord}(f_{2}^{}))    \over   (f_{3}^{}, f_{4}^{})^{\ord}_{\Pi} }  
= 
   \calL_{p}'(V_{(\pi, \chi)}, 0)
\cdot
Q^{\ord}\left( {  f_{1} \ot  f_{2}\over f_{3}\ot f_{4}}\right).
\eeq
\end{customthm}
\begin{rema}  In contrast to Theorem \ref{GZ thm}:
\begin{itemize}
\item
Theorem \ref{GZ thm'} also holds for exceptional $\Pi$;
\item we have only included the Gross--Zagier formula and omitted an analogue to the first statement of Theorem \ref{GZ thm}, that is that $\wtil P_{\Pi}^{\ord}$ takes values in $\wtil H^{1}_{f}(E, V)$, as that has already been established. 
\end{itemize} 
\end{rema}

\begin{lemm} \lb{B Bord} Suppose that $\Pi$ is not exceptional. Theorem \ref{GZ thm'} is equivalent to Theorem \ref{GZ thm}.
\end{lemm}
\begin{proof}  Using freely the notation and results of Appendix  \ref{app A}, we first show that Theorem \ref{GZ thm'} for $f_{1}$, $f_{2}$, $f_{3}$, $f_{4}$ is equivalent to  Theorem B for 
$$f_{1}'= \gamma_{H'}^{\ord} f_{1} \in \Pi_{H'_{p\infty}}, \quad f_{2}'= \gamma_{H'}^{\ord} f_{2}^{} \in \Pi^{\vee}_{H'_{p\infty}},\qquad f_{3}'=w_{\rm a}^{\ord} f_{3}^{}\in \Pi^{\rm a}, \quad f_{4}'=f_{4}^{}\in \Pi^{\vee, \ord};$$
let us call such $(f_{1}', f_{2}', f_{3}', f_{4}' )$  a `special quadruple'.

Indeed, by the definitions \eqref{Pord text}, \eqref{()'}, the left hand side of \eqref{formula B'} equals 
$$  {  h_{V}(P_{\Pi}^{}(f_{1}'), P_{\Pi^{\vee}}^{\ord}(f_{2}')) \over \dim W \cdot (f_{3}', f_{4}')};$$
whereas by  Proposition \ref{compare Qs}, 
\beqq
Q^{\ord}\left( {  f_{1} \ot  f_{2}\over f_{3}\ot f_{4}}\right).
= {e_{p}(V_{(\pi, \chi)})^{-1} \over\dim W }  \cdot Q^{}\left(   {f_{1}'\ot  f_{2}'\over f_{3}'\ot f_{4}' } \right).
\eeqq

By the multiplicity-one principle,  Theorem \ref{GZ thm} for special quadruples implies Theorem  \ref{GZ thm} in general, since under our assumptions the functional $Q$ is non-vanishing on special quadruples: this again follows from Proposition \ref{compare Qs} and Lemma \ref{6444}.1.(a)-(b). 
\end{proof}

\subsubsection{Comparison of $p$-adic $L$-functions}
We describe how, upon restricting  $ \calL_{p}(\cV^{\sharp})$ to the cyclotomic line through a point of trivial weight, we recover the $p$-adic $L$-function of \cite{dd-pyzz, nonsplit}.

\begin{lemm} \lb{compare Lp} 
Let $z=(x, y)\in \cE_{\G_{0}}^{\ord, \rm cl}\ts \cE_{\H}^{\ord, \cl}$ be a point corresponding to a representation $\pi_{0 ,x}\otimes \chi_{y}$ of weights $(0; (2, \ldots, 2))$, $(0;0, \ldots, 0)$. 
 Let $A$ denote the modular abelian variety attached to $\pi_{0,x}$, and let  $$\calL_{p}( V_{(A , \chi)})\in \mathscr{K}(\cE_{\rm Z})$$
  be
  the $p$-adic $L$-function of \cite[Theorem A]{nonsplit}.
   Consider the map 
 \beqq
 j_{x, y}\colon \cE_{\rm Z}&\to \{x\} \ts \cE_{\H}^{\ord} \subset \cE_{\G_{0}}^{\ord, \rm cl}{\ts} \cE_{\H}^{\ord}\\
 \chi_{F} &\mapsto \pi_{0, x}\otimes \chi_{y}\cdot \chi_{F\circ N_{E_{\A}^{\ts}/F_{\A}^{\ts}}}.\eeqq
 Then 
\beq
\lb{LL'} \calL_{p}( V_{(\pi_{0} , \chi)}):=  j_{(x, y)}^{*}  \calL_{p}(\cV^{\sharp}) &=   \calL_{p}( V_{(A, \chi)}).\eeq  
\end{lemm} 
\begin{proof}  This is immediate from the respective  interpolation properties. (Note that the first equality in
  \eqref{LL'} is just a reminder of \eqref{def Lp}.)
\end{proof}

\subsubsection{Interpolation argument and proof of the main theorems}
Let $\X \subset \cE^{\ord}_{K^{p}}$ be a locally distinguished Hida family for $(\G\ts \H)'$, as in Definition \ref{hida loc dist}. Fix a level $K^{p}{}'\subset K^{p}$.  Let 
$$\X'=\X'_{K^{p}{}'}:=\X^{(6)}\cap \X^{(3,f)}\supset \X^{\cl}$$  be  the intersection of the open subsets of $\X$ of Theorem \ref{X6} and Corollary \ref{X3f}.

Recall that we denote by $\X^{\cl, W}$ the set of classical points of weight $W$, and by $\X^{\cl, \textup{ n-exc}}$ the set of non-exceptional classical points. When $W=\Q_{p}$ is the trivial weight, we  also define
$$\X^{\cl,\ \textup{p-crys}, \  \Q_{p}, \textup{ram}} \subset \X^{\cl, \textup{p-crys}, \, \Q_{p}} \subset  \X^{\cl, \text{n-exc}} $$
by the following conditions on the classical point $z=(x, y)$ (equivalently, on the representation $\Pi_{z}$):
\begin{enumerate}
\item[(p-crys)] for all $v\vert p$, the representation $V_{x|G_{F_{v}}}$ is potentially crystalline (equivalently, $\pi_{x,v}$ is a principal series; the second inclusion above follows from Lemma \ref{6444});
\item[(ram)] $\chi_{y,p}$ is \emph{sufficiently ramified} in the following sense: let $r_{v}\geq1$ be minimal such that 
 $1+\vpi_{v}^{r_{v}} \OO_{F_{v}}$ is contained in the kernel of $\omega_{x,v}$, and let $U_{F,p}^{\circ}=\prod_{v\vert p}   (1+\vpi_{v}^{r_{v}} \OO_{F_{v}})$; then $\chi_{y, p}$ is is nontrivial on $N_{E_{p}/F_{p}}^{-1}(U_{F,p}^{\circ})\cap \OO_{E_{p}}^{\ts}$.
\end{enumerate}

 \begin{lemm} \lb{dense'}
 The subset $\X^{\cl,\ \textup{p-crys}, \  \Q_{p}, \textup{ram}}  \subset \X'$ is dense. 
 \end{lemm}
 \begin{proof} 
Denote by ${\rm p}_{\G} \colon \X' \to \cE_{\G}$ the natural projection.
 If  `$?$' is any relevant decoration, let  $\X_{\G}^{?}:= {\rm p}_{\G}(\X^{?})$; for  $x\in \X_{\G}^{\cl}$,  let $\X_{x, \H}^{?}:= {\rm p}_{\G}^{-1}(x)\cap  \X^{?}$.

For each $x\in \X_{\G}^{\cl, \textup{ p-crys, } \Q_{p}}$, the set 
$ \X_{x,\H}^{\cl,  \ \Q_{p},   \textup{ ram }}$ contains contains all but finitely many points in 
 $ \X_{x, \H}^{\cl, \, \Q_{p}}$, which is dense in  $\X_{x, \H}$. Thus the closure of $\X^{\cl,\ \text{p-crys}, \  \Q_{p}, \text{ ram}}$ contains all of $\X^{\cl,\ \text{p-crys},\  \Q_{p}}$.  
 
 Now we observe that $\X_{\G}^{\cl, \ \text{p-crys},\  \Q_{p}} = \mathscr{Y}_{\G}\cap \X_{\G}^{\cl,\  \Q_{p}}$ for the open subset  $$\mathscr{Y}_{\G}:=\X_{\G}- \{ x\ | \cV_{\G, v|x }^{+}(-1)  \cong \cV_{\G, v|x}^{-} \text{for some $v\vert p$}\} ,$$ which is non-empty as it contains $\X_{\G}^{\cl, W_{\G}}$ for any representation $W_{\G}$ whose partial weights are all $\geq 3$ (cf. the proof of Lemma \ref{6444}.1 (c)). 
Therefore  $\X^{\cl,\ \text{p-crys},\  \Q_{p}}$ is the intersection of the  non-empty open ${\rm p}_{\G}^{-1}(\Y_{\G})$ with  $\X^{\cl, \, \Q_{p}}$, which  is dense in $\X'$ by Lemma \ref{dense}.  We conclude that   $\X^{\cl,\ \text{p-crys},\  \Q_{p}}$ and  $\X^{\cl,\ \text{p-crys}, \  \Q_{p}, \text{ ram}}$ are dense in $\X'$. 
 \end{proof}
 

\begin{prop} \lb{final lemma}
The following are equivalent.
\begin{enumerate}
\item[1.] \lb{it1} Theorem \ref{ugz thm} holds over $\X'_{K^{p}{}'}$ for all $K^{p}{}'\subset K^{p}$.

\medskip

\item[2a.] \lb{it3a}
Theorem \ref{GZ thm'} holds for all representations $\Pi$ corresponding to points of $\X^{\cl}$ satisfying $\textup{(wt)}$.
\item[2b.] \lb{it3c}
Theorem \ref{GZ thm'} holds for all representations $\Pi$ corresponding to points of $\X^{\cl, \textup{ p-crys}, \ \Q_{p}, \textup{ ram}}$.

\medskip

\item[3a.] \lb{it4a} Theorem \ref{GZ thm}   holds for all representations $\Pi$ corresponding to points of $\X^{\cl, \textup{ n-exc}}$ satisfying $\textup{(wt)}$.
\item[3b.] \lb{it4c} Theorem \ref{GZ thm}  holds for all  representations $\Pi$ corresponding to points of 
$\X^{\cl, \textup{ p-crys}, \ \Q_{p}, \textup{ ram}}$.
\end{enumerate}
\end{prop}
\begin{proof}
 For any point   $z\in \X^{\cl, \textup{n-exc}}$ satisfying $\textup{(wt)}$,  denote  by $\Pi_{z}$ the associated automorphic representation, by $V_{z}$  the associated Galois representation . We have   proved the following specialisation-at-$z$ properties of objects defined  over (open subsets of)  $\X$:
\begin{itemize}
\item the $\OO_{\X'}$-module $\Pi_{H'_{\Sg}}^{K^{p}{}', \ord}$ (respectively $(\Pi_{H'_{\Sg}}^{K^{p}{}', \ord})^{\iota}$  specialises to $\Pi_{z,H'_{\Sg}}^{K^{p}{}', \ord}$ (respectively  $\Pi_{z,H'_{\Sg}}^{\vee, K^{p}{}', \ord}$), by Proposition \ref{cofinal}.\ref{Piord} (and the definition of the involution $\iota$);
\item the Galois representation $\cV$  (respectively $\cV^{\iota}$) and its ordinary filtrations  specialise to $V=V_{\Pi}$ (respectively $V_{\Pi^{\vee}}$)  with its ordinary filtrations, by construction (Proposition \ref{galois2});
\item there is a natural map $\wtil{H}^{1}_{f}(E, \cV)_{|z}\to \wtil{H}^{1}_{f}(E, V)$;
\item  the class $\cP_{K^{p}{}'}$ specialises to the restriction of $P^{\ord}_{\Pi_{z}}$ to $\Pi_{z,H'_{\Sg}}^{K^{p}{}', \ord}$ under the above map, by  Theorem   \ref{theo heeg univ} whose proof is completed in \S~\ref{645};
\item the product of local terms $\cQ$ specialises to the restriction of $Q^{\ord} $ to the spaces of $H'_{\Sg}$-coinvariants, $K^{p}{}'$-invariants in $\Pi_{z}^{\ord}$, $\Pi_{z}^{\vee,\ord}$, 
 by Theorem \ref{X6}.
\end{itemize}
Let us  complete the proof  that either side of Theorem \ref{ugz thm} specialises to $1/2$ times the corresponding side of Theorem \ref{GZ thm'}.  Consider the diagram $\X_{0}\to \X_{0}^{\sharp}\to \cE_{\rm Z}$. It is \emph{not} a product, even Zariski-locally; however  the
conormal sheaf is trivial. (This is dual to the fact that $\G\ts\H\to (\G\ts\H)'$ is a ${\rm Z}$-torsor for the \'etale topology but  {not} for the Zariski topology.)  The    immersion $\X\hat{\ts} \cE_{\rm Z}\to \X^{\sharp}$ given by `$(\Pi, \chi_{F})\mapsto \Pi\ot \chi_{F}$' induces the map on conormal sheaves 
$$\calN_{\X/\X^{\sharp}}^{*}=\OO_{\X}\hat{\ot} \Gamma_{F} \to
\calN_{\X/\X\ts \cE_{\rm Z}}^{*} = \OO_{\X}\hat{\ot} \Gamma_{F} $$
  that is multiplication by $1/2$
 under the canonical identifications. Hence:
\begin{itemize}
\item the $p$-adic height pairing ${h}_{\cV}=h_{\cV^{\sharp}|\X}$ specialises to 
${1\over 2} h_{V}={1\over 2} h_{V\ot \chi_{F, {\rm univ}|G_{E} }| z}$, by \S~\ref{sec: ht};
\item the derivative ${\rm d}^{\sharp} \calL_{p}(\cV)$   specialises to ${1\over 2} \calL_{p}'(V, 0)$ in $\Q_{p}(z)\hat{\ot} \Gamma_{F}$,  by the definition in \eqref{def Lp}.
\end{itemize}

We may now  complete the proof. By the specialisation properties summarised above, we have $1.\Rightarrow 2.a$ ($\Rightarrow 2.b$). By Lemma \ref{B Bord}, we have $2.a \Rightarrow 3.a$, $2.b \Leftrightarrow 3.b$.  
 By Lemma \ref{dense'} and the specialisation properties, $2.b\Rightarrow 1.$
\end{proof}

\begin{proof}[{Proof of Theorems \ref{GZ thm},  \ref{ugz thm}, and \ref{GZ thm'}}] The first assertion of Theorem \ref{GZ thm} was proved in Proposition \ref{in H1f}.
For a representation $\Pi$ of trivial weight satisfying the conditions (p-crys), (ram), the  formula of Theorem \ref{GZ thm} is \cite[Theorem B]{nonsplit} (cf. Lemma \ref{compare Lp}).
  By Proposition \ref{final lemma}, this implies  Theorem \ref{ugz thm} and the general case  of Theorems \ref{GZ thm},   \ref{GZ thm'}.  
\end{proof}

\subsubsection{Applications to non-vanishing / 1: self-dual CM families} \lb{sec: nvCM}
We prove the generic non-vanishing result of  Theorem \ref{nv CM}. Recall that $\Y$ is a component of the subvariety $\cE_{\H}^{\ord , \rm sd}\subset\cE_{\H}^{\ord}$ cut out by the condition $\chi_{F_{\A}^{\ts}}=\eta \chi_{{\rm cyc}, F}$, and such that $\eps(\chi_{y}, 1)=-1$ generically along $\Y$.

\begin{proof}[Proof of Theorem \ref{nv CM}] Recall that a $p$-adic CM type of $E$ over  $\baar{\Q}_{p}$ is a choice   $\Sg$ of a place $w\vert v$ of $E$ for each place $v\vert p$ of $F$ (we identify primes above $p$ with embeddings into   $\baar{\Q}_{p}$). For each of the $p$-adic CM types $\Sg$ of $E$ and each connected component $\Y^{\sharp}$ of $\cE^{\ord}_{\H}$,  there is a Katz $p$-adic $L$-function
$$L_{\Sg} \in \OO(\Y^{\sharp}).$$
It is characterised (see \cite{katz, HT}) by its  values at the subset $\Y^{\sharp, \cl.\Sg}\subset \Y^{\sharp, \cl}$ of those $y$ such that  the algebraic part of $\chi_{y}$ is   $t\mapsto\prod_{\sg\in \Sg} \sg(t)^{w} \sg(t/t^{c})^{k_{\sg}}$ for integers $w,  k_{\sg}$ such that either $w\geq 1$, or  $w<1$ and $w+k_{\sg}>0$ for all $\sg\in \Sg$. The interpolation property relates $L_{\Sg}(y)$ to $L(1, \chi_{y})$. It is easy to see that for  a given $y\in \cE^{\ord, \rm sd}_{\H}$, there is a unique CM type $\Sg$ such that $y$ belongs to the interpolation subset of $L_{\Sg}$. For such $y$ and $\Sg$,  we denote by $L_{p}(\chi_{y}, s)\in \OO(\cE_{{\rm Z}/\Q_{p}(y)})$ the function $s \mapsto L_{\Sg}(y(s))$ where $\chi_{y(s)}=\chi\cdot \chi_{F,s}\circ N_{E/F})$.

Consider now the setup of the theorem, and  let $\Y^{\sharp}\subset \cE_{\H}^{\ord}$ be the component containing $\Y$. 
By \cite{ashay}, under our assumptions the normal derivative $d^{\sharp} L_{\Sg}\in \mathscr{N}^{*}_{\Y/\Y^{\sharp}}$ is non-vanishing. Let $\Y'_{\Sg}$ be the complement of its zero locus, and let $$\Y':=\bigcap_{\Sg} \Y_{ \Sg}'.$$
 Let $y\in \Y^{\cl}\cap \Y'$, and let $\chi:=\chi_{y}$. It is easy to see that there is a unique $\Sg$ such that $y\in \Y^{\sharp, \cl , \Sg}$. 
 
 We \emph{claim} that there exists a a finite-order character $\chi_{0}\in \cE_{\H}^{\ord, \cl}$, such that  the character 
 $$\chi'= \chi^{c} \chi_{0}^{c}\chi_{0}^{-1} $$
 (that has  the same algebraic part as $\chi^{c}:=\chi\circ c$, and    defines a point $y'\in \cE_{\H}^{\ord, {\rm sd}, \cl}$) satisfies the following properties:
\begin{itemize}
 \item $L(1, \chi')\doteq L_{p}(\chi', 0)\neq 0$, where $\doteq $ denotes equality up to a nonzero constant;
 \item $H^{1}_{f}(E, \chi')=0$.
 \end{itemize}

 Granted the claim, we have a decomposition of $G_{E}$-representations
$$ (\chi\chi_{0}\oplus \chi^{c}\chi_{0}^{c})\ot \chi_{0}^{-1} =\chi \oplus \chi' $$ 
and a corresponding factorisation
$$ \calL_{p}(V_{(\pi_{0}, \chi_{0}^{-1})},s)\doteq L_{p}(\chi,s) L_{p}(\chi',s ), $$
 where $\pi_{0}=\theta(\chi\chi_{0})$ (the theta lift), and $\pi_{0}\otimes\chi_{0}^{-1}$ descends to a representation of $(\G_{0}\ts\H)'(\A)$.  It follows that  $\calL_{p}'(V_{(\pi_{0}, \chi_{0}^{-1})},0)\neq 0$. By Theorem \ref{BK thm}, we have a class 
 $$Z \in H^{1}_{f}({E}, V_{(\pi_{0}, \chi_{0}^{-1})}) \ot H^{1}_{f}({E}, V_{(\pi_{0}, \chi_{0}^{-1})}) =    (H^{1}_{f}({E},\chi) \otimes H^{1}_{f}({E},\chi^{c}))
 \oplus (H^{1}_{f}({E},\chi')\otimes H^{1}_{f}({E},\chi'^{c})) $$
  whose $p$-adic height is non-vanishing. Since $H^{1}_{f}({E},\chi')=H^{1}_{f}({E},\chi'^{c})) =0$, the class $Z$ is as desired.

It remains to prove the claim. Let $\Y_{1}\subset \cE_{\H}^{\ord, {\rm sd}}$ be a component over which the anticyclotomic Main Conjecture is known -- that is, one containing a finite-order character satisfying the properties of \cite{hida-quad}). By applying  \cite[Lemma 2.5]{bur-d} to any character corresponding to a point of $\Y_{1}^{\cl}$, we find another component $\Y_{2}\subset \cE_{\H}^{\ord, {\rm sd}}$,  whose classical points correspond to characters $\chi_{2}$ with $$\eps(1, \chi_{2})=1;$$
 moreover from the proof in \emph{loc. cit.}  one sees that $\Y_{2}$ may be taken to still satisfy the assumptions of \cite{hida-quad}. Then the function $L_{\Sg c|\Y_{2}}$ is non-vanishing by \cite{hsieh-mu}; hence, by the density of classical points with a given weight, we may find $y'\in \Y_{2}^{\cl}$ corresponding to a character $\chi'$ satisfying the first among the required  conditions. By the anticylcotomic Main Conjecture for $\Y_{2}$ proved in \cite{hida-mc, hida-quad}, that is equivalent to the second condition. Finally, the ratio $\chi'/\chi^{c} $ is an anticyclotomic character (that is, trivial on $F_{\A}^{\ts}$), hence (\cite[Lemma 5.3.1]{28}) of the form $\chi_{0}^{c}\chi_{0}^{-1}$ for some finite-order character $\chi_{0}$.  
\end{proof}

\subsection{A universal  Waldspurger formula}\lb{sec: wald}
We describe  the complementary picture over locally distinguished families  attached to \emph{coherent} quaternion algebras over $F$.  Unexplained notions and  notation will be entirely parallel to what defined in the introduction. 

 Let $B$ be a totally definite quaternion algebra \emph{over $F$}, let  $\Sg$ be the set of finite places where $B$ is ramified, and let $\G_{/\Q}$ be the algebraic group with $\G(R)=(B\ot R)^{\ts}$ for any $\Q$-algebra $R$. Let $E$ be a CM quadratic extension of $F$, admitting an $F$-algebra embedding ${\rm e}\colon E\into B$  which we fix. We use the same symbols as in the introduction for the towers of Shimura varieties associated to the groups as in \eqref{list}.  Here, all those Shimura varieties  are $0$-dimensional.

 Let $L$ be a $p$-adic field and let $\Pi$ be an automorphic representation of $(\G\ts\H)'(\A):=(\G\ts\H)'(\A^{\infty})\ts (\G\ts\H')_{/\Q_{p}}$ over $L$ of weight $W$, by which we mean one occurring in $\H^{0}(\baar{Z}, \cW^{\vee})\ot W$. The normalised  fundamental class of $Y$ gives rise to an element $P\in \H_{0}(Z, \cW)^{\H'(\A)}$ and to an  $\H'(\A)$-invariant functional 
 $$P_{\Pi}\colon \Pi\to L, $$
 which may be nonzero only if $\Pi$ is locally distinguished by $\H'$. If $\Pi$ is ordinary, we again define $P^{\ord}:= P\tord$.

 Let $\X$ be a locally distinguished  Hida family for $(\G\ts \H)'$, which via a Jacquet--Langlands  map is isomorphic to  a Hida family $\X_{0}$ for $(\G_{0}\ts\H)'$. For each compact open subgroup $K^{p}\subset (\G\ts\H)'(\A^{p})$, there is a `universal ordinary representation' $\Pi^{K^{p},\ord}_{H'_{\Sg}}$ of $(\G\ts\H)'(\A)$ over $\X$. As in Theorem \ref{theo heeg univ}, there exists an $\H'(\A^{p\infty})$-invariant, $\OO_{\X}$-linear functional 
\beq\lb{cP as functional2} \cP\colon  \Pi^{K^{p},\ord}_{H'_{\Sg}}\to \OO_{\X} \eeq
 interpolating (the restrictions of) $P_{\Pi_{z}}^{\ord}$ at all $z\in \X^{\cl}$ satisfying (wt). 
 
Starting from the natural pairings $\H_{0} (\baar{Z}_{K} , \cW)\ot \H_{0} (\baar{Z}_{K} ,\cW^{\vee})\to L$, we may define pairings $(\ , \ )_{\Pi}$ on each  representation $\Pi$ over a field  by the formula  \eqref{pair pi} (using the counting measure for ${\rm v}(K)$);   then we obtain modified pairings $(\ , \ )^{\ord}_{\Pi}$ on each $\Pi^{\ord}\ot \Pi^{\vee, \ord}$, and a  pairing $((\ ,  \ ))$ on the universal representations over $\X$, interpolating modified pairings  $(\ , \ )^{\ord}_{\Pi_{z}}$.

Finally, the functional $\cQ$, over an open $\X'\subset \X$ containing $\X^{\cl}$,  is also constructed as  in  \S~\ref{sec 44}; in the argument using the local Langlands correspondence, we use the rank-$2$ family of Galois representations pulled back from $\X_{0}$. 

 \begin{theoA}\label{wald univ} Let $\X$ be a locally distinguished Hida family for $(\G\ts\H)'$. 
 Abbreviate $\Pi^{(\iota)}:=  \Pi^{K^{p}{}', \ord, (\iota)}_{H'_{\Sg}}$, $\OO:=\OO_{\X}$, $\cK:=\cK_{\X}$. 

There is an
open subset $\X'\subset \X$ containing $\X^{\cl}$, such that 
$${\cP(f_{1})\cdot \cP^{\iota}(f_{2})    \over   ((f_{3}, f_{4}))    }
=
 \calL_{p}(\cV^{\sharp})_{|\X}
\cdot
\cQ\left( {  f_{1} \ot  f_{2}\over f_{3}\ot f_{4}}\right),
$$
 an equality of $\cK$-valued $\OO$-linear  functionals on $(\Pi\ot_{\OO} \Pi^{\iota})\ot_{\OO^{\times}} (\Pi\ot_{\OO} \Pi^{\iota})^{\ts, -1}$.
\end{theoA}
Similarly to \S~\ref{sec: final}, this universal formula follows from its specialisations at all classical points satisfying (wt); those are known by modifying the main result of  \cite{wald} as in Theorem \ref{GZ thm'}. 

 The formula essentially reduces the study of $\calL_{p}(\cV^{\sharp})_{|\X}$ to the study of the universal Waldspurger periods $\cP$. This is particularly interesting in the case of exceptional zeros, as we discuss next.

\subsection{Bertolini--Darmon conjectures and exceptional zeros} \lb{sec bd conj}
We first  formulate a conjecture on the behaviour of $\cP$ at a point $z\in \X^{\cl}$ and gather some old and new evidence in its favour.  In view of \S~\ref{ss 648}, the conjecture is often particularly interesting when $z$ is exceptional.  Then we deduce from our constructions and a known exceptional case of the conjecture a proof of Theorem \ref{nv exc}.

The conjecture requires some algebraic preliminaries.
 
 \subsubsection{Pfaffian regulators}  
 Let $L$ be a field  of characteristic $0$, and let $M$, $T$ be a finite-dimensional $L$-vector spaces.  Let $h\colon M\ot M\to T$ be a skew-symmetric pairing, and let $r=\dim_{L}M$.
 
 If $r$ is even, we define the Pfaffian regulator  $${\rm Pf}^{+}(M, h)\in ({\rm Sym}^{r/2}T )/ L^{\ts}$$
 to be the Pfaffian of the skew-symmetric matrix $h(x_{i}, x_{j})_{ij}$ for any $L$-basis $x_{i}$ of $M$. It is well-defined modulo $L^{\ts}$. 
 
 If $r$ is odd, we define an enhanced Pfaffian regulator 
  $${\rm Pf}^{-}(M, h)\in (M\ot {\rm Sym}^{(r-1)/2}T )/ L^{\ts}$$
as follows. It suffices to define $\partial^{e}{\rm Pf}^{-}(h) $ for any basis $\partial_{1}, \ldots, \partial_{d}$ of $T^{\vee}$ and all tuples $e=(e_{i})_{i=1}^{d}$ with $\sum_{i=1}^{d}e_{i}= (r-1)/2$; here $\partial^{e}:= \prod_{i=1}^{d} \partial_{i}^{e_{i}} $. Let $I\subset \{1, \ldots, d\}$ be the support of the tuple $e$. Let $M_{I}   \subset M$ be the sum, over $i \in I$,  of the radicals of the pairings $\partial_{i} h$. If $\dim M_{I}\geq 2$, we define $
\partial^{e} {\rm Pf}^{-}(h) :=0$. If $\dim M_{I}=1$ and $x\in M_{i}$ is a generator,  denote by $\baar{h}$ the induced pairing on $\baar{M}:= M/M_{I}$; we define 
$$\partial^{e} {\rm Pf}^{-}(M, h)  = x\ot \partial^{e} {\rm Pf}^{+}(\baar{M}, \baar{h}). $$

\begin{rema}\lb{rel reg} In the even case, we have of course ${\rm Pf}^{+}(M, h) ^{2}  = R(M, h)\in ( {\rm Sym}^{r}T )/ L^{\ts,2}$, where $R(M, h)$ is the discriminant of the pairing $h$. In the odd case, assume further given a symmetric bilinear pairing $h^{\sharp}\colon M\ot M\to T^{\sharp}$. Let $h'=h\oplus h^{\sharp}\colon M\ot M\to T':-T\oplus T^{\sharp}$, and let $R(M, h')\in {\rm Sym}^{r}( T')/L^{\ts, 2}$ be its discriminant. Then it is easy to verify that 
$$h^{\sharp}( {\rm Pf}^{-}(M, h), {\rm Pf}^{-}(M, h))\in (T^{\sharp}\ot {\rm Sym}^{r-1}T/) L^{\ts, 2}$$
 is the image of $R(M, h')$ under the natural projection.
\end{rema}

\begin{rema}\lb{ambi int} If $L$ is the fraction field of a domain $\OO$ and $M$ is endowed with an $\OO$-lattice, it is possible to lift the ambiguity in the definitions to an element of $\OO^{\ts}$. 
\end{rema}

\subsubsection{A conjecture \`a la Bertolini--Darmon} \lb{sec 733}
Let  $\G$ be either as in the introduction or as in \S~\ref{sec: wald}. We define a sign $\epsilon := -1$ in the former case and $\epsilon := +1$ in the latter case. Let $\X$ be a locally distinguished Hida family for $(\G\ts\H)'$, and let $$z\in \X^{\cl}$$ be a classical point.  We denote by $\cP$ the universal Heegner class (if $\epsilon=-1$) or toric period (if $\epsilon =+1$), viewed as a family of functionals as in \eqref{cP as functional}, \eqref{cP as functional2}, parametrised by a subset $\X'\subset \X$.  Assume that $\X'$ can be taken to be a neighbourhood  of $z\in \X$.

Let  $T_{z}^{*}\X ={\frakm}_{z}/\frakm_{z}^{2}$ be the cotangent space to $\X$ at $z$ (where $\frakm_{z}\subset \OO_{\X, z}$ is the  maximal ideal), and for any $r\in \N$ let  
$$({\rm d}_{z}^{\parallel})^{r} \colon {\frakm}_{z}^{r}\subset \OO_{\X, z}\to \frakm_{z}^{r}/\frakm_{z}^{r+1} = {\rm Sym}^{r} T_{z}^{*}\X  $$
be the   natural projection. It is easy to see that the involution $\iota$ of \S~\ref{ssec invo} satisfies
${\rm d}_{z}^{\parallel}\iota =\id$ (whereas ${\rm d}_{z}^{\sharp} \iota =-\id$).  

Let $V=\cV_{|z}$, and let
 ${\rm c} \colon  \wtil{H}^{1}_{f}(E, V^{\iota}) \to  \wtil{H}^{1}_{f}(E, V)$ be the isomorphism induced by the adjoint action, on $G_{E, S}$, of a lift of the complex conjugation in $\Gal(E/F)$. Let 
\beq\lb{hpar}h^{\parallel}:= h_{z/\X}^{\Box} \colon \wtil{H}^{1}_{f}(E, V) \ot  \wtil{H}^{1}_{f}(E, V)  \to  T_{z}^{*}\X\eeq
 be the \nek--Venerucci height pairing as in \eqref{skewsym}.  Since the pairing of Proposition \ref{duality} is skew-hermitian, by Proposition \ref{prop ht2} the pairing $h^{\parallel}$ is skew-symmetric.
Define  $${\rm Pf}^{\, \parallel, \pm}(V):= {\rm Pf}^{\pm} (\wtil{H}^{1}_{f}(E, V), h^{\parallel}).$$

  \begin{customconj}{Pf} \lb{pf conj}  Let $\wtil{r}:= \dim_{L}\wtil{H}^{1}_{f}(E, V)$.
We have $(-1)^{\wtil{r}}=\epsilon$, the universal element  $\cP$ vanishes to order at least $\lfloor{\wtil{r}/2}\rfloor$ at $z$, and for any generator  $\wp\in (\Pi^{\ord})^{*, \H'(\A^{p\infty)}}$ and all $f\in \Pi^{\ord}$, we have
\beq\lb{pf conj eq}
({{\rm d}}^{\parallel}_{z})^{ \lfloor{\wtil{r}/2}\rfloor} \, \cP(f) = {\rm Pf}^{\, \parallel, \epsilon}(V) \cd \wp(f)\eeq
in  $[\wtil{H}^{1}_{f}(E, V)^{(1+\epsilon)/2}\ot_{\Q_{p}(z)} {\rm Sym}_{\Q_{p}(z)}^{{\lfloor{\wtil{r}/2}\rfloor}}T_{z}^{*}\X ]\, / \Q_{p}(z)^{\ts}$.
\end{customconj}

\subsubsection{Relation to the original conjectures of Bertolini--Darmon}\lb{ssec bd} Define $\X^{\rm a}:= \X\cap ( \{x\}\ts \cE_{\H'})$ (the anticyclotomic family), $\X^{\rm wt} := \X\cap (\cE_{\G}\ts\{y\})$ (the weight family). In the classical case when $E$ is imaginary quadratic, $\pi$ is associated with an elliptic curve over $\Q$, $\chi=\one$,  and we restrict to the anticyclotomic variable $\X^{\rm a}$, Conjecture \ref{pf conj} is a variant of conjectures of Bertolini--Darmon, surveyed in \cite{bdsurvey} and (in a form slightly closer to the one of the present work) in \cite[\S~4]{dd-exc}. The Bertolini--Darmon conjectures were partly generalised to higher-weight modular forms  in \cite{LV}.

\begin{rema} By using natural  $G_{E, S}$-stable lattices in $V$, or better lattices in $\wtil{H}^{1}_{f}(E, V)$ spanned by motivic elements, it is possible to define the Pfaffian regulators up to an ambiguity that is a unit in the ring of integers of a local field or of a number field (recall Remark \ref{ambi int}). It should then possible to refine the conjecture up to such ambiguity (by including the appropriate constants), as in the original works of Bertolini--Darmon (see \cite[Conjecture 4.2.1]{dd-exc}).

In view of Remark \ref{rel reg}, inserting Conjecture \ref{pf conj} into Theorem \ref{ugz thm} would yield a multivariable formula relating higher partial derivatives of $p$-adic $L$-functions  with suitable height regulators, in the spirit of the  Birch and Swinnerton-Dyer conjecture; the argument is the same as that of  \cite[Proposition 5.1.1]{dd-exc}. We plan to return to formulate such conjectural formulas in the appropriate generality in future work. 
\end{rema}

\subsubsection{Evidence for Conjecture \ref{pf conj} in low rank} The preliminary parity conjecture 
\beq \lb{par1} 
(-1)^{\wtil{r}}=\epsilon\eeq 
is known in many cases as a consequence of the work of \nek\ (see \cite[Theorem 12.2.3]{nek-selmer}. Indeed, the statement proved  in \emph{loc. cit.}, in view of the functional equation of $\calL(V_{(\pi, \chi)},s)$, is that 
\beq \lb{par2}
(-1)^{r}= \eps
\eeq
 where $r=\dim H^{1}_{f}(E, V)$ and $\eps =\eps (V)$. Now if $r^{\rm exc}= r^{\rm exc, \, s}+ {r}^{\rm exc,\, ns}$ denotes the number of exceptional  primes of $F$ above $p$ (respectively, the number of those that moreover are split or nonsplit in $E$),  by Lemma \ref{6444}.2 we have $\epsilon=\eps\cd (-1)^{r^{\rm exc,\, ns}}$, and $\wtil{r}= r+2r^{\rm exc,\, s}+ r^{\rm exc, \, ns}$. Thus \eqref{par2} is equivalent to \eqref{par1}.

Let us now review the conjectural vanishing and leading-term formula. Most of the available evidence is concentrated in the case where $V$ arises from the classical context of \S~\ref{ssec bd}, to which we restrict unless otherwise noted for the rest of this discussion. (All the results mentioned below hold under various additional assumptions, which we will not recall.)

 If $\X$ is replaced with $\X^{\rm a}$ (the original Bertolini--Darmon case), the conjecture is  known if $\calL(V_{(\pi, \chi)},s)$ vanishes  to order $0$ or $1$ at $s=0$, see \cite[Theorem 4.2.5]{dd-exc} and references therein, as well as \cite{LP}. Some of those results have been generalised to higher weight (\cite{IoS, Mas}) or to totally real fields (\cite{Hu, BGe, MB}).

  When $\X$ is replaced by $\X^{\rm wt}$, $p$ is inert in $E$, $V$ is exceptional, and $\epsilon =+1$, Bertolini--Darmon \cite{bd-hida} proved a formula  for  ${\rm d}_{z}^{\parallel} \cP_{|\X^{\rm wt}}$, which implies the projection of \eqref{pf conj eq} to $T_{z}^{*}\X^{\rm wt}$ when ${\rm ord}_{s=0}\calL(V_{(\pi, \chi)}, s)=1$. The interpretation of the formula of Bertolini--Darmon in terms of height pairings was observed by Venerucci (see \cite[Theorem 2.1 and Theorem 4.2.2]{ven}), whose work was a second important influence in the formulation of Conjecture \ref{pf conj}.  The Bertolini--Darmon formula was generalised to higher-weight modular forms by Seveso \cite{Sev}  and to elliptic curves over totally real fields by Mok \cite{Mok}.

  \subsubsection{Evidence  for Conjecture \ref{pf conj} in higher rank}
Lower bounds for the order of vanishing of $\cP_{|\X^{\rm a}}$ have been obtained in two recent works for $\epsilon =1$. In the context of elliptic curves over totally real fields, \cite[Theorem 5.5]{BGe} gives a  bound (that is coarser than predicted by Conjecture \ref{pf conj}) in terms of the number of exceptional primes. In a classical context (and if if $p$ splits in $E$),  Agboola--Castella \cite[Corollary 6.5]{Agb-Ca} prove a bound that is finer than that of Conjecture \ref{pf conj}. 
 (That refined  bound is predicted by Bertolini--Darmon;  it is related to some trivial  degeneracies of the anticyclotomic height pairing, as touched upon also in the paragraphs preceding \eqref{exppl} below.)

 Regarding  the formula of Conjecture  \ref{pf conj}, an interesting anticyclotomic case can be deduced from  the recent work \cite{FG}, as we now explain. 

 \paragraph{The work of Fornea--Gehrmann}
Suppose that $A$ is a modular elliptic curve over the totally real field $F$,  such that the set $S_{p}^{\rm exc}$
of places  $v|p$ where $A$ has multiplicative reduction consists of exactly  $r$ primes $v_{1}, \ldots, v_{r}$ inert in $E$.  
Let $\phi\colon\bigotimes_{i=}^{r} E_{v_{i}}^{\ts}\to \bigotimes_{i=1}^{r} {A}(E_{v_{i}})$  be the product of Tate unifomisations, and let  $\hat{\phi}$ be the induced map on $p$-adic completions.

 One of the main constructions of \cite{FG} produces an explicit  element $Q_{A}\in \bigotimes_{i=1}^{r} E_{v_{i}}^{\ts}\hat{\ot} \Q_{p}$, which carries a precise conjectural relation to the arithmetic of $A$. Partition $S_{p}^{\exc}= S_{p}^{\exc, +}\sqcup S_{p}^{\exc, -}$ according to whether the multiplicative reduction is,  respectively, split or nonsplit, and let $r^{+}:= |S_{p}^{\exc, +}|$.  Denote   $ \hat{A}(E_{v})=A(E_{v})\hat{\ot}\Q_{p}$ and let $\hat{A}(E_{v})^{\pm}$ be its $\pm$-eigenspaces for the conjugation $c_{v}\in \Gal(E_{v}/F_{v})$.

 Assume that $A_{E}$ has rank  $r_{A}\geq r$ and let $r_{A}^{+}$ be the rank of $A$. 
According to \cite[Conjectures 1.3, 1.5]{FGM}, we have\footnote{In \emph{locc. citt.}, some restrictive assumptions are made (in particular that $E$ is not CM), but the conjectures make sense even without those and indeed closely related conjectures appear in \cite{FG} without those assumptions. Moreover, our statement  slightly differs from the ones of \cite{FG}, which instead of postulating that $(x_{1}, \ldots, x_{r})$ is a basis, postulates that $\hat{\phi}(Q_{A})\neq 0 $ under the extra assumption that $\Res_{F/\Q}A$ is simple (equivalently $L(A, s)$ is primitive). Our slight reformulation appears more uniform and still addresses \cite[Remark 1.1]{FG} (cf. the comment following \cite[Conjecture 1.5]{FG}). }
 \beq \lb{conj FG}
  r_{A}>r \text{ or }  r_{A}^{+}+r^{+}\neq r &\Longrightarrow \hat\phi(Q_{A}) = 0, \\
  r_{A}= r \text{ and }  r_{A}^{+}+r^{+}=r &\Longrightarrow
\hat\phi(Q_{A}) \doteq \det( (\hat{x}^{\rm a}_{i,v_{j}})_{1\leq i, j\leq r}) \eeq
where $(x_{1}, \ldots, x_{r})$ is a basis of $A(E)_{\Q}$, and for $v\in S_{p}^{\exc , \pm} $ 
we denote by $\hat{x}_{v}^{\rm a}$ the eigen-projection  of $x\in A(E)$ to $ \hat{A}(E_{v})^{\mp}$. The symbol $\doteq $ denotes equality up to a constant in $\Q^{\ts}$. 

For  $V=V_{p}A_{E}$, assuming the finiteness of $\Sha(A_{E})[p^{\infty}]$ we have $\wtil{r}:=\dim  \wtil{H}^{1}_{f}(E, V)  =r+r_{A}$ (see Lemma \ref{6444}.1(a) or \eqref{exppl} below), and   the parity conjecture is known. Hence if $r_{A}\equiv r\pmod 2$,  which we henceforth assume,
  $V$ corresponds to a point $z$ of a locally  distinguished Hida family $\X$ of sign $\epsilon =+1$ (for a unique choice of the coherent quaternionic group $\G$).
 Let 
 $$\Gamma^{\rm a}:= (E_{\A^{\infty}}^{\ts}/F_{\A^{\infty}}^{\ts}\widehat{\OO}_{E}^{p, \ts}) \hat{\ot}\Q_{p}\cong T_{z}^{*}\X^{\rm a},  
 \qquad  \ell^{\rm a}=\prod_{w} \ell^{\rm a}_{w}\colon E_{\A^{\infty}}^{\ts}/E^{\ts} \to \Gamma^{\rm a}$$
  be the natural projection, and let 
  $\ell_{\rm exc, \ot}^{\rm a}:=\bigotimes_{i=1}^{r}\ell^{\rm a}_{v_{i}}\colon \bigotimes_{i=1}^{r} E_{v_{i}}^{\ts}\hat{\ot}\Q_{p} \to {\rm Sym}^{r}\Gamma^{\rm a}$.  Denoting by ${\rm d}_{z}^{\rm a}$ the component of ${\rm d}_{z}^{\parallel}$ along $\X^{\rm a}$,  
Theorem A of \cite{FG}  (which relies on the aforementioned lower bound from \cite{BGe}) proves 
\beq \lb{eq FG} ({\rm d}_{z}^{\rm a})^{ r} \cP(f^{\circ}) \doteq \ell_{\rm exc, \ot}^{\rm a}(Q_{A}) \eeq
for a suitable test vector $f^{\circ}\in \Pi^{\ord}$. 

 \paragraph{Comparison with Conjecture \ref{pf conj}}
We show that granted \eqref{conj FG} and the finiteness of $\Sha(A_{E})[p^{\infty}]$, the formula \eqref{eq FG} is equivalent to the conjectured \eqref{pf conj eq}.

Let $h^{\rm a}\colon \wtil{H}^{1}_{f}(E, V) \ot \wtil{H}^{1}_{f}(E, V) \to T_{z}^{*}\X^{\rm a}=\Gamma^{\rm a}$
be the projection of $h^{\parallel}$, and let 
$${\rm Pf}^{\, \rm a, +}(V) ={\rm Pf}^{+} ( \wtil{H}^{1}_{f}(E, V) , h^{\rm a}) \in {\rm Sym}^{r} \Gamma^{\rm a}.$$
Let $\wtil{H}^{1}_{f}(E, V)^{\pm}$ be the eigenspaces for the complex conjugation $c\in \Gal(E/F)$.   Since $h_{z/\X^{\rm a}}$ enjoys the $c$-equivariance property $h_{z/\X^{\rm a}}(cx, cx')=c. h_{z/\X^{\rm a}}(x, x')$ and $c$ acts by $-1$ on $\Gamma^{\rm a}$, by construction the pairing $h^{\rm a}$ satisfies $h(cx, cx')= -h(x, x')$,  so that each of  $\wtil{H}^{1}_{f}(E, V)^{\pm}$ is $h^{\rm a}$-isotropic. In particular, Conjecture \ref{pf conj} agrees with \eqref{eq FG} and \eqref{conj FG} that, in the first case of the latter, we have $({\rm d}_{z}^{\rm a})^{r}\cP=0$.

Assume now we are in the second case of \eqref{conj FG}, so that each of  $\wtil{H}^{1}_{f}(E, V)^{\pm}$  
has dimension $r$ and $h^{\rm a}$ need not be degenerate.
For $1\leq i\leq r$,  let $q_{i}=q_{v_{i}}\in E_{v_{i}}^{\ts}$ be a Tate parameter for $A/E_{v}$. By
  \cite[\S~7.14]{nekheights}
 we explicitly have
\beq \lb{exppl}
\wtil{H}^{1}_{f}(E, V) &=  H^{1}_{f}(E, V) \oplus \bigoplus_{i=1}^{r} \Q_{p}\cd[q_{v_{i}}] \\
 h^{\rm a}(q_{v}, q_{v'})=0,\qquad  h^{\rm a}(x, q_{v}) &= \log^{\rm a}_{A, v} (\hat{x}_{v}) , \qquad x\in A(E),\quad  v\neq v',
\eeq
where $\log_{A, v}^{\rm a}\colon \hat{A}(E_{v})\ot \Q_{p} 
\cong E_{v}^{\ts}/q_{v}^{\Z}\hat{\ot}\Q_{p} 
\cong  \OO_{E_{v}}^{\ts}\hat{\ot}\Q_{p} \stackrel{\ell_{v}^{\rm a}}{\longrightarrow} \Gamma^{\rm a}$. Note that $\log_{A, v}^{\rm a}$ factors through $x\mapsto \hat{x}_{v}^{\rm a}$. 

Up to changing the basis  $x_{i}$ of $A(E)_{\Q}$ and reordering the $v_{i}$,  we may assume that the basis 
 $$x_{r^{+}+1}, \ldots, x_{r}, q_{1}, \ldots, q_{r^{+}}, \qquad  x_{1}, \ldots, x_{r^{+}}, q_{r^{+}+1}, \ldots, q_{r} $$
of $\wtil{H}^{1}_{f}(E, V)$ is the concatenation of a basis of  $\wtil{H}^{1}_{f}(E, V)^{+}$ and a basis of $\wtil{H}^{1}_{f}(E, V)^{-}$, respectively. Using this basis, \eqref{exppl},  and the identity ${\rm pf} \smalltwomat {}{M}{-M^{\rm t}}{} =\pm \det M$, we have\footnote{All the equalities to  follow ignore signs and in fact, by our coarse definitions, only make sense at best  up to $\Q^{\ts}$.}
$${\rm Pf}^{\, \rm  a, +}(V)=\det M= \det M_{1}\det M_{2},$$
 where the $r\ts r$ matrix $M$ is block-left-upper-triangular with anti-diagonal blocks $M_{k} = (\log_{A,v_{j}}^{\rm a}(x_{i}))_{i, j \in I_{k}}$ for $I_{1}=\{ r^{+}+1,\ldots, r\}$, $I_{2}=\{ 1, \ldots, r^{+}\}$.

On the other hand, we note that under \eqref{conj FG}, we have $\ell_{\rm exc, \ot}^{\rm a}(Q_{A}) = \det  N$ where   the $r\ts r$ matrix $N_{ij}= \log^{\rm a}_{A,v_{j}}(\hat{x}_{i, v}) = h^{\rm a}(x_{i}, q_{j})$ is block-diagonal with  blocks $N_{1}=M_{2}$, $N_{2}=M_{1}$. Thus $\ell_{\rm exc, \ot}^{\rm a}(Q_{A})  = {\rm Pf}^{\, \rm  a, +}(V)$, and \eqref{eq FG} is equivalent to \eqref{pf conj eq}.

\subsubsection{Applications to non-vanishing / 2: exceptional families} \lb{sec G} We prove Theorem \ref{nv exc}.\footnote{A less interesting variant of it was sketched in \cite{nonsplit}.} Recall that $\X_{0}$ is a Hida family for ${\rm PGL}_{2/\Q}$, that contains a classical point $x_{0}\in \X(\Q_{p})$ corresponding to an elliptic curve $A$ with split multiplicative reduction at $p$ satisfying $L(A, 1)=L(V_{p}A, 0)\neq 0$. 

\begin{proof}[Proof of Theorem \ref{nv exc}]
 Let $E$ be an imaginary quadratic field, with associated quadratic character $\eta$, satisfying the following:  $p$ is inert in $A$, all other primes dividing the conductor of $A$ split in $E$, and the twisted $L$-value $L(A, \eta, 1)\neq 0$. Then $A$ has split multiplicative reduction over $E$ with Tate parameter
  $$q=q_{A}\in E_{p}^{\ts}.$$ 
Let $\Omega_{A_{E}}\in \bC^{\ts}$ be the N\'eron period, and let 
  let $\H:={\rm Res}_{E/\Q}{\bf G}_{m}$. 

By construction,  $\eps_{v}(V_{p}A_{E})=1$ for all finite $v\nmid p$, hence the Hida family $\X\subset \cE_{(\GL_{2}\ts\H)'}^{\ord}$ containing the image of $\X_{0}\ts \{\one\}$ is locally distinguished. Let $\X^{\sharp}\subset  \cE_{(\GL_{2}\ts\H)}^{\ord}$ be the Hida family containing $\X$.
Let $\Pi$ be the universal ordinary representation over $\X$ and let $f\in \Pi$ be such that $f_{|x_{0}}$ is a test vector (that is, a vector not annihilated by any $\H'(\A^{p})$-invariant  functional $\lm\colon \Pi_{|x_{0}}\to \Q_{p}$). Let $\cP_{0, E}$ be the pullback of $\cP(f)$ to $\X_{0}$, and let $\cP_{0}:={1\over 2}\Tr_{E/\Q}\cP_{0, E}$.  By Corollary \ref{X3f}, 
$$\cP_{0,E}\in \wtil{H}_{f}^{1}(G_{E}, \cV_{0}),\qquad \cP_{0}\in  \wtil{H}_{f}^{1}(G_{\Q}, \cV_{0}).$$

By the main result of \cite{BDrigid} (as reformulated in \cite[Theorem 5.4, \S~5.2]{bdsurvey}),\footnote{In the works of Bertolini--Darmon, an explicit test vector $f$ is chosen;  cf. \cite{dd-exc} for more details on bridging the setups.} there is a constant $c\in \Q_{p}^{\ts}$ such that 
$$\qquad \cP_{0, E}(x_{0})\ot \cP_{0, E}^{\iota}(x_{0})=c\cdot {L(A_{E}, 1)\over \Omega_{A_{E}} }\cdot [q] \ot [q]\qquad \text{in }  \wtil{H}^{1}_{f}({\Q}, V_{p}A_{E})\ot  \wtil{H}^{1}_{f}({\Q},V_{p} A_{E})  $$
using the  description
$ \wtil{H}^{1}_{f}({\Q}, V_{p}A_{E})= \Q_{p}\cdot [q] \oplus H^{1}_{f}({\Q}, V_{p}A_{E})$ as in \eqref{exppl}.

In particular, $\cP_{0, E}(x_{0})= \cP_{0}(x_{0})$ is a nonzero multiple of $[q]$, which   is $\Gal(E/\Q)$-invariant. Hence $\cP_{0, E}$ and $\cP_{0}$ are  non-vanishing.  Then by  \cite{nek-koly},  $\wtil{H}^{1}_{f}({\Q}, \cV_{0})$ has generic rank~1.

Moreover, 
\beq\lb{GSa}
h_{\cV_{0}/\cV_{0}^{\sharp}}(\cP_{0} , \cP_{0}^{\iota})(x_{0}) =c\cdot {L(A_{E}, 1)\over \Omega_{A_{E}} }\cdot h([q], [q]) = c\cdot \ell (q) \cdot {L(A_{E}, 1) \over \Omega_{A_{E}}}\in \Gamma_{\Q}\ot_{\Z_{p}}\Q_{p},\eeq
where $\ell\colon \Q_{p}^{\ts}\to \Gamma_{\Q}$ is the universal logarithm (see again  \cite[\S~7.14]{nekheights} for the  second equality). By \cite{st et}, the right-hand side is nonzero, hence $h_{\cV_{0}/\cV_{0}^{\sharp}}(\cP_{0}, \cP_{0}^{\iota})\neq 0$. 
\end{proof}
\begin{rema} \lb{rmk GS}
As noted in \cite{dd-exc, nonsplit}, the combination of Theorem \ref{ugz thm} (or rather Theorem \ref{GZ thm'}) and a precise form of \eqref{GSa} gives a new proof ot the following  theorem of Greenberg--Stevens \cite{GS}: for $A_{/\Q_{p}}$ an elliptic curve of split multiplicative reduction at $p$ and $L_{p}(V_{p}A)\in \Z_{p}\llb \Gamma_{\Q}\rrb_{\Q_{p}}$ its $p$-adic $L$-function,    $$L_{p}'(V_{p}A, 0)= {\ell(q)\over{\rm ord}_{p}(q)} \cdot {L(A, 1)\over \Omega_{A}}.$$
\end{rema}

\appendix

\section{$p$-adic semilocal  constructions} \lb{app A}

\subsection{Preliminaries}\lb{A1}
Throughout this appendix, unless otherwise noted  $L$ denotes a field of characteristic zero (admitting embeddings into $\bC$).

\subsubsection{Admissible and coadmissible representations} Let $\G$ be a reductive group over $\Q_{p}$. We denote 
\beq \lb{G ccc}
G_{p}:=\G(\Q_{p}), \qquad G_{\infty}:= \G(\Q_{p}), \qquad G=G_{p\infty}:=G_{p}\ts G_{\infty}, \qquad G_{\Delta}:= \Delta(\G(\Q_{p}))\subset G, \eeq
where $G_{p}$ and $G_{\Delta}$ have the $p$-adic topology, $G_{\infty}$ has the Zariski topology, and  $\Delta $ is the (continuous) diagonal embedding. The difference between $G_{p}$, $G_{\infty}$, $G_{\Delta}$ will be in the category of modules we choose to consider.  Namely, we consider the categories of smooth admissible representations of $G_{p}$ over $L$, of algebraic representations of $G_{\infty}$ over $L$, and the products of such for $G$; we call the  latter locally algebraic representations of $G$ over $L$.  

\begin{defi} \lb{def admiss} Suppose  that $L$ is a finite extension of $\Q_{p}$. 
A \emph{$p$-adic} locally algebraic admissible  representation $\Pi$ of $G$ over $L$ is one such that for each compact open subgroup $K\subset G_{\Delta}$, 
there exists a family of  $\OO_{L}$-lattices $\Pi^{K, \circ}\subset \Pi^{K}$, for $K\subset G_{\Delta}$,  with the property that  $\Pi^{K', \circ}\cap \Pi^{K}= \Pi^{K, \circ}$ for all $K'\subset K$. 
\end{defi}
The typical example of  a $p$-adic locally algebraic admissible  representation is 
$\varinjlim_{K_{p}\subset G_{p}} H^{i}(Y_{K^{p}K_{p}}, \cW)\ot W^{\vee}$, where $Y_{K}$ is the system of  locally symmetric spaces attached to a model $\G_{\Q}$ of $\G$ over $\Q$, and $\cW$ is the automorphic local system attached to the algebraic representation  $W$ of $G_{\infty}$.

\medskip

There is a dual notion, introduced in  \cite[p. 152]{Sch-T}, see also \cite{Sch-T2}. Assume that $L$ is endowed with a discrete valuation (possibly trivial), giving it a  norm $|\cd|$. Let  $G'$ be one of the groups \eqref{G ccc} or an open subgroup. For $K\subset G'$ a compact open subgroup,  let  $\cD_{G',K}=\cH_{G', K}:=C^{\infty}_{c}(K\bks G'/K,L)$ and $\cD_{G'}=\varprojlim \cD_{G',K}$ be the Hecke algebras of distributions; they are endowed with a natural topology as $L$-vector space, respectively as the inverse limit. A \emph{coadmissible} $G'$-representation $M$ over $(L, |\cd|)$ is  a topological right $\cD_{G'}$-module such that, for any compact subgroup $G^{\circ}\subset G'$,  the $\cD_{G^{\circ}}$-module $M$   admits a presentation  of the following form: there exists a  system of  topological $\cD_{G^{\circ}, K}$-modules $M_{K}$ and isomorphisms  $M_{K}\cong M_{K'}\ot_{\cD_{G^{\circ}, K'}} \cD_{G^{\circ}, K}$ for  $K'\subset K\subset G^{\circ}$, such that   $M\cong \varprojlim_{K} M_{K}$.

Considering first a field  $L$ as endowed with a trivial valuation, we shall consider coadmissible representations $M$ of $G_{p}$ over $L$ that are \emph{smooth} in the sense  the Lie algebra $\frak g$ of $G_{p}$ acts trivially; coadmissible representations $W$  of $G_{\infty}$ that are algebraic (those are just algebraic representations); and the products of such as representations of $G$, which we call locally algebraic coadmissible representations of $G$.

\begin{defi}  \lb{def coadmiss} Suppose  that $L$ is a finite extension of $\Q_{p}$; denote by $|\cd|$ the $p$-adic norm and by $|\cd|_{\rm triv}$ the trivial norm on $L$. 
A \emph{$p$-adic} locally algebraic coadmissible  representation $M$ of $G$ over $L$  is one as above for $(L, |\cd|_{\rm triv})$, whose restriction to $G_{\Delta}$ is coadmissible for $(L, |\cd|)$. 
\end{defi}

The typical example of  a $p$-adic locally algebraic coadmissible  representation is 
$\varprojlim_{K_{p}} H_{i}(Y_{K^{p}K_{p}}, \cW)\ot W^{\vee}$, where the notation is as after Definition \ref{def admiss}.

\subsubsection{Notation} 
Consider the groups \eqref{list}. For a place $v\vert p$ of $F$, we let
 $$G_{v}:=\B_{v}^{\ts},\quad H_{v}:=E_{v}^{\ts},\quad H_{v}':=E_{v}^{\ts}/F_{v}^{\ts},\quad (G\ts H)'_{v}:= (G_{v}\ts H_{v})/F_{v}^{\ts}$$
 as topological groups. We use the parallel notation $G_{*, v, \infty}$ for $G_{*,v}$ viewed as the group of points of an algebraic group over $F_{v}$.

We assume from now on that $\B_{p}$ is split and fix an isomorphism $\G_{\Q_{p}}\cong \Res_{F_{p}/\Q_{p}}\GL_{2}$, giving a model of $\G_{*}$ over $\Z_{p}$. We define involutions 
$$g^{\iota}:= g^{{\rm T}, -1} \quad \text{on $\G(\Q_{p})$}, \qquad h^{\iota}:=h^{c, -1} \quad \text{ on $\H(\Q_{p})$},$$ that induce involutions $\iota$ on all our groups. The embedding $\H'\into (\G\ts \H)'$ is compatible with the involutions.

For $t\in T_{\G_{*}, p}$,  let $\Up_{t}:= K_{p, r}
t K_{p,r}\in \cH^{K_{p,r}}_{G_{*, p}}$ for any $r\geq 1$, and 
$$\Up_{t,p\infty}:=\Up_{t}\ot t_{\infty}^{}. $$
When $x\in F_{p}^{\ts}$, we abuse notation by writing $\Up_{x}= \Up_{\smalltwomat {x}{}{}1}$; we also write 
$$\Up_{p\infty}:= \Up_{\smalltwomat p{}{}1, p\infty}$$
 for short.

\subsubsection{Ordinary parts of admissible or coadmissible $G_{*}$-modules}  \lb{ssec ord}
Let $L$ be a finite extension of $\Q_{p}$. Let $\Pi=\Pi_{p}\ot W$ be a $p$-adic locally algebraic admissible representation of $G_{*}$
 Let us write 
$$\Pi^{N_{0, (r)}}:= \Pi^{N_{0, (r)}}\ot W^{N} ,$$ where $N_{0, r}:=K_{p, r}$. Choose  $\OO_{L}$-lattices $W^{\circ}\subset W$, $\Pi_{p}^{\circ, K}\subset \Pi_{p}^{K}$, stable under the Hecke action, and compatibly with    the transition maps associated with  $K'\subset K$.
 Then   $\Pi^{\circ, N_{0}}:=\Pi_{p}^{\circ, N_{0}} \ot W^{\circ, N}= \varinjlim_{r} \Pi_{p}^{\circ, K_{p,r}} \ot W^{\circ, N}$ 
is  stable under the action of $\Up_{p\infty}$.
As shown by Hida, the idempotent 
 $$e^{\ord}: = \lim_{n}\Up_{p\infty}^{n!} \colon \Pi^{\circ, N_{0}}\to \Pi^{\circ, N_{0}}$$ is then well-defined and its image is denoted by $\Pi^{\circ, \ord}$.  The space $\Pi^{\circ, \ord}$ is   the maximal split $\OO_{L}$-submodule of $\Pi^{\circ, N_{0}}$ over which $\Up_{p\infty}$ acts invertibly. We also write $e^{\ord}$ for $e^{\ord}\ot 1\colon \Pi^{N_{0}}=\Pi^{\circ, N_{0}}\ot L\to \Pi^{N_{0}}$, and we let $\Pi^{\ord}=e^{\ord}\Pi^{N_{0}}$ be its image. If $\Pi_{p}$ and $W$ are irreducible, then $\Pi^{\ord}$ has dimension either~$0$ or~$1$; in the latter case we say that $\Pi$ is \emph{ordinary}.  (This notion is independent of the choice of lattices.)
 
 Let ${\rm M}={\rm M}_{p}\ot W^{\vee}$ be a $p$-adic locally algebraic coadmissible  right module for $G_{*}$ over $L$. By definition of coadmissibility, the system $({\rm M}_{p,K})_{K\subset G_{*,p}}$ is endowed with a compatible system $\cH_{\G_{*}(\Z_{p}), K}$-stable  lattices ${\rm M}_{p,K}^{\circ}$, so that  for some $\G_{*}(\Z_{p})$-stable lattice $W^{\vee, \circ}$, 
 ${\rm M}^{\circ}_{N_{0}}:= \varprojlim {\rm M}^{\circ}_{p, K_{p, r}}\ot W^{\vee, \circ}_{N_{0}}$
is  stable under $\Up_{p\infty}$. Then we can again define $e^{\ord}\colon {\rm M}^{(\circ)}_{N_{0}} \to {\rm M}^{(\circ)}_{N_{0}}$. Its image 
$${\rm M}^{(\circ), \ord}:= {\rm M}^{(\circ)}_{N_{0}} e^{\ord}$$ is called the \emph{ordinary part} of ${\rm M}^{(\circ)}_{N_{0}} $.
 
The ordinary parts $\Pi^{\ord}$, ${\rm M}^{\ord}$ retain an action of the operators $\Up_{t, p\infty}$.

\subsubsection{Special group elements, and further notation} The following notation will be in use throughout this appendix. Let  $v\vert p$ be a place of $F$. We denote by    $e_{v}$ be the ramification degree of $E_{v}/F_{v}$, and fix a uniformiser   $\vpi_{v}\in F_{v}$  chosen so that $\prod_{v\vert p}\vpi_{v}^{e_{v}}=p$.   Let ${\rm Tr}_{v}=\Tr_{E_{v}/F_{v}}$ and ${\rm Nm}_{v}:={\rm Nm}_{E_{v}/F_{v}}$ be the trace and norm.
Fix an isomorphism $\OO_{E, v}=\OO_{F, v}\times \OO_{F, v}$ if $v$ is split. If $v$ is nonsplit,  let $c$ be the Galois conjugation of $E_{v}/F_{v}$, and  fix an element $\theta_{v}\in \OO_{E,v}$ such that $\OO_{E, v}=\OO_{F,v}[\theta_{v}]$ (thus $\theta_{v}$ is a unit if $v$ is inert and a uniformiser if $v$ is ramified). We define a purely imaginary ${\rm j}_{v}\in E_{v}^{\ts}$ to be 
\beq\lb{jv}
{\rm j}_{v}:= \begin{cases} (-1_{w},1_{w^{c}}) & \text{if $E_{v}=E_{w}^{\ts}\ts E_{w^{c}}^{\ts}$,} \\
\tht_{v}^{c}-\tht_{v} & \text{if $E_{v}$ is a field}.
\end{cases}
\eeq

We assume that $E_{v}$ embeds in $\B_{v}$ and fix the embedding $E_{v}\to \B_{v}$ to be 
\beqq
t=(t_{w},t_{w^{c}}) \mapsto  \twomat {t_{w}}{}{}{t_{w^{c}}} & &\text{if $E_{v}=E_{w}^{\ts} \ts E_{w^{c}}^{\ts}$}, \\
 t=a+\theta b\mapsto \twomat {a+b\Tr_{v}\theta_{v}} {b{\rm Nm}_{v}\theta_{v}} {-b} a& & \text{if $E_{v}$ is a field}.
\eeqq

For $r\geq 0$, let 
\beqq 
w_{r, v}&:=\twomat {}1{-p^{r}}{} \in \GL_{2}(F_{v}), \qquad
\gamma_{r,v} :=
\begin{cases}
	 {\twomat {p^{r}} 1{}{1}}   &\text{if $v$ splits}\\
	{\twomat {p^{r}{\rm Nm}_{v}(\theta_{v})} {}{}1 }  &\text{if $v$ is nonsplit}
\end{cases}
\in (G\ts H)_{v}'
\eeqq
and 
\beqq 
w_{r}&:=\prod_{v\vert p}w_{r, v}\in \G(\Q_{p}), \qquad \gamma_{r}:= \prod_{v\vert p}\gamma_{r,v}\in (\G\ts\H)'(\Q_{p}).
\eeqq
\subsection{Toric,  ordinary, and anti-ordinary parts} \lb{sec: A1} Let $L$ be a finite extension of $\Q_{p}$. We perform some twists. 
\subsubsection{Ordinary and anti-ordinary parts}
 Let $w:=\smalltwomat {}1{-1}{}\in G_{*,\Delta}$ and let  $\pi^{w}$ be the representation on the same space as $\pi$ but with $G$-action given by $\pi^{w}(g)v:=\pi(w^{-1}gw)v$.  Let $N^{-}:= w^{-1}Nw$, and $U_{p\infty}^{-}:= U_{w^{-1}\smalltwomat p {}{}1 w, p\infty}$. 
 
 Let $\pi=\pi_{p}\ot W$ be a $p$-adic admissible locally algebraic   representation of $G$ over $L$. The \emph{anti-ordinary part} of $\pi$ is  the space  $$\pi^{\rm a}:=\pi_{p}^{\rm a}\ot W^{N^{-}}\subset \pi$$  of `ordinary' elements  with respect to $N^{-}$ and $U_{p\infty}^{-}$. Because $\pi^{w}$ is isomorphic to $\pi$, the spaces $\pi^{\rm a}$ and  $\pi^{\ord}$ have the same dimension. 
 
 Let $M=M_{p}\ot W$ be a  $p$-adic coadmissible locally algebraic   representation of $G$ over $L$. The \emph{anti-ordinary part} of $M$ is the quotient  $$M^{\rm a}:= M_{p}^{\rm a}\ot W_{N^{-}}$$ of $M$ that is its `ordinary' part with respect to $N_{0}^{-}$ and and $U_{p\infty}^{-}$.

\begin{prop}\lb{w ord} Let $W$ be an algebraic representation of $G_{\infty}$. 
\begin{enumerate}
\item
Let $\pi$ be a $p$-adic locally algebraic admissible representation of $G$. There is an isomorphism 
\beqq
w^{\ord}_{\rm a}\colon \pi^{\ord}&\to \pi^{\rm a}\\
 f&\mapsto  \lim_{r\to \infty}  p^{r[F:\Q]} w_{r,p}w_{0, \infty}^{\iota}\Up_{p}^{-r}f,
\eeqq
where the sequence stabilises as soon as $r\geq 1 $ is such that $f_{p}\in \pi_{p}^{U_{1}^{1}(p^{r})}$.
\item 
Let $M$ be a  $p$-adic locally algebraic coadmissible representation of $G$. There is a map
\beqq
w^{\ord}_{\rm a}\colon M^{\ord}&\to M^{\rm a}\\
 m=m_{p}\ot m_{\infty}&\mapsto  \lim_{r\to \infty}  p^{r[F:\Q]} [m  (U_{p}^{})^{-r} w_{r,p}]_{N_{0}^{-}} \ot [m_{\infty}w_{0, \infty}^{\iota}]_{N^{-}},
 \eeqq
 where before applying $w_{*}$, we take arbitrary lifts from $N_{0}^{}$-coinvariants to the module $M$.
\end{enumerate}
\end{prop}
\begin{proof}
That the maps are well-defined is a standard result left to the reader. At least for admissible representations, the map  is an isomorphism (equivalently, nonzero) because of Lemma \ref{lemm: loc pair} below.
\end{proof}

Let $\pi_{v}^{\ord}$ (respectively $\pi_{v}^{\rm a}$) denote  the preimage of $\pi^{\ord}$   (respectively $\pi^{\rm a}$) in $\pi_{v}$, and let $W_{v}$ be the $G_{v, \infty}$-component of $W$. The following local components of the above map are similarly well-defined:
\beq\lb{waov}
w_{{\rm a}, v}^{\ord}\colon \pi_{v}^{\ord}&\to \pi_{v }^{\rm a},
&\qquad  w_{{\rm a}, v, \infty}^{\ord}\colon W_{v}^{N_{v}}&\to W_{v}^{N_{v}^{-}}\\
 f_{v}&\mapsto \lim_{r\to \infty} p^{r[F_{v}:\Q_{p}]} w_{r , v} \Up_{p,v}^{-r}f_{v}, 
 & 
 f_{v, \infty}& \mapsto w_{0,v, \infty}^{\iota}f_{v, \infty}.
 \eeq

\begin{lemm} \lb{cor w ord} Let $\pi$ be an ordinary representation of $G$. 
If $(\ , \ )\colon \pi \ot \pi^{\vee}\to L$ is a nondegenerate  $G$-invariant  pairing, then  the pairing 
\beqq
(\ , \ )^{\ord} \colon \pi^{\ord}\ot \pi^{\vee, \ord}&\to L \\
(f_{1}, f_{2})^{\ord}&:= (w_{\rm a}^{\ord} f_{1}, f_{2})\eeqq  is a nondegenerate pairing.
\end{lemm}
\begin{proof} It suffices to see this for a specific pairing $(\ , \ )$: we may take the product of the pairings \eqref{kir pair v} below, that are known to be nondegenerate,  and any nondegenerate pairing on $W\ot W^{\vee}$. Then the result follows from Lemmas \ref{lemm: loc pair} and  \ref{wao alg} below.
\end{proof}

\subsubsection{Ordinary and toric parts}\lb{ssec tord}
We  construct a  map from the ordinary part of a representation of $(G\ts H)'$ to its toric coinvariants, as  well as a dual map in the opposite direction for coadmissibe modules. These map are the key to the interpolation of toric periods. 

Suppose that   $W_{(v)}$ (respectively $W=\bigotimes_{v\vert p}W_{v}$ is  an algebraic representation of $(G\ts H)_{(v) , \infty}$ (respectively $(G\ts H)'_{\infty}$) over $L$  such that, for a field extension $L'/L$ splitting $E$,  $W_{(v), \infty}\ot_{L} L'=\bigotimes_{\sg\colon F_{(v)}\into L} W_{\sg}$ with 
\beq\lb{Wwkl}
W_{\sg}=W_{\sg, w, k, l}:= {\rm Sym}^{k_{\sg}-2}{\rm Std} \cdot {\det}^{w-k_{\sg}+2\over 2} \ot \sg^{{l_{\sg}-w}\over 2}{(\sg^{c})}^{-l_{\sg}-w\over 2}
\eeq
for some integers $k_{\sg }\geq 2, |l_{\sg}|<k_{\sg}, w$ of the same parity. (Here we have chosen, for each $\sg\colon F\into L$, an extension $\sg\colon E\into L$.) Then we define a constant
\beq\lb{cW}
c(W_{\sg}) &:=  {\rm j}_{v}^{-w-k_{\sg}+2 }\cdot {k_{\sg}-2 \choose (k_{\sg}-2-l_{\sg})/2} \cdot 
\begin{cases} 1& \text{if $v $ splits in $E$}\\ 
 \tht^{c,(k-2-l)/2} \tht^{(k-2+l)/2}  & \text{if $v$ does not split in $E$},
\end{cases}\\
   c(W_{(v)})&:=\prod_{\sg\colon F_{(v)}\into L} c(W_{\sg}).
\eeq
(Note that   ${\rm j}_{v}^{-w-k_{\sg}+2 }=1$ if $v$ splits in $E$, as $w+k_{\sg}-2$ is even.)

\begin{lemm}\lb{ex mat id} 
Recall the congruence subgroups $V_{v,r}'$, $K_{v,r}$ defined in \S~\ref{ssec cong}. For all $r\geq 1$, we  have the identity of Hecke operators in the Hecke algebra for $(G\ts H)'_{v}$:
$$ V_{v, r+1}'\left( \sum_{t\in V_{v,r}'/V_{v,r+1}'} t \right) \cdot
\gamma_{r+1,v}   K_{v,r}=   V_{v, r+1}' \gamma_{r,v}\cdot \Up_{\vpi_{v}}  K_{v, r}.$$
\end{lemm}
\begin{proof}
This is a consequence of   the following matrix identity.

Let $v\vert p$ be a prime of $F$.
For   $r\in \Z_{\geq 1}$, $j\in \OO_{F,v}$ 
 let 
  $b_{j,v}:= \twomat {\vpi} j {}1$. In the split case, let 
    $$ t_{j,r, v}= k_{j,r,v}:= \twomat {1+j\vpi^{r}}{}{}1\in E_{v}^{\ts}$$
 In the nonsplit case, let
 $$t_{j,r, v}=1+\theta_{v}\vpi_{v}^{r},
  \qquad 
  k_{j, r, v}=\twomat {1+j{\Tr}_{v}(\theta_{v}) \vpi_{v}^{r}} {{\rm Tr}_{v}(\theta_{v})-j\vpi_{v}^{r}}
  {-j{\rm N}_{v}(\theta_{v}) \vpi_{v}^{2r}} { 1+j^{2}{\rm N}_{v}(\theta_{v} )\vpi_{v}^{r}}.
  $$
   
    Then 
$$t_{j,r, v}\gamma_{r+1,v}=\gamma_{r,v} b_{j,v} k_{j,r, v}$$
in $\GL_{2}(F_{v})$.
\end{proof}

\begin{prop}\lb{gHo} Let $W$ be an algebraic representation of $(G\ts H)'_{\infty}$. 
\begin{enumerate}
\item Let $\Pi=\Pi_{p}\ot W$  be a $p$-adic locally algebraic  admissible representation  of $(G\ts H)'$. 
There is a map 
\begin{align}
\nonumber
\tord \colon \Pi^{\ord} &{\to}  \Pi_{H'}\\
\lb{H'N}
  f  &\mapsto \lim_{r} \ 
    {H'_{}}[ p^{r[F:\Q]} \cdot c(W)^{-1}\cdot \gamma_{r, p\infty}  \Up_{p\infty}^{-r} f]
 \end{align}
where
  ${H'_{}}[-]\colon \Pi\to \Pi_{H'}$ is the natural projection. 

 The sequence in the right hand side of \eqref{H'N} stabilises as soon as $f_{p}\in \Pi^{K_{p, r}}$, where $K_{p, r}\subset (G\ts H)_{p}'$ is  defined at the end of \S~\ref{sec: not}.
\item
Let   ${\rm M}:= {\rm M}_{p}\ot W^{\vee}$ be a $p$-adic locally algebraic  coadmissible representation of $(G\ts H)'$.
There  is a map 
\beqq
\tord \colon {\rm  M}^{H'}&\to {\rm M}^{\ord}\\
m  &\mapsto \lim_{r} 
\
[ p^{r[F:\Q]} \cdot  c(W)^{-1}\cdot  m
 \gamma_{r,p\infty}]
{N_{0,r}} e^{\ord} \Up_{p\infty}^{-r},
 \eeqq
 where $[ -]{N_{0,r}}\colon  {\rm M} \to {\rm M}_{N_{0}}$ is the natural projection.
\end{enumerate}
\end{prop}

The constant $c(W)$ is justified by Lemma \ref{unitary} below.

\begin{proof}
  For part 1, let  $f\in \Pi_{p}^{K_{p, r}}$. 
   Then it follows from Lemma \ref{ex mat id} that,  denoting by  $[f_{r}]_{H'}$
  the sequence in the right hand side of \eqref{H'N},   we have 
  $${1\over p^{[F:\Q]}} \sum_{t\in V_{p,r}'/V_{p,r+1}'} \Pi(t) f_{r+1}=f_{r},$$
 hence $[f_{r+1}-f_{r}]_{H'}=0 $ and the sequence stabilises. 
 
 For part 2,  Lemma \ref{ex mat id} similarly implies (the boundedness and)
  the convergence of the sequence in $\varprojlim_{r} {\rm M}_{N_{0, r}}^{\ord}$. 
\end{proof}

Let $\Pi_{v}^{\ord}$ denote  the preimage of $\Pi^{\ord}$ in $\Pi_{v}$, and let $W_{v}$ be the $(G\ts H)'_{v, \infty}$-component of $W$. 
The following  local components of the above maps are similarly well-defined:
\beq\lb{tord v}
\gamma_{H',v}^{\ord}\colon \Pi_{v}^{\ord}&\to \Pi_{v, H'_{v}}
&\qquad
\gamma_{H',v, \infty}^{\ord}\colon W_{v}^{N}&\to W_{v, H'_{v}}\\
 f_{v}&\mapsto \lim [p^{r[F_{v}:\Q_{p}]} \gamma_{r , v} \Up_{\vpi_{v}}^{-r}f_{v}]_{H'_{v}}, 
 & 
 f_{v, \infty}& \mapsto
   c(W_{v})^{-1}
   \cdot \gamma_{0,v, \infty}^{\iota}f_{v, \infty}. 
 \eeq

\subsubsection{Exceptional representations and vanishing of $P^{\ord}$}   We show that $\tord$ acts by zero precisely on those representations that are exceptional.
\begin{lemm} \lb{exc vanishing}
Let $\Pi=\pi\ot\chi$ be an ordinary, distinguished, irreducible representation of $(G\ts H)'$.  The following are equivalent:
\begin{enumerate}
\item $\Pi$ is exceptional; 
\item $e_{v}(V_{(\pi, \chi)}) = 0$;
\item there exists $P\in \Pi^{*, H'} -\{ 0\}$ such that $P^{\ord}:=P \tord=0$;
\item for all $P\in \Pi^{*, H'}$,  we have $P^{\ord}=0$;
\end{enumerate}
\end{lemm}
\begin{proof} The equivalence of 1. and 2. is  a reminder from  Lemma \ref{6444}. The equivalence of 3. and 4. is a consequence of multiplicity-one. Consider 3. Let $P\in \Pi^{*, H'}$. Identify  $\Pi^{\vee}=\Pi^{\iota}$ (the representation on the same space as $\Pi$, with group action twisted by the involution $\iota$). Then  the identity map on spaces yields  isomorphisms  $\Pi^{*, H'}\cong \Pi^{\vee, *, H'}$ and $\Pi^{\ord, *}\cong \Pi^{\vee, \ord, *}$, and it follows from the explicit description of $\tord$ that if $P^{\vee}$ denotes the image of $P$, then the image of $P^{\vee}$ is  $P^{\vee,\ord}$.  Hence,  $P^{\ord}$ is zero if and only if so is $P^{\vee, \ord}$, if and only if so is $P\ot P^{\vee}\circ \tord \ot \tord$. Now by the theory recalled in \S~\ref{ssec loc tp} , $P \ot P^{\vee}$ is necessarily a multiple of the explicit functional $Q_{dt, (, )}$ defined there. Therefore  it suffices to show that \emph{$Q_{dt, (, )}$ vanishes on the line $\tord \Pi^{\ord}\ot \tord \Pi^{\vee, \ord}$ if and only if $e_{v}(V_{(\pi, \chi)}) = 0$}. This follows from the explicit computations of Propositions \ref{toric period} and \ref{compare toric inf} below, cf. also Proposition \ref{compare Qs}.
\end{proof}

\subsection{Pairings at $p$} 
\lb{A3}
The goal of this subsection is to relate the $p$-components of the toric terms  $Q$ and their ordinary variants $Q^{\ord}$, as defined in \S\S~\ref{sec 42}-\ref{sec 43}.

 Let $v \vert p$ be a place of $F$. 
\subsubsection{Integrals and gamma factors}
If $\pi$  (respectively $\chi$) is an irreducible representation of $G_{v}$ over $L$, we denote by $V_{\pi}$ (respectively $V_{\chi})$ the associated  $2$- (respectively $1$-) dimensional Frobenius-semisimple representation of ${\rm WD}_{F_{v}}$ (respectively of  ${\rm WD}_{E_{v}}:=\prod_{w\vert v}{\rm WD}_{E_{w}}$; we choose the  ``Hecke'' normalisation, so that $\det V_{\pi}$ is the cyclotomic character if $\pi$ is self-dual. 
If $\Pi=\pi\ot \chi$ is an irreducible representation of $(G\ts H)'_{v}$, we denote by $V_{\Pi}= V_{\pi|{\rm WD}_{E_{v}}}\ot V_{\chi}$ the  associated  $2$-dimensional representation of ${\rm WD}_{E_{v}}$. 
If $E_{*} $ is $F$ or $E$, $w\vert p$ is a prime of $E_{*}$   and $V$ is any representation of ${\rm WD}_{E_{*,v}}$ as above, we let $V_{w}:= V_{|{\rm WD}_{E_{*, w}}}$.

 If $\psi\colon F_{v}\to \bC^{\ts}$ is a nontrivial character, we denote by $d_{\psi} y$ the selfdual Haar measure on $F_{v}$ and $d^{\ts}_{\psi}y:= d_{\psi}y/|y|$. The \emph{level} of $\psi$ is the largest $n$ such that $\psi_{|\vpi^{-n}\OO_{F, v}}=1$. We recall that if $\psi$ has level~0, then $\vol(\OO_{F, v},d_{\psi}y)=1$. 
 
 Recall the Deligne--Langlands $\gamma$-factor  of \eqref{gamma intro}.

\begin{lemm}[{\cite[Lemma A.1.1]{nonsplit}}] \lb{int gamma}
Let $\mu\colon F_{v}^{\ts}\to \bC^{\ts}$ and $\psi\colon F_{v}\to \bC^{\ts}$ be  characters, with $\psi_{v}\neq 1$. Let $d^{\ts}y$ be a Haar measure on $F_{v}^{\ts}$.
Then 
$$\int_{F_{v}^{\ts}} \mu(y) \psi(y) d^{\ts}y={d^{\ts}y \over d^{\ts}_{\psi}y}  \cdot \mu(-1)\cdot  \gamma(\mu, \psi)^{-1} .$$
\end{lemm}

\subsubsection{Local pairing}  The following isolates those representations that can be components of an ordinary representation.
\begin{defi}\lb{def-fs}
A \emph{refined} representation $(\pi, \alpha)$ of $G_{v}$ over a field $L$ consists of  a smooth irreducible infinite-dimensional representation $\pi$ and   a character $\alpha\colon F_{v}\to L^{\ts}$, such that  $\pi$ embeds into  the un-normalised induction ${\rm Ind}( |\ |\alpha , |\ |^{-1}\omega\alpha^{-1}))$ for some  other  character $\omega\colon F_{v}^{\ts}\to L^{\ts}$.\footnote{Note that $\pi$ admits a refinement  if and only if it is neither supercuspidal nor $1$-dimensional.} Sometimes we abusively simply write $\pi$ instead of $(\pi, \alpha)$. A refined   representation $\Pi=\pi\ot \chi$ $(G\ts H)'_{v}$ is the product of a refined representation $\pi=(\pi, \alpha)$ of $G$ and a character $\chi$ of $H$, such that $\omega\chi_{|F_{v}^{\ts}}=\one$. 
\end{defi}
If $(\pi, \alpha)$ is a refined representation of $G_{v}$, we let $\pi^{\ord}\subset \pi^{N_{0}}$ be the unique line on which the operator $\Up_{t}$ acts by $\alpha(t)$. If $\Pi=\pi\ot\chi$  is a refined representation of $(G\ts H)'_{v}$, we let $\Pi^{\ord}:=\pi^{\ord}\ot\chi$. The  associated Weil--Deligne representation $V_{\pi}$ is reducible, and we have a unique filtration 
$$0\to V_{\pi}^{+}\to V_{\pi} \to V_{\pi}^{-}\to 0$$
such that ${\rm WD}_{F_{v}}$ acts on $V_{\pi}^{+}$ through the character $\alpha|\cdot|$.

Let $\pi$ be a refined  representation of $G_{v}$ over $L$, and let $(\ , \  )_{\pi}\colon \pi\ot \pi^{\vee}\to L$  be a  $G$-invariant   pairing. Then we define 
\beqq
(\ , \  )_{\pi}^{\ord}\colon \pi^{\ord}\ot (\pi^{\vee})^{\ord} & \to L \\
f\ot f^{\vee} &\mapsto  (w^{\ord}_{\rm a}f, f^{\vee}),
\eeqq
where $w^{\ord}_{\rm a} $ is the operator denoted  $w_{{\rm a}, v}^{\ord}$ in \eqref{waov}. If $\Pi$ is a refined representation of $(G\ts H)_{v}'$ over $L$ and $(\ , \ ) =(\ , \ )_{\pi}(\ ,\ )_{\chi}\colon \Pi\ot \Pi^{\vee} \to L$ is a pairing, we define  $(\ , \ )^{\ord}:= (\ , \  )_{\pi}^{\ord} (\ ,\ )_{\chi}$, a pairing on $\Pi^{\ord}\ot \Pi^{\vee, \ord}$.

\begin{lemm} \lb{lemm: loc pair} Let $(\pi, \alpha)$ be a refined representation of $G_{p}$ over $\bC$, with central character  $\omega$ as in Definition \ref{def-fs}. Let $\alpha^{\vee}=\alpha\omega^{-1}$.
 Let
$$\ad(V_{\pi})^{++}(1)= \Hom(V_{\pi}^{-}, V_{\pi}^{+})(1).$$
Fix  a   character $\psi\colon F_{v}\to \bC^{\ts}$ of level~$0$, and Kirillov models of $\iota\pi_{v}$, $\pi_{v}^{\vee}$ with respect to $\psi_{v}$, $-\psi_{v}$.
Let
\beq\lb{f kir}
f_{v}^{(\vee)}(y):= \one_{\OO_{F, v}}(y) \alpha_{v}^{(\vee)} |\ |_{v}(y) \qquad \in \pi^{\ord}_{v}. 
\eeq
Suppose that $( \ , \ )_{\pi,{v}}$ is, in the Kirillov models, the pairing 
\beq
\lb{kir pair v}
(f, f^{\vee})_{\pi}:= \int_{F^{\ts}}  f(y)f^{\vee}(y) d^{\ts}_{\psi}y .\eeq
Then 
$$(f, f^{\vee})^{\ord}_{\pi,v}= \omega_{v}(-1)\cdot \gamma(\ad(V_{\pi})^{++}(1), \psi)^{-1} .$$
\end{lemm}
\begin{proof}
We omit all remaining subscripts $v$ and argue    similarly to \cite[Lemma 2.8]{hsieh3}.  The inner product
$(f, f^{\vee})^{\ord}_{\pi}$ is the value at $s=0$ of 
$$\qquad \alpha|\ |(\vpi)^{-r}Z(s+1/2,w_{r}f, \alpha^{\vee}|\  |), 
\qquad Z(s+1/2, w_{r}f, \alpha^{\vee}|\ | ):=   \int_{F^{\ts}} w_{r}f(y)\alpha^{\vee}|\ |(y) |y|^{s}d^{\ts}_{\psi}y.$$
  By  the functional equation for $\GL_{2}$, this equals 
\beqq
&\omega(-1)\cdot \gamma(s+1/2,\pi\ot \alpha^{\vee}|\ |, \psi)^{-1}  \cdot 
\int_{p^{-r}\OO_{F}-\{0\}} \alpha\alpha^{\vee, -1}\omega^{-1}|\ |^{-s}(y) d^{\ts}_{\psi}y\\
= \ & \omega(-1)\cdot  \gamma(s,\alpha \alpha^{\vee}|\ |^{2}, \psi)^{-1} \cdot \gamma(s, |\ |, \psi)^{-1}   
\cdot \zeta_{F}(1)^{-1}
\zeta_{F}(-s),
\eeqq
using the fact that the domain of integration can be replaced with $F^{\ts}$, the additivity of gamma factors, and the relation $\alpha^{\vee}=\alpha\omega^{-1}$. 
Evaluating at $s=0$ we find 
$\gamma(\ad(V_{\pi})^{++}(1), \psi)^{-1}$ as desired.
\end{proof}

\subsubsection{Local toric period} 
\lb{A3mu}
We compute the value of the local toric periods on the lines of interest to us. Let $\Pi=\pi\ot\chi$ be a refined representation of $(G\ts H)'_{v}$.
 Let $dt$ be a measure on $H'_{v}$, and set as in \eqref{vol circ}
$$ {\vol^{\circ}(H'_{v}, dt_{})} :=
 {\vol(\OO_{E, v}^{\ts}/\OO_{F, v}^{\ts}, dt_{})\over e_{v} 
  L(1, \eta_{v})^{-1}}.$$

Then for all $f_{1}, f_{3}\in \Pi^{\ord}$, $f_{2},f_{4}\in \Pi^{\vee, \ord}$ with $f_{3}, f_{4}\neq 0$, we define
\beq \lb{Q orddd v}
Q_{dt}^{\ord}\left( {f_{1}\ot f_{2} \over f_{3}\ot f_{4}}\right)   : =  \mu^{+}({\rm j}_{v})^{}\cdot  {\vol^{\circ}(H'_{v}, dt_{})} \cdot {f_{1}\ot f_{2} \over f_{3}\ot f_{4}},
\eeq 
 where  ${\rm j}_{v}=\eqref{jv}$ and $\mu^{+}=\chi_{v}\cdot \alpha|\cdot|\circ N_{E_{v}/F_{v}}$ is the character giving the action of  $E_{v}^{\ts}$ on $V^{+}:=V_{\pi}^{+}\ot \chi$.

\begin{prop} \lb{toric period} Let $\Pi=\pi\ot\chi$ be a refined   representation of $(G\ts H)'_{v}$ over $L$, with associated Weil--Deligne representation $V=V_{\pi|{\rm WD}_{E_{v}}}\ot \chi$.  Let ${\tord}={\tord}_{,v}$ be as defined in \eqref{tord v}. Then  
$$ 
 Q_{dt}^{}\left({ \tord f_{1}\ot \tord f_{2} \over w_{\rm a}^{\ord} f_{3}\ot f_{4}}\right) =  e_{v}(V_{(\pi, \chi)}) \cdot  
 Q_{dt}^{\ord}\left( {f_{1}\ot f_{2} \over f_{3}\ot f_{4}}\right) . $$
Here $$  e_{v}(V_{(\pi, \chi)}) = \calL(V_{(\pi, \chi)}, 0)^{-1}\cdot
 \iota^{-1}\left( |d|_{v}^{-1/2} \gamma({\rm ad}(\iota V_{\pi}^{++})(1),  \psi_{v})\cdot \prod_{w\vert v} \gamma(\iota V^{+}_{|{\rm WD}_{E_{w}}}, \psi_{E_{w}})^{-1}\right)$$
is defined independently of  any choice of  an embedding $\iota\colon L\into \bC$ and nontrivial character $\psi\colon F_{v}\to \bC^{\ts}$.
\end{prop}
\begin{proof}
Identify $\chi^{\pm 1}$ with $L$ and assume that $f_{i}=f_{i, \pi}f_{i, \chi}$ with $f_{i, \chi}$ identified with $1$. Fix $\iota\colon L\into \bC$ (omitted from the notation) and $\one\neq \psi\colon F_{v}\to \bC^{\ts}$. Identify $\pi$, $\pi^{\vee}$ with Kirillov models with respect to $\psi$, $-\psi$. 
 Let $(\ , \ )=(\ ,  \ )_{\pi}\cdot(\, \ )_{\chi}$ be the invariant pairing on $\Pi \ot \Pi^{\vee}$
such that     $( \ , \ )_{\pi} = \eqref{kir pair} $ and $(1 \ ,1  \ )_{\chi}=1$.  Assume, after a harmless extension of scalars, that $dt=|D_{v}|^{-1/2}
d^{\ts}_{\psi_{E}}z/d^{\ts}_{\psi}y$, which gives  $\vol^{\circ}(H', dt)=1$.  Let $f_{1}=f_{3}=f_{\pi}$, $f_{2}=f_{4}=f_{\pi}^{\vee}$ with $f_{\pi}^{(\vee)}$ as in \eqref{f kir}. 

In view of the definitions  \eqref{Qv sec4}, \eqref{Q orddd v} and of Lemma \ref{lemm: loc pair}, it suffices to show that 
$$Q^{\sharp}(\tord f, \tord f^{\vee}) := \int_{H'_{v}}(\pi(t)\tord f,  \tord f^{\vee}) \chi(t) \, dt
 =  \omega_{}(-1)\cd \mu^{+}({\rm j}_{v}) \cdot  
 \prod_{w\vert v} \gamma(V^{+}_{|{\rm WD}_{E_{w}}}, \psi_{E_{w}})^{-1}.$$

We denote by $\alpha$ the refinement of $\pi$, and we  fix  $r\geq 1$ to be larger than the valuations of the  conductors of $\pi$ and  of the norm of the conductor of $\chi$. 
\paragraph{Split case} Suppose first that $E_{v}/F_{v}$ is split and identify $E_{v}^{\ts}=F_{v}^{\ts}\ts F_{v}^{\ts}$ as usual. 
 Then as in  \cite[Lemma 10.12]{dd-pyzz} we find
\beqq 
Q^{\sharp}(\tord f, \tord f^{\vee}) &= 
\prod_{w\vert v}
\int_{E_{w}^{\ts}}\alpha\chi_{w}|\ |_{w}(y_{w}) \psi_{w}(y_{w}) d^{\ts}y_{w}
\int_{E_{w^{c}}^{\ts}}\alpha_{}\chi_{w^{c}}|\ |_{w^{c}}(y_{w^{c}}) \psi_{w^{c}}(-y_{w^{c}}) d^{\ts}y_{w^{c}}\\
&=
\omega_{v}(-1)\cdot 
\mu^{+}_{}({\rm j}_{v})
 \cdot \gamma(V^{+}_{v}, \psi_{v})^{-1},
 \eeqq
 where we have used Lemma \ref{int gamma}.

\paragraph{Nonsplit case} Now suppose that $E_{v}=E_{w}$ is a field and drop all subscripts $v$, $w$.
We   abbreviate ${\rm T}:=\Tr(\theta)$, $\Nm:={\rm Nm}(\theta)$.

 We have
\beq\lb{nsplitQv} Q^{\sharp} (\tord f,\tord f^{\vee})
 = 
  \int_{H'}
    \alpha\alpha^{\vee}|\ |^{2}(\vpi)^{-r}
  \cdot
  (\pi(\gamma_{r}^{-1}t \gamma_{r})f_{\pi} , f_{\pi}^{\vee})\chi(t) \, dt.
\eeq
 There is a decomposition 
\beqq
H'=H'_{1}\sqcup H'_{2}, \qquad H'_{1}=\{  1+b\theta  \ |\  b\in \OO_{F}\} ,
\quad H'_{2}=
\{a{\rm N} +\theta\  |\  a\in {\rm N}^{-1}\vpi\OO_{F}\}, 
\eeqq
  that is an isometry when  $H'_{1}$, $H_{2}'$ are endowed with the measures $d_{\psi}b$, $d_{\psi}a$.
  
  Let $r':= r+e-1$ and let us redefine, for the purposes of this proof, $w_{r'}:= \smalltwomat {}1{-\Nm^{-1}\vpi^{-r}}{}$.
 Let  $\sim_{r'} $ denote the relation in $\GL(2, F)$ of equality up to right multiplication by an element of $U_{1}^{1}(\vpi^{r'})$, 
 and let $t^{(r)}:= \gamma_{r}^{-1}t\gamma_{r}$. 

\paragraph{Contribution from  $H_{1}'$.} For $t=1+b\theta\in H_{1}'$,  we have
\beqq
t^{(r)}=
 \twomat {1+b{\rm T}} {b\vpi^{-r}}
{-b\Nm\vpi^{r}}   {1} \sim_{r'}
 \twomat {1+b{\rm T}+b^{2}\Nm}  {b\vpi^{-r}} {}1.
\eeqq
Hence the integral over $H_{1}'$ equals 
\beqq
\omega^{-1}\alpha^{2}|\ |^{2}(\vpi)^{-r}
& \int_{\OO_{F}} \int_{\OO_{F}-\{0\}}
\psi(by\vpi^{-r}) \alpha|\ |({\rm Nm}(1+b\theta)y) \alpha\omega^{-1}|\ |(y) \chi(1+b\theta)
\,d^{\ts}_{\psi}y
\, d_{\psi}b
\\
= &
\int_{\OO_{F}}
\int_{ \vpi^{-r}\OO_{F}-\{0\}} 
\chi\cdot \alpha |\ |\circ {\rm Nm} ((1+b\tht)y)\cdot \psi(by)
\,
d_{\psi}^{\ts}y
\, d_{\psi}b.
\eeqq
 
 We show that the domain of integration in $y$ can be harmlessly extended to $F^{\ts}$, i.e. that
$$\int_{\OO_{F}} \int_{v(y)\leq  -r-1} \mu^{+}((1+b\tht) y)\psi(by) \, d^{\ts}_{\psi}y \, db$$
vanishes. Consider first the contribution from $v(b)\geq r$.  On this domain, $\mu^{+}(1+b\tht)=1$ and integration in $db$ yields $\int_{\vpi^{r}\OO_{F}}\psi(by)\, db= \one_{\vpi^{-r}\OO_{F}}(y)$, that vanishes on $v(y)\leq -r-1$. Consider now the contribution from $v(b)\leq r-1$
\beq \lb{rerb}
\int_{0\leq v(b) \leq r-1}  \mu^{+}(1+b\tht) \int_{v(y)\leq -r-1} \mu^{+}(y)\psi(by) \, d^{\ts}_{\psi}y db.\eeq
Let $n$ be the conductor of $\mu^{+}_{|F^{\ts}}$.  Then \eqref{rerb} vanishes if $n=0$; otherwise only the annulus $v(y)=-n-1$ contributes, and after a change of variable $y'=by$ we obtain 
$$\eps(\mu^{+}_{|F^{\ts}}, \psi)^{-1}\cdot \int_{r- n \leq v(b) \leq r-1}   \mu^{+}(1+b\theta)\mu^{+}(b)^{-1} db.$$
On  our domain $\mu^{+}(1+b\theta)=1$, and $\int \mu^{+}(b)^{-1}=0$ as $\mu^{+}_{|F^{\ts}} $ is ramified.

We conclude that the contribution from $H_{1}'$ is 
$$
\int_{\OO_{F}} \int_{F^{\ts}} \mu^{+}((1+b\tht) y)\psi(by) \, d^{\ts}_{\psi}y \, db
\int_{H_{1}'}
\int_{ F^{\ts}}
\mu^{+}(ty)\cdot \psi_{E}(ty/(\tht-\tht^{c}))
\, d_{\psi}^{\ts}y
\, dt.
$$

\paragraph{Contribution from  $H_{2}'$.}  For $t=a{\rm N}+\tht \in H_{2}'$, we have
\beqq
t^{(r)} 
&=
 \twomat {a\Nm+{\rm T}} {\vpi^{-r}}
{-\Nm\vpi^{r}}   {a\Nm} 
= 
w_{r}'
\twomat 
1 {-a\vpi^{-r}}
{a\Nm +{\rm T}} {\vpi^{-r}}
\sim_{r}
w_{r'}
\twomat 
{1+a{\rm T}+a^{2}\Nm } {-a\vpi^{-r}}
{} {\vpi^{-r}} .
\eeqq

Then the integral over $H_{2}'$ is
\beqq
\ &\omega^{-1}\alpha^{2}|\ |^{2}(\vpi)^{-r}
\int_{\Nm^{-1}\vpi\OO_{F}} 
 \left( \pi(\left( \smalltwomat {{\rm Nm}(1+a\tht)} {-a\vpi^{-r}}
{} {\vpi^{-r}} \right) f, 
\pi^{\vee}(w_{r'}^{-1}f^{\vee} \right) \cdot\chi(a\Nm +\tht)
\,d_{\psi }a\\
=&
\omega^{-1}\alpha^{2}|\ |^{2}(\vpi)^{-r}
\int_{\Nm^{-1}\vpi\OO_{F}} 
\int_{\OO_{F}-\{0\}}
\omega(\vpi)^{-r} 
\psi(-ay)
\alpha|\ |(y\vpi^{r}{\rm Nm}(1+a\tht))
\cdot\pi^{\vee}(w_{r'}^{-1})f^{\vee}(y) 
\cdot
 \chi(a\Nm +\tht)
 \,d_{\psi}^{\ts}y\,
\,d_{\psi }a\\
=&
\alpha|\  |(\vpi)^{-r}
\int_{\Nm^{-1}\vpi\OO_{F}} 
\int_{F^{\ts}}
\alpha|\ |(y{\rm Nm}(1+a\tht))
\cdot\pi^{\vee}(w_{r'}^{-1})f^{\vee}(y)
\cdot \chi(a\Nm +\tht)
 \,d_{\psi}^{\ts}y
\,d_{\psi }a\\
=&
\alpha|\  |(\vpi^{-r}{\rm N}^{-1})  \cdot 
Z(1/2,\pi^{\vee}( w_{r'}^{-1})f^{\vee}, \al|\ |)
\cdot
\int_{\Nm^{-1}\vpi\OO_{F}} 
\alpha|\  |({\rm Nm}(a{\rm N}+ \tht))
\cdot \chi(a\Nm +\tht)
\, d_{\psi }a,
\eeqq 
where we have observed that $w_{r'}f^{\vee}$ vanishes outside $\OO_{F}$, and that $\psi(-ay)=1$ for $y\in \OO_{F}$. 
Applying first the same argument as in the proof of Lemma \ref{lemm: loc pair},  then   Lemma \ref{int gamma}, this equals
\beqq
& \gamma(\ad(V_{\pi})^{++}(1),- \psi)^{-1}
 \cdot   \int_{\Nm^{-1}\vpi\OO_{F}} \al|\  | \circ {\rm Nm} \cdot \chi(a\Nm +\tht) d_{\psi}a  \\
 = & \int_{F^{\ts}}\mu^{+}(y) \psi(y) \, d^{\ts}_{\psi}y 
  \cdot   \int_{\Nm^{-1}\vpi\OO_{F}} 
\mu^{+}(a\Nm +\tht) d_{\psi}a  
= \int_{H_{2}'} \int_{F^{\ts}}
\mu^{+}(ty) \psi_{E}(ty/(\tht-\tht^{c})) \, d_{\psi}^{\ts}y.
\eeqq

\paragraph{Conclusion} Summing up the two contributions to \eqref{nsplitQv} yields
$$\mu^{+}(\tht^{c}-\tht) \cdot \int_{H'} \int_{F^{\ts}}  \mu^{+}(u) \psi_{E}(u) \, d^{\ts}u 
= \omega_{}(-1)\cdot  \mu^{+}({\rm j})\cdot    \gamma(\mu^{+}, \psi_{E})^{-1},
$$
as desired.
 \end{proof}

\subsection{Pairings at infinity}\lb{A4}
 Fix a place $v\vert p$ of $F$. 
\subsubsection{Models for algebraic representations and pairings} Suppose  that $W$ is the  representation \eqref{Wwkl} of $(G\ts H)_{v, \infty}'$ over $L\stackrel{\sg}{\hookleftarrow} E$. We identify $W$ with the space of homogeneous  polynomials $p(x, y)$ of degree $k-2$ in $L[x, y]$, where $x$ and  $y$ are considered as the components of a  column (respectively  row) vector if $W$ is viewed as a right (respectively left) representation. In those two cases, the action is respectively:
\beq \lb{W act poly}
p. (g, h)(x, y)&= \det(g)^{w-k+2\over 2}\sg(h)^{l-w\over 2}\sg^{c}(h)^{-l-w \over 2}\cdot  p(g(x, y)^{\rm T}) \\
(g, h).p(x, y)& = \det(g)^{w-k+2\over 2}\sg(h)^{l-w\over 2}\sg^{c}(h)^{-l-w \over 2}\cdot  p((x, y)g) .
\eeq
In either case, we   fix the invariant pairing 
\beq\lb{inv pair W} (x^{k-2-a}y^{a}, x^{a'}y^{k-2-a'}) = (-1)^{a} {k-2 \choose a}^{-1}\delta_{a, a'}.\eeq

\begin{lemm} \lb{wao alg} Let $W=\eqref{Wwkl}$, viewed as a left representation  of $G_{v, \infty}$ only.  Let $w_{\rm a}^{\ord }\colon W^{N}\to W_{N}$ be the map denoted by $w_{{\rm a}, v, \infty}^{\ord}$ of \eqref{waov}. 
Fix the models and pairing described above. Then $W^{N}$ is spanned by $x^{k-2}$ and $W_{N}$ is spanned by the image of $y^{k-2}$, and 
$$(w_{\rm a}^{\ord } (x^{k-2}), x^{k-2})=1.$$
\end{lemm}

\subsubsection{The map $\tord $ is unitary on algebraic representations} We start with  a lemma completing the proof of Proposition \ref{indep of wt}.

Suppose that  $M_{p}=M_{p, 0}\ot W_{p}$ is a decomposition of a locally algebraic coadmissible right $(G\ts H)'_{p}$-representation over $L$, into the product of a smooth and an irreducible algebraic representation, respectively. Let $W^{\vee}_{\infty}$ be the dual representation to $W_{p}$, viewed as a right  representation of  $(G\ts H)'_{\infty}$. Assume that $L$ is a $p$-adic field and that  the $(G\ts H)'$-module  $M_{p}\ot W^{\vee}_{\infty}$ is $p$-adic coadmissible. Then the operator $\tord$ on it   (whose definition of Proposition \ref{gHo} extends \emph{verbatim} to the case where $M_{p}$ is only locally algebraic) decomposes as 
$$\tord= \lim_{r\to \infty} (p^{r[F:\Q]}\cdot\gamma_{r, p}\Up_{p}^{-r})\ot \gamma_{r, p}\Up_{p}^{-r}\ot c(W)^{-1} \gamma_{0, \infty}^{\iota}.$$
according to the decomposition $M_{p}\ot W^{\vee}_{\infty}=M_{p, 0}\ot W_{p}\ot W^{\vee}_{\infty}$

\begin{lemm}\lb{unitary}  In relation to the situation just described, the operator 
$${}^{\rm alg}\tord:=  \lim_{r\to \infty} \gamma_{r, p}\Up_{p}^{-r}\ot  c(W)^{-1}\gamma_{0, \infty}^{\iota}\colon W^{H'}\ot W^{\vee, H'} \to W^{N}\ot W_{N}$$
is unitary.  That is, for any invariant pairing $(\ , \ )$ on $W\ot W^{\vee}$ and $\xi\in W^{H'}$, $\xi^{\vee}\in W^{\vee, H'}$, the images of $\xi\ot \xi^{\vee}$ and  ${}^{\rm alg}\tord (\xi\ot \xi^{\vee})$  under the pairings induced by $(\ , \  )$ coincide. 
\end{lemm}

\begin{proof}
We may fix a place $v\vert p$, and consider the factor representations $W_{v}\ot W_{v,\infty}^{\vee}$  of $(G\ts H)'_{v} \ts(G\ts H)'_{v, \infty}$.  After extension of scalars, we may decompose $W_{v}=\bigotimes_{\sg\colon F\to \baar{Q}_{p}}W_{v}^{\sg}$ where each $W_{v}^{\sg}$ is one of the representations \eqref{Wwkl} for suitable integers $w, k, l$. Thus we are reduced to proving the unitarity of the relevant component of ${}^{\rm alg}\tord$ on the  representation $W_{v}^{\sg}\ot W_{v, \infty}^{\vee, \sg}$. We omit all subscripts.

\paragraph{Split case}
Suppose first that $v$ splits in $E$.  Then $W^{H'}= Lx^{(k-2-l)/2} y^{(k-2+l)/2}$, and if 
$$\xi:= x^{(k-2-l)/2} y^{(k-2+l)/2} $$ then $$\xi^{\vee}:=   (-1)^{(k-2+l)/2} {k-2 \choose  {(k-2-l)/2}} x^{(k-2+l)/2}y^{ (k-2-l)/2}$$ satisfies $( \xi, \xi^{\vee})=1$. 
We have 
$$ \xi {\tord}_{, p}:=\lim_{r\to \infty} \xi \gamma_{r, p}\Up_{p}^{-r}  = y^{k-2},  $$
and
\beqq
\xi^{\vee}{\tord}_{,\infty} 
= (-1)^{(k-2+l)/2}   c(W)^{-1} {k-2 \choose  {(k-2-l)/2}}  x^{(k-2+l)/2}(-x+y)^{ (k-2-l)/2}
\eeqq
projects into $W^{\vee}_{N} \stackrel{\cong}{\leftarrow} Lx^{k-2}$ to 
$$\xi^{\vee}{\tord}_{,\infty} =   (-1)^{k-2}   c(W)^{-1}   {k-2 \choose  {(k-2-l)/2}}  x^{k-2}.$$
Hence 
$$(  \xi {\tord}_{, p}, \xi^{\vee}{\tord}_{,\infty})=   c(W)^{-1}   {k-2 \choose  {(k-2-l)/2}} =1.$$

\paragraph{Nonsplit case}
Suppose now that $v$ does not split in $E$. Let $z:= x+\tht^{c} y$, $\baar{z}:=x+\tht y$. Then $W^{H'} = L z^{(k-2-l)/2}\baar{z}^{ (k-2+l)/2}$, and if 
$$ \xi :=z^{(k-2-l)/2}\baar{z}^{ (k-2+l)/2}=  x^{(k-2-l)/2}y^{ (k-2+l)/2} .\twomat 1{\tht^{c}} 1{\tht}  (-{\rm j})^{(w+k-2)/2} \in W^{H'}$$
then 
$$\xi^{\vee}:=  (-1)^{(k-2+l)/2} {k-2 \choose  {(k-2-l)/2}} x^{(k-2+l)/2}y^{ (k-2-l)/2}  .\twomat 1{\tht^{c}} 1{\tht}  (-{\rm j})^{(-w-k+2)/2} \in W^{\vee, H'}$$
satisfies $( \xi, \xi^{\vee})=1$. 
We have 
$$ \xi {\tord}_{, p}= {\rm N}^{(w-k+2)/2}   \tht^{c,(k-2-l)/2} \tht^{(k-2+l)/2} y^{k-2}  ,$$
and  
\beqq
\xi^{\vee}{\tord}_{,\infty} 
&=  (-1)^{(k-2-l)/2}  (-{\rm j})^{-(w+k-2)/2} c(W)^{-1}
 {k-2 \choose  {(k-2-l)/2}} x^{(k-2+l)/2}y^{ (k-2-l)/2}  \twomat 1{\Nm^{-1}\tht^{c}} 1{\Nm^{-1}\tht} \eeqq
projects into $W^{\vee}_{N} \stackrel{\cong}{\leftarrow} Lx^{k-2}$ to 
$$\xi^{\vee}{\tord}_{,\infty} =  (-1)^{(k-2-l)/2}   (-{\rm j})^{-w-k+2}  c(W)^{-1} {k-2 \choose  {(k-2-l)/2}} \Nm^{(w+k-2)/2}
x^{k-2}.
$$
Then 
$$(  \xi {\tord}_{, p}, \xi^{\vee}{\tord}_{,\infty})=   (-{\rm j})^{-w-k+2 }  \tht^{c,(k-2-l)/2} \tht^{(k-2+l)/2} {k-2 \choose  {(k-2-l)/2}} c(W)^{-1}=1 .$$
\end{proof}

\subsubsection{Algebraic  toric period}\lb{A4mu}
Let  $W =W_{G} \ot W_{H}$ be an  algebraic representation of $(G\ts H)'_{v,\infty}$ over $L$.
For any $\iota\colon L\into \bC$, let $\iota V_{ }$ (respectively $\iota V_{G}$) be the Hodge structure associated with $W$ (respectively $W_{G}$),
 and let\footnote{To compare with \eqref{calLv}, we have $\zeta_{\R}(2)/L(1, \eta_{\bC/\R})=1.$}
$$\calL(V_{(W_{G}, W_{H})}, 0):= \iota^{-1}\left({\pi^{-[F_{v}:\Q_{p}]}L(\iota V_{}, 0) \over L(\ad(\iota V_{\G), \infty}), 1)}\right).$$ 

Let $dt$ be a `measure' on $H'_{v, \infty}$, by which we simply mean a value $\vol(H'_{v, \infty }, dt)$ similarly to \S~\ref{ssec loc tp}, and set as in \eqref{vol circ}
$$ {\vol^{\circ}(H'_{v}, dt_{})} :=2^{-[F_{v}:\Q]} \vol(H'_{v, \infty }, dt).$$
 Let $(\ , \ )=(\ ,  \ )_{W_{G}}\cdot(\, \ )_{W_{H}}$ be a nondegenerate invariant pairing on $W. \ot W^{\vee}$. 

Then for all $f_{1}, f_{3}\in W$, $f_{2},f_{4}\in W^{\vee}$ with $(f_{3}, f_{4})\neq 0$, we define
\beq \lb{Q inf}
Q_{dt}^{}\left( {f_{1}\ot f_{2} \over f_{3}\ot f_{4}}\right)   : =  \calL(V_{(W_{G}, W_{H})}, 0)^{-1}\cdot \vol^{}(H'_{v, \infty}, dt) \cdot {( {\rm p}_{H'} (f_{1}), {\rm p}_{H'} (f_{2})) \over (f_{3}, f_{4})},
\eeq
where ${\rm p}_{H'}$ denotes the idempotent projector onto $H'_{v, \infty}$-invariants. 

Let $\sg_{W_{G}}\colon F_{v}^{\ts}\to L^{\ts}$ be the character giving the action of $\twomat {F_{v}^{\ts}}{}{}1$ on $W_{G}^{N}$, let $\chi\colon E_{v}^{\ts}\to L^{\ts} $ be the algebraic character attached to $W_{H}$, and let 
$$\mu^{+}=\chi\cdot \sg_{W_{G}}\circ N_{E_{v}/F_{v}}.$$
  Let ${\rm j}_{v}=\eqref{jv}$.
 Then for all $f_{1}, f_{3}\in W^{N}:=W_{G}^{N}\ot W_{H}$, $f_{2}, f_{4}\in W^{\vee , N}$ with $f_{3}, f_{4}\neq 0$, we define
\beq \lb{Q ord inf v}
Q_{dt}^{\ord}\left( {f_{1}\ot f_{2} \over f_{3}\ot f_{4}}\right) :=
\mu^{+}({\rm j}_{v})^{}\cdot  {\vol^{\circ}(H'_{v}, dt_{})} \cdot {f_{1}\ot f_{2} \over f_{3}\ot f_{4}}.
\eeq

\begin{prop} \lb{compare toric inf} 
Let $W$  be a    representation of $(G\ts H)'_{v, \infty}$ over $L$.
Let ${\tord}={\tord}_{,v, \infty}$ be as defined in \eqref{tord v}, and let $w_{\rm a}^{\ord}=w_{{\rm a}, v, \infty}^{\ord}$ be as defined in \eqref{waov}. Then  for all $f_{1}, f_{3}\in W^{N}$, $f_{2}, f_{4}\in W^{\vee , N}$ with $f_{3}, f_{4}\neq 0$, 
$$ 
 Q_{dt}^{}\left({ \tord f_{1}\ot \tord f_{2} \over w_{\rm a}^{\ord} f_{3}\ot f_{4}}\right) =  \dim W \cdot   
 Q_{dt}^{\ord}\left( {f_{1}\ot f_{2} \over f_{3}\ot f_{4}}\right) . $$
\end{prop}
\begin{proof} 
After possibly extending scalars we may assume that $L$ splits $E$ and pick an extensions of each   $\sg\colon F\into L$ to a $\sg\colon  E\into L$. 
We then have  $W=\bigotimes_{\sg\colon F\into L}W_{\sg} $ with  $W_{\sg}=\eqref{Wwkl}$ for suitable integers $w, k_{\sg},l_{\sg}$, and analogously $\mu^{+}(t)= \prod_{\sg\colon F\into L}\mu_{\sg}^{+}$ with 
\beq \lb{xi+j}
\mu_{\sg}^{+}(t)=\sg (t)^{(k_{\sg}-2+l_{\sg})/2} \sg (t^{c})^{(k_{\sg}-2-l_{\sg})/2}, \qquad  \mu_{\sg}^{+}({\rm j})=(-1)^{(k_{\sg}-2-l_{\sg})/2}\cdot  {\rm j}_{v}^{k_{\sg}-2}.
\eeq
If $v$ splits in $E$, this simplifies to $\mu_{\sg}^{+}({\rm j}) = (-1)^{(k_{\sg}-2+l_{\sg})/2}$.

Moreover, 
 $\calL(V, 0)= \prod_{\sg}\calL(V_{\sg}, 0)$ with
$$\calL(V_{\sg}, 0)= {\pi^{-1}\Gamma_{\bC}({k_{\sg}+l_{\sg}\over 2})\Gamma_{\bC}({k_{\sg}-l_{\sg}\over 2}) \over \Gamma_{\bC}(k_{\sg}) \Gamma_{\R}(2)} 
= {2\over k_{\sg}-1} \cdot {k_{\sg}-2 \choose {k_{\sg}-2+l_{\sg}\over 2 }}^{-1}.
$$

Fix a $\sg \colon F\into L$ for the rest of this proof, work with $W_{\sg} $ only,  and we drop $\sg $  from the notation.  
We may  assume  that $f:=f_{1}$, $f^{\vee}:= f_{2}$ both equal $x^{k-2}$ in the models \eqref{W act poly}, and that $\vol(H', dt)=1$.  By the definitions above and Lemma \ref{wao alg}, we then need to prove that 
$$ Q(\tord f_, \tord  f^{\vee}):= 
{k-1\over 2 }\cdot {k_{}-2 \choose {k-2+l_{}\over 2 }} \cdot ( {\rm p}_{H'}( \tord f_{1}), {\rm p}_{H'} (\tord  f_{2}) ) ={k-1 \over 2} \cdot \mu^{+}({\rm j}).$$
Recall  in what follows that $\tord$ contains  the factor $c(W) =\eqref{cW}$.
\paragraph{Split case}
Suppose  first that $v$ splits in $E$.
Then $W^{H'} =  L x^{(k-2-l)/ 2}y^{(k-2+l)/ 2 }$, and
 $c(W)^{-1}\gamma_{0}^{\iota} f= c(W)^{-1} (x-y)^{k-2}$
 projects to 
$$
{}\ \qquad \tord f =(-1)^{ (k-2+l )/ 2}  c(W)^{-1} {k-2 \choose (k-2-l)/2} x^{(k-2-l)/ 2}y^{(k-2+l )/ 2 } \qquad \in W^{H'}.$$
 It follows that 
\beqq
Q^{}(\tord f_, \tord  f^{\vee})
&=   { k_{}-1\over 2} \cdot {k_{}-2 \choose ( {k_{}-2+l)/ 2 }}^{2}\cdot 
(-1)^{ {k-2+l\over 2}} \cdot c(W)^{-1}c(W^{\vee})^{-1}\\
&=     {k-1\over 2} \cdot  \mu^{+}({\rm j}) .
\eeqq

\paragraph{Nonsplit case}
Suppose now  that $v$ is nonsplit in $E$. Let $z:= x-\tht^{c,-1} y$, $\baar{z}:=x-\tht^{ -1}y$, then  $W^{H'}=L   z^{(k-2-l)/ 2}\baar{z}^{(k-2+l)/ 2 }$ and 
$$  \gamma_{0}^{\iota} f=c(W)^{-1} \Nm^{(w+k-2)/2} x^{k-2} =c(W)^{-1} \Nm^{(w+k-2)/2} {\rm j}^{2-k}  (\tht^{c}z-\tht \baar{z})^{k-2}$$
projects to
\beqq 
\tord f&= c(W^{\vee})^{-1}{k-2 \choose {k-2-l\over 2}}^{} (-1)^{(k-2+l)/  2}\cdot
  \Nm^{(w+k-2)/2} {\rm j}^{2-k} \theta^{c, (k-l-2)/2} \theta^{(k+l-2)/2}\cdot z^{(k-2-l)/ 2}\baar{z}^{(k-2+l)/ 2 }\\
&=
c(W^{\vee})^{-1}{k-2 \choose {k-2-l\over 2}}^{} 
(-1)^{(k-2+l)/  2}\cdot
 {\rm j}^{(-w-k+2)/2} \theta^{c, (k-l-2)/2} \theta^{(k+l-2)/2}\cdot 
 \twomat 1 1 {-\tht^{c,-1}} {-\tht^{-1}} . x^{(k-2-l)/ 2} {y}^{(k-2+l)/ 2 } .
\eeqq
By the invariance of the pairing, 
$$Q(\tord f, \tord f^{\vee}) = (-1)^{(k-2+l)/2}    c(W)^{-1}c(W^{\vee})^{-1} {k-2 \choose {k-2-l\over 2}}
(-{\rm j})^{2-k} \Nm^{k-2}  = (-1)^{k-2}{k-2 \choose {k-2-l\over 2}}^{-1}
\mu^{+}({\rm j})$$
so that again
\beqq
Q^{}(\tord f_, \tord  f^{\vee})= \dim W\cdot  \mu^{+}({\rm j}) .
\eeqq
\end{proof}

 \section{Correction}

Denote by $S_{p, \rm ns}$ the set of $p$-adic places of $F$ that are nonsplit in $E$.
 For a Hecke character $\chi$ of $E$, we consider the following condition:
$$(\star) \qquad  \text{for each  $v\in S_{p, \rm ns}$,  $v$ is inert in $E$ and $\chi_{v}$ is unramified}.$$
After a correction, the  article \cite{nonsplit} only proves the formula of Theorem B for  (ordinary, locally distinguished, potentially crystalline, non-exceptional) representations $\Pi=\pi\ot\chi$ of trivial weight  \emph{satisfying $(\star)$}. As a consequence, the main results of this paper are affected as follows.
\begin{itemize}
\item Theorem A  holds for $\pi_{0} \ot \chi$ under the extra assumption $(\star)$.
\item Theorem B, as well as Theorem $\text{B}^{\text{ord}}$ of \S~7.1.1, hold for   $\Pi=\pi \ot \chi$ under the extra assumption $(\star)$.
\item Theorem C is not affected.
\item Assume that \emph{no $p$-adic place of $F$ is ramified in $E$}. Define a completed homology module $M_{K^{p}}^{\star}$  as in \S~1.3.1, but with the limit being taken only over subgroups $K_{p}$ containing $N_{\G , 0} \ts \prod_{v \in S_{p, \rm ns}} \OO_{E,v}^{\ts}$. Define a space $\cE_{K^{p}}^{\rm ord, \star}$ as in \emph{loc. cit.} with $M_{K^{p}}$ replaced by $M_{K^{p}}^{\star}$. (This is simply the closure in $\cE^{\rm ord, \star}_{K^{p}}$ of the classical points of $\cE^{\rm ord}_{K^{p}}$ that correspond to representations satisfying $(\star)$.)  Define a $\text{Hida}^{\star}$ family for $(\G\ts \H)'$ to be an irreducible component of $\cE_{K^{p}}^{\rm ord, \star}$. 

Then Theorem D and Theorem E hold for \emph{$\text{Hida}^{\star}$ families} rather than $\text{Hida}$ families.
\item Theorem F is not affected. 
\item Theorem G is not affected, and the related alternative proof of the theorem of Greenberg--Stevens mentioned in Remark 7.3.4 remains valid.
\end{itemize}

The original proofs still apply. For the  main theorems (B and D), the proofs in \S~7.1 need to replace Lemma 7.1.4 with the statement that \emph{in a $\text{Hida}^{\star}$ family, the classical points of trivial weight that satisfy condition \textup{(p-crys)}  are dense}; the last paragraph in the proof of Lemma 7.1.4 applies verbatim to prove that  statement. (The condition (ram) is no longer relevant, since the corrected  version of \cite[Theorem B]{nonsplit} does not require it.)

\begin{enonce*}[remark]{Errata} In  footnote 12, the citation to \cite{nek-heeg}  should  not point  to \S~I.2 (I am grateful to M.H. Nicole for bringing this to my attention), but to Proposition II.2.4 (2). In the third-last line of the proof of Lemma 4.1.1, the symbol `$\cong$' should be replaced by `$\into$'.
\end{enonce*}

\backmatter
\addtocontents{toc}{\medskip}

\end{document}